\documentclass[11pt,oneside]{amsart}

\usepackage{fullpage}

\usepackage{amsthm,amsfonts,amssymb,amsmath,amsxtra}
\usepackage[all]{xy}
\SelectTips{cm}{}
\usepackage{tikz}
\usepackage{verbatim}
\usepackage{mathrsfs}

\RequirePackage{xspace}
\RequirePackage{etoolbox}
\RequirePackage{varwidth}
\RequirePackage{enumitem}
\RequirePackage{tensor}
\RequirePackage{mathtools}
\RequirePackage{longtable}
\RequirePackage{multirow}
\RequirePackage{makecell}
\RequirePackage{bigints}

\usepackage[backref=page]{hyperref} %

\newcommand{\sA}{\ensuremath{\mathscr{A}}\xspace}

\newcommand{\sH}{\ensuremath{\mathscr{H}}\xspace}

\newcommand{\sM}{\ensuremath{\mathscr{M}}\xspace}

\newcommand{\sS}{\ensuremath{\mathscr{S}}\xspace}

\newcommand{\sV}{\ensuremath{\mathscr{V}}\xspace}

\newcommand{\fkm}{\ensuremath{\mathfrak{m}}\xspace}

\newcommand{\BC}{\ensuremath{\mathbb{C}}\xspace}

\newcommand{\BE}{\ensuremath{\mathbb{E}}\xspace}

\newcommand{\BG}{\ensuremath{\mathbb{G}}\xspace}

\newcommand{\BL}{\ensuremath{\mathbb{L}}\xspace}

\newcommand{\BP}{\ensuremath{\mathbb{P}}\xspace}

\newcommand{\BV}{\ensuremath{\mathbb{V}}\xspace}
\newcommand{\BW}{\ensuremath{\mathbb{W}}\xspace}
\newcommand{\BX}{\ensuremath{\mathbb{X}}\xspace}
\newcommand{\BY}{\ensuremath{\mathbb{Y}}\xspace}
\newcommand{\BZ}{\ensuremath{\mathbb{Z}}\xspace}

\newcommand{\CD}{\ensuremath{\mathcal{D}}\xspace}
\newcommand{\CE}{\ensuremath{\mathcal{E}}\xspace}
\newcommand{\CF}{\ensuremath{\mathcal{F}}\xspace}

\newcommand{\CL}{\ensuremath{\mathcal{L}}\xspace}

\newcommand{\CN}{\ensuremath{\mathcal{N}}\xspace}
\newcommand{\CO}{\ensuremath{\mathcal{O}}\xspace}

\newcommand{\CU}{\ensuremath{\mathcal{U}}\xspace}
\newcommand{\CV}{\ensuremath{\mathcal{V}}\xspace}

\newcommand{\CX}{\ensuremath{\mathcal{X}}\xspace}
\newcommand{\CY}{\ensuremath{\mathcal{Y}}\xspace}
\newcommand{\CZ}{\ensuremath{\mathcal{Z}}\xspace}

\newcommand{\nat}{{\natural}}

\DeclareMathOperator{\Aut}{Aut}

\newcommand{\Ch}{{\mathrm{Ch}}}

\newcommand{\cl}{{\mathrm{cl}}}

\DeclareMathOperator{\diag}{diag}
\newcommand{\disc}{{\mathrm{disc}}}

\DeclareMathOperator{\End}{End}

\newcommand{\Fil}{\ensuremath{\mathrm{Fil}}\xspace}

\newcommand{\Fr}{\mathbf{F}}

\DeclareMathOperator{\Gal}{Gal}

\newcommand{\GL}{\mathrm{GL}}

\DeclareMathOperator{\Gr}{Gr}
\newcommand{\gr}{\mathrm{gr}}

\newcommand{\GU}{\mathrm{GU}}

\DeclareMathOperator{\Hom}{Hom}

\DeclareMathOperator{\im}{im}

\newcommand{\Ind}{{\mathrm{Ind}}}

\DeclareMathOperator{\Int}{\ensuremath{\mathrm{Int}}\xspace}

\DeclareMathOperator{\length}{length}
\DeclareMathOperator{\Lie}{Lie}

\newcommand{\mult}{{\mathrm{mult}}}

\DeclareMathOperator{\rank}{rank}

\renewcommand{\Re}{{\mathrm{Re}}}
\newcommand{\red}{\ensuremath{\mathrm{red}}\xspace}

\DeclareMathOperator{\Res}{Res}

\newcommand{\Sh}{\mathrm{Sh}}

\newcommand{\SL}{{\mathrm{SL}}}

\DeclareMathOperator{\Spec}{Spec}
\DeclareMathOperator{\Spf}{Spf}

\newcommand{\SO}{{\mathrm{SO}}}

\newcommand{\val}{{\mathrm{val}}}

\newcommand{\Ver}{{\mathrm{Vert}}}

\DeclareMathOperator{\supp}{supp}

\DeclareMathOperator{\tr}{tr}

\newcommand{\U}{\mathrm{U}}

\DeclareMathOperator{\vol}{vol}

\newcommand{\jiao}{\stackrel{\BL}{\cap}}

\newcommand{\wit}{\widetilde}
\newcommand{\wh}{\widehat}

\newcommand{\pair}[1]{\langle {#1} \rangle}

\newcommand{\ov}{\overline}
\newcommand{\incl}{\hookrightarrow}
\newcommand{\lra}{\longrightarrow}
\newcommand{\imp}{\Longrightarrow}

\newenvironment{altenumerate}
   {\begin{list}
      {(\theenumi) }
      {\usecounter{enumi}
       \setlength{\labelwidth}{0pt}
       \setlength{\labelsep}{0pt}
       \setlength{\leftmargin}{0pt}
       \setlength{\itemsep}{\the\smallskipamount}
       \renewcommand{\theenumi}{\roman{enumi}}
      }}
   {\end{list}}

\setitemize[0]{leftmargin=*,itemsep=\the\smallskipamount}
\setenumerate[0]{leftmargin=*,itemsep=\the\smallskipamount}

\renewcommand{\to}{%
   \ifbool{@display}{\longrightarrow}{\rightarrow}%
   }
\let\shortmapsto\mapsto
\renewcommand{\mapsto}{%
   \ifbool{@display}{\longmapsto}{\shortmapsto}%
   }
\newlength{\olen}
\newlength{\ulen}
\newlength{\xlen}
\newcommand{\xra}[2][]{%
   \ifbool{@display}%
      {\settowidth{\olen}{$\overset{#2}{\longrightarrow}$}%
       \settowidth{\ulen}{$\underset{#1}{\longrightarrow}$}%
       \settowidth{\xlen}{$\xrightarrow[#1]{#2}$}%
       \ifdimgreater{\olen}{\xlen}%
          {\underset{#1}{\overset{#2}{\longrightarrow}}}%
          {\ifdimgreater{\ulen}{\xlen}%
             {\underset{#1}{\overset{#2}{\longrightarrow}}}
             {\xrightarrow[#1]{#2}}}}%
      {\xrightarrow[#1]{#2}}
   }
\makeatother
\newcommand{\xyra}[2][]{%
   \settowidth{\xlen}{$\xrightarrow[#1]{#2}$}%
   \ifbool{@display}%
      {\settowidth{\olen}{$\overset{#2}{\longrightarrow}$}%
       \settowidth{\ulen}{$\underset{#1}{\longrightarrow}$}%
       \ifdimgreater{\olen}{\xlen}%
          {\mathrel{\xymatrix@M=.12ex@C=3.2ex{\ar[r]^-{#2}_-{#1} &}}}%
          {\ifdimgreater{\ulen}{\xlen}%
             {\mathrel{\xymatrix@M=.12ex@C=3.2ex{\ar[r]^-{#2}_-{#1} &}}}
             {\mathrel{\xymatrix@M=.12ex@C=\the\xlen{\ar[r]^-{#2}_-{#1} &}}}}}%
      {\mathrel{\xymatrix@M=.12ex@C=\the\xlen{\ar[r]^-{#2}_-{#1} &}}}%
   }
\makeatletter
\newcommand{\xla}[2][]{%
   \ifbool{@display}%
      {\settowidth{\olen}{$\overset{#2}{\longleftarrow}$}%
       \settowidth{\ulen}{$\underset{#1}{\longleftarrow}$}%
       \settowidth{\xlen}{$\xleftarrow[#1]{#2}$}%
       \ifdimgreater{\olen}{\xlen}%
          {\underset{#1}{\overset{#2}{\longleftarrow}}}%
          {\ifdimgreater{\ulen}{\xlen}%
             {\underset{#1}{\overset{#2}{\longleftarrow}}}
             {\xleftarrow[#1]{#2}}}}%
      {\xleftarrow[#1]{#2}}
   }
\newcommand{\isoarrow}{%
   \ifbool{@display}{\overset{\sim}{\longrightarrow}}{\xrightarrow\sim}%
   }
\renewcommand{\lra}{%
   \ifbool{@display}{\longleftrightarrow}{\leftrightarrow}%
   }

\newcommand{\barE}{{\bar{\mathbb{E}}}}
\newcommand{\Fb}{{\breve F}}
\newcommand{\OFb}{{O_{\breve F}}}
\newcommand{\Herm}{\mathrm{Herm}}
\newcommand{\rd}{\mathrm{d}}
\newcommand{\Rep}{\mathrm{Rep}}
\newcommand{\iden}{\langle1\rangle}

\newcommand{\Den}{\mathrm{Den}}
\newcommand{\pDen}{\partial\mathrm{Den}}
\newcommand{\wt}{\mathfrak{m}}

\newcommand{\ch}{\mathrm{ch}}
\newcommand{\Tate}{\mathrm{Tate}}
\newcommand{\barK}{\overline{K}_0(\CV(\Lambda))}
\newcommand{\barCh}{\overline{\Ch}}
\newcommand{\Intch}{c_{1,{\mathcal{V}(\Lambda)}}}
\newcommand{\Intc}{c_{{\mathcal{V}(\Lambda)}}}
\newcommand{\IntK}{\mathrm{K}_{\CV(\Lambda)}}

\newcommand{\KG}{K}
\newcommand{\tildeG}{}
\newcommand{\tildeGh}{}

\newcommand{\pEis}{\partial\mathrm{Eis}}
\newcommand{\Diff}{\mathrm{Diff}}
\newcommand{\Ei}{\mathrm{Ei}}
\newcommand{\wdeg}{\widehat\deg}
\newcommand{\wCh}{\widehat{\Ch}}
\newcommand{\sz}{\mathsf{z}}
\newcommand{\sx}{\mathsf{x}}
\newcommand{\sy}{\mathsf{y}}

\DeclareFontFamily{U}{matha}{\hyphenchar\font45}
\DeclareFontShape{U}{matha}{m}{n}{
      <5> <6> <7> <8> <9> <10> gen * matha
      <10.95> matha10 <12> <14.4> <17.28> <20.74> <24.88> matha12
      }{}
\DeclareSymbolFont{matha}{U}{matha}{m}{n}
\DeclareFontFamily{U}{mathx}{\hyphenchar\font45}
\DeclareFontShape{U}{mathx}{m}{n}{
      <5> <6> <7> <8> <9> <10>
      <10.95> <12> <14.4> <17.28> <20.74> <24.88>
      mathx10
      }{}
\DeclareSymbolFont{mathx}{U}{mathx}{m}{n}

\DeclareMathSymbol{\obot}         {2}{matha}{"6B}

\newtheorem{theorem}[subsubsection]{Theorem}
\newtheorem{proposition}[subsubsection]{Proposition}
\newtheorem{lemma}[subsubsection]{Lemma}
\newtheorem {conjecture}[subsubsection]{Conjecture}
\newtheorem{corollary}[subsubsection]{Corollary}

\theoremstyle{definition}
\newtheorem{definition}[subsubsection]{Definition}
\newtheorem{example}[subsubsection]{Example}

\newtheorem{remark}[subsubsection]{Remark}

\numberwithin{equation}{subsubsection}

\title{Kudla--Rapoport cycles and derivatives of local densities}

\author[Chao Li]{Chao Li}
\address{Columbia University, Department of Mathematics, 2990 Broadway,	New York, NY 10027, USA}
\email{chaoli@math.columbia.edu} 

\author[Wei Zhang]{Wei Zhang}
\address{Massachusetts Institute of Technology, Department of Mathematics, 77 Massachusetts Avenue, Cambridge, MA 02139, USA}
\email{weizhang@mit.edu}

\date{\today}

\begin{document}

\maketitle{}

\begin{abstract}
We prove the local Kudla--Rapoport conjecture, which is a precise identity between the arithmetic intersection numbers of special cycles on unitary Rapoport--Zink spaces and the derivatives of local representation densities of hermitian forms. As a first application, we prove the global Kudla--Rapoport conjecture, which relates the arithmetic intersection numbers of special cycles on unitary Shimura varieties and the central derivatives of the Fourier coefficients of incoherent Eisenstein series. Combining previous results of Liu and Garcia--Sankaran, we also prove cases of the arithmetic Siegel--Weil formula in any dimension.
\end{abstract}

\tableofcontents{}

\section{Introduction}

\subsection{Background}

The classical \emph{Siegel--Weil formula} (\cite{Siegel1951,Weil1965}) relates certain Siegel Eisenstein series with the arithmetic of quadratic forms, namely expressing special \emph{values} of these series as theta functions --- generating series of representation numbers of quadratic forms. Kudla (\cite{Kudla1997a, Kudla2004}) initiated an influential program to establish the \emph{arithmetic Siegel--Weil formula} relating certain Siegel Eisenstein series with objects in arithmetic geometry, which among others, aims to express the \emph{central derivative} of these series as the arithmetic analogue of theta functions --- generating series of arithmetic intersection numbers of $n$ special divisors on Shimura varieties associated to $\SO(n-1,2)$ or $\U(n-1,1)$. These special divisors include Heegner points on modular or Shimura curves appearing in the Gross--Zagier formula (\cite{Gross1986, Yuan2013}) ($n=2$), modular correspondence on the product of two modular curves in the Gross--Keating formula (\cite{Gross1993}) and Hirzebruch--Zagier cycles on Hilbert modular surfaces (\cite{Hirzebruch1976}) ($n=3$).

The arithmetic Siegel--Weil formula was established by Kudla, Rapoport and Yang (\cite{Kudla1999, Kudla1997a, Kudla2000, Kudla2006}) for $n=1,2$ (orthogonal case) in great generality. The \emph{archimedean} component of the formula was also known, due to Liu \cite{Liu2011} (unitary case), and Garcia--Sankaran \cite{Garcia2019} in full generality (cf. Bruinier--Yang \cite{Bruinier2018} for an alternative proof in the orthogonal case). However, the full formula (in particular, the nonarchimedean part) was widely open in higher dimension.

In the works \cite{Kudla2011,Kudla2014} Kudla--Rapoport made the nonarchimedean part of the conjectural formula more precise by defining arithmetic models of the special cycles (for any $n$ in the unitary case), now known as \emph{Kudla--Rapoport cycles}. They formulated the \emph{global Kudla--Rapoport conjecture} for the nonsingular part of the formula, and explained how it would follow (at least at an unramified place) from the \emph{local Kudla--Rapoport conjecture}, relating the derivatives of local representation densities of hermitian forms and arithmetic intersection numbers of Kudla--Rapoport cycles on unitary Rapoport--Zink spaces. They further proved the conjectures in the special case when the arithmetic intersection is \emph{non-degenerate} (i.e., of the expected dimension 0). Outside the non-degenerate case, the only known result was due to Terstiege \cite{Terstiege2013}, who proved the Kudla--Rapoport conjectures for $n=3$. Analogous results were known in the orthogonal case, see \cite{Gross1993,Kudla1999a,Kudla2000a, Bruinier2018} (non-degenerate case) and \cite{Terstiege2011} ($n=3$). 

The main result of this paper settles the local Kudla--Rapoport conjecture for any $n$ in the unitary case. As a first application, we will be able to deduce the global Kudla--Rapoport conjecture, and prove the first cases of the arithmetic Siegel--Weil formula in any higher dimension. In a companion paper \cite{LZ2020}, we will also use similar methods to prove analogous results in the orthogonal case.

As explained in \cite{Kudla1997a} and \cite{Liu2011}, the arithmetic Siegel--Weil formula (together with the doubling method) has important application to the \emph{arithmetic inner product formula}, relating the central derivative of the standard $L$-function of cuspidal automorphic representations on orthogonal or unitary groups to the height pairing of certain cycles on Shimura varieties constructed from arithmetic theta liftings. It can be viewed as a higher dimensional generalization of the Gross--Zagier formula, and an arithmetic analogue of the Rallis inner product formula. Further applications to the arithmetic inner product formula are investigated in \cite{LL2020}. We also mention that the local Kudla--Rapoport conjecture has application to the so-called \emph{unitary arithmetic fundamental lemma} for cycles on unitary Shimura varieties arising from the embedding $\U(n)\times \U(n)\hookrightarrow \U(2n)$.

\subsection{The local Kudla--Rapoport conjecture} Let $p$ be an odd prime. Let $F_0$ be a finite extension of $\mathbb{Q}_p$ with residue field $k=\mathbb{F}_q$ and a uniformizer $\varpi$. Let $F$ be an unramified quadratic extension of $F_0$. Let $\breve F$ be the completion of the maximal unramified extension of $F$. For any integer $n\geq 1$, the \emph{unitary Rapoport--Zink space} $\mathcal{N}=\mathcal{N}_n$ (\S\ref{sec:rapoport-zink-spaces}) is the formal scheme over $S=\Spf O_{\breve F}$ parameterizing hermitian formal $O_F$-modules of signature $(1,n-1)$ within the supersingular quasi-isogeny class. Let $\mathbb{E}$ and $\mathbb{X}$ be the framing hermitian $O_F$-module of signature $(1,0)$ and $(1,n-1)$ over $\bar k$. The space of \emph{quasi-homomorphisms} $\mathbb{V}=\mathbb{V}_n\coloneqq \Hom_{O_F}^\circ(\barE, \mathbb{X})$ carries a natural $F/F_0$-hermitian form, which makes $\mathbb{V}$ the unique (up to isomorphism) nondegenerate nonsplit (see \S\ref{ss:notation}) $F/F_0$-hermitian space of dimension $n$ (\S\ref{sec:herm-space-mathbbv}). For any subset $L\subseteq \mathbb{V}$, the local \emph{Kudla--Rapoport cycle} $\mathcal{Z}(L)$ (\S\ref{sec:kudla-rapop-cycl}) is a closed formal subscheme of $\mathcal{N}$, over which each quasi-homomorphism $x\in L$ deforms to homomorphisms.

Let $L\subseteq \mathbb{V}$ be an $O_F$-lattice (of full rank $n$). We now associate to $L$ two integers: the \emph{arithmetic intersection number} $\Int(L)$ and the \emph{derivative of the local density} $\pDen(L)$.

Let $x_1,\ldots, x_n$ be an $O_F$-basis of $L$. Define the \emph{arithmetic intersection number}
\begin{equation}
  \label{eq:intL}
  \Int(L)\coloneqq \chi(\mathcal{N},\mathcal{O}_{\mathcal{Z}(x_1)} \otimes^\mathbb{L}\cdots \otimes^\mathbb{L}\mathcal{O}_{\mathcal{Z}(x_n)} ),
\end{equation}
 where $\mathcal{O}_{\mathcal{Z}(x_i)}$ denotes the structure sheaf of the Kudla--Rapoport divisor $\mathcal{Z}(x_i)$, $\otimes^\mathbb{L}$ denotes the derived tensor product of coherent sheaves on $\mathcal{N}$, and $\chi$ denotes the Euler--Poincar\'e characteristic (\S\ref{sec:arithm-inters-numb}). By \cite[Proposition 3.2]{Terstiege2013} (or \cite[Corollary D]{Howard2018}), we know that $\Int(L)$ is independent of the choice of the basis $x_1,\ldots, x_n$ and hence is a well-defined invariant of $L$ itself.

For $M$ another hermitian $O_F$-lattice (of arbitrary rank), define $\Rep_{M, L}$ to be the \emph{scheme of integral representations of $M$ by $L$}, an $O_{F_0}$-scheme such that for any $O_{F_0}$-algebra $R$, $\Rep_{M,L}(R)=\Herm(L \otimes_{O_{F_0}}R, M \otimes_{O_{F_0}}R)$, where $\Herm$ denotes the set of hermitian module homomorphisms. The \emph{local density} of integral representations of $M$ by $L$ is defined to be $$\Den(M,L)\coloneqq \lim_{N\rightarrow +\infty}\frac{\# \Rep_{M,L}(O_{F_0}/\varpi^N)}{q^{N\cdot\dim (\Rep_{M,L})_{F_0}}}.$$ Let $\iden^k$ be the self-dual hermitian $O_F$-lattice of rank $k$ with hermitian form given by the identity matrix $\mathbf{1}_k$. Then $\Den(\iden^k, L)$ is a polynomial in $(-q)^{-k}$ with $\mathbb{Q}$-coefficients. Define the (normalized) \emph{local Siegel series} of $L$ to be the polynomial $\Den(X,L)\in \mathbb{Z}[X]$ (Theorem \ref{thm: Den(X)}) such that $$\Den((-q)^{-k},L)=\frac{\Den(\iden^{n+k}, L)}{\Den(\iden^{n+k}, \iden^n)}.$$ It satisfies a functional equation relating $X\leftrightarrow\frac{1}{X}$,
\begin{equation}
  \label{eq:FEmain}
  \Den(X,L)=(-X)^{\val(L)}\cdot \Den\left(\frac{1}{X},L\right).
\end{equation}
Since $\mathbb{V}$ is nonsplit, we know that $\val(L)$ is odd (see \S\ref{ss:notation} for the notation) and so the value $\Den(1,L)=0$. We thus consider the \emph{derivative of the local density} $$\pDen(L)\coloneqq -\frac{\rd}{\rd X}\bigg|_{X=1}\Den(X,L).$$

Our main theorem in Part \ref{part:local-kudla-rapoport} is a proof of the local Kudla--Rapoport conjecture \cite[Conjecture 1.3]{Kudla2011}, which asserts an exact identity between the two integers just defined.
\begin{theorem}[Theorem \ref{thm: main}, local Kudla--Rapoport conjecture]\label{thm:intro1}
  Let $L\subseteq \mathbb{V}$ be an $O_F$-lattice of full rank $n$. Then $$\Int(L)=\pDen(L).$$
\end{theorem}
We refer to $\Int(L)$ as the \emph{geometric side} of the identity (related to the geometry of Rapoport--Zink spaces and Shimura varieties) and $\pDen(L)$ the \emph{analytic side} (related to the derivatives of Eisenstein series and $L$-functions).

Our main theorem  in Part \ref{part:local-kudla-rapoport-1} proves a variant of the local Kudla--Rapoport conjecture in the presence of a minimal nontrivial level structure, given by the stabilizer of an almost self-dual lattice (see \S\ref{ss:notation}) in a nonsplit $F/F_0$-hermitian space. The relevant Rapoport--Zink space on the geometric side is no longer formally smooth.  See Theorems \ref{thm: main2} and \ref{thm: main2'} for the precise statement.

\subsection{The arithmetic Siegel--Weil formula} Next let us describe some global applications of our local theorems. We now switch to global notations. Let $F$ be a CM number field, with $F_0$ its totally real subfield of index 2. Fix  a CM type $\Phi\subseteq \Hom(F, \overline{\mathbb{Q}})$ of $F$. Fix an embedding $\overline{\mathbb{Q}}\hookrightarrow \mathbb{C}$ and identify the CM type $\Phi$ with the set of archimedean places of $F$, and also with the set of archimedean places of $F_0$. Let $V$ be an $F/F_0$-hermitian space of dimension $n$ and $G=\Res_{F_0/\mathbb{Q}}\U(V)$. Assume the signatures of $V$ are $\{(n-1,1)_{\phi_0},(n,0)_{\phi\in\Phi-\{\phi_0\}}\}$ for some distinguished element $\phi_0\in \Phi$. Define a torus $Z^\mathbb{Q}=\{z\in \Res_{F/\mathbb{Q}}\mathbb{G}_m: \mathrm{Nm}_{F/F_0}(z)\in \mathbb{G}_m\}$. Associated to $\wit G\coloneqq Z^\mathbb{Q}\times G$ there is a natural Shimura datum $(\wit G,\{h_{\wit G}\})$ of PEL type (\S\ref{sec:shimura-varieties}). Let $K=K_{Z^\mathbb{Q}}\times K_G\subseteq \wit G(\mathbb{A}_f)$ be a compact open subgroup. Then the associated Shimura variety $\Sh_{\KG}=\Sh_{\KG}(\wit G,\{h_{\wit G}\})$ is of dimension $n-1$ and has a canonical model over its reflex field $E$.

Assume $K_{Z^\mathbb{Q}}\subseteq Z^\mathbb{Q}(\mathbb{A}_f)$ is the unique maximal open compact subgroup, and $K_{G,v}\subseteq \U(V)(F_{0,v})$ ($v$ a place of $F_0$) is given by
\begin{itemize}
\item the stabilizer of a self-dual or almost self-dual lattice $\Lambda_v\subseteq V_v$ if $v$ is inert in $F$,
\item the stabilizer of a self-dual lattice $\Lambda_v\subseteq V_v$ if $v$ is ramified in $F$,
\item a principal congruence subgroup if $v$ is split in $F$.
\end{itemize}
Then we construct a global regular integral model $\mathcal{M}_K$ of $\Sh_K$ over $O_E$ following \cite{Rapoport2017} (\S\ref{sec:glob-integr-models}). When $F_0=\mathbb{Q}$, we have $E=F$ and the integral model $\mathcal{M}_K$ recovers that in \cite{Bruinier2017} when $K_G$ is the stabilizer of a global self-dual lattice, which is closely related to that in \cite{Kudla2014}.

Let $\mathbb{V}$ be the \emph{incoherent} $\mathbb{A}_{F}/\mathbb{A}_{F_0}$-hermitian space nearby $V$ such that $\mathbb{V}$ is totally positive definite and $\mathbb{V}_v \cong V_v$ for all finite places $v$. Let $\varphi_K\in \sS(\mathbb{V}^n_f)$ be a $K$-invariant (where $K$ acts on $\mathbb{V}_f$ via the second factor $K_G$) factorizable Schwartz function such that $\varphi_{K,v}=\mathbf{1}_{(\Lambda_v)^n}$ at all $v$ inert in $F$. Let $T\in \Herm_n(F)$ be a nonsingular hermitian matrix of size $n$. Associated to $(T,\varphi_K)$ we construct arithmetic cycles $\mathcal{Z}(T,\varphi_K)$ over $\mathcal{M}_K$ (\S\ref{sec:glob-kudla-rapop-1}) generalizing the \emph{Kudla--Rapoport cycles} $\mathcal{Z}(T)$ in  \cite{Kudla2014}. Analogous to the local situation (\ref{eq:intL}), we may define its local arithmetic intersection numbers $\Int_{T,v}(\varphi_K)$ at finite places $v$ (\S\ref{sec:local-arithm-inters}). Using the star product of Kudla's Green functions, we also define its local arithmetic intersection number  $\Int_{T,v}(\sy,\varphi_K)$ at infinite places (\S\ref{sec:local-arithm-sieg}), which depends on a parameter $\sy\in \Herm_n(F_\infty)_{>0}$  where  $F_\infty=F \otimes_{F_0}\mathbb{R}^\Phi\cong \mathbb{C}^\Phi$. Combining all the local arithmetic numbers together, define the \emph{global arithmetic intersection number}, or the \emph{arithmetic degree} of the Kudla--Rapoport cycle $\mathcal{Z}(T,\varphi_K)$, $$\wdeg_T(\sy,\varphi_K)\coloneqq \sum_{v\nmid\infty}\Int_{T,v}(\varphi_K)+\sum_{v\mid \infty}\Int_{T,v}(\sy,\varphi_K).$$ It is closely related to the usual arithmetic degree on the Gillet--Soul\'e arithmetic Chow group $\wCh^n_\mathbb{C}(\mathcal{M}_K)$ (\S\ref{sec:arithm-degr-kudla}).

On the other hand, associated to $\varphi=\varphi_K \otimes \varphi_\infty\in\sS(\mathbb{V}^n)$, where $\varphi_{\infty}$ is the Gaussian function, there is a classical \emph{incoherent Eisenstein series} $E(\sz, s,\varphi_K)$ (\S\ref{sec:incoh-eisenst-seri-1}) on the hermitian upper half space $$\mathbb{H}_n=\{\sz=\sx+i\sy:\ \sx\in\Herm_n(F_\infty),\ \sy\in\Herm_n(F_\infty)_{>0}\}.$$ This is essentially the Siegel Eisenstein series associated to a standard Siegel--Weil section of the degenerate principal series (\S\ref{sec:sieg-eisenst-seri}). The Eisenstein series here has a meromorphic continuation and a functional equation relating $s\leftrightarrow -s$.  The central value $E(\sz, 0, \varphi_K)=0$ by the incoherence. We thus consider its \emph{central derivative} $$\pEis(\sz, \varphi_K)\coloneqq \frac{\rd}{\rd s}\bigg|_{s=0}E(\sz, s,\varphi_K).$$ Associated to an additive character $\psi: \mathbb{A}_{F_0}/F_0\rightarrow \mathbb{C}^\times$, it has a decomposition into the central derivative of the Fourier coefficients $$\pEis(\sz,\varphi_K)=\sum_{T\in \Herm_n(F)}\pEis_T(\sz,\varphi_K).$$

Now we can state our first application to the global Kudla--Rapoport conjecture  \cite[Conjecture 11.10]{Kudla2014}, which asserts an identity between the arithmetic degree of Kudla--Rapoport cycles and the derivative of nonsingular Fourier coefficients of the incoherent Eisenstein series. 

\begin{theorem}[Theorem \ref{thm:totallypositive}, global Kudla--Rapoport conjecture]
Let $\Diff(T, \mathbb{V})$ be the set of places $v$ such that $\mathbb{V}_{v}$ does not represent $T$.  Let $T\in\Herm_n(F)$ be nonsingular such that $\Diff(T,\mathbb{V})=\{v\}$ where $v$ is inert in $F$ and not above 2. Then $$\wdeg_T(\sy, \varphi_K)q^T=c_K\cdot \pEis_T(\sz,\varphi_K),$$ where $q^T\coloneqq\psi_\infty(\tr T\sz)$, $c_K=\frac{(-1)^n}{\vol(K)}$ is a nonzero constant independent of $T$ and $\varphi_K$, and $\vol(K)$ is the volume of $K$ under a suitable Haar measure on $\wit G(\mathbb{A}_f)$.
\end{theorem}

We form the generating series of arithmetic degrees  $$\wdeg(\sz, \varphi_K)\coloneqq \sum_{T\in\Herm_n(F)\atop \det T\ne0}\wdeg_T(\sy,\varphi_K) q^T.$$ Now we can state our second application to the arithmetic Siegel--Weil formula, which relates this generating series to the central derivative of the incoherent Eisenstein series.

\begin{theorem}[Theorem \ref{sec:arithm-sieg-weil}, arithmetic Siegel--Weil formula]
  Assume that $F/F_0$ is unramified at all finite places and split at all places above 2. Further assume that $\varphi_K$ is nonsingular (\S\ref{sec:incoh-eisenst-seri}) at two places split in $F$. Then $$\wdeg(\sz, \varphi_K)=c_K\cdot \pEis(\sz,\varphi_K).$$ In particular, $\wdeg(\sz, \varphi_K)$ is a nonholomorphic hermitian modular form of genus $n$.
\end{theorem}

\begin{remark}
  The unramifiedness assumption on $F/F_0$ forces $F_0\ne \mathbb{Q}$. To treat the general case, one needs to formulate and prove an analogue of Theorem \ref{thm:intro1} when the local extension $F/F_0$ is ramified.  We remark that at a ramified place, in addition to the \emph{Kr\"amer model} with level given by the stabilizer of a self-dual lattice, we may also consider the case of \emph{exotic good reduction} with level associated to an (almost) $\varpi$-modular lattice. In a future work we hope to extend our methods to cover these cases, which in particular requires an extension of the local density formula of Cho--Yamauchi \cite{CY} to the ramified case.
\end{remark}

\begin{remark}
The nonsingularity assumption on $\varphi_K$ allows us to kill all the singular terms on the analytic side. Such $\varphi_K$ exists  for a suitable choice of $K$ since we allow arbitrary Drinfeld levels at split places.
\end{remark} 

\subsection{Strategy of the proof of the main Theorem \ref{thm:intro1}} 
The previously known special cases of the local Kudla--Rapoport conjecture (\cite{Kudla2011,Terstiege2013}) are proved via explicit computation of both the geometric and analytic sides. Explicit computation seems infeasible for the general case. Our proof instead proceeds via induction on $n$ using the \emph{uncertainty principle}.

More precisely, for a fixed $O_F$-lattice $L^\flat\subseteq \mathbb{V}=\mathbb{V}_n$ of rank $n-1$ (we assume $L_F^\flat$ is non-degenerate throughout the paper), consider functions on $x\in \mathbb{V}\setminus L^\flat_F$, $$\Int_{L^\flat}(x)\coloneqq \Int(L^\flat+\langle x\rangle),\quad \pDen_{L^\flat}(x)\coloneqq \pDen(L^\flat+\langle x\rangle).$$ Then it remains to show the equality of the two functions $\Int_{L^\flat}=\pDen_{L^\flat}$. Both functions vanish when $x$ is non-integral, i.e., $\val(x)<0$. Here $\val(x)$ denotes the valuation of the norm of $x$. By utilizing the inductive structure of the Rapoport--Zink spaces and local densities, it is not hard to see that if $x\perp L^\flat$ with $\val(x)=0$, then $$\Int_{L^\flat}(x)=\Int(L^\flat),\quad \pDen_{L^\flat}(x)=\pDen(L^\flat)$$ for the lattice $L^\flat\subseteq \mathbb{V}_{n-1}\cong \langle x\rangle^\perp_F$ of full rank $n-1$. By induction on $n$, we have $\Int(L^\flat)=\pDen(L^\flat)$, and thus the difference function $\phi=\Int_{L^\flat}-\pDen_{L^\flat}$ vanishes on $\{x \in\mathbb{V}: x\perp L^\flat, \val(x)\le0\}$. We would like to deduce that $\phi$ indeed vanishes identically.

The uncertainty principle (Proposition \ref{uncert}), which is a simple consequence of the Schr\"odinger model of the local Weil representation of $\SL_2$, asserts that if $\phi\in C_c^\infty(\mathbb{V})$ satisfies that both $\phi$ and its Fourier transform $\hat\phi$ vanish on $\{x\in \mathbb{V}: \val(x)\le0\}$, then $\phi=0$. In other words, $\phi,\hat\phi$ cannot simultaneously have ``small support'' unless $\phi=0$. We can then finish the proof by applying the uncertainty principle to $\phi=\Int_{L^\flat}-\pDen_{L^\flat}$, if we can show that both $\Int_{L^\flat}$ and $\pDen_{L^\flat}$ are invariant under the Fourier transform (up to the Weil constant $\gamma_{\mathbb{V}}=-1$). However, both functions have singularities along the hyperplane $L^\flat_F\subseteq \mathbb{V}$, which cause trouble in computing their Fourier transforms or even in showing that $\phi\in C_c^\infty(\mathbb{V})$.

To overcome this difficulty, we isolate the singularities by decomposing $$\Int_{L^\flat}=\Int_{L^\flat,\sH}+\Int_{L^\flat,\sV},\quad \pDen_{L^\flat}=\pDen_{L^\flat,\sH}+\pDen_{L^\flat,\sV}$$ into ``horizontal'' and ``vertical'' parts. Here on the geometric side $\Int_{L^\flat,\sH}$ is the contribution from the horizontal part of the Kudla--Rapoport cycles, which we determine explicitly in terms of quasi-canonical lifting cycles (Theorem \ref{thm:horizontal}). On the analytic side we define $\pDen_{L^\flat,\sH}$  to match with $\Int_{L^\flat,\sH}$. We show the horizontal parts have logarithmic singularity along $L^\flat_F$, and vertical parts are indeed in $C_c^\infty(\mathbb{V})$ (Corollary \ref{cor:LC int}, Proposition \ref{prop:LC Den}). We can then finish the proof if we can determine the Fourier transforms as
\begin{equation}
  \label{eq:FCmain}
\widehat{\Int}_{L^\flat,\sV}=- \Int_{L^\flat,\sV},\quad \widehat{\pDen}_{L^\flat,\sV}=-\pDen_{L^\flat,\sV}.  
\end{equation}

On the geometric side we show (\ref{eq:FCmain}) (Corollary \ref{cor:FT int}) by reducing to the case of intersection with Deligne--Lusztig curves. This reduction requires the Bruhat--Tits stratification of $\mathcal{N}^\mathrm{red}$ into certain Deligne--Lusztig varieties (\S\ref{sec:bruh-tits-strat}, due to Vollaard--Wedhorn \cite{Vollaard2011}) and the Tate conjecture for these Deligne--Lusztig varieties (Theorem \ref{thm:tate}, which we reduce to a cohomological computation of Lusztig \cite{Lusztig1976/77}).

On the analytic side we are only able to show (\ref{eq:FCmain}) (Theorem \ref{thm: pDen=0}) directly when $x\perp L^\flat$ and $\val(x)<0$. The key ingredient is a local density formula (Theorem \ref{thm: Den(X)}) due to Cho--Yamauchi \cite{CY} together with the functional equation (\ref{eq:FEmain}). We then deduce the general case by performing another induction on $\val(L^\flat)$ (\S\ref{ss:proof}).

We remark the extra symmetry (\ref{eq:FCmain}) under the Fourier transform can be thought of as a \emph{local modularity}, in analogy with the global modularity of arithmetic generating series (such as in \cite{Bruinier2017}) encoding an extra global $\SL_2$-symmetry. The latter global modularity plays a crucial role in the second author's recent proof \cite{Zhang2019}  of the \emph{arithmetic fundamental lemma}. In contrast to \cite{Zhang2019}, our proof of the local Kudla--Rapoport conjecture does not involve global arguments, thanks to a more precise understanding of the horizontal part of Kudla--Rapoport cycles. In other similar (non-arithmetic) situations, induction arguments involving Fourier transforms and the uncertainty principle are not unfamiliar:  here we only mention the second author's proof \cite{Zhang2014} of the Jacquet--Rallis smooth transfer conjecture, and more recently Beuzart-Plessis' new proof \cite{BP} of the Jacquet--Rallis fundamental lemma.

\subsection{The structure of the paper} In Part \ref{part:local-kudla-rapoport}, we review necessary background  on the local Kudla--Rapoport conjecture and  prove the main Theorem  \ref{thm:intro1}. In Part \ref{part:local-kudla-rapoport-1}, we prove a variant of the local Kudla--Rapoport conjecture in the almost self-dual case (Theorems \ref{thm: main2}, \ref{thm: main2'}), by relating both the geometric and analytic sides in the almost self-dual to the self-dual case (but in one dimension higher). In Part \ref{part:semi-global-global}, we review semi-global and global integral models of Shimura varieties and Kudla--Rapoport cycles, and incoherent Eisenstein series. We then apply the local results in Parts \ref{part:local-kudla-rapoport} and  \ref{part:local-kudla-rapoport-1} to prove the local arithmetic Siegel--Weil formula (Theorem \ref{thm:semi-global-identity}), the global Kudla--Rapoport conjecture (Theorem \ref{thm:totallypositive}), and cases of the arithmetic Siegel--Weil formula (Theorem \ref{sec:arithm-sieg-weil}).

\subsection{Acknowledgments} The authors would like to thank X. He, B. Howard, S. Kudla, Y. Liu, G. Lusztig, M. Rapoport, L. Xiao, Z. Yun, X. Zhu and Y. Zhu for useful conversations and/or comments. The authors are also grateful to referees for careful reading and helpful suggestions.
C.~L.~was partially supported by an AMS travel grant for ICM 2018 and the NSF grant DMS-1802269.  W. Z. was partially supported by the NSF grant DMS-1838118 and 1901642. The authors would like to thank the Morningside Center of Mathematics, where part of this work is done, for its hospitality.

\subsection{Notation on hermitian lattices}\label{ss:notation}
Let $p$ be a prime. In the local parts of the paper (Part \ref{part:local-kudla-rapoport} and \ref{part:local-kudla-rapoport-1}), we let $F_0$ be a non-archimedean local field of residue characteristic $p$, with ring of integers $O_{F_0}$, residue field $k=\mathbb{F}_q$ of size $q$, and uniformizer $\varpi$. Unless otherwise specified, we let $F$ be a quadratic extension of $F_0$, with ring of integers $O_{F}$ and residue field $k_F$. Let $\sigma$ be the nontrivial automorphism of $F/F_0$. Let $\Fb$ be the completion of the maximal unramified extension of $F$, and $\OFb$ its ring of integers.

Unless otherwise specified, we assume that $F/F_0$ is unramified (with an exception of \S \ref{sec:uncert-princ} on the uncertainty principle). We further assume that $F_0$ has characteristic 0 and residue characteristic $p>2$ (with exceptions of  \S \ref{sec:local-densities}, \S \ref{s:FT ana}, \S \ref{sec:uncert-princ}, \S \ref{sec:local-density-an}, which concern only the analytic side).

Let $\mathbb{V}$ be an $F/F_0$-hermitian space with hermitian form $(\ ,\ )$. We write $\val(x):=\val((x,x))$ for any $x\in \mathbb{V}$, where $\val$ is the valuation on $F_0$. Recall that the $F/F_0$-hermitian spaces are classified up to isomorphism by it dimension $n$ and its discriminant $\disc(\mathbb{V})=(-1)^{{n \choose 2}}\det(\mathbb{V})\in F_0^\times/\mathrm{Nm}_{F/F_0}F^\times$ (\cite[Theorem 3.1]{Jacobowitz1962}). We say $\mathbb{V}$ is \emph{split} if $\disc(\mathbb{V})=1\in F_0^\times/\mathrm{Nm}_{F/F_0}F^\times$, and \emph{nonsplit} otherwise.

Let $L\subseteq \mathbb{V}$ be an $O_F$-lattice of rank $n$. We denote by $L^\vee$ its dual lattice under $(\ ,\ )$. We say that $L$ is \emph{integral} if $L\subseteq L^\vee$. If $L$ is integral, define its \emph{fundamental invariants} to be the unique sequence of integers $(a_1,\ldots, a_n)$ such that $0\leq a_1\le \cdots\le a_n$, and $L^\vee/L\simeq \oplus_{i=1}^n O_F/\varpi^{a_i}$ as $O_F$-modules; define its \emph{valuation} to be $\val(L)\coloneqq \sum_{i=1}^na_i$; and define its \emph{type}, denoted by $t(L)$, to be the number of nonzero terms in its fundamental invariants $(a_1,\ldots, a_n)$.

We say $L$ is \emph{minuscule} or a \emph{vertex lattice} if it is integral and $L^\vee\subseteq \varpi^{-1}L$. Note that $L$ is a vertex lattice of type $t$ if and only if it has fundamental invariants $(0^{(n-t)},1^{(t)})$, if and only if $L\subseteq^t L^\vee\subseteq \varpi^{-1}L$, where $\subseteq^t$ indicates that the $O_F$-colength is equal to $t$. The set of vertex lattices of type $t$ in $\mathbb{V}$ is denoted by $\Ver^t=\Ver^t(\mathbb{V})$. We say $L$ is \emph{self-dual} if $L=L^\vee$, or equivalently $L$ is a vertex lattice of type 0. We say $L$ is \emph{almost self-dual} if $L$ is a vertex lattice of type 1. When $F/F_0$ is unramified, if $\mathbb{V}$ is split then $\val(L)$ is even and $\mathbb{V}$ contains a self-dual lattice; if $\mathbb{V}$ is nonsplit then $\val(L)$ is odd and $\mathbb{V}$ contains an almost self-dual lattice.

We denote by $L^\flat\subseteq \mathbb{V}$ an $O_F$-lattice of rank $n-1$, and we always assume that $L_F^\flat$ is non-degenerate. Here we use the subscript $(-)_F$ to stand for the base change to $F$, so $L_F^\flat=L^\flat\otimes_{O_F}F$.

Fix an unramified additive character $\psi: F_0\to \BC^\times$. Here ``unramifiedness'' means that the conductor of $\psi$ (i.e., the largest fractional ideal in $F_0$ on which $\psi$ is trivial) is $O_{F_0}$. For an integrable function $f$ on $\mathbb{V}$, we define its Fourier transform $\wh f$ to be $$\wh f(x):=\int_\mathbb{V} f(y)\psi (\tr_{F/F_0}(x,y))\rd y,\quad x\in \mathbb{V}.$$ We normalize the Haar measure  on $\BV$  to be self-dual, so $\hat{\hat{f}}(x)=f(-x)$. For an $O_F$-lattice $L\subseteq \mathbb{V}$ of rank $n$, we have (under the assumption that  $F/F_0$ is unramified)
$$
\wh{\bf 1}_{L}=\vol(L){\bf 1}_{L^\vee},\quad\text{and} \quad\vol(L)=[L^\vee:L]^{-1/2}=q^{-\val(L)}.
$$Note that $\val(L)$ can be defined for any lattice $L$ (not necessarily integral) so that the above equality for $\vol(L)$ holds.

\subsection{Notation on formal schemes}\label{sec:notat-form-schem}

Let $X$ be a formal scheme.  Denote by $X^{\red}$ the underlying reduced scheme. For closed formal subschemes $\CZ_1,\cdots,\CZ_m$ of $X$, denote by $\cup_{i=1}^m\mathcal{Z}_i$ the formal scheme-theoretic union, i.e., the closed formal subscheme with ideal sheaf $\cap_{i=1}^m\mathcal{I}_{\mathcal{Z}_i}$, where $\mathcal{I}_{\mathcal{Z}_i}$ is the ideal sheaf of $\mathcal{Z}_i$.  A closed formal subscheme on $X$ is called a Cartier divisor if it is defined by an {\em invertible} ideal sheaf.  

Let  $X$ be a formal scheme over $\Spf \OFb$. Then $X$ defines a functor on the category of $\Spf \OFb$-schemes (i.e. $\OFb$-schemes on which $\varpi$ is locally nilpotent). For a noetherian $\varpi$-adically complete $\OFb$-algebra $R$, write $X(R):=\Hom_{\Spf {\OFb}}(\Spf R, X)=\varprojlim_nX(\Spec R/\varpi^n)$.

 When $X$ is noetherian, denote by $K_0^Y(X)$ the Grothendieck group of finite complexes of coherent locally free $\mathcal{O}_X$-modules acyclic outside $Y$. As defined in \cite[(B.1), (B.2)]{Zhang2019},  denote by $\mathrm{F}^i K_0^Y(X)$ be the (descending) codimension filtration on $K_0^Y(X)$, and denote by $\Gr^{i}K_0^Y(X)$ its $i$-th graded piece. As in \cite[Appendix B]{Zhang2019}, the definition of $K_0^Y(X)$, $\mathrm{F}^i K_0^Y(X)$ and $\Gr^{i}K_0^Y(X)$ can be extended to locally noetherian formal schemes $X$ such that $X$ is an increasing union of open noetherian formal subschemes. Similarly, we let $K_0'(X)$ denote the Grothendieck group of coherent sheaves of $\mathcal{O}_X$-modules. Now let $X$ be regular. Then there is a natural isomorphism $K_0^Y(X)\simeq K_0'(Y)$. For closed formal subschemes $\CZ_1,\cdots,\CZ_m$ of $X$, denote by $\CZ_1\jiao_X\cdots\jiao_X\CZ_m$ (or simply $\CZ_1\jiao\cdots\jiao\CZ_m$) the derived tensor product $\CO_{\CZ_1}\otimes^\BL_{\CO_X}\cdots \otimes^\BL_{\CO_X}\CO_{\CZ_m}$, viewed as an element in $K_0^{\CZ_1\cap\cdots\cap \CZ_m}(X)$.

For $\mathcal{F}$ a finite complex of coherent $\CO_X$-modules, we define its Euler--Poincar\'e characteristic $$\chi(X, \mathcal{F}):=\sum_{i,j}(-1)^{i+j}\length_{\OFb}H^i(X,H_j(\mathcal{F}))$$
if the lengths are all finite. Assume that $X$ is regular with pure dimension $n$. If $\mathcal{F}_i\in \mathrm{F}^{r_i}K_0^{\mathcal{Z}_i}(X)$ with $\sum_i r_i\geq n$, then by \cite[(B.3)]{Zhang2019} we know that $\chi(X, \bigotimes_i^\mathbb{L}\mathcal{F}_i)$ depends only on the image of $\mathcal{F}_i$ in $\Gr^{r_i}K_0^{\mathcal{Z}_i}(X)$.  In fact, we will only need this assertion when $X$ is a scheme (cf. Remark \ref{rem:avoid B3}). When $X$ is a formal scheme, the assertion holds trivially when one of the $r_i$ is $\dim X$; this special case will be used repeatedly. 

For a morphism $\pi: X\to Y$ between two formal schemes and a closed formal subscheme $\CZ\incl \CY$, let $\pi^{-1}(\CZ)\incl X$ be the preimage of $\CZ$.
Let $\pi^\ast: K_0^{\CZ}(Y)\to K_0^{\pi^{-1}(\CZ)}(X)$ be the homomorphism induced by pulling back locally free sheaves. If $\pi$ is proper (i.e., the induced morphism $X^{\red}\to Y^{\red}$ on the reduced schemes is proper), there is a direct image homomorphism $\pi_\ast:  K_0'(X)\to K_0'(Y)$ sending (the class of) a coherent $\CO_X$-module $\CF$ to $\sum_{i\geq0} (-1)^i{\rm R}^i\pi_\ast \CF$. 



\subsection{Reminder on hermitian spaces over finite fields}

 Let $V$ be a non-degenerate $k_F/k=\mathbb{F}_{q^2}/\mathbb{F}_q$-hermitian space of dimension $m$ (which is unique up to isomorphism). The following (well-known) formula will be used throughout this article often without explicit reference.

\begin{lemma}\label{lem:Smb}Let $\mathcal{S}_{b}(V)$ be set of totally isotropic $k_F$-subspaces of dimension $b$ in $V$, and $S_{m,b}:=\#\mathcal{S}_b(V)$. Then $$S_{m,b}=\frac{\prod_{i=m-2b+1}^m(1-(-q)^i)}{\prod_{i=1}^b(1-q^{2i})}.$$  
\end{lemma}

\begin{proof}
  The unitary group $\U(V)(k)=\U_{m}(k)$ acts transitively on the set $\mathcal{S}_b(V)$, with stabilizer given by a parabolic subgroup $P_b(k)\subseteq \U(V)(k)$. As an affine variety, we have $$P_b\simeq \Res_{k_F/k} \GL_{b}\times \Res_{k_F/k}\BG_a^{b\delta} \times \BG_{a}^{b^2}\times \U_\delta,$$ where $\delta=m-2b$. Therefore $$S_{m,b}=\frac{\# \U_m(k)}{\# P_{b}(k)}=\frac{q^{m^2}\prod_{i=1}^m (1-(-q)^{-i})}{[q^{2b^2} \prod_{i=1}^b (1-q^{-2i} )]\cdot q^{2b\delta}\cdot q^{b^2}\cdot[ q^{\delta^2}\prod_{i=1}^
\delta (1-(-q)^{-i})]},$$ which simplifies to the desired formula.
\end{proof}

\part{Local Kudla--Rapoport conjecture: the self-dual case}\label{part:local-kudla-rapoport}

\section{Kudla--Rapoport cycles}


\subsection{Rapoport--Zink spaces $\mathcal{N}$}\label{sec:rapoport-zink-spaces}

Let $n\ge1$ be an integer. A \emph{hermitian $O_F$-module of signature $(1,n-1)$} over a $\Spf \OFb$-scheme $S$ is a triple $(X, \iota,\lambda)$ where
\begin{enumerate}
\item $X$ is a formal $p$-divisible $O_{F_0}$-module over $S$ of relative height $2n$ and dimension $n$,
\item $\iota: O_F\rightarrow\End(X)$ is an action of $O_F$ extending the $O_{F_0}$-action and satisfying the Kottwitz condition of signature $(1,n-1)$: for all $a\in O_F$, the characteristic polynomial of $\iota(a)$ on $\Lie X$ is equal to $(T-a)(T-\sigma(a))^{n-1}\in \mathcal{O}_S[T]$,
\item $\lambda: X\xrightarrow{\sim} X^\vee$ is a principal polarization on $X$ whose Rosati involution induces the automorphism $\sigma$ on $O_F$ via $\iota$.
\end{enumerate}

Up to $O_F$-linear quasi-isogeny compatible with polarizations, there is a unique such triple $(\mathbb{X}, \iota_{\mathbb{X}}, \lambda_{\mathbb{X}})$ over $S=\Spec \bar k$. Let $\mathcal{N}=\mathcal{N}_n=\mathcal{N}_{F/F_0, n}$ be the (relative) \emph{unitary Rapoport--Zink space of signature $(1,n-1)$}, parameterizing 
hermitian $O_F$-modules of signature $(1,n-1)$ within the supersingular quasi-isogeny class. More precisely, $\mathcal{N}$ is the formal scheme over $\Spf \OFb$ which represents the functor sending each $S$ to the set of isomorphism classes of tuples $(X, \iota, \lambda, \rho)$, where the \emph{framing} $\rho: X\times_S \bar S\rightarrow \mathbb{X}\times_{\Spec \bar k}\bar S$ is an $O_F$-linear quasi-isogeny of height 0 such that $\rho^*((\lambda_\mathbb{X})_{\bar S})=\lambda_{\bar S}$. Here $\bar S\coloneqq S_{\bar k}$ is the special fiber.

The Rapoport--Zink space $\mathcal{N}=\mathcal{N}_n$ is formally locally of finite type and formally smooth of relative dimension $n-1$ over $\Spf \OFb$ (\cite{RZ96}, \cite[Proposition 1.3]{Mihatsch2016}).

\subsection{The hermitian space $\mathbb{V}$}\label{sec:herm-space-mathbbv}

Let $\mathbb{E}$ be the formal $O_{F_0}$-module of relative height 2 and dimension 1 over $\Spec \bar k$. Then $D\coloneqq \End_{O_{F_0}}^\circ(\mathbb{E})$ is the quaternion division algebra over $F_0$. We fix an $F_0$-embedding $\iota_\mathbb{E}:F\rightarrow D$, which makes $\mathbb{E}$ into a formal $O_F$-module of relative height 1. We fix an $O_{F_0}$-linear principal polarization $\lambda_\mathbb{E}: \mathbb{E}\xrightarrow{\sim} \mathbb{E}^\vee$. Then $(\mathbb{E}, \iota_\mathbb{E},\lambda_\mathbb{E})$ is a hermitian $O_F$-module of signature $(1,0)$. We have $\mathcal{N}_1\simeq \Spf \OFb$ and there is a unique lifting (\emph{the canonical lifting}) $\mathcal{E}$ of the formal $O_F$-module $\mathbb{E}$ over $\Spf \OFb$, equipped with its $O_F$-action $\iota_\mathcal{E}$, its framing $\rho_\mathcal{E}: \mathcal{E}_{\bar k}\xrightarrow{\sim}\mathbb{E}$, and its principal polarization $\lambda_\mathcal{E}$ lifting $\rho_\mathcal{E}^*(\lambda_\mathbb{E})$. Define $\barE$ to be the same $O_{F_0}$-module as $\mathbb{E}$ but with $O_F$-action given by $\iota_{\barE}\coloneqq \iota_\mathbb{E}\circ \sigma$, and $\lambda_{\barE}\coloneqq \lambda_{\mathbb{E}}$, and similarly define $\bar{\mathcal{E}}$ and $\lambda_{\bar{\mathcal{E}}}$. 

Define $\mathbb{V}=\mathbb{V}_n\coloneqq \Hom_{O_F}^\circ(\barE, \mathbb{X})$ to be the space of \emph{special quasi-homomorphisms} (\cite[Definition 3.1]{Kudla2011}). Then $\mathbb{V}$ carries a $F/F_0$-hermitian form: for $x,y\in \mathbb{V}$, the pairing $(x,y)\in F$ is given by  $$ (\barE\xrightarrow{x} \mathbb{X}\xrightarrow{\lambda_\mathbb{X}} {\mathbb{X}^\vee}\xrightarrow{y^\vee} \barE^\vee\xrightarrow{\lambda_\mathbb{E}^{-1}}\barE)\in\End_{O_{F}}^\circ(\barE)=\iota_\barE(F)\simeq F.$$ The hermitian space $\mathbb{V}$ is the unique (up to isomorphism) nondegenerate non-split $F/F_0$-hermitian space of dimension $n$. The space of special homomorphisms $\Hom_{O_F}(\barE, \mathbb{X})$ is an integral hermitian $O_F$-lattice in $\mathbb{V}$. The unitary group $\U(\mathbb{V})(F_0)$ acts on the framing hermitian $O_F$-module $(\mathbb{X}, \iota_\mathbb{X},\lambda_\mathbb{X})$ and hence acts on the Rapoport--Zink space $\mathcal{N}$.

\subsection{Kudla--Rapoport cycles $\mathcal{Z}(L)$}\label{sec:kudla-rapop-cycl}

For any subset $L\subseteq \mathbb{V}$, define the \emph{Kudla--Rapoport cycle} (or \emph{special cycle}) $\mathcal{Z}(L)\subseteq \mathcal{N}$ to be the closed formal subscheme which represents the functor sending each $S$ to the set of isomorphism classes of tuples $(X, \iota, \lambda,\rho)$ such that for any $x\in L$, the quasi-homomorphism $$\rho^{-1}\circ x\circ \rho_{\bar{\mathcal{E}}}: \bar{\mathcal{E}}_S\times_S\bar S\rightarrow X\times_S \bar S$$  extends to a homomorphism $\bar{\mathcal{E}}_S\rightarrow X$ (\cite[Definition 3.2]{Kudla2011}). Note that $\mathcal{Z}(L)$ only depends on the $O_F$-linear span of $L$ in $\mathbb{V}$.





\subsection{Arithmetic intersection numbers $\Int(L)$}\label{sec:arithm-inters-numb}
Let $L\subseteq \mathbb{V}$ be an $O_F$-lattice of rank $r\ge1$. Let $x_1,\ldots, x_{r}$ be an $O_F$-basis of $L$. Since each $\mathcal{Z}(x_i)$ is a Cartier divisor on $\mathcal{N}$ (\cite[Proposition~3.5]{Kudla2011}), we know that $\mathcal{O}_{\mathcal{Z}(x_i)}\in \mathrm{F}^1K_0^{\mathcal{Z}(x_i)}(\mathcal{N})$ (see \S\ref{sec:notat-form-schem}), and hence by \cite[(B.3)]{Zhang2019} we obtain $$\mathcal{O}_{\mathcal{Z}(x_1)} \otimes^\mathbb{L}\cdots \otimes^\mathbb{L}\mathcal{O}_{\mathcal{Z}(x_r)}\in \mathrm{F}^{r}K_0^{\mathcal{Z}(L)}(\mathcal{N}).$$ This is independent of the choice of the basis $x_1,\ldots, x_n$ by \cite[Corollary C]{Howard2018} and hence is a well-defined invariant of $L$ itself. We will provide a different proof of this independence (see Corollary \ref{cor:ind L}) after recalling the structure of  the reduced scheme of $\CZ(L)$.

\begin{definition}
Define the \emph{derived Kudla--Rapoport cycle} 
$^\BL\CZ(L)$ to be the image of $\mathcal{O}_{\mathcal{Z}(x_1)} \otimes^\mathbb{L}\cdots \otimes^\mathbb{L}\mathcal{O}_{\mathcal{Z}(x_r)}$ in the $r$-th graded piece $\Gr^{r}K_0^{\CZ(L)}(\mathcal{N})$. 
\end{definition}

\begin{definition}
  When $L\subseteq \mathbb{V}$ has rank $r=n$, define the \emph{arithmetic intersection number} \begin{align}\label{eq:def Int}
\Int(L)\coloneqq \chi\bigl(\mathcal{N},{}^\BL\CZ(L)),
\end{align}  where  $\chi$ denotes the Euler--Poincar\'e characteristic (\S\ref{sec:notat-form-schem}).
\end{definition}

\begin{example}[The case $\rank L=1$]\label{ex int rank=1} If $\rank L=1$, then by the theory of canonical lifting (\cite{Gross1986a}), we have 
 $$
 \Int(L)=\frac{\val(L)+1}{2}.
 $$
\end{example}


\subsection{Generalized Deligne--Lusztig varieties $Y_V$}\label{sec:gener-deligne-luszt}

Let $V$ be the unique (up to isomorphism) $k_F/k$-hermitian space of odd dimension $2d+1$. Define $Y_V$ to be the closed $k_F$-subvariety of the Grassmannian $\Gr_{d+1}(V)$ parameterizing subspaces $U\subseteq V$ of dimension $d+1$ such that $U^\perp\subseteq U$ (\cite[(2.19)]{Vollaard2010}). It is a smooth projective variety of dimension $d$, and has a locally closed stratification $$Y_V=\bigsqcup_{i=0}^d X_{P_i}(w_i),$$ where each $X_{P_i}(w_i)$ is a generalized Deligne--Lusztig variety of dimension $i$ associated to a certain parabolic subgroup $P_i\subseteq\U(V)$ (\cite[Theorem 2.15]{Vollaard2010}). The open stratum $Y_V^\circ\coloneqq X_{P_d}(w_d)$ is a classical Deligne--Lusztig variety associated to a Borel subgroup $P_d\subseteq \U(V)$ and a Coxeter element $w_d$. Each of the other strata $X_{P_i}(w_i)$ is also isomorphic to a parabolic induction of a classical Deligne--Lusztig variety of Coxeter type for a Levi subgroup of $\U(V)$ (\cite[Proposition 2.5.1]{He2019}).

\subsection{Minuscule Kudla--Rapoport cycles $\mathcal{V}(\Lambda)$}
\label{ss:minu KR}
Let $\Lambda\subseteq\mathbb{V}$ be a vertex lattice. Then $V_\Lambda\coloneqq \Lambda^\vee/\Lambda$ is a $k_F$-vector space of dimension $t(\Lambda)$, equipped with a nondegenerate $k_F/k$-hermitian form induced from $\mathbb{V}$. Since $\mathbb{V}$ is a non-split hermitian space, the type $t(\Lambda)$ is odd. Thus we have the associated generalized Deligne--Lusztig variety $Y_{V_\Lambda}$ of dimension $(t(\Lambda)-1)/2$. The reduced subscheme of the minuscule Kudla--Rapoport cycle $\mathcal{V}(\Lambda)\coloneqq \mathcal{Z}(\Lambda)^\mathrm{red}$ is isomorphic to $Y_{V_\Lambda,\bar k}$\footnote{We naturally identify $\mathbb{V}$ with $\mathbf{N}_0$ in \cite{Vollaard2011} and $C$ in \cite{Kudla2011} via \cite[Lemma 3.9]{Kudla2011}. Notice that $\mathcal{V}(\Lambda)$ in \cite{Vollaard2011} and \cite{Kudla2011} is the same as our $\mathcal{V}(\Lambda^\vee)$.}. In fact $\mathcal{Z}(\Lambda)$ itself is already reduced (\cite[Theorem B]{Li2017}), so $\mathcal{V}(\Lambda)=\mathcal{Z}(\Lambda)$.

\subsection{The Bruhat--Tits stratification on $\mathcal{N}^\mathrm{red}$}\label{sec:bruh-tits-strat}

The reduced subscheme of $\mathcal{N}$ satisfies $\mathcal{N}^\mathrm{red}=\bigcup_{\Lambda} \mathcal{V}(\Lambda)$, where $\Lambda$ runs over all vertex lattices $\Lambda\subseteq \mathbb{V}$. For two vertex lattices $\Lambda, \Lambda'$, we have $\mathcal{V}(\Lambda)\subseteq \mathcal{V}(\Lambda')$ if and only if $\Lambda\supseteq \Lambda'$; and $\mathcal{V}(\Lambda)\cap \mathcal{V}(\Lambda')$ is nonempty if and only if $\Lambda+\Lambda'$ is also a vertex lattice, in which case it is equal to $\mathcal{V}(\Lambda+\Lambda')$. In this way we obtain a \emph{Bruhat--Tits stratification} of $\mathcal{N}^\mathrm{red}$ by locally closed subvarieties (\cite[Theorem B]{Vollaard2011}), $$\mathcal{N}^\mathrm{red}=\bigsqcup_{\Lambda}\mathcal{V}(\Lambda)^\circ, \quad \mathcal{V}(\Lambda)^\circ\coloneqq \mathcal{V}(\Lambda)-\bigcup_{\Lambda\subsetneq \Lambda'}\mathcal{V}(\Lambda').$$ Each Bruhat--Tits stratum $\mathcal{V}(\Lambda)^\circ\simeq Y_{V_\Lambda,\bar k}^\circ$ is a classical Deligne--Lusztig of Coxeter type associated to $\U(V_\Lambda)$, which has dimension $(t(\Lambda)-1)/2$. It follows that the irreducible components of $\mathcal{N}^\mathrm{red}$ are exactly the projective varieties $\mathcal{V}(\Lambda)$, where $\Lambda$ runs over all vertex lattices of maximal type (\cite[Corollary C]{Vollaard2011}). The points in the 0-dimensional Bruhat--Tits strata, i.e., $\mathcal{V}(\Lambda)$ for type 1 vertex lattices $\Lambda$, are known as \emph{superspecial} points.

For $L\subseteq \mathbb{V}$ be an $O_F$-lattice of rank $r\ge1$. By \cite[Proposition 4.1]{Kudla2011}, the reduced subscheme $\mathcal{Z}(L)^\mathrm{red}$ of a Kudla--Rapoport cycle $\mathcal{Z}(L)$ is a union of Bruhat--Tits strata,
\begin{equation}
  \label{eq:KRstrat}
\mathcal{Z}(L)^\mathrm{red}=\bigcup_{L\subseteq\Lambda}\mathcal{V}(\Lambda).  
\end{equation}

When $n\geq 3$, a point $z\in \CN(\ov k)$ is called {\em super-general} if there is no special homomorphism $u$ of valuation $0$ such that $z\in \CZ(u)(\ov k)$. By (\ref{eq:KRstrat}), we know that $z$ is super-general if and only if $z\in \mathcal{V}(\Lambda)^\circ$ for $\Lambda$ a vertex lattice of type $n$.  In particular, there is no super-general point on $\CN$ when $n\ge3$ is even.

\subsection{Independence of the choice of the basis}
We generalize the results of Terstiege \cite[Lemma~3.1, Proposition~3.2]{Terstiege2013}.
\begin{lemma}\label{lem:two div}
Let $x,y\in\BV=\BV_n$ be linearly independent. Then the sheaves ${\rm Tor}_i^{\CO_{\CN_n}}(\CO_{\CZ(x)}, \CO_{\CZ(y)})$ vanish for all $i\geq 1$. In particular,
$$
\mathcal{O}_{\mathcal{Z}(x)} \otimes^\mathbb{L} \mathcal{O}_{\mathcal{Z}(y)}=\mathcal{O}_{\mathcal{Z}(x)} \otimes \mathcal{O}_{\mathcal{Z}(y)}.
$$
\end{lemma}

\begin{proof}The proof is similar to \cite[Lemma~3.1]{Terstiege2013}. Let $z\in \CN_n(\ov k)$ and let $R=\CO_{\CN_{n},z}$ be the local ring at $z$. Let $f,g\in R$ be the local equations at $z$ of $\CZ(x), \CZ(y)$ respectively. Then (cf. {\it loc. cit.})
$$
{\rm Tor}_1^{\CO_{\CN_n}}(\CO_{\CZ(x)}, \CO_{\CZ(y)})_z=  \ker\!\!\xymatrix{(R/(g)\ar[r]^-{\cdot f}&R/(g)),}
$$
and ${\rm Tor}_i^{\CO_{\CN_n}}(\CO_{\CZ(x)}, \CO_{\CZ(y)})_z=0$ for $i>1$.  We {\em claim} that $f$ and $g$ have no common divisor in the regular ring $R$  for every $z\in \CN_n(\ov k)$. The claim implies the desired vanishing of ${\rm Tor}_1$.

 We prove the claim by induction on $n$. When $n\leq 3$, this is known by the proof of  \cite[Lemma~3.1]{Terstiege2013}. Now assume that $n\geq 4$. 
 
Let $\mathfrak{V}$ be the set of $z\in \CN_n(\ov k)$ where the local equations of $\CZ(x)$ and $\CZ(y)$ share a common divisor.

First suppose that $\mathfrak{V}$ contains a point $z$ that is not super-general (\S\ref{sec:bruh-tits-strat}). Choose $u$  with valuation $0$ such that $z\in \CZ(u)(\ov k)$. We may further assume that $u$ is linearly independent from $x$ and $y$. In fact, if $u\in\pair{x,y}_F$,  we may choose a non-zero $u'\in \pair{x,y}^\perp_F$ such that $z\in\CZ(u')(\ov k)$, and then  we replace $u$ by $u+u'$. Denote by $x^\flat$ (resp. $y^\flat$) the orthogonal projection of $x$ (resp. $y$) to $\pair{u}^\perp$. Then $x^\flat$ and $y^\flat$ remain  linearly independent. The restrictions of $\CZ(x)$ and $\CZ(y)$ to $\CZ(u)\simeq\CN_{n-1}$ (cf. \eqref{eq:ind Z(u)}) are the special divisors $\CZ(x^\flat)$ and  $\CZ(y^\flat)$. By our assumption $z\in\mathfrak{V}$, the local equations at $z\in\CN_{n-1}(\ov k)$ of $\CZ(x^\flat)$ and $\CZ(y^\flat)$ share a common divisor. This contradicts the induction hypothesis.
 
Now suppose that  $\mathfrak{V}$ consists of only super-general points. In particular, $n$ is odd. Recall that the difference divisor $\CD(y):=\CZ(y)-\CZ(y/\varpi)$ (as Cartier divisors) is effective and regular by \cite{Terstiege2013b}. We have an equality of Cartier divisors  $\CZ(y)=\sum_{i\geq 0} \CD(y/\varpi^i)$ (this is a locally finite sum). Let $z_0\in \mathfrak{V}$ and let $\Lambda$ be the unique vertex lattice of type $n$ such that $z_0\in\CV(\Lambda)(\ov k)$. Possibly replacing $y$ by $y/\varpi^i$ for some $i\geq 0$, we may assume that locally at $z_0$ the divisor $\CD(y)$ is a component of $\CZ(x)$. By the argument of \cite[Lemma~3.6]{Kudla2011}, the set of points $z\in \CD(y)^{\red}$ where the local equations at $z$ of $\CZ(x)$ and $\CD(y)$ share a common divisor is open and closed in $\CD(y)^{\red}$. In fact we can directly apply {\it loc. cit} by letting $\mathfrak{X}$ be the formal completion of $\CN_n$ along $\CD(y)^{\red}$ and noting that the local equation of $\CD(y)$ is given by an irreducible element. It follows that there exists an irreducible component of the scheme $\CD(y)\cap\CV(\Lambda)$, denoted by $\mathfrak{D}_0$,  passing through $z_0$.  Then $\mathfrak{D}_0(\ov k)\subset\mathfrak{V}$, and $\mathfrak{D}_0$ is closed subscheme of $\CV(\Lambda)$. By our assumption on $\mathfrak{V}$,  we have $\mathfrak{D}_0\subset \CV(\Lambda)^\circ$.  However, by \cite[Corollary 2.8]{Lusztig1976/77}, the open variety $\CV(\Lambda)^\circ$ is affine with $\dim\CV(\Lambda)=\frac{n-1}{2}\geq 2$ and hence can not have any positive dimensional projective irreducible subscheme (such as $\mathfrak{D}_0$). Contradiction! This completes the induction.
\end{proof}

\begin{corollary}\label{cor:ind L}Let $L\subseteq \mathbb{V}$ be an $O_F$-lattice of rank $r\ge1$. Let $x_1,\ldots, x_{r}$ be an $O_F$-basis of $L$. Then $\mathcal{O}_{\mathcal{Z}(x_1)} \otimes^\mathbb{L}\cdots \otimes^\mathbb{L}\mathcal{O}_{\mathcal{Z}(x_r)}\in K_0^{\mathcal{Z}(L)}(\mathcal{N})$ is independent of the choice of the basis.
\end{corollary} 
\begin{proof}This is similar to \cite[Proposition~3.2]{Terstiege2013}. We can transform a basis  into any other basis by a suitable sequence of the following operations: permutations; the multiplication on a basis vector by a unit in $O_F^\times$; for every pair $(i,j), i\neq j$, the substitution of $x_i$ by $x_i+\alpha x_j$ for $\alpha\in O_F$.
\end{proof}

\subsection{Horizontal and vertical parts of $\mathcal{Z}(L)$}\label{sec:horiz-vert-parts}

\begin{definition}
  A formal scheme $Z$ over $\Spf \OFb$ called \emph{vertical} (resp. \emph{horizontal}) if $\varpi$ is locally nilpotent on $Z$ (resp. flat over $\Spf \OFb$). Clearly the formal scheme-theoretic union of two vertical (resp. horizontal) formal subschemes of a formal scheme is also vertical (resp. horizontal).

  We define the \emph{horizontal part} $Z_\sH\subseteq Z$ to be the closed formal subscheme defined by the ideal sheaf $\mathcal{O}_{Z,\mathrm{tor}}$ of torsion sections of $\mathcal{O}_Z$. Then $Z_\sH$ is the maximal vertical closed formal subscheme of $Z$.

  When $Z$ is noetherian, there exists $N\gg0$ such that $\varpi^N \mathcal{O}_{Z,\mathrm{tor}}=0$, and we define \emph{vertical part} $Z_\sV\subseteq Z$ to be the closed formal subscheme defined by the ideal sheaf $\varpi^N\mathcal{O}_Z$. Since $\mathcal{O}_{Z,\mathrm{tor}}\cap \varpi^N\mathcal{O}_Z=0$, we have a decomposition  $$Z=Z_\sH\cup Z_\sV,$$ as a union of horizontal and vertical formal subschemes. Notice that the horizontal part $Z_\sH$ is canonically defined, while the vertical part $Z_\sV$ depends on the choice of $N$. 
\end{definition}


\begin{lemma}\label{lem:ZLnoetherian}
  Let $L\subseteq \mathbb{V}$ be an $O_F$-lattice of rank $r\ge n-1$. Then $\mathcal{Z}(L)$ is noetherian.
\end{lemma}

\begin{proof}
As a closed formal subscheme of the locally noetherian formal scheme $\mathcal{N}$, we know that $\mathcal{Z}(L)$ is locally noetherian. Since $L$ has rank $r\ge n-1$, the number of vertex lattices $\Lambda\subseteq \mathbb{V}$ such that $L\subseteq \Lambda$ is finite. By (\ref{eq:KRstrat}), we know that $\mathcal{Z}(L)^\mathrm{red}$ is a closed subset in finitely many irreducible components of $\mathcal{N}^\mathrm{red}$. Since each irreducible component of $\mathcal{N}^\mathrm{red}$ is quasi-compact, we know that $\mathcal{Z}(L)$ is quasi-compact, hence noetherian.
\end{proof}

By Lemma \ref{lem:ZLnoetherian}, for $L\subseteq \mathbb{V}$ an $O_F$-lattice of rank $r\ge n-1$, we obtain a decomposition of the Kudla--Rapoport cycle  into \emph{horizontal} and \emph{vertical} parts $$\mathcal{Z}(L)=\mathcal{Z}(L)_\sH\cup \mathcal{Z}(L)_\sV.$$ Again notice that the vertical part $\mathcal{Z}(L)_\sV$ depends on the choice of an integer $N\gg 0$. Since the choice of $N$ is not important for our purpose we suppress it from the notation (cf. \S\ref{sec:horiz-vert-parts-1}).

\subsection{Finiteness of $\Int(L)$}

The following result should be well-known to the experts.

\begin{lemma}\label{lem:properscheme}
Let $L\subseteq \mathbb{V}$ be an $O_F$-lattice of rank $n$. Then the formal scheme $\mathcal{Z}(L)$ is a proper scheme over $\Spf \OFb$. In particular,   $\Int(L)$ is finite.
\end{lemma}

\begin{proof}
  The vertical part $\mathcal{Z}(L)_\sV$ is a scheme by Lemma \ref{prop:supp Z(L) V} below. We show that the horizontal part $\mathcal{Z}(L)_\sH$ is empty. If not, there exists $z\in \mathcal{Z}(L)(O_K)$ for some finite extension $K$ of $\Fb$. Let $\mathcal{X}$ be the corresponding $O_F$-hermitian module of signature $(1,n-1)$ over $O_K$. Since $L$ has rank $n$, we know that $\mathcal{X}$ admits $n$ linearly independent special homomorphisms $\tilde x_i:\bar{\mathcal{E}}\rightarrow \mathcal{X}$, which gives rise to an $O_F$-linear isogeny $$(\tilde x_1,\ldots, \tilde x_{n}): \bar{\mathcal{E}}^{n}\rightarrow \mathcal{X}.$$ It then follows that the $O_F$-action on $\mathcal{X}$ satisfies the Kottwitz signature condition $(0,n)$ rather than $(1,n-1)$ in characteristic 0, a contradiction. Thus $\mathcal{Z}(L)_\sH$ is empty, and so $\mathcal{Z}(L)$ is a scheme. Since $\mathcal{Z}(L)^\mathrm{red}$ is contained in finitely many irreducible components of $\mathcal{N}^\mathrm{red}$ and each irreducible component of $\mathcal{N}^\mathrm{red}$ is proper over $\Spec\bar k$, it follows that the scheme $\mathcal{Z}(L)$ is proper over $\Spf \OFb$. The finiteness of $\Int(L)$ then follows from the discussion before \cite[(B.4)]{Zhang2019}.
\end{proof}

\subsection{A cancellation law for $\Int(L)$}\label{sec:cancelation-law-intl}

Let $ M\subset \BV_n$ be a self-dual lattice of rank $r$. The map $(X, \iota, \lambda, \rho)\mapsto (X\times \bar{\mathcal{E}}^r,\iota\times \iota_{\bar{\mathcal{E}}^r}, \lambda\times \lambda_{\bar{\mathcal{E}^r}}, \rho\times \rho_{\bar{\mathcal{E}}^r})$ gives a natural embedding 
 \begin{align}\label{eq:inc M} 
 \delta_M\colon 
 \xymatrix{\CN_{n-r}\ar[r]& \CN_n},
 \end{align}
 which identifies $ \CN_{n-r}$ with the special cycle $\CZ(M)$ (\cite[Remark 4.5]{Rapoport2018}).  Let $\BV_n=M_F\obot\BV_{n-r}$ be the induced orthogonal decomposition. For $u\in\BV_n$, denote by $u^\flat$ the projection to $\BV_{n-r}$.  If $u^\flat\neq 0$, then the special divisor $\CZ(u)$ intersects transversely with $\CN_{n-r}$ and its pull-back to $\CN_{n-r}$ is the special divisor $\CZ(u^\flat)$. By \cite[Lemma B.2 (i)]{Zhang2019}, we obtain
 \begin{align}\label{eq:ind Z(u)}
 \CN_{n-r}\jiao \CZ(u)=\CZ(u^\flat).
 \end{align}

\begin{lemma}\label{lem: cancel}
Let $ M\subset \BV_n$ be a self-dual lattice of rank $r$ and $L^\flat$ an integral lattice in $\BV_{n-r}$. Then
$$\Int(L^\flat\obot M)=\Int(L^\flat).$$
\end{lemma}
\begin{proof}
This follows from the equation \eqref{eq:ind Z(u)} and the definition of $\Int$ by \eqref{eq:def Int}. 
\end{proof}

\section{Local densities}\label{sec:local-densities}
In this section (except \S \ref{sec:local-kudla-rapoport}) we allow $F_0$ to be a non-archimedean local field of characteristic not equal to $2$ (but possibly with residue characteristic $2$), and $F$ an unramified quadratic extension.

\subsection{Local densities for hermitian lattices}\label{ss:loc den} Let $L, M$ be two hermitian $O_F$-lattices. Let $\Rep_{M,L}$ be the \emph{scheme of integral representations of $M$ by $L$}, an $O_{F_0}$-scheme such that for any $O_{F_0}$-algebra $R$, \begin{align}\label{def: Rep}
\Rep_{M,L}(R)=\Herm(L \otimes_{O_{F_0}}R, M \otimes_{O_{F_0}}R),\end{align} where $\Herm$ denotes the set of hermitian module homomorphisms. The \emph{local density} of integral representations of $M$ by $L$ is defined to be $$\Den(M,L)\coloneqq \lim_{N\rightarrow +\infty}\frac{\#\Rep_{M,L}(O_{F_0}/\varpi^N)}{q^{N\cdot\dim (\Rep_{M,L})_{F_0}}}.$$ Note that if $L, M$ have rank $n, m$ respectively and the generic fiber $(\Rep_{M,L})_{F_0}\ne\varnothing$, then $n\le m$ and
\begin{equation}
  \label{eq:dimRep}
\dim (\Rep_{M,L})_{F_0}=\dim \U_m-\dim \U_{m-n}= n\cdot (2m-n).  
\end{equation}

\subsection{Local Siegel series for hermitian lattices} 

Let $k\ge0$ be an integer. Let $\iden^k$ be the self-dual hermitian $O_F$-lattice of rank $k$ with hermitian form given the identity matrix $\mathbf{1}_k$. Let $L$ be a hermitian $O_F$-lattice of rank $n$. By \cite[Theorem II]{Hironaka1998}, $\Den(\iden^{n+k}, L)$ is a polynomial in $(-q)^{-k}$ with $\mathbb{Q}$-coefficients (zero if $L$ is not integral). A special case (see \cite[p.677]{Kudla2011}) is 
\begin{equation}
  \label{eq: iden}
\Den(\iden^{n+k}, \iden^n)=\prod_{i=1}^n(1-(-q)^{-i}X)\bigg|_{X= (-q)^{-k}}.
\end{equation}
Define the (normalized) \emph{local Siegel series} of $L$ to be the polynomial $\Den(X,L)\in \mathbb{Z}[X]$ (Theorem \ref{thm: Den(X)}) such that $$\Den((-q)^{-k},L)=\frac{\Den(\iden^{n+k}, L)}{\Den(\iden^{n+k}, \iden^n)}.$$ 

The local Siegel series satisfies a functional equation (\cite[Theorem 5.3]{Hironaka2012})
\begin{equation}
  \label{eq:functionalequation}
  \Den(X,L)=(-X)^{\val(L)}\cdot \Den\left(\frac{1}{X},L\right).
\end{equation}

\begin{definition}
Define the central value of the local density to be $$\Den(L)\coloneqq\Den(1,L).$$ In particular, if $\val(L)$ is odd, then $\Den(L)=0$. In this case, define the \emph{central derivative of the local density} or \emph{derived local density} by $$\pDen(L)\coloneqq-\frac{\rd}{\rd X}\bigg|_{X=1}\Den(X,L).$$  
\end{definition}

Notice that by definition $\Den(M,L)$ only depends on the isometry classes of $M$ and $L$, and hence only depends on the fundamental invariants of $M$ and $L$. In particular, $\Den(X,L)$ and $\pDen(L)$ only depends on the fundamental invariants of $L$.  Moreover, there is an analog of Lemma \ref{lem: cancel}: for any self-dual lattice $M$ of $\rank(M)=m$ and any integral lattice $L^\flat$ of $\rank(L^\flat)=n$, we have
$$
\Den(\iden^{n +m+k}, L^\flat\obot M)=\Den(\iden^{n+k},  L^\flat)$$
and therefore we obtain a cancellation law:
\begin{align}
\label{eq:cancel den}
 \Den(X,L^\flat\obot M)= \Den(X,L^\flat).
\end{align}

\subsection{Relation with local Whittaker functions}\label{sec:relation-with-local} Let $\Lambda=\langle 1\rangle^n$ be an self-dual hermitian $O_F$-lattice. Let $L$ be a hermitian $O_F$-lattice of rank $n$. Let $T=((x_i, x_j))_{1\le i,j\le n}$ be the fundamental matrix of an $O_F$-basis $\{x_1,\ldots, x_n\}$ of $L$, an $n\times n$ hermitian matrix over $F$. Associated to the standard Siegel--Weil section of the characteristic function $\varphi_0=\mathbf{1}_{\Lambda^n}$ and the unramified additive character $\psi: F_0\rightarrow \mathbb{C}^\times$, there is a local (generalized) Whittaker function $W_T(g, s, \varphi_0)$ (see \S\ref{sec:four-coeff-deriv}, \S\ref{sec:incoh-eisenst-seri} for the precise definition). By \cite[Proposition 10.1]{Kudla2014}, when $g=1$, it satisfies the interpolation formula for integers $s=k\ge0$ (notice $\gamma_p(V)=1$ in the notation there), $$W_T(1, k,\varphi_0)=\Den(\langle 1\rangle^{n+2k}, L).$$ So its value at $s=0$ is $$W_T(1, 0, \varphi_0)=\Den(\langle 1\rangle^n, L)=\Den(L)\cdot \Den(\langle 1\rangle^n, \langle 1\rangle^n),$$ and its derivative at $s=0$ is\footnote{In \cite[Proposition 9.3]{Kudla2014}, the factor $\log p$ should be $\log p^2$.} $$W_T'(1, 0, \varphi_0)=\pDen(L)\cdot \Den(\langle 1\rangle^n, \langle 1\rangle^n)\cdot\log q^2.$$ Plugging in \eqref{eq: iden}, we obtain
\begin{align}
W_T(1, 0, \varphi_0)&=\Den(L)\cdot \prod_{i=1}^n(1-(-q)^{-i}),\label{eq:localWhittaker0}\\ W_T'(1, 0, \varphi_0)&=\pDen(L)\cdot \prod_{i=1}^n(1-(-q)^{-i})\cdot \log q^2.    \label{eq:localWhittaker1}
\end{align}

\subsection{The local Kudla--Rapoport conjecture}\label{sec:local-kudla-rapoport} Now we can state the main theorem of this article, which proves the Kudla--Rapoport conjecture on the identity between arithmetic intersection numbers of Kudla--Rapoport cycles and central derivatives of local densities. Recall that $\BV=\BV_n$ is the hermitian space defined in \S \ref{sec:herm-space-mathbbv}.
\begin{theorem}[local Kudla--Rapoport conjecture]\label{thm: main}
Let $L\subseteq \mathbb{V}$ be an $O_F$-lattice of full rank $n$. Then $$\Int(L)=\pDen(L).$$
\end{theorem}
This will be proved in \S\ref{ss:proof}.

\begin{remark}\label{rem:localKR1}
In the notation of \S\ref{sec:relation-with-local},  it follows immediately from Theorem \ref{thm: main} and \eqref{eq:localWhittaker1} that $$\Int(L)=\frac{W_T'(1, 0, \varphi_0)}{\log q^2}\cdot  \prod_{i=1}^n(1-(-q)^{-i})^{-1}.$$ 
\end{remark}

\subsection{Formulas in terms of weighted lattice counting: Theorem of Cho--Yamauchi}

Define weight factors $$\wt(a; X)\coloneqq \prod_{i=0}^{a-1}(1-(-q)^iX), \quad \wt(a)\coloneqq -\frac{\rd}{\rd X}\bigg|_{X=1}\wt(a; X)=\prod_{i=1}^{a-1}(1-(-q)^i),$$ where by convention $\wt(0; X)=1$ and $\wt(0)=0$, $\wt(1)=1$. 
Then we have the following explicit formula for the local Siegel series.

\begin{theorem}[Cho--Yamauchi]\label{thm: Den(X)}
The following identity holds: $$\Den(X,L)=\sum_{L\subset L'\subset L'^\vee} X^{2\ell( L'/L)}\cdot  \wt(t(L'); X),
$$ 
where the sum runs over all integral lattices $L'\supset L$. Here $$\ell(L'/L)\coloneqq {\rm length}_{O_F}\,L'/L.$$
\end{theorem}
\begin{proof}
  This is proved the same way as in the orthogonal case \cite[Corollary 3.11]{CY}, using the following hermitian analogue of \cite[\S 5.6 Exercise 4]{Kitaoka1993}. Let $U$ be an $\mathbb{F}_{q^2}/\mathbb{F}_q$-hermitian space of dimension $n$ whose radical has dimension $a$. Let $V$ be a non-degenerate $\mathbb{F}_{q^2}/\mathbb{F}_q$-hermitian space of dimension $m\ge n$. Then the number of isometries from $U$ to $V$ is equal to $$q^{n(2m-n)}\cdot \prod_{i=0}^{n+a-1}(1-(-q)^{i-m}).$$ Writing $m=n+k$, this is equal to $$q^{n(2m-n)}\cdot\Den(\langle1\rangle^{n+k},\langle 1\rangle^n)\cdot \mathfrak{m}(a; (-q)^{-k}),$$ which explains the correct weight factor $\mathfrak{m}(a; X)$ appearing in the theorem.

  We remark that since $F/F_0$ is unramified, the analogue of the smoothness theorem \cite[Theorem 3.9]{CY} is valid in the hermitian case even when the residue characteristic is $p=2$, as \cite[Lemma 5.5.2]{Gan2000} is still valid for $p=2$ by \cite[\S 9]{Gan2000}.
\end{proof}

\begin{example}[The case $\rank L=1$]\label{ex rank=1} If $\rank L=1$, the formula specializes to 
$$\Den(X,L)=\sum_{i=0}^{\val(L)}(-X)^i.
$$
In particular, if $\val(L^\flat)$ is odd,  we obtain $\Den(L)=0$ and $$\pDen(L)=\frac{\val(L)+1}{2}.$$
\end{example}

Also note that if $L'\supset L$, then $\val(L')$ and $\val(L)$ have the same parity (see \S\ref{ss:notation}). In particular, if $\val(L)$ is odd, then $t(L')>0$ and hence $\wt(t(L'); 1)=0$. Thus we obtain the following explicit formula for $\pDen(L)$.

\begin{corollary}\label{cor: pDen}
If $\val(L)$ is odd, then $$\pDen(L)=\sum_{L\subset L'\subset L'^\vee}  \wt(t(L')).$$
\end{corollary}

\subsection{Some special cases}
Since $\wt(a; (-q)^{-k})=0$ if $0\le k\le (a-1)$, we also obtain

\begin{corollary}\label{cor:densityat1} 
For $k\ge0$, 
  $$\Den((-q)^{-k},L)=\sum_{L\subset L'\subset L'^\vee\atop t(L')\leq k} q^{-2\ell(L'/L)k}\cdot\wt(t(L'); (-q)^{-k})$$
  In particular, for $k=0$, 
  \begin{equation}
    \label{eq: Den L}
    \Den(L)=\Den(1,L)=\sum_{L\subset L'\subset L'^\vee\atop t(L')=0}1=\#\{L' \text{ self-dual}: L\subseteq L' \}.
  \end{equation}
 For $k=1$, 
  \begin{align}
    \frac{1}{\vol(L)}\Den((-q)^{-1},L)&=\sum_{L\subset L'\subset L'^\vee\atop t(L')=0}1+\sum_{L\subset L'\subset L'^\vee\atop t(L')=1}(1+q^{-1})\frac{1}{\vol(L')} \label{eq:valueat1}
  \end{align}
\end{corollary}

\begin{corollary} The following identities hold: \begin{align}\label{eq: Den -q}
\Den(-q,L)=   \sum_{L\subset L'\subset L'^\vee} [L':L]\cdot\wt(t(L')+1),
  \end{align}
and
\begin{align}\label{eq: FE}
\Den(-q,L)=    \frac{1}{\vol(L)}\Den((-q)^{-1},L).
\end{align}
\end{corollary}

\begin{proof}
The first part follows from Theorem \ref{thm: Den(X)} and the fact that $$\wt(t(L'); -q)=\wt(t(L')+1).$$  The second part follows from  the functional equation (\ref{eq:functionalequation}).
\end{proof}

\subsection{An induction formula}

\begin{proposition}\label{prop: ind}
Let $L^\flat$ be a lattice of rank $n-1$ with fundamental invariants $(a_1,\cdots,a_{n-1})$. Let $L=L^\flat+\pair{x}$ and  $L'=L^\flat+\pair{\varpi^{-1} x}$ where $x\perp L^\flat$ with $\val(x)>a_{n-1}$. Then
$$
\Den(X,L)=X^2\Den(X,L')+(1-X)\Den(-qX,L^\flat).
$$
\end{proposition}
This is \cite[Theorem 5.1]{Terstiege2013} in the hermitian case, and 
Katsurada \cite[Theorem 2.6 (1)] {Katsurada1999} in the orthogonal case (see also \cite{CY}).

%
%
%

\section{Horizontal components of Kudla--Rapoport cycles}

\subsection{Quasi-canonical lifting cycles}\label{sec:quasi-canon-lift}

Let $\langle y\rangle\subseteq \mathbb{V}_2$ be a rank one $O_F$-lattice. By \cite[Proposition 8.1]{Kudla2011}, we have a decomposition as Cartier divisors on $\mathcal{N}_2$, $$\mathcal{Z}(y)=\sum_{i=0}^{\lfloor \val(y)/2\rfloor}\mathcal{Z}_{\val(y)-2i}.$$ Here $\mathcal{Z}_s$ ($s\ge0$) is the \emph{quasi-canonical lifting cycle} of level $s$ on $\mathcal{N}_2$, the horizontal divisor corresponding to the \emph{quasi-canonical lifting of level $s$} of the framing object $(\mathbb{X}, \iota_\mathbb{X}, \lambda_\mathbb{X})$ of $\mathcal{N}_2$ (the quasi-canonical lifting of level $s=0$ is the \emph{canonical lifting}). We define the \emph{primitive part} of $\mathcal{Z}(y)$ to be $$\mathcal{Z}(y)^\circ\coloneqq \mathcal{Z}_{\val(y)}\subseteq \mathcal{Z}(y).$$

Let $O_{F,s}=O_{F_0}+\varpi^s O_F\subseteq O_{F}$. Let $\Fb_s$ be the finite abelian extension of $\Fb$ corresponding to the subgroup $O_{F,s}^\times$ under local class field theory. Let $O_{\Fb,s}$ be the ring of integers of the ring class field $\Fb_s$. Then $O_{\Fb,0}=\OFb$, and the degree of $O_{\Fb,s}$ over $\OFb$ is equal to $q^s(1+q^{-1})$ when $s\ge1$. We have $$\mathcal{Z}_s\cong \Spf O_{\Fb,s}.$$

\subsection{Horizontal cycles}

Let $L^\flat\subseteq\mathbb{V}_n$ be a hermitian $O_F$-lattice of rank $n-1$. Let $M^\flat$ be an integral hermitian $O_F$-lattice of rank $n-1$ such that $L^\flat\subseteq M^\flat$. When $t(M^\flat)\le1$, we can construct a horizontal formal subscheme in $\mathcal{N}_n$ using quasi-canonical lifting. In fact, since $t(M^\flat)\le1$, we may find a rank $n-2$ $O_F$-lattice $M_{n-2}$, which is self-dual in the hermitian space $M_{n-2,F}$, and a rank one $O_F$-lattice $\langle y\rangle$, such that we have an orthogonal direct sum decomposition $$M^\flat= M_{n-2}\obot \langle y\rangle.$$ Let $M_{n-2,F}^\perp \subseteq \mathbb{V}_n$ be the orthogonal complement of $M_{n-2,F}$ in $\mathbb{V}_n$. Then we have an isomorphism $M_{n-2,F}^\perp\simeq\mathbb{V}_2$, and thus an isomorphism (see \S\ref{sec:cancelation-law-intl})  $$\mathcal{Z}(M_{n-2})\simeq \mathcal{N}_2.$$ Under this isomorphism, we can identify the Cartier divisor $\mathcal{Z}(M^\flat)\subseteq \mathcal{Z}(M_{n-2})$ with the Cartier divisor $\mathcal{Z}(y)\subseteq \mathcal{N}_2$.

We define  the \emph{primitive part}  $\mathcal{Z}(M^\flat)^\circ\subseteq \mathcal{Z}(M^\flat)$ to be the primitive part $\mathcal{Z}(y)^\circ\subseteq \mathcal{Z}(y)$ under the above identification. Since $\val(y)=\val(M^\flat)$, we have a decomposition as Cartier divisors on $\mathcal{Z}(M_{n-2})$,
\begin{equation}
  \label{eq:ZMflatdecomp}
  \mathcal{Z}(M^\flat)\simeq\sum_{i=0}^{\lfloor \val(M^\flat)/2\rfloor} \mathcal{Z}_{\val(M^\flat)-2i},
\end{equation}
 and we can characterize $\mathcal{Z}(M^\flat)^\circ$ as the unique component of $\mathcal{Z}(M^\flat)$ isomorphic to $\mathcal{Z}_{\val(M^\flat)}$ (the component of the maximal degree). In particular, $\mathcal{Z}(M)^\circ$ is independent of the choice of the self-dual lattice $M_{n-2}$ and we have
 \begin{equation}
   \label{eq:lengthsquasicanonical}
   \deg_{\OFb}(\mathcal{Z}(M^\flat)^\circ)=
   \begin{cases}
     1, & t(M^\flat)=0, \\
     \vol(M^\flat)^{-1}(1+q^{-1}), & t(M^\flat)=1.
   \end{cases}
 \end{equation} 

\begin{theorem}\label{thm:horizontal}
Let $L^\flat\subseteq\mathbb{V}_n$ be a hermitian $O_F$-lattice of rank $n-1$. Then 
\begin{equation}
  \label{eq:horizontal}
  \mathcal{Z}(L^\flat)_\sH=\bigcup_{L^\flat\subseteq M^\flat\subseteq (M^\flat)^\vee\atop t(M^\flat)\le1}\mathcal{Z}(M^\flat)^\circ.
\end{equation}
Moreover, the identity
\begin{equation}
  \label{eq:horizontalK}
\mathcal{O}_{\mathcal{Z}(L^\flat)_\sH}=\sum_{L^\flat\subseteq M^\flat\subseteq (M^\flat)^\vee\atop t(M^\flat)\le1}\mathcal{O}_{\mathcal{Z}(M^\flat)^\circ}
\end{equation}
 holds in $\Gr^{n-1}K_0^{\mathcal{Z}(L^\flat)_\sH}(\mathcal{N}_n)$.  
\end{theorem}

\begin{lemma}\label{lem:alldistinct}
The primitive cycles $\mathcal{Z}(M^\flat)^\circ$ on the right-hand-side of (\ref{eq:horizontal}) are all distinct.   
\end{lemma}

\begin{proof}
If not, suppose $\mathcal{Z}(M_1^\flat)^\circ=\mathcal{Z}(M_2^\flat)^\circ$. Let $M^\flat=M_1^\flat+M_2^\flat$, which also has type $t(M^\flat)\le 1$. Then by definition we have  $\mathcal{Z}(M^\flat)= \mathcal{Z}(M_1^\flat)\cap \mathcal{Z}(M_2^\flat)$.  By the assumption $\mathcal{Z}(M_1^\flat)^\circ=\mathcal{Z}(M_2^\flat)^\circ$ we know that $\mathcal{Z}(M_i^\flat)^\circ\subseteq \mathcal{Z}(M^\flat)$ (for $i=1,2$). So by \eqref{eq:ZMflatdecomp} the inclusion $\mathcal{Z}(M_i^\flat)^\circ\simeq\mathcal{Z}_{\val(M_i^\flat)}\subseteq \mathcal{Z}(M^\flat)$ implies that $\val(M_i^\flat)\le \val(M^\flat)$. But $M^\flat\supseteq M_i^\flat$, it follows that $M^\flat=M_i^\flat$, and so $M_1^\flat=M_2^\flat$, a contradiction.   
\end{proof}

By Lemma \ref{lem:alldistinct}, we know that (\ref{eq:horizontal}) implies (\ref{eq:horizontalK}). It is clear from construction that in (\ref{eq:horizontal}) the right-hand-side is contained in the left-hand-side. To show the reverse inclusion, we will use the Breuil modules and Tate modules.

\subsection{Breuil modules}\label{sec:breuil-modules}
First let us review the (absolute) Breuil modules (\cite{Breuil2000}, \cite[Appendix]{Kisin2006}, \cite[\S12.2]{Brinon-Conrad}).  Let $W=W(\bar k)$. Let $O_K$ be a totally ramified extension of $W$ of degree $e$ defined by an Eisenstein polynomial $E(u)\in W[u]$. Let $S$ be Breuil's ring, the $p$-adic completion of $W[u][\frac{E(u)^i}{i!}]_{i\ge1}$ (the divided power envelope of $W[u]$ with respect to the ideal $(E(u))$. The ring $S$ is local and $W$-flat, and $S/uS\cong W$. Let $\Fil^1S\subseteq S$ be the ideal generated by all $\frac{E(u)^i}{i!}$. Then $S/\Fil^1S\cong O_K$.  By Breuil's theorem, $p$-divisible groups $G$ over $O_K$ are classified by their Breuil modules $\sM(G)=\mathbb{D}(G)(S)$ (\cite[Proof of A.6]{Kisin2006}), where $\mathbb{D}(G)$ is the Dieudonn\'e crystal of $G$. It is a finite free $S$-module together with an $S$-submodule $\Fil^1\sM(G)$, and a $\phi_S$-linear homomorphism $\phi_\sM:\Fil^1\sM(G)\rightarrow G$ satisfying certain conditions. The classical Dieudonn\'e module $M(G_{\bar k})$ of the special fiber $G_{\bar k}$ is given by $\mathbb{D}(G_{\bar k})(W)=\mathbb{D}(G)(S) \otimes_S W=\sM(G)/u\sM(G)$, with Hodge filtration $\Fil^1M(G_{\bar k})$ equal to the image of $\Fil^1\sM(G)$. We also have $\mathbb{D}(G)(O_K)=\mathbb{D}(G)(S) \otimes_S O_K=\sM(G) \otimes_S O_K$.

  For $\varpi$-divisible $O_{F_0}$-modules, one has an analogous theory of relative Breuil modules (see \cite{Henniges2016}) by replacing $W=W(\bar k)$ with $\OFb=W_{O_{F_0}}(\bar k)$, and by defining $S$ to be the $\varpi$-adic completion of the $O_{F_0}$-divided power envelope (in the sense of \cite{Faltings2002}) of $\OFb[u]$ with respect to the ideal $(E(u))$.

\subsection{Tate modules}

Let $K$ be a finite extension of $\Fb$. Let $z\in \mathcal{N}_n(O_K)$ and let $G$ be the corresponding $O_F$-hermitian module of signature $(1,n-1)$ over $O_K$. Let $$L\coloneqq \Hom_{O_F}(T_p\bar{\mathcal{E}}, T_pG),$$ where $T_p(-)$ denotes the integral $p$-adic Tate modules.  Then $L$ is a \emph{self-dual} $O_F$-hermitian lattice of rank $n$, where the hermitian form $\{x,y\}\in O_F$ is defined to be $$(T_p\bar{\mathcal{E}}\xrightarrow{x}T_pG\xrightarrow{\lambda_G}T_pG^\vee\xrightarrow{y^\vee}T_p\bar{\mathcal{E}}^\vee\xrightarrow{\lambda_{\bar{\mathcal{E}}}^\vee}T_p\bar{\mathcal{E}})\in\End_{O_F}(T_p\bar{\mathcal{E}})\cong O_F.$$ 
There are two injective $O_F$-linear homomorphisms (preserving their hermitian forms)
\[
   \xymatrix{
	       &\Hom_{O_F}(\ov \CE,G) \ar[dl]_{i_K} \ar[dr]^{i_{\ov k}}\\
L =\Hom_{O_F}(T_p\bar{\mathcal{E}}, T_pG) &  &\mathbb{V}_n=\Hom_{O_F}^\circ(\barE, \mathbb{X}_n),
	}
\]
where the right map $i_{\ov k}$ is induced by the reduction to $\Spec \bar k$ and the framings $\rho_{\bar\CE}$ and $\rho_z: G_{\ov k}\rightarrow \mathbb{X}_n$  corresponding to $\ov \CE$ and $z\in \mathcal{N}_n(O_K)$ respectively. These extend to $F$-linear homomorphisms (still denoted by the same notation)
\begin{equation}\label{diag:SW}
   \xymatrix{
	       &\Hom_{O_F}^\circ(\ov \CE,G) \ar[dl]_{i_K} \ar[dr]^{i_{\ov k}}\\
L_F &  &\mathbb{V}_n.
	}
\end{equation}

\begin{lemma}\label{lem:tate} The following identity holds:
\begin{align}\label{eq:SW0}
\Hom_{O_F}(\ov \CE,G)=i_K^{-1}(L),
\end{align}

\end{lemma}
\begin{proof}
We may identify $\Hom_{O_F}^\circ(\ov \CE,G)$ as subspaces of the bottom two vector spaces. So $$i_K^{-1}(L)\cong L\cap \Hom_{O_F}^\circ(\ov \CE,G)$$ where the intersection is taken inside the $F$-vector space $L_F$. By \cite[Theorem 4, Corollary 1]{Tate1967}, $i_K$ induces an isomorphism $$\Hom_{O_F}(\ov \CE,G)\cong \Hom_{O_F[\Gamma_K]}(T_p\ov \CE,, T_pG),$$ where $\Gamma_K=\Gal(\ov K/K)$, and so an isomorphism $$\Hom_{O_F}^\circ(\ov \CE,G)\cong \Hom_{O_F[\Gamma_K]}(V_p\ov \CE,, V_pG),$$ where $V_p(-)$ denotes the rational $p$-adic Tate module. Thus we obtain
  \begin{align*}
    L\cap \Hom_{O_F}^\circ(\ov \CE,G)&\cong\Hom_{O_F}(T_p\ov \CE,T_pG)\cap \Hom_{O_F[\Gamma_K]}(V_p\ov \CE, V_pG)\\
    &=\Hom_{O_F[\Gamma_K]}(T_p\ov \CE,, T_pG)\\
    &\cong\Hom_{O_F}(\ov \CE,G),
  \end{align*}
which proves the result.
\end{proof}

  Let $M\subseteq \mathbb{V}_n$ be an $O_F$-lattice (of arbitrary rank). By definition we have $z\in \mathcal{Z}(M)(O_K)$ if and only if $M \subseteq i_{\bar k}(\Hom_{O_F}(\ov \CE,G))$. It follows from Lemma \ref{lem:tate} that $z\in\mathcal{Z}(M)(O_K)$ if and only if
  \begin{align}\label{eq:SW}
    M \subseteq i_{\ov k}( i_K^{-1}(L)).
  \end{align}

\subsection{Proof of Theorem \ref{thm:horizontal}}

Let $z\in \mathcal{Z}(L^\flat)(O_K)$ and let $G$ be the corresponding $O_F$-hermitian module of signature $(1,n-1)$ over $O_K$. By  \eqref{eq:SW}, we know that $$L^\flat\subseteq i_{\ov k}( i_K^{-1} (L)).$$ Define $M^\flat\coloneqq L^\flat_F \cap i_{\ov k}( i_K^{-1} (L))$. By   \eqref{eq:SW} again, we obtain that $z\in \mathcal{Z}(M^\flat)(O_K)$. 
Moreover, the diagram \eqref{diag:SW} induces an isomorphism  $$
\xymatrix{M^\flat\ar[r]^-\sim &L\cap i_K(i_{\ov k}^{-1}(L^\flat_F)).}$$
Set $ \BW=i_K(i_{\ov k}^{-1}(L^\flat_F))$. Then it has  the same dimension as  $L^\flat_F$. 
\begin{lemma}\label{lem:typele1}
  Assume that $L$ is a self-dual $O_F$-hermitian lattice of rank $n$ and $\BW\subset L_F$ is a vector subspace of dimension $n-1$. Let $M^\flat\coloneqq \BW\cap L$. Then $t(M^\flat)\le1$.
\end{lemma}

\begin{proof}
  Since $M^\flat=\BW\cap L$, we know that $L/M^\flat$ is a free $O_F$-module of rank one. Hence we may write $L=M^\flat+\langle x\rangle$ for some $x\in L$. Choose an orthogonal basis $\{e_1,\ldots,e_{n-1}\}$ of $M^\flat$ such that $(e_i,e_i)=\varpi^{a_i}$. The fundamental matrix of $\{e_1,\ldots,e_{n-1},x\}$ has the form $$T=
  \begin{pmatrix}
    \varpi^{a_1} & & & (e_1,x)\\
    & \varpi^{a_2} & & (e_2,x)\\
    && \ddots & \vdots\\
    (x,e_1) & (x,e_2) &\cdots & (x,x)
  \end{pmatrix}.$$ If $t(M^\flat)\ge2$ (i.e., at least two $a_i$'s are $>0$ ), then the rank of $T$ mod $\varpi$ is at most $n-1$, contradicting that $L$ is self-dual.
\end{proof}

It follows from Lemma \ref{lem:typele1} that $z\in\mathcal{Z}(M^\flat)(O_K)$ is a quasi-canonical lifting contained in the right-hand-side of (\ref{eq:horizontal}). By construction, $M^\flat$ is the largest lattice in $L_F^\flat$ contained in $i_{\ov k}( i_K^{-1}(L))$, thus in fact we have $z\in \mathcal{Z}(M^\flat)^\circ(O_K)$ by the equation \eqref{eq:SW}. Therefore the $O_K$-points of both sides of \eqref{eq:horizontal} are equal.

To finish the proof of Theorem \ref{thm:horizontal}, it remains to check that the $O_K[\varepsilon]$-points of both sides of \eqref{eq:horizontal} are equal (where $\varepsilon^2=0$). Namely, we would like to show that for each $z\in \mathcal{Z}(L^\flat)(O_K)$, there is a unique lift of $z$ in $\mathcal{Z}(L^\flat)(O_K[\varepsilon])$ . Let $\mathbb{D}(G)$ be the (covariant) $O_{F_0}$-relative Dieudonn\'e crystal of $G$. The action of $O_F$ via $\iota: O_F\rightarrow\End(G)$ induces an action $O_F \otimes_{O_{F_0}}O_K \simeq O_K \oplus O_K$ on $\mathbb{D}(G)(O_K)$, and hence a $\mathbb{Z}/2 \mathbb{Z}$-grading on $\mathbb{D}(G)(O_K)$. Let $\sA=\gr_0\mathbb{D}(G)(O_K)$ be the 0th graded piece of $\mathbb{D}(G)(O_K)$, a free $O_K$-module of rank $n$. By the Kottwitz signature condition, it is equipped with an $O_K$-hyperplane $\sH=\Fil^1\sA:=\Fil^1\mathbb{D}(G)(O_K)\cap \sA$. The $O_K$-hyperplane $\sH$ contains the image of $L^\flat$ under the identification of \cite[Lemma 3.9]{Kudla2011}. Let $\wit\sA=\gr_0\mathbb{D}(G)(O_K[\varepsilon])$. Since the kernel of $O_K[\varepsilon]\rightarrow O_K$ has a nilpotent divided power structure, by Grothendieck--Messing theory, a lift $\tilde z\in \mathcal{Z}(L^\flat)(O_K[\varepsilon])$ of $z$ corresponds to an $O_K[\varepsilon]$-hyperplane $\wit \sH$ of $\wit\sA$ lifting the $O_K$-hyperplane $\sH$ of $\sA$  and contains the image of $L^\flat$ in $\wit\sA$ (cf. \cite[Theorem 3.1.3]{Li2017}, \cite[Proof of Proposition 3.5]{Kudla2011}). By Breuil's theorem (\S\ref{sec:breuil-modules}), the image of $L^\flat$ in $\gr_0\mathbb{D}(G)(S)$ has rank $n-1$ over $S$ and thus its image in the base change $\sA$ has rank $n-1$ over $O_K$, we know that there is a unique choice of such hyperplane $\wit\sH$. Hence the lift $\tilde z$ is unique as desired.

\subsection{Relation with the local density} Notice that $\deg_\OFb(\mathcal{Z}(L^\flat)_\sH)$ is equal to the degree of the 0-cycle $\mathcal{Z}(L^\flat)_{\Fb}$ in the generic fiber $\mathcal{N}_{\Fb}$ of the Rapoport--Zink space, which may be interpreted as a \emph{geometric} intersection number on the generic fiber. We have the following identity between this geometric intersection number and a local density.

\begin{corollary}\label{cor:genericfiber}
$\deg_{\OFb}(\mathcal{Z}(L^\flat)_\sH)=\vol(L^\flat)^{-1}\Den((-q)^{-1}, L^\flat)=\Den(-q, L^\flat)$.
\end{corollary}

\begin{proof}
  The first equality follows immediately from Theorem \ref{thm:horizontal}, Equation (\ref{eq:lengthsquasicanonical}), and Equation (\ref{eq:valueat1}). The second equality follows from the functional equation (\ref{eq: FE}).
\end{proof}

\begin{remark}
  Using the $p$-adic uniformization theorem (\S\ref{sec:p-adic-unif}) and the flatness of the horizontal part of the global Kudla--Rapoport cycles, one may deduce from Corollary \ref{cor:genericfiber} an identity between the geometric intersection number (i.e. the degree) of a special 0-cycle on a compact Shimura variety associated to $\U(n,1)$ and the value of a Fourier coefficient of a \emph{coherent} Siegel Eisenstein series on $\U(n,n)$ at the near central point $s=1/2$. This should give a different proof  (of a unitary analogue) of a theorem of Kudla \cite[Theorem 10.6]{Kudla1997} for compact orthogonal Shimura varieties. 
\end{remark}

\section{Vertical components of Kudla--Rapoport cycles}

\subsection{The support of the vertical part $\CZ(L^\flat)_\sV$} \label{sec:supp-vert-part} Let $L^\flat$ be an $O_F$-lattice of rank $n-1$ in $\mathbb{V}_n$. Recall that $\mathcal{Z}(L^\flat)_\sV$ is the vertical part of the Kudla--Rapoport cycle $\mathcal{Z}(L^\flat)\subseteq \mathcal{N}_n$ (\S\ref{sec:horiz-vert-parts}).

\begin{lemma}\label{prop:supp Z(L) V}
$\mathcal{Z}(L^\flat)_\sV$ is supported on $\mathcal{N}_n^\mathrm{red}$, i.e., $\mathcal{O}_{\mathcal{Z}(L^\flat)_\sV}$ is annihilated by a power of the ideal sheaf of $\mathcal{N}_n^\mathrm{red}\subseteq \mathcal{N}_n$.
\end{lemma}

\begin{proof}
  If not, we may find an affine formal curve $C=\Spf R$ (where $R$ is an integral domain) as a closed formal subscheme of $\mathcal{Z}(L^\flat)_\sV$  such that $C^\mathrm{red}$ consists of a single point $z\in \mathcal{N}_n^\mathrm{red}$. The universal $p$-divisible $O_{F_0}$-module  $X^\mathrm{univ}$ over $\mathcal{N}_n$ pulls back to a $p$-divisible $O_{F_0}$-module  $\mathcal{X}_\eta$ over a geometric generic point $\eta$ of $C$. Since $C^\mathrm{red}=\{z\}$, we know that the $p$-divisible $O_{F_0}$-module $\mathcal{X}_\eta$ is not supersingular (otherwise a nonempty open subset of $C$ is contained in $\mathcal{N}_n^\mathrm{red}$ by the definition of $\mathcal{N}_n$). On the other hand, if $L^\flat=\langle x_1,\ldots,x_{n-1}\rangle$, then $\mathcal{X}_\eta$ admits $n-1$ linearly independent special homomorphisms $\tilde x_i:\bar{\mathcal{E}}_\eta\rightarrow \mathcal{X}$, which gives rise to a homomorphism $$(\tilde x_1,\ldots, \tilde x_{n-1}): \bar{\mathcal{E}}_\eta^{n-1}\rightarrow \mathcal{X}_\eta.$$ Since $\eta$ has characteristic $p$, by the Dieudonn\'e--Manin classification we know that $\mathcal{X}_\eta$ is isogenous to $\bar{\mathcal{E}}_\eta^{n-1}\times \mathcal{X}_\eta'$ for  $\mathcal{X}'_\eta$ a $p$-divisible $O_{F_0}$-module of relative height 2 and dimension 1 with an $O_F$-action. It follows that $\mathcal{X}_\eta'$ is supersingular, and so $\mathcal{X}_\eta$ itself is also supersingular, a contradiction.
\end{proof}

\subsection{Horizontal and vertical parts of $^\BL\CZ(L^\flat)$}\label{sec:horiz-vert-parts-1}

Since $\CZ(L^\flat)_\sH$ is one dimensional, the intersection $\CZ(L^\flat)_\sH\cap \CZ(L^\flat)_\sV$ must be zero dimensional (if non-empty). It follows that there is a decomposition of the $(n-1)$-th graded piece 
\begin{align}\label{eq:gradedK0}
 \Gr^{n-1}K_0^{\CZ(L^\flat)}(\mathcal{N}_n)= \Gr^{n-1}K_0^{\CZ(L^\flat)_\sH}(\mathcal{N}_n)\oplus \Gr^{n-1}K_0^{\CZ(L^\flat)_\sV}(\mathcal{N}_n).
\end{align}

\begin{definition}
The decomposition (\ref{eq:gradedK0})  induces a decomposition of the derived Kudla--Rapoport cycle into \emph{horizontal} and \emph{vertical} parts
\begin{align*}
^\BL\CZ(L^\flat)=\,^\BL\CZ(L^\flat)_\sH+ ~ ^\BL\CZ(L^\flat)_\sV\in \Gr^{n-1}K_0^{\CZ(L^\flat)}(\mathcal{N}_n).
\end{align*}
From this decomposition, we see that even though the vertical part $\mathcal{Z}(L^\flat)_\sV$ depends on the choice of an integer $N\gg0$ (\S\ref{sec:horiz-vert-parts}), the element $^\BL\CZ(L^\flat)_\sV\in \Gr^{n-1}K_0^{\CZ(L^\flat)}(\mathcal{N}_n)$ is canonical and independent of the choice of $N$.

Since $\CZ(L^\flat)_\sH$ has the expected dimension, the first summand $^\BL\CZ(L^\flat)_\sH$ is represented by the structure sheaf of $\CZ(L^\flat)_\sH$ by \cite[Lemma B.2 (ii)]{Zhang2019}. Abusing notation we shall write the sum as
\begin{align}\label{eq:H+V}
^\BL\CZ(L^\flat)=\CZ(L^\flat)_\sH+ ~ ^\BL\CZ(L^\flat)_\sV.
\end{align}
\end{definition}

By Proposition \ref{prop:supp Z(L) V}, we have a change-of-support homomorphism 
$$
\xymatrix{\Gr^{n-1}K_0^{\CZ(L^\flat)_\sV}(\mathcal{N}_n)\ar[r]& \Gr^{n-1}K_0^{\mathcal{N}_n^\mathrm{red}}(\mathcal{N}_n).}
$$
Abusing notation we will also denote the image of $^\BL\CZ(L^\flat)_\sV$ in the target by the same symbol.

\begin{corollary}\label{cor:curves}
There exist finitely many curves $C_i\subseteq \mathcal{N}_n^\mathrm{red}$ and $\mult_{C_i}\in \mathbb{Q}$ such that $$^\mathbb{L}\mathcal{Z}(L^\flat)_\sV=\sum_i \mult_{C_i}[\mathcal{O}_{C_i}]\in \Gr^{n-1}K_0^{\mathcal{N}_n^\mathrm{red}}(\mathcal{N}_n).
$$
\end{corollary}

\begin{proof}
It follows immediately from Lemma \ref{prop:supp Z(L) V}, where the finiteness of such curves $C_i$ is due to Lemma \ref{lem:ZLnoetherian}. 
\end{proof}

\subsection{The Tate conjecture for certain Deligne--Lusztig varieties} Consider the generalized Deligne--Lusztig variety $Y_{d}\coloneqq Y_V$ and the classical Deligne--Lusztig $Y_d^\circ\coloneqq Y_V^\circ$ as defined in \S \ref{sec:gener-deligne-luszt}, where $V$ is the unique $k_F/k$-hermitian space of dimension $2d+1$. Recall that we have a stratification $$Y_d=\bigsqcup_{i=0}^d X_{P_i}(w_i).$$ Let $$X_i^\circ\coloneqq X_{P_i}(w_i), \quad X_i\coloneqq \overline{X_i^\circ}=\bigsqcup_{m=0}^i X_m^\circ.$$ Then $X_i^\circ$ is a disjoint union of the classical Deligne--Lusztig variety $Y_i^\circ$, and each irreducible component of $X_i$ is isomorphic to $Y_i$.

For any $k_F$-variety $S$, we write  $H^{j}(S)(i)\coloneqq H^{j}(S_{\bar k_F}, \overline{\mathbb{Q}_\ell}(i))$ ($\ell\ne p$ is a prime). Let $\Fr=\mathrm{Fr}_{k_F}$ be the $q^2$-Frobenius acting on $H^{j}(S)(i)$.

\begin{lemma} \label{lem:lusztig}For any $d, i\ge0$ and $s\ge1$, the action of $\Fr^s$ on the following cohomology groups are semisimple, and the space of $\Fr^s$-invariants is zero when $j\ge1$.
  \begin{altenumerate}
  \item\label{item:1} $H^{2j}(Y_d^\circ)(j)$.
  \item\label{item:2} $H^{2j}(X_i^\circ)(j)$.
  \item\label{item:3} $H^{2j}(Y_d-X_i)(j)$.
  \end{altenumerate}
\end{lemma}

\begin{proof}
  \begin{altenumerate}
  \item   By \cite[7.3 Case $^2A_{2n}$]{Lusztig1976/77} (notice the adjoint group assumption is harmless due to \cite[1.18]{Lusztig1976/77}), we know that there are exactly $2d+1$ eigenvalues of $\Fr$ on $H^{*}_c(Y_d^\circ)$, given by $(-q)^{m}$ where $m=0,1,\ldots, 2d$, and the eigenvalue $(-q)^m$ appear exactly in $H^j_c(Y_d^\circ)$ for $j=\lfloor m/2\rfloor +d$. By the Poincare duality, we have a perfect pairing  $$H^{2d-j}_c(Y_d^\circ) \times H^{j}(Y_d^\circ)(d)\rightarrow H^{2d}_c(Y_d^\circ)(d)\simeq\overline{\mathbb{Q}_\ell}.$$ Thus the eigenvalues of $\Fr$ on $H^{2j}(Y_d^\circ)(j)$ are given by $q^{2(d-j)}$ times the inverse of the eigenvalues in $H_c^{2(d-j)}(Y_d^\circ)$, which is equal to $\{(-q)^{2j}, (-q)^{2j-1}\}$ when $d\ge 2j>0$, and $\{(-q)^{2j}=1\}$ when $j=0$. Hence the eigenvalue of $\Fr^s$ is never equal to 1 when $j\ge1$.  The semisimplicity of the action of $\Fr^s$ follows from \cite[6.1]{Lusztig1976/77}.
  \item It follows from (\ref{item:1}) since $X_i^\circ$ is a disjoint union of $Y_i^\circ$.
  \item It follows from (\ref{item:2}) since $Y_d-X_i=\bigsqcup_{m=i+1}^d X_m^\circ$.\qedhere
  \end{altenumerate}
\end{proof}

\begin{theorem}\label{thm:tate}
  For any $0\le i \le d$ and any $s\ge1$, we have
  \begin{enumerate}
  \item   The space of Tate classes $H^{2i}(Y_d)(i)^{\Fr^s=1}$ is spanned by the cycle classes of the irreducible components of $X_{d-i}$. In particular, the Tate conjecture (\cite[Conjecture 1]{Tate1965}, or \cite[Conjecture $T^i$]{Tate1994}) holds for $Y_d$.
  \item\label{item:eigen1}   Let $H^{2i}(Y_d)(i)_1\subseteq H^{2i}(Y_d)(i)$ be the the generalized eigenspace of $\Fr^s$ for the eigenvalue 1. Then $H^{2i}(Y_d)(i)_1=H^{2i}(Y_d)(i)^{\Fr^s=1}$.
  \end{enumerate}

\end{theorem}

\begin{proof}
The assertion is clear when $i=0$. Assume $i>0$. Associated to the closed embedding $X_{d-i}\hookrightarrow Y_d$ we have a long exact sequence
  \begin{equation}
    \label{eq:longexact}
    \cdots\rightarrow H^{j}_{X_{d-i}}(Y_d)\rightarrow H^j(Y_d)\rightarrow H^j(Y_d-X_{d-i})\rightarrow H_{X_{d-i}}^{j+1}(Y_d)\rightarrow \cdots
  \end{equation}
 Take $j=2i$. We have a Gysin isomorphism
 \begin{equation}
   \label{eq:gysin}
   \bigoplus_{Z\in \mathrm{Irr}(X_{d-i})}H^0(Z)\xrightarrow{\sim}H_{X_{d-i}}^{2i}(Y_d)(i),
 \end{equation}
 where the sum runs over all the irreducible components of $X_{d-i}$. By \eqref{eq:gysin} and Lemma \ref{lem:lusztig}, the actions of $\Fr^s$ on $H_{X_{d-i}}^{2i}(Y_d)$ and $H^{2i}(Y_d-X_{d-i})$ are semisimple, and thus $$H^{2i}_{X_{d-i}}(Y_d)(i)_1=H^{2i}_{X_{d-i}}(Y_d)(i)^{\Fr^s=1}, \quad H^{2i}(Y_d-X_{d-i})(i)_1=H^{2i}(Y_d-X_{d-i})(i)^{\Fr^s=1}.$$ Taking the $i$-th Tate twist and taking the generalized eigenspace of $\Fr^s$ for the eigenvalue 1 of (\ref{eq:longexact}), we obtain a 3-term exact sequence $$H^{2i}_{X_{d-i}}(Y_d)(i)^{\Fr^s=1}\rightarrow H^{2i}(Y_d)(i)_1\rightarrow H^{2i}(Y_d-X_{d-i})(i)^{\Fr^s=1}.$$ The last term is 0 by Lemma \ref{lem:lusztig} (\ref{item:3}) as $i>0$. Thus $H^{2i}(Y_d)(i)_1=H^{2i}(Y_d)(i)^{\Fr^s=1}$, and we have a surjection onto Tate classes $$\bigoplus_{Z\in\mathrm{Irr}(X_{d-i})}H^0(Z)\simeq H^{2i}_{X_{d-i}}(Y_d)(i)^{\Fr^s=1}\twoheadrightarrow H^{2i}(Y_d)(i)^{\Fr^s=1}.$$ So $H^{2i}(Y_d)(i)^{\Fr^s=1}$ is spanned by the cycle classes of the irreducible components of $X_{d-i}$.
\end{proof}

Let us come back to the situation of \S \ref{sec:supp-vert-part}.

\begin{corollary} \label{cor:ZV}
  There exists finitely many Deligne--Lusztig curves $C_i\subseteq \mathcal{N}_n^\mathrm{red}$ and $\mult_{C_i}\in \mathbb{Q}$, such that for any $x\in \mathbb{V}_n\setminus L^\flat_F$,
  $$
  \chi(\CN_n,\,^\mathbb{L}\mathcal{Z}(L^\flat)_\sV\jiao\mathcal{Z}(x))=\sum_i \mult_{C_i}\cdot  \chi(\CN_n,\,C_i\jiao \mathcal{Z}(x)).
  $$ 
\end{corollary}

\begin{proof}
  By the Bruhat--Tits stratification of $\mathcal{N}_n^\mathrm{red}$ (\S \ref{sec:bruh-tits-strat}), any curve $C$ in $\mathcal{N}_n^\mathrm{red}$ lies in some Deligne--Lusztig variety $\mathcal{V}(\Lambda)\cong Y_d$. By Theorem \ref{thm:tate} (for $i=d-1$), the cycle class of such a curve $C$ can be written as a $\mathbb{Q}$-linear combination of the cycle classes of Deligne--Lusztig curves on $\mathcal{V}(\Lambda)$. Notice that $\chi(\mathcal{N}, C \jiao \mathcal{Z}(x))$ only depends on the cycle class of $C$. In fact, since $\mathcal{Z}(x)$ is a Cartier divisor on $\mathcal{N}_n$, $\mathcal{V}(\Lambda)\jiao \mathcal{Z}(x)$ is explicitly represented by the two-term complex of locally free sheaves $$[\mathcal{O}_{\mathcal{N}_n}(-\mathcal{Z}(x))|_{\mathcal{V}(\Lambda)}\rightarrow \mathcal{O}_{\mathcal{N}_n}|_{\mathcal{V}(\Lambda)}]\in \mathrm{F}^1K_0(\mathcal{V}(\Lambda)).$$ Hence by \cite[(B.3)]{Zhang2019}, $\chi(\mathcal{N}, C \jiao \mathcal{Z}(x))$ only depends on the image of $\mathcal{O}_C$ in $\Gr^{d-1}K_0(\mathcal{V}(\Lambda))_\mathbb{Q}\cong \Ch^{d-1}(\mathcal{V}(\Lambda))_\mathbb{Q}$, and the image of $\mathcal{V}(\Lambda)\jiao \mathcal{Z}(x)$ in $\Gr^1K_0(\mathcal{V}(\Lambda))_\mathbb{Q}\cong\Ch^1(\mathcal{V}(\Lambda))_\mathbb{Q}$. 
  As the cycle class map intertwines the intersection product and the cup product, cf. (\ref{eq:chcl}), we know that $\chi(\mathcal{N}, C \jiao \mathcal{Z}(x))$ only depends on the cycle class of $C$. The result then follows from Corollary~\ref{cor:curves}.
\end{proof}

\subsection{The vertical cycle in the case $n=3$, and Theorem \ref{thm: main} in the case $n=2$}

Now let $n=3$, and let $L^\flat\subset\BV_3$ be a rank two integral lattice. Denote by $\Ver^t(L^\flat)$
the set of vertex lattices $\Lambda$ of type $t$ containing $L^\flat$. For any integral lattice $\Lambda$, we denote $L^\flat_{\Lambda}\coloneqq L^\flat_F\cap \Lambda$, which is an integral lattice in $L^\flat_F$.
\begin{theorem}\label{thm:n=2}
\begin{altenumerate}
\item\label{item:n=21} Let $L^\flat\subset\BV_3$ be a rank two lattice. Then the following identity
$$
\CZ(L^\flat)_\sV=\sum_{\Lambda\in \Ver^3(L^\flat)}\mult_{L^\flat}(\Lambda)\cdot{\CV(\Lambda)},
$$
holds in $\Gr^{2}K_0^{\CZ(L^\flat)_\sV}(\mathcal{N}_3)$,
where
$$
\mult_{L^\flat}(\Lambda)=\#\{L'^\flat \mid L^\flat\subset L'^\flat\subset L^\flat_\Lambda\}.
$$
Similarly, the following identity $$^\mathbb{L}\mathcal{Z}(L^\flat)_\sV=\sum_{\Lambda\in \Ver^3(L^\flat)}\mult_{L^\flat}(\Lambda)\cdot{\CO_{\CV(\Lambda)}}
$$ holds in $\Gr^{2}K_0^{\CZ(L^\flat)_\sV}(\mathcal{N}_3)$.
\item\label{item:n=22}  Theorem \ref{thm: main} holds when $n=2$, i.e., $\Int(L^\flat)=\pDen(L^\flat)$ for all $L^\flat\subset \BV_2$.
\end{altenumerate}
\end{theorem}

\begin{remark} \begin{altenumerate}
\item
Part (\ref{item:n=22}) is known by \cite[Theorem 1.1]{Kudla2011}. However, our proof is logically independent from {\em loc. cit.}.  
\item Later we will only need  (in the proof of Lemma \ref{int Lam}) a very special case of part (\ref{item:n=21}) of Theorem \ref{thm:n=2}, i.e., the minuscule case in the proof below. 
\end{altenumerate}
\end{remark}
We first establish two lemmas. The first one is trivial and we state it because it will be used repeatedly.
\begin{lemma}\label{lem: simple}
 Let $e\in\BV_3$ be a unit-normed vector. Then there is a unique vertex lattice $\Lambda_e$ of type $1$ in $\BV_3$ containing $e$. Moreover, if $L\subset \BV_3$ is any integral lattice (not necessarily of full rank) and $e\in L$, then $L\subset \Lambda_e$.  
 \end{lemma}
 \begin{proof}The hermitian space $\pair{e}^\perp\subset\BV_3$ is two dimensional and non-split, hence has a unique maximal integral lattice $\Lambda^\flat$ (consisting of all vectors with integral norms). Then we see that $\Lambda_e=\pair{e}\obot \Lambda^\flat$ has the desired property.
 \end{proof}
\begin{lemma}\label{lem: refl}
Fix $\Lambda_0\in \Ver^3(L^\flat)$. Then there exists a vector $e\in \BV_3$ with unit norm such that, when denoting  $M=\pair{e}$,
  \begin{altenumerate}
\item\label{item:refl1}
The lattice $\Lambda_0+M$ is equal to the vertex lattice $\Lambda_e$ of type $1$ in Lemma \ref{lem: simple}, and $\Lambda_e=L^\flat_{\Lambda_e}\oplus M$;
\item\label{item:refl2}
$\Lambda_0=L^\flat_{\Lambda_0}+\varpi M$ and $L^\flat_{\Lambda_0}=L^\flat_{\Lambda_e}$;
\item\label{item:refl3}
  For any other $\Lambda\neq \Lambda_0$ in $\Ver^3(L^\flat+\varpi M  )$, the lattice $L_{\Lambda}^\flat$ is equal to $L_{\varpi\Lambda^\vee_e}^\flat$ and is a sub-lattice of $L_{\Lambda_0}^\flat=L_{\Lambda_e}^\flat$  of colength one;
\item\label{item:refl4}
  For any lattice $L'^\flat$ such that  $L^\flat\subset L'^\flat\subset L_{\Lambda_e}^\flat$,  we have
$$
t (L'^\flat\oplus M)=\begin{cases}
2,& \text{if } L'^\flat\subset L_{\varpi\Lambda^\vee_e}^\flat,  \\
1,&\text{otherwise.} \end{cases}
$$
\end{altenumerate}

\end{lemma}
\begin{remark}
Before presenting the proof, we indicate the geometric picture of the lemma. The reduced scheme $ \CZ(L^\flat)^{\red}$ of $\CZ(L^\flat)$ is a (connected, a fact we do not need) union of the curves $\CV(\Lambda)$ for $\Lambda\in \Ver^3(L^\flat)$. The lemma implies that, on any given irreducible component $\CV(\Lambda_0)$, there exists a (superspecial) point $\CV(\Lambda_e)$, such that among all the curves $\CV(\Lambda)\subset \CZ(L^\flat)^{\red}$ passing through $\CV(\Lambda_e)$ (noting that such $\Lambda$ necessarily belongs to $\Ver^3(L^\flat+\varpi M)$ due to the implication $\CV(\Lambda_e)\subset\CV(\Lambda)\imp e\in \Lambda^\vee\imp \varpi e \in \Lambda$), the given one $\CV(\Lambda_0)$ has the (strictly) largest  associated lattice $L^\flat_{\Lambda_0}$. This suggests the possibility to determine the multiplicity $\mult_{L^\flat}(\Lambda)$ by induction on $[L^\flat_{\Lambda_0}:L^\flat]$.
\end{remark}

\begin{proof}We pick a vector $x$ of valuation one in $L_{\Lambda_0}^\flat$ (such $x$ necessarily exists) and denote by $E$ the rank one lattice $\pair{x}$. Denote by $M'$ its orthogonal complement in $L_{\Lambda_0}^\flat$, so that
$$
L_{\Lambda_0}^\flat=E\obot M'.
$$

We {\em claim} that there exists a vector $e\perp E$ such that 
  \begin{altenumerate}
\item[(a)]  The norm of $e$ is a unit;
\item[(b)]  Denoting $M=\pair{e}$, then the rank two lattice $M'\oplus M$ is self-dual;
\item[(c)]   $\Lambda_0= E\obot(M'\oplus \varpi M)$.
\end{altenumerate}
To show the claim, we consider the two dimensional subspace $\pair{x}^\perp_F$. From $\val(x)=1$, it follows that $\pair{x}^\perp_F$ is a split Hermitian space, and $\Lambda_0$ is an orthogonal direct sum $E\obot E^\perp$ for  a vertex lattice $E^\perp$ of type $2$ in $\pair{x}^\perp_F$. The sub-lattice $M'$ is saturated in $E^\perp$. Consider the two dimensional $k_F$-vector space $V\coloneqq  \varpi^{-1}E^\perp/ E^\perp$ with the induced hermitian form. The $q+1$ isotropic lines in $V$ are bijective to self-dual lattices containing $ E^\perp$. Since $q+1>1$, there exists an isotropic line not containing the image of $\varpi^{-1}M'$ in $V$, or equivalently, there exists a self-dual lattice $\Xi\subset \pair{x}^\perp_F $ containing $ E^\perp$ but not $\varpi^{-1}M'$ (i.e., $M'$ remains saturated in $\Xi$). Finally, we choose a unit-normed $e$ lifting a generator of the free $O_F$-module   $\Xi/M'$ of rank one. It is easy to verify that such a vector $e$ satisfies all the conditions $(a), (b)$ and $(c)$ and this proves the claim. 

We fix such a vector $e$ and we now verify that it satisfies the desired conditions. Parts (\ref{item:refl1}) and (\ref{item:refl2}) are clear by the claim above. Now let $\Lambda$ be a lattice in $\Ver^3(L^\flat+\pair{\varpi e})$. Then $\Lambda+\pair{e}$ is an integral lattice containing a unit-normed vector, hence is a vertex lattice of type $1$ (recall from \S\ref{ss:minu KR} that the type of a vertex lattice is always odd). Therefore, by Lemma \ref{lem: simple}, we obtain $\Lambda+\pair{e}=\Lambda_e$ for all $\Lambda\in \Ver^3(L^\flat+\pair{\varpi e})$.  Now assume that $\Lambda\neq \Lambda_0$. Then we obtain the following diagram
$$
\xymatrix@C=-2em@R=3em{
& &  \Lambda_e= E\obot(M'\oplus M)&  
\\ &\Lambda_0= E\obot(M'\oplus \varpi M) \ar@{^{(}->}[ur]&&\ar@{_{(}->}[ul] \hspace{3em}\Lambda \longleftarrow\joinrel\rhook  L^\flat_{\Lambda}\oplus  \varpi M 
\\ &&\varpi\Lambda_e^\vee=E\obot( \varpi M'\oplus  \varpi M) \ar@{_{(}->}[ul] \ar@{^{(}->}[ur]& 
\\ &&\varpi\Lambda_e=\varpi E\obot(\varpi M'\oplus  \varpi M) \ar@{^{(}->}[u]^1. & }
$$
It is easy to see that 
$$
E\obot \varpi M'\subset L^\flat_\Lambda\subset E\obot M',
$$
and hence either $L^\flat_\Lambda=E\obot  M'$ or $L^\flat_\Lambda=E\obot \varpi M'$. In the former case,  we must have $\Lambda\supset E\obot ( M' \oplus  \varpi M )=\Lambda_0$, contradicting $\Lambda\neq \Lambda_0$. This shows that  $L_{\Lambda}^\flat=E\obot \varpi M'= L_{ \varpi \Lambda^\vee_e}^\flat$, and hence completes the proof of (\ref{item:refl3}).

Let $L'^\flat\subset L_{\Lambda_e}^\flat=E\obot M'$. Then the type of $L'^\flat\oplus M$ is either $1$ or $2$. To show part (\ref{item:refl4}), we first assume that $L'^\flat\subset L_{ \varpi \Lambda^\vee_e}^\flat=E\obot \varpi M'$. Then we have $$t (L'^\flat\oplus M) \geq t (E\obot (\varpi M'\oplus M))= t (E)+t (\varpi M'\oplus M)$$
and $t (E)=1$.  Now noting that $M'\oplus M$ is self-dual, its proper sub-lattice $\varpi M'\oplus M$ can not be self-dual, and hence $t(\varpi M'\oplus M)\geq 1$. Since $M$ is self-dual, it follows that $t(L'^\flat\oplus M)<3$, and hence $t(L'^\flat\oplus M)=2$.

Finally we assume that $L'^\flat \subset  E\obot M'$ but $L'^\flat\not\subset  E\obot \varpi M'$,  then there must be a vector $u\in L'^\flat$ whose projection to $M'$ is a generator of $M'$. It follows that $\pair{u}\oplus M$ is a rank-two self-dual  sub-lattice of $L'^\flat \oplus M$, forcing the type $t(L'^\flat \oplus M)\leq 1$. Since $\BV_3$ is a non-split hermitian space, it does not contain any self-dual lattice of full rank. Therefore $t(L'^\flat \oplus M)=1$ and this completes the proof of (\ref{item:refl4}).
\end{proof}

\begin{proof}[Proof of Theorem \ref{thm:n=2}]
The proof is rather involved and will be divided into four steps: \begin{enumerate}
\item both parts hold when $t(L^\flat)\leq 1$.
\item part \eqref{item:n=21} holds for minuscule $L^\flat$ and we deduce \eqref{eq: Z(e) L} that will be used repeated later.
\item part \eqref{item:n=21} for $L^\flat$ with odd $\val(L^\flat)$ implies part \eqref{item:n=22} for the same $L^\flat$.
\item part \eqref{item:n=21} for all $L^\flat$ holds by induction on $\val(L^\flat)$.
\end{enumerate}

We start with a remark. By Lemma \ref{lem:two div}, $^\mathbb{L}\mathcal{Z}(L^\flat)_\sV\in \Gr^{2}K_0^{\mathcal{N}_3^\mathrm{red}}(\mathcal{N}_3)$ is represented by the class of $\CO_{[\mathcal{Z}(L^\flat)_\sV]}$. Therefore
\begin{align}\label{eq:LZ V}
^\mathbb{L}\mathcal{Z}(L^\flat)_\sV=\mathcal{Z}(L^\flat)_\sV=\sum_{\Lambda\in \Ver^3(L^\flat)}\mult_{L^\flat}(\Lambda)\cdot\CV(\Lambda)
\end{align}
where the multiplicity $\mult_{L^\flat}(\Lambda)$ is a positive integer to be determined.

We first prove that both parts hold in the special case $t(L^\flat)\leq 1$. In fact, if $t(L^\flat)\leq 1$, we may write $L^\flat=\pair{u}\obot \pair{e}$ for a unit-normed vector $e$. Then $\CZ(u)$ and $\CZ(e)\simeq\CN_2$ intersect transversely (cf. the discussion before \eqref{eq:ind Z(u)}). Therefore $\CZ(L^\flat)$ is flat over $\Spf \OFb$ and  $\CZ(L^\flat)_\sV=\,^\mathbb{L}\CZ(L^\flat)_\sV=0$. Noting that $\mult_{L^\flat}(\Lambda)=0$ for any vertex lattice $\Lambda\subset\BV_3$ of type $3$, we have proved part \eqref{item:n=21}.  Similarly, part (\ref{item:n=22}) is reduced to the case $n=1$ by the cancellation law Lemma \ref{lem: cancel} and \eqref{eq:cancel den}. By Example \ref{ex int rank=1} and \ref{ex rank=1} we have 
\begin{align}\label{eq: n=1}
\Int(L^\flat)=\frac{\val(L^\flat)+1}{2}=\pDen(L^\flat).
\end{align}

Next we consider the minuscule case of part (\ref{item:n=21}), i.e., when the fundamental invariants of  $L^\flat$ are $(1,1)$. Then $\Ver^3(L^\flat)$ consists of a single type $3$ lattice $\Lambda=L^\flat\obot \pair{u}$ for a vector $u$ of valuation one.   By Theorem \ref{thm:horizontal} the horizontal part is the sum of quasi-canonical lifting cycles $\CZ(L'^\flat)\simeq \CN_1$ corresponding to the $q+1$ self-dual  lattices $L'^\flat$ containing $L^\flat$. Therefore by \eqref{eq:KRstrat} we have an equality in $
\Gr^{2}K_0^{\CZ(L^\flat)}(\mathcal{N}_3)$,
\begin{align}\label{eq:minu}
\CZ(L^\flat)=m \cdot \CV(\Lambda)+ \sum_{L^\flat\subset L'^\flat=(L'^\flat) ^\vee}\CZ(L'^\flat),
\end{align}
where the multiplicity $m$ of $\CV(\Lambda)$ is a positive integer to be determined. Now let $x_1,x_2$ be an orthogonal basis  of $L^\flat$, so that $\val(x_1)=\val(x_2)=1$. Now choose vector $e\perp x_1$  such that $e$ has unit norm and $\pair{x_2}\oplus \pair{e}$ is a self-dual lattice. It follows that $L^\flat\oplus \pair{e}$ is a vertex lattice of type $1$, and that, for any $L'^\flat=(L'^\flat) ^\vee$, the strictly larger lattice $L'^\flat\oplus \pair{e}$ can not be integral. Therefore, $\CZ(e)$ does not intersect with any of the quasi-canonical lifting cycles $\CZ(L'^\flat)$ appearing in \eqref{eq:minu}. Now consider 
$$
\Int(L^\flat\oplus \pair{e})=\chi(\CN_3, \CZ(L^\flat)\jiao \CZ(e)).
$$ 
On one hand, by Lemma  \ref{lem: cancel} applied to the self-dual lattice $\pair{x_2}\oplus \pair{e}$, we obtain $\Int(L^\flat\oplus \pair{e})=\Int(\pair{x_1})=1$ by Example \ref{ex int rank=1}. On the other hand, using the decomposition \eqref{eq:minu}, we have 
$$
\Int(L^\flat\oplus \pair{e})=m\cdot \chi(\CN_3, \CV(\Lambda)\jiao \CZ(e)).
$$
We deduce that the multiplicity $m=1$ in \eqref{eq:minu}, and
\begin{align}\label{eq: Z(e) L0}
\chi(\CN_3, \CV(\Lambda)\jiao \CZ(e))=1.
\end{align}
We note that, choosing $L^\flat$ appropriately, the argument above shows that \eqref{eq: Z(e) L0} holds for any $\Lambda\in\Ver^3$ and a unit-normed $e$ such that $\Lambda+\pair{e}$ is an integral lattice (necessarily a vertex lattice of type $1$). Moreover, it is obvious that $\chi(\CN_3, \CV(\Lambda)\jiao \CZ(e))=0$ if $\Lambda+\pair{e}$ is not integral (equivalently $e\not\in \Lambda^\vee$). Therefore we obtain that, for any $\Lambda\in\Ver^3$ and any $e$ with unit-norm,
\begin{align}\label{eq: Z(e) L}
\chi(\CN_3, \CV(\Lambda)\jiao \CZ(e))=\begin{cases}
1,& e \in \Lambda^\vee\\
0,&e\not\in \Lambda^\vee.
\end{cases}
\end{align}
This equation will be repeated used later.

Next we show that part (\ref{item:n=22}) for $L^\flat\subset\BV_2$ (necessarily with odd $\val(L^\flat)$, as $\BV_2$ is non-split) follows from part (\ref{item:n=21}) with the same $L^\flat\subset\BV_3$.  Here we have implicitly fixed an isomorphism $\BV_3\simeq \BV_2\obot M_F$ and an embedding $\CN_2\to\CN_3$ of the form \eqref{eq:inc M} induced by a self-dual lattice $M=\pair{e}$  of rank one. Let $L'^\flat$ be a type one lattice containing $L^\flat$, then by Lemma \ref{lem: cancel} and \eqref{eq: n=1},
$$
\Int(L'^\flat\obot M)=\Int(L'^\flat)=\frac{\val(L'^\flat)+1}{2}.
$$ 
Let $L''^\flat$ be the unique lattice such that $L'^\flat\subset L''^\flat\subset (L'^\flat)^\vee$ and $L''^\flat/L'^\flat$ has length one. Then by part (\ref{item:n=21}) for $L'^\flat$ and $L''^\flat$ we have
$$
\CZ(L'^\flat)=\CZ(L'^\flat)_\sH,\quad \CZ(L''^\flat)=\CZ(L''^\flat)_\sH.
$$
By Theorem \ref{thm:horizontal} we obtain an equality in $
\Gr^{2}K_0^{\CZ(L'^\flat)}(\mathcal{N}_3)$,
$$
 \CZ(L'^\flat)=\CZ(L''^\flat)+\CZ(L'^\flat)^\circ,
$$where $ \CZ(L'^\flat)^\circ$ is the associated quasi-canonical lifting cycle (cf. \S\ref{sec:quasi-canon-lift}).
It follows that
\begin{align}\label{eq: qc M}
\chi(\CN_3,\CZ(M)\jiao \CZ(L'^\flat)^\circ)=\Int(L'^\flat)-\Int(L''^\flat)=1.
\end{align}
Therefore by Theorem \ref{thm:horizontal} we obtain
\begin{align}\label{int n=2 H}
\chi(\CN_3,\CZ(M)\jiao \CZ(L^\flat)_\sH)=\#\{ L'^\flat\mid L^\flat \subset L'^\flat \subset(L'^\flat)^\vee, t(L'^\flat)=1 \}.
\end{align}
By part (\ref{item:n=21}) for $L^\flat$ (note that we are assuming this part to hold), and by \eqref{eq: Z(e) L}, we obtain  
\begin{align*}
\chi(\CN_3,\CZ(M)\jiao \CZ(L^\flat)_\sV)=&\sum_{\Lambda\in \Ver^3(L^\flat)\atop M\subset\Lambda^\vee  }\mult_{L^\flat}(\Lambda)\notag\\=&
\sum_{L^\flat\subset L'^\flat \subset(L'^\flat)^\vee\atop t(L'^\flat)=2} \#\{\Lambda\in \Ver^3(L^\flat)\mid L'^\flat\subset \Lambda, M\subset\Lambda^\vee  \}.
\end{align*}
Here, for $\Lambda\in\Ver^3$, the condition $M=\pair{e}\subset\Lambda^\vee$ (i.e., $M+\Lambda$ being integral) is equivalent to  $\Lambda\subset \Lambda_e$, where  $\Lambda_e$ is the lattice in Lemma \ref{lem: simple}. Note that $\Lambda_e=M\obot \Lambda^\flat$ where  $\Lambda^\flat$ is the unique maximal integral lattice in $M_F^\perp$.    If $L'^\flat$ is of type $2$, then $\pair{L'^\flat, \Lambda^\flat}\subset \varpi O_F$ (we leave the proof to the reader), or equivalently $L'^\flat\subset \varpi(\Lambda^\flat)^\vee$. Therefore any $L'^\flat$ of type $2$ is automatically contained in $\varpi\Lambda_e^\vee$, hence contained in any type $3$ vertex lattice $\Lambda\subset \Lambda_e$. It follows that the condition $L'^\flat\subset \Lambda$ in the sum above is redundant. Since there are $q+1$ of type $3$ lattices $\Lambda\subset \Lambda_e$, we obtain
\begin{align}\label{int n=2 V}
\chi(\CN_3,\CZ(M)\jiao \CZ(L^\flat)_\sV)=(q+1)\,\#\{ \text{integral } L'^\flat\mid L^\flat\subset L'^\flat, t(L'^\flat)=2\}.
\end{align}
Then the desired assertion for part (\ref{item:n=22}) for $L^\flat\subset\BV_2$ follows from \eqref{int n=2 H}, \eqref{int n=2 V} and the formula in Corollary \ref{cor: pDen}:
$$
\pDen(L^\flat)=\sum_{ L^\flat\subset L'^\flat \subset(L'^\flat)^\vee}\fkm(t( L'^\flat)), \text{ where }\fkm(t( L'^\flat))=\begin{cases}1,& t( L'^\flat)=1,\\q+1,& t( L'^\flat)=2.
\end{cases}
$$

Finally, we prove part (\ref{item:n=21}) for $L^\flat\subset \BV_3$ by induction on $\val(L^\flat)$. We have proved it when $t(L^\flat)=1$.
Now we fix $L^\flat\subset \BV_3$ of type $2$. By induction, we may assume that part (\ref{item:n=21}) holds for all $L'^\flat\subset \BV_3$ with $\val(L'^\flat)<\val(L^\flat)$. Note that, by what we have proved, the induction hypothesis also implies that part (\ref{item:n=22}) holds for all $L'^\flat\subset \BV_3$ with {\em odd} $\val(L'^\flat)<\val(L^\flat)$ (here $L'^\flat$ need not to be a lattice in $L^\flat_F$). 

To determine the multiplicity, we fix $\Lambda_0\in \Ver^3(L^\flat)$. Choose $e$ as in Lemma \ref{lem: refl} and follow the same notation. We {\em claim} that $\CZ(M)$ does not intersect the horizontal part $\CZ(L^\flat)_\sH$. In fact, if $\CZ(M)$ were to intersect non-emptily with $\CZ(L^\flat)_\sH$, then by Theorem \ref{thm:horizontal}, $\CZ(M)$ would intersect  non-emptily with  $\CZ(L'^\flat)^\circ$ for a type $1$ lattice $L'^\flat$  containing $L^\flat$.  This would imply that $M\oplus L'^\flat$ is an integral lattice, and hence $M\oplus L'^\flat\subset \Lambda_e$ by Lemma \ref{lem: simple}. Therefore $L'^\flat\subset L^\flat_{\Lambda_e}$ and hence the type of $ L^\flat_{\Lambda_e}$ is at most one. By Lemma \ref{lem: refl} \eqref{item:refl2}, $L^\flat_{\Lambda_e}=L^\flat_{\Lambda_0}$ is a sub-lattice of $\Lambda_0\in \Ver^3(L^\flat)$ and hence it has type $2$. Contradiction!

 It follows from the claim and \eqref{eq:LZ V} that
 $$
 \Int(L^\flat\oplus M)=\sum_{\Lambda\in \Ver^3(L^\flat)}\mult_{L^\flat}(\Lambda)\cdot \chi(\CN_3, \CV(\Lambda)\jiao \CZ(e)).
 $$
By \eqref{eq: Z(e) L}  and noting that $e\in \Lambda^\vee \Longleftrightarrow \varpi e\in\Lambda$, we obtain
\begin{align}\label{eq: Int M1}
\Int(L^\flat\oplus M)=\mult_{L^\flat}(\Lambda_0)+ \sum_{\Lambda\in \Ver^3(L^\flat+\varpi M)\atop \Lambda\neq\Lambda_0}\mult_{L^\flat}(\Lambda).
\end{align}
Note that the second summand is understood as zero if $\Lambda_0$ is the only element in $\Ver^3(L^\flat+\varpi M)$. 

By Lemma \ref{lem: refl} \eqref{item:refl1}, we obtain $
[\Lambda_e: L^\flat\oplus M]=[L^\flat_{\Lambda_e}:L^\flat]
$.  From $\val(\Lambda_e)=1$ and $\val(L^\flat_{\Lambda_e})\geq 2$, it follows that $\val(L^\flat\oplus M)<\val (L^\flat)$. Since $M$ is self-dual, we can decompose $L^\flat\oplus M=M^\perp\obot M$ orthogonally for a rank two lattice $M^\perp$ (note that here $M^\perp\subset\BV_3$ is not necessarily a lattice in $L^\flat_F$). Then $\val(M^\perp)$ is odd and $\val(M^\perp)<\val (L^\flat)$.  By induction hypothesis, part \ref{item:n=22} holds for $M^\perp$. It follows that, by the cancellation law Lemma \ref{lem: cancel} and \eqref{eq:cancel den},
$$
\Int(L^\flat\oplus M)=\Int(M^\perp)=\pDen(M^\perp)=\pDen(L^\flat\oplus M).
$$
By Corollary \ref{cor: pDen}, $\pDen(L^\flat\oplus M)$ is the sum 
$$
\pDen(L^\flat\oplus M)=\sum_{ L^\flat\oplus M\subset  L'\subset L'^\vee}\fkm(t( L')).
$$
Note that, by Lemma \ref{lem: simple}, any integral lattice $L'$ containing $M$ is necessarily contained in $\Lambda_e$. Since $\Lambda_e=L^\flat_{\Lambda_e}\oplus M$ by Lemma \ref{lem: refl} \eqref{item:refl1},  every $L'$ in the sum must be a direct sum $L'^\flat\oplus M$ for a unique integral lattice $L'^\flat$ lying between $L^\flat$ and $L^\flat_{\Lambda_e}$. In other words,  $\pDen(L^\flat\oplus M)$ is the sum 
$$
\#\{L'^\flat \mid L^\flat\subset L'^\flat\subset L^\flat_{\Lambda_e}\}+
q\cdot \#\{L'^\flat \mid L^\flat\subset L'^\flat\subset L^\flat_{\Lambda_e}, t(L'^\flat\oplus M)=2\}.
$$
By (\ref{item:refl2}), (\ref{item:refl3}), and (\ref{item:refl4}) of  Lemma \ref{lem: refl}, the above sum is equal to 
\begin{align}\label{eq: Int M2}
\#\{L'^\flat \mid L^\flat\subset L'^\flat\subset L^\flat_{\Lambda_0}\}+
\sum_{\Lambda\in \Ver^3(L^\flat+\varpi M)\atop \Lambda\neq \Lambda_0} \#\{L'^\flat \mid L^\flat\subset L'^\flat\subset L^\flat_{\Lambda}\}.
\end{align}
If $L^\flat_{\Lambda_0}=L^\flat$, then both \eqref{eq: Int M1} and \eqref{eq: Int M2} have only one term and we obtain 
$$\mult_{L^\flat}(\Lambda_0)=\#\{L'^\flat \mid L^\flat\subset L'^\flat\subset L^\flat_{\Lambda_0}\}=1.
$$
If $L^\flat_{\Lambda_0}\neq L^\flat$, by Lemma \ref{lem: refl} \eqref{item:refl3},  the index $[L^\flat_{\Lambda}: L^\flat]$ is strictly smaller than $[L^\flat_{\Lambda_0}:L^\flat]$ for $\Lambda\neq \Lambda_0$ in the sum of \eqref{eq: Int M2}. Therefore,   by induction on $[L^\flat_{\Lambda_0}:L^\flat]$, comparing \eqref{eq: Int M1} and \eqref{eq: Int M2} we finish the proof of the multiplicity formula for $\Lambda_0$, i.e., $\mult_{L^\flat}(\Lambda_0)= \#\{L'^\flat \mid L^\flat\subset L'^\flat\subset L^\flat_{\Lambda_0}\}$. This completes the proof.
\end{proof}

\begin{corollary}\label{cor:qc int}
Let $L^\flat\subset \BV_n$ be  an integral lattice of rank $n-1$ and type $t(L^\flat)\leq 1$. Then for any $x\in \BV_n\setminus L^\flat_F$, 
$$
\chi(\CN_n,\CZ(x)\jiao \CZ(L^\flat)^\circ)=\sum_{L^\flat+\pair{ x}\subset L'\subset  L'^\vee, \atop L'\cap L^\flat_F=L^\flat } \fkm(t(L')).
$$
\end{corollary}
\begin{proof}By assumption that $t(L^\flat)\leq 1$, there exists a self-dual lattice $M$ of rank $n-2$ such that $L^\flat=M\obot \pair{u}$. We then reduce the question to the case $n=2$, in which case $L^\flat=\pair{u}$. By Theorem \ref{thm:horizontal}, we have an equality in $
\Gr^{1}K_0^{\CZ(L^\flat)}(\mathcal{N}_2)$ (or as Cartier divisors on $\CN_2$ in this case),
$$ \CZ(L^\flat)= \CZ(\varpi^{-1}L^\flat)+ \CZ(L^\flat)^\circ.$$
 By Theorem \ref{thm:n=2} part (\ref{item:n=22}), 
$$
\Int(L^\flat\oplus\pair{ x}) =\pDen(L^\flat\oplus\pair{ x}),
$$
and 
$$
\Int(\varpi^{-1} L^\flat\oplus\pair{ x}) =\pDen(\varpi^{-1}L^\flat\oplus\pair{ x}).
$$
Therefore 
$$
\chi(\CN_2,\CZ(x)\jiao \CZ(L^\flat)^\circ)=\Int( L^\flat\oplus\pair{ x})-\Int(\varpi^{-1} L^\flat\oplus\pair{ x})
$$and the assertion follows from the formula for local density  in Corollary \ref{cor: pDen}.
\end{proof}

\section{Fourier transform: the geometric side}

 Let  $L^\flat\subset \BV=\BV_n$ be an $O_F$-lattice of rank $n-1$. Let $L_F^\flat=L^\flat\otimes_{O_F}F\subset \BV_n$ be the $F$-vector subspace of dimension $n-1$. Assume that $L_F^\flat$ is non-degenerate throughout the paper. 


\subsection{Horizontal versus Vertical cycles}
Recall from \eqref{eq:H+V} that there is a decomposition of the derived special cycle $^\BL\CZ(L^\flat)$ into a sum of vertical and horizontal parts
$$
^\BL\CZ(L^\flat)=\CZ(L^\flat)_\sH+ ~ ^\BL\CZ(L^\flat)_\sV,
$$and by Theorem \ref{thm:horizontal}, the horizontal part is a sum of quasi-canonical lifting cycles
$$
\CZ(L^\flat)_\sH=\sum_{L'^\flat}\CZ(L'^\flat)^\circ,
$$
where the sum runs over all lattices $L'^\flat$ such that
$$
L^\flat\subset L'^\flat \subset (L'^\flat)^\vee \subset L^\flat_F,\quad t(L'^\flat)\leq 1.
$$

\begin{definition}
  Define the \emph{horizontal part of the arithmetic intersection number}
  \begin{align}
    \Int_{L^\flat,\sH}(x)\coloneqq\chi(\CN_n,\CZ(x)\jiao
    \CZ(L^\flat)_\sH),\quad x\in \BV\setminus L^\flat_F.
  \end{align}
\end{definition}  

\begin{definition}\label{def:horpDen}
Analogously, define \emph{the horizontal part of the derived local density}
  \begin{align}\label{pDen H}
    \pDen_{L^\flat,\sH}(x)\coloneqq\sum_{ L^\flat\subset L'\subset
      L'^\vee \atop t(L'^\flat)\leq1 }\fkm(t(L')){\bf 1}_{L'}(x), \quad  x\in \BV\setminus L^\flat_F,
  \end{align}
  where we denote
  \begin{align}\label{eq:L' flat}
    L'^\flat\coloneqq L'\cap L_F^\flat\subset L_F^\flat.
  \end{align}
\end{definition}

\begin{theorem}\label{thm H}
\label{thm:Int H}As functions on  $\BV\setminus L^\flat_F$, 
$$
\Int_{L^\flat,\sH}=\pDen_{L^\flat,\sH}.
$$
\end{theorem}
\begin{proof}
By  Corollary \ref{cor:qc int}, for a fixed integral lattice $L'^\flat\subset L^\flat_F$ of type $t\leq 1$,  we have
$$
\chi(\CN_n,\CZ(x)\jiao \CZ(L'^\flat)^\circ)=\sum_{L'^\flat+\pair{ x}\subset L'\subset  L'^\vee, \atop L'\cap L^\flat_F=L'^\flat } \fkm(t(L')).
$$
The assertion follows from Theorem \ref{thm:horizontal} and the corresponding formula \eqref{pDen H} for the horizontal part of the local density 
$\pDen_{L^\flat,\sH}$. \qedhere
\end{proof}

\begin{definition}
 Define the \emph{vertical part of the arithmetic intersection number}
\begin{align}\label{def: Int V}
\Int_{L^\flat,\sV}(x)\coloneqq\chi(\CN_n,\CZ(x)\jiao {}^\BL\CZ(L^\flat)_\sV), \quad x\in\BV\setminus L^\flat_F.
\end{align}
  
\end{definition}
Then by \cite[(B.3)]{Zhang2019} there is a decomposition
\begin{align}\label{eq:Intdecomp}
\Int_{L^\flat}(x)=\Int_{L^\flat,\sH}(x)+\Int_{L^\flat,\sV}(x),\quad x\in\BV\setminus L^\flat_F.
\end{align}

We will defer the vertical part of the derived local density to the next section (Definition \ref{def:verpDen}).


\subsection{Computation of $\Int_{\CV(\Lambda)}$}
Let $\Lambda\subseteq \mathbb{V}$ be a vertex lattice. Let $\CV(\Lambda)$ be the Deligne--Lusztig variety in the Bruhat--Tits stratification of $\CN_n^{\rm red}$ (\S\ref{sec:bruh-tits-strat}). Define 
\begin{equation}\label{eq:Int CV}
\Int_{\CV(\Lambda)}(x)\coloneqq \chi\bigl(\CN_n,\mathcal{V}(\Lambda) \jiao\mathcal{Z}(x)\bigr),\quad x\in \mathbb{V}\setminus \{0\}.
\end{equation}
Next we explicitly compute $\Int_{\CV(\Lambda)}$  for $\Lambda\in \Ver^3$, i.e., for $\mathcal{V}(\Lambda)$ a Deligne--Lusztig curve.


\begin{lemma}\label{int Lam}
Let $\Lambda\in \Ver^3$.
Then
$$
\Int_{\CV(\Lambda)}=-q^2(1+q) {\bf 1}_{\Lambda}+\sum_{\Lambda\subset\Lambda',\, t(\Lambda')=1}{\bf 1}_{\Lambda'}.
$$
\end{lemma}
\begin{proof}

We note that 
\begin{equation}
  \label{eq:intVLambda}
-q^2(1+q) {\bf 1}_{\Lambda}(x)+\sum_{\Lambda\subset\Lambda',\, t(\Lambda')=1}{\bf 1}_{\Lambda'}(x) =\begin{cases}(1-q^2) ,& x\in\Lambda ,\\
1,& x\in \Lambda^\vee\setminus \Lambda, \text{and }\val(x)\geq 0,\\
0,&\text{otherwise}. 
\end{cases}
\end{equation}
We start with the simple observation that $ \CZ(x)\cap \CV(\Lambda)$ is empty unless $x\in \Lambda^\vee$ and $\val(x)\ge0$. In fact, if $ \CZ(x)\cap \CV(\Lambda)$  is non-empty, the intersection must contain a point $\CV(\Lambda')$ for a type $1$ vertex lattice $\Lambda'$.  Then $\Lambda\subset \Lambda'$ and by \eqref{eq:KRstrat} we have $x\in \Lambda'$. It follows that $x\in \Lambda'\subset \Lambda'^\vee\subset V^\vee$. 

To show \eqref{eq:intVLambda}, we first consider the special case $n=3$.   If $u\notin \Lambda$, then $\CZ(u)\cap \CV(\Lambda) $ is non-empty only when $u$ lies in one of the type $1$ lattices nested between 
$\Lambda$ and $ \Lambda^\vee$. Then the intersection number is equal to one by \eqref{eq: Z(e) L}, and the desired equality follows.

Now assume $u\in \Lambda$ and $u\neq 0$.  Choose an orthogonal basis $\{e_1,e_2,e_3\}$ of $\Lambda$ (so the norms of them all have valuation one). Let $L$ be the rank two lattice generated by $e_1,e_2$. Now we note that,  by Theorem \ref{thm:horizontal} and Theorem \ref{thm:n=2} part (i),
$$
\CZ(L)=\CV(\Lambda)+ \sum_{L\subset M=M^\vee\subset L^\vee}\CZ(M),
$$
where each of $\CZ(M)\simeq \CN_1$ since $M$ is self-dual. There are exactly $q+1$ such $M$.

Let $u\in  \Lambda \setminus \{0\}$, and write it in terms of the chosen basis
$$
u=\lambda_1 e_1+\lambda_2 e_2+\lambda_ 3e_3, \quad \lambda_i\in O_F.
$$
Assume that $\lambda_3\neq 0$, and let $a_3=2\,\val(\lambda_3)+1$ (an odd integer). By \cite{Terstiege2013}, we may calculate all of the intersection numbers
\begin{align*}
\chi(\CN_3,\CZ(L)\jiao \CZ(u))&=\frac{a_3+1}{2}(q+1)+(1-q^2),\\ 
\chi(\CN_3,\CZ(M)\jiao \CZ(u))&=\frac{a_3+1}{2}.
\end{align*}
It follows that 
$$
\chi(\mathcal{N}_3,\CV(\Lambda)\jiao \CZ(u))=(1-q^2).
$$
If $\lambda_3=0$, then we choose $L$ to be the span of some other pairs of basis vectors, and we run the same computation. This proves the desired equality if  $u\in  \Lambda \setminus \{0\}$ and completes the proof when $n=3$.

Now assume that $n>3$.
Since $\Lambda$ is a vertex lattice of type $3$, it admits an orthogonal direct sum decomposition
\begin{align}\label{M}
\Lambda=\Lambda^\flat\oplus M
\end{align}
where $\Lambda^\flat$ is a rank $3$ vertex lattice of type $3$, and $M$ is a type $0$ (i.e., self-dual) lattice of rank $n-3$.
Then $$
\Lambda^\vee=\Lambda^{\flat,\vee}\obot M
$$
and any element $u\in \Lambda^\vee$ has a unique decomposition 
$$
u=u^\flat+ u_M, \quad u^\flat\in \Lambda^{\flat,\vee}, \, u_M\in M.
$$

First assume that $u^\flat\neq 0$, i.e., $u\notin M$.
Since $M$ is self-dual, we have a natural embedding \eqref{eq:inc M} 
 $$ \delta_M\colon 
 \xymatrix{\CN_3\ar[r]& \CN_n}
 $$
 which identifies $ \CN_3$ with the special cycle $\CZ(M)$. 
 Moreover, the Deligne--Lusztig curve $\CV(\Lambda^\flat)$ on $\CN_3$ is sent to $\CV(\Lambda)$, and the special divisor $\CZ(u)$ intersects properly with $\CN_3$ and its pull-back to $\CN_3$ is the special divisor $\CZ(u^\flat)$, cf. \eqref{eq:ind Z(u)}. 
 
We obtain (by the projection formula for the morphism $\delta_M$)
$$
\chi(\CN_n,\CV(\Lambda)\jiao \CZ(u))=\chi(\CN_3,\CV(\Lambda^\flat)\jiao \CZ(u^\flat)).
$$
This reduces the case $u^\flat\neq 0$ to the case $n=3$. In particular, when $u^\flat\in  \Lambda^\flat\setminus \{0\}$,
\begin{align}\label{q2-1}
\chi(\CN_n,\CV(\Lambda)\jiao \CZ(u))=1-q^2.
\end{align}

Finally it remains to show that the intersection number is the constant $(1-q^2)$ when $u\in \left( \Lambda^\flat\oplus M\right)\setminus \{0\}$. It suffices to show this when $u\in M\setminus \{0\}$.  Choose an orthogonal basis $\{e_1,e_2,e_3\}$ for $\Lambda^\flat$, and $\{f_1,\cdots,f_{n-3}\}$ for $M$. Write $$
u=\mu_1 f_1+\cdots+\mu_{n-3} f_{n-3}, \quad \mu_j\in O_F.
$$
One of the $\mu_i$ is non-zero, and without loss of generality we assume $\mu_1\neq 0$. Now define $\widetilde M$ to be the new lattice generated by $e_1+f_1, f_2,\cdots, f_{n-3}$. It is self-dual, and its orthogonal complement $\widetilde\Lambda^\flat$  in $\Lambda$ is again a type $3$-lattice. Now replace  the decomposition $\Lambda=\Lambda^\flat\oplus M$ by the new one $\Lambda=\widetilde\Lambda^\flat\oplus \widetilde M$. Then $u\notin \widetilde M$, and hence  we can apply \eqref{q2-1}. This completes the proof.
\end{proof}

\begin{corollary}\label{cor:LC int}
The function $\Int_{L^\flat,\sV}$ extends to a (necessarily unique) function in $C_c^\infty(\mathbb{V})$.\end{corollary}

\begin{proof}
  This follows from Corollary \ref{cor:ZV} and  Lemma \ref{int Lam}.\end{proof}

\subsection{Fourier transform: the geometric side; ``Local modularity''}
We compute the Fourier transform of $\Int_{L^\flat,\sV}$ as a function on $\BV$.

\begin{lemma}\label{lem: FT Int V}
Let $\Lambda\in \Ver^3$. Then the function $\Int_{\CV(\Lambda)}\in C_c^\infty(\mathbb{V})$ satisfies
$$
\wh{\Int_{\CV(\Lambda)}}=\gamma_\BV\Int_{\CV(\Lambda)}.
$$
Here $\gamma_\BV=-1$ is the Weil constant.
\end{lemma}
\begin{proof}By Lemma \ref{int Lam}, we obtain
  \begin{align*}
    \wh{\Int_{\CV(\Lambda)}}&=-\vol(\Lambda)\cdot q^2(1+q)\cdot 1_{\Lambda^\vee}+\sum_{\Lambda\subset\Lambda',\, t(\Lambda')=1}\vol(\Lambda')\cdot 1_{\Lambda'^\vee}\\
&=-(1+q^{-1})\cdot 1_{\Lambda^\vee}+\sum_{\Lambda\subset\Lambda',\, t(\Lambda')=1}q^{-1}\cdot1_{\Lambda'^\vee}.
  \end{align*}

Now we compute its value at $u\in\BV$ according to four cases.
\begin{altenumerate}
\item If $u\in \Lambda$, then there are exactly $q^3+1$ type 1 lattices $\Lambda'$ containing $\Lambda$, and the value is 
$$q^{-1} (q^3+1)-(1+q^{-1})=q^2-1.$$
\item Assume that $u\in \Lambda_1\setminus \Lambda$ for some $\Lambda_1\in \Ver^1$, i.e., the image of $\bar u$ of $u$ in $\Lambda^\vee/\Lambda$ is an isotropic vector. Notice that $u\in \Lambda'^\vee $ if and only if $\ov u$ is orthogonal to the line given by the image of $(\Lambda')^\vee$ in $\Lambda^\vee/\Lambda$. So there is exactly one such $ \Lambda'\in \Ver^1$, i.e., $\Lambda'=\Lambda_1$, and we obtain the value
 $$q^{-1} -(1+q^{-1})=-1.$$
 \item Assume that $u\in \Lambda^\vee\setminus\Lambda$ but $u\not\in \Lambda_1\setminus \Lambda$ for any $\Lambda_1\in\Ver^1$. Then $\ov u$ is anisotropic in $\Lambda^\vee/\Lambda$. Notice that $\ov u^\perp$ is a non-degenerate hermitian space of dimension two, and $\Lambda'$ corresponds to an isotropic line in $\ov u^\perp$. So there are exactly  $q+1$ of such $ \Lambda'\in \Ver^1$, and we obtain the value
 $$q^{-1}(q+1) -(1+q^{-1})=0.$$
\item If $u\not\in \Lambda^\vee$, then the value at $u$ is $$q^{-1}\cdot 0-(1+q^{-1})\cdot 0=0.$$
\end{altenumerate}
This completes the proof by comparing with (\ref{eq:intVLambda}).
\end{proof}

\begin{remark}
It follows from Lemma \ref{lem: FT Int V} that $\Int_{\CV(\Lambda)}$ is  $\SL_2(O_{F_0})$-invariant under the Weil representation. This invariance may be viewed  as a  ``local modularity'', an analog of the global modularity of arithmetic generating series of special divisors (such as in \cite{Bruinier2017}).
\end{remark}

\begin{corollary}\label{cor:FT int}
The function $\Int_{L^\flat,\sV}\in C_c^\infty(\BV)$ satisfies
 $$\wh{\Int_{L^\flat,\sV}}=\gamma_\BV\Int_{L^\flat,\sV}.
 $$
\end{corollary}

\begin{proof}
This follows from Corollary \ref{cor:ZV} and Lemma \ref{lem: FT Int V}.
\end{proof}

\subsection{Fourier transform: the geometric side; ``Higher local modularity''} In this subsection we generalize Lemmas \ref{int Lam} and \ref{lem: FT Int V} on the function $\Int_{\CV(\Lambda)}$ for vertex lattices $\Lambda$ of type 3 to vertex lattices $\Lambda$ of arbitrary type $t(\Lambda)=2d+1\ge3$. 

Let $$\ch: K_0(\mathcal{V}(\Lambda))_\mathbb{Q}\rightarrow \bigoplus_{i=0}^d\Ch^i(\mathcal{V}(\Lambda))_\mathbb{Q}$$ be the Chern character from the Grothendieck ring to the Chow ring of $\mathcal{V}(\Lambda)$, which is an isomorphism of graded rings.  In particular, it induces an isomorphism $$\ch_i: \Gr^iK_0(\mathcal{V}(\Lambda))_\mathbb{Q}\isoarrow \Ch^i(\mathcal{V}(\Lambda))_\mathbb{Q},$$ for $0\le i\le d$. Let $$\cl_i: \Ch^i(\mathcal{V}(\Lambda))_\mathbb{Q}\rightarrow H^{2i}(\mathcal{V}(\Lambda), \mathbb{Q}_\ell)(i)$$ be the $\ell$-adic cycle class map and let $$\cl=\bigoplus_{i=0}^d\cl_i: \bigoplus_{i=0}^d\Ch^i(\mathcal{V}(\Lambda))_\mathbb{Q}\rightarrow \bigoplus_{i=0}^d H^{2i}(\mathcal{V}(\Lambda), \mathbb{Q}_\ell)(i).$$ Then $\cl$ intertwines the intersection product on the Chow ring and the cup product on the cohomology ring, namely the following diagram commutes,
\begin{equation}
  \label{eq:chcl}
\begin{gathered}
  \xymatrix@C=0.5em{\Gr^iK_0(\mathcal{V}(\Lambda))_\mathbb{Q} \ar[d]^{\ch_i}_{\rotatebox{90}{$\sim$}} &\times & \Gr^jK_0(\mathcal{V}(\Lambda))_\mathbb{Q} \ar[d]^{\ch_j}_{\rotatebox{90}{$\sim$}} \ar[rr]^-{\cdot} && \Gr^{i+j}K_0(\mathcal{V}(\Lambda))_\mathbb{Q} \ar[d]^{\ch_{i+j}}_{\rotatebox{90}{$\sim$}}\\\Ch^i(\mathcal{V}(\Lambda))_\mathbb{Q} \ar[d]^{\cl_i} &\times & \Ch^j(\mathcal{V}(\Lambda))_\mathbb{Q} \ar[d]^{\cl_j} \ar[rr]^-{\cdot} && \Ch^{i+j}(\mathcal{V}(\Lambda))_\mathbb{Q} \ar[d]^{\cl_{i+j}} \\ H^{2i}(\mathcal{V}(\Lambda)),\mathbb{Q}_\ell)(i) &\times & H^{2j}(\mathcal{V}(\Lambda), \mathbb{Q}_\ell)(j) \ar[rr]^-{\cup} &&  H^{2(i+j)}(\mathcal{V}(\Lambda), \mathbb{Q}_\ell)(i+j).}
\end{gathered}
\end{equation}

Denote by $\Tate^{2i}_\ell(\mathcal{V}(\Lambda))\subseteq H^{2i}(\mathcal{V}(\Lambda), \mathbb{Q}_\ell)(i)$ to be the subspace of Tate classes, i.e., the elements fixed by $\Fr^s$ for some power $s\ge1$. Then by Theorem \ref{thm:tate}, we have the identity $$\im(\cl_i)_{\mathbb{Q}_\ell}=\Tate^{2i}_\ell(\mathcal{V}(\Lambda)),$$ and moreover $\Tate^{2i}_\ell(\mathcal{V}(\Lambda))$ is spanned by the cycle classes of $\mathcal{V}(\Lambda')\subseteq \mathcal{V}(\Lambda)$, where $\Lambda'\supseteq \Lambda$ runs over vertex lattices of type $2(d-i)+1$. Denote by
\begin{equation}
  \label{eq:barK}
  \barK:=K_0(\mathcal{V}(\Lambda))_\mathbb{Q}/\ker({\cl\circ\ch}),\quad \barCh^i(\mathcal{V}(\Lambda)):=\Ch^i(\mathcal{V}(\Lambda))_\mathbb{Q}/\ker \cl_i.
\end{equation}
 Then $\ch$ and $\cl$ induce isomorphisms
 \begin{equation}
   \label{eq:Tateclasses}
   \ch: \barK\isoarrow \bigoplus_{i=0}^d \barCh^i(\mathcal{V}(\Lambda)), \quad \cl: \bigoplus_{i=0}^d \barCh^i(\mathcal{V}(\Lambda))_{\mathbb{Q}_\ell}\isoarrow \bigoplus_{i=0}^d \Tate^{2i}_\ell(\mathcal{V}(\Lambda)).
 \end{equation}
  By Theorem \ref{thm:tate}~(\ref{item:eigen1}) and that the cup product 
  is $\Fr$-equivariant, the Poincar\'e duality induces a perfect pairing
  \begin{equation}
    \label{eq:tatepairing}
    \cup: \Tate_\ell^{2i}(\mathcal{V}(\Lambda)) \times \Tate_\ell^{2d-2i}(\mathcal{V}(\Lambda))\rightarrow \mathbb{Q}_\ell.
  \end{equation}

  \begin{definition}\label{def:K CV}
  For $x\in \mathbb{V}\setminus\{0\}$, define $\IntK(x)\in\barK$ to be the image of $$\mathcal{V}(\Lambda)\jiao \mathcal{Z}(x)\in K^{\mathcal{V}(\Lambda)}_0(\mathcal{N}_n)\isoarrow K_0(\mathcal{V}(\Lambda))$$ under (\ref{eq:barK}).  
\end{definition}
\begin{remark}
Our main result in this subsection (Theorem \ref{thm:highermodularity}) shows that the function $\IntK$ satisfies the local modularity analogous to Lemma \ref{lem: FT Int V}.
  \end{remark}
 Since $\mathcal{Z}(x)$ is a Cartier divisor on $\mathcal{N}_n$,  we know that $\mathcal{V}(\Lambda)\jiao \mathcal{Z}(x)$ is explicitly represented by the two-term complex of line bundles on $\mathcal{V}(\Lambda)$, $$[\mathcal{O}_{\mathcal{N}_n}(-\mathcal{Z}(x))|_{\mathcal{V}(\Lambda)}\rightarrow \mathcal{O}_{\mathcal{N}_n}|_{\mathcal{V}(\Lambda)}]\in \mathrm{F}^1K_0(\mathcal{V}(\Lambda)).$$ Thus we have 
 the Chern character
 \begin{align*}
   \ch(\mathcal{V}(\Lambda)\jiao \mathcal{Z}(x))&=\ch(\mathcal{O}_{\mathcal{N}_n}(-\mathcal{Z}(x))|_{\mathcal{V}(\Lambda)})-\ch(\mathcal{O}_{\mathcal{V}(\Lambda)})\\
   &=\exp(c_1(\mathcal{O}_{\CN_n}(-\mathcal{Z}(x))|_{\mathcal{V}(\Lambda)}))-\exp(c_1(\mathcal{O}_{\mathcal{V}(\Lambda)}))\\
   &=\sum_{i=1}^d\frac{ c_1(\mathcal{O}_{\CN_n}(-\mathcal{Z}(x))|_{\mathcal{V}(\Lambda)})^i}{i!}.  
 \end{align*}

\begin{definition}
For $x\in \mathbb{V}\setminus\{0\}$,  define $$\Intch(x):=c_1(\mathcal{O}_{\CN_n}(-\mathcal{Z}(x))|_{\mathcal{V}(\Lambda)})\in \Ch^1(\mathcal{V}(\Lambda))_\mathbb{Q},$$ and $$\Intc(x):=c_1(\mathcal{O}_{\CN_n}(-\mathcal{Z}(x))|_{\mathcal{V}(\Lambda)})^d\in \Ch^d(\mathcal{V}(\Lambda))_\mathbb{Q}\xrightarrow{\stackrel{\deg}{\sim}} \mathbb{Q}.$$
\end{definition}

\begin{lemma}\label{lem:Lambdainvariant}
The function $\Intch$ (resp. $\Intc$) is $\Lambda$-invariant, under the translation by $\Lambda$. In particular, the function $\Intch$ (resp. $\Intc$) extends uniquely to an $\Lambda$-invariant function on $\mathbb{V}$, or equivalently, a function on $\mathbb{V}/\Lambda$ (still denoted by the same symbol).
\end{lemma}


\begin{proof}
It suffices to show that $\mathcal{V}(\Lambda)\jiao \mathcal{Z}(x)\in K^{\mathcal{V}(\Lambda)}_0(\mathcal{N}_n)$ is $\Lambda$-invariant, i.e.,
\begin{equation}
  \label{eq:Lambdainv}
\mathcal{O}_{\mathcal{Z}(\Lambda)} \otimes^\mathbb{L} \mathcal{O}_{\mathcal{Z}(x)}=\mathcal{O}_{\mathcal{Z}(\Lambda)} \otimes^\mathbb{L} \mathcal{O}_{\mathcal{Z}(x+y)}
\end{equation}
 for any $x\ne0$ and any $y\in \Lambda$ such that $x+y\ne0$. 

First assume that $x,y$ are linearly independent. Then by Corollary \ref{cor:ind L}, we have $$\mathcal{O}_{\mathcal{Z}(y)} \otimes^\mathbb{L} \mathcal{O}_{\mathcal{Z}(x)}=\mathcal{O}_{\mathcal{Z}(y)} \otimes^\mathbb{L} \mathcal{O}_{\mathcal{Z}(x+y)}.$$ Thus the left-hand-side of \eqref{eq:Lambdainv} $$\mathcal{O}_{\mathcal{Z}(\Lambda)} \otimes^\mathbb{L} \mathcal{O}_{\mathcal{Z}(x)}=\mathcal{O}_{\mathcal{Z}(\Lambda)} \otimes^\mathbb{L}_{\mathcal{O}_{\mathcal{Z}(y)}}\mathcal{O}_{\mathcal{Z}(y)}\otimes^\mathbb{L}\mathcal{O}_{\mathcal{Z}(x)}$$ equals the right-hand-side of \eqref{eq:Lambdainv} $$\mathcal{O}_{\mathcal{Z}(\Lambda)} \otimes^\mathbb{L} \mathcal{O}_{\mathcal{Z}(x+y)}=\mathcal{O}_{\mathcal{Z}(\Lambda)} \otimes^\mathbb{L}_{\mathcal{O}_{\mathcal{Z}(y)}}\mathcal{O}_{\mathcal{Z}(y)}\otimes^\mathbb{L}\mathcal{O}_{\mathcal{Z}(x+y)}$$ as desired. 

It remains to consider the case that $x,y$ are linearly dependent. Choose $x_1\in\Lambda$ linearly independent of $x$. Then
$$\mathcal{O}_{\mathcal{Z}(\Lambda)} \otimes^\mathbb{L} \mathcal{O}_{\mathcal{Z}(x)}=\mathcal{O}_{\mathcal{Z}(\Lambda)} \otimes^\mathbb{L} \mathcal{O}_{\mathcal{Z}(x+x_1)}.$$
Since $x+x_1$ is linearly independent of $y$, we obtain
$$\mathcal{O}_{\mathcal{Z}(\Lambda)} \otimes^\mathbb{L} \mathcal{O}_{\mathcal{Z}(x+x_1)}=\mathcal{O}_{\mathcal{Z}(\Lambda)} \otimes^\mathbb{L} \mathcal{O}_{\mathcal{Z}(x+x_1+y)}.$$ 
Since $x+y\neq 0$, $x+y$ is linearly independent of $x_1$ and hence
$$\mathcal{O}_{\mathcal{Z}(\Lambda)} \otimes^\mathbb{L} \mathcal{O}_{\mathcal{Z}(x+x_1+y)}=\mathcal{O}_{\mathcal{Z}(\Lambda)} \otimes^\mathbb{L} \mathcal{O}_{\mathcal{Z}(x+y)}.$$ This completes the proof.
\end{proof}

\begin{lemma}\label{lem:firstchern}
  Let $\Lambda\in \Ver^{2d+1}$. Then for any $x\in \Lambda$, we have $$\Intch(x)=-\frac{1}{1+q^{2d+1}}\sum_{y\in V_\Lambda\setminus \{0\}\atop (y,y)=0}\Intch(y)\in \Ch^1(\mathcal{V}(\Lambda))_\mathbb{Q},$$
  where $V_\Lambda=\Lambda^\vee/\Lambda$ is a $k_F/k$-hermitian space of dimension $2d+1$ (see \S\ref{ss:minu KR}).
\end{lemma}

\begin{proof}
  Since the cycle class map for divisors $\cl_1: \Ch^1(\mathcal{V}(\Lambda))_\mathbb{Q}\rightarrow H^2(\mathcal{V}(\Lambda), \mathbb{Q}_\ell)(1)$  is injective, we know from (\ref{eq:Tateclasses}) that $$
  \Ch^1(\mathcal{V}(\Lambda))_{\mathbb{Q}_\ell}\isoarrow \Tate^2_\ell(\mathcal{V}(\Lambda)).$$ It follows from the perfect pairing (\ref{eq:tatepairing}) that to show the desired identity it suffices to show that for any Deligne--Lusztig curve $\mathcal{V}(\Lambda')\subseteq \mathcal{V}(\Lambda)$ ($t(\Lambda')=3$), the following identity
  \begin{equation}
    \label{eq:cupwithtypethree}
    \Intch(x)\cdot [\mathcal{V}(\Lambda')]=-\frac{1}{1+q^{2d+1}}\sum_{y\in V_\Lambda\setminus \{0\}\atop (y,y)=0}\Intch(y)\cdot [\mathcal{V}(\Lambda')]
  \end{equation}
  holds in $\Ch^d(\mathcal{V}(\Lambda))_\mathbb{Q}\isoarrow \mathbb{Q}$. By the projection formula, $$\Intch(x)\cdot [\mathcal{V}(\Lambda')]=c_{1, \mathcal{V}(\Lambda')}(x),\quad \Intch(y)\cdot [\mathcal{V}(\Lambda')]=c_{1, \mathcal{V}(\Lambda')}(y).$$ Since $t(\Lambda')=3$, we know from Lemma \ref{int Lam} that $$c_{1, \mathcal{V}(\Lambda')}(x)=\Int_{\mathcal{V}(\Lambda')}(x)=(1-q^2),\quad c_{1, \mathcal{V}(\Lambda')}(y)=\Int_{\mathcal{V}(\Lambda')}(y)=
  \begin{cases}
    (1-q^2), & y\in \Lambda'/\Lambda,\\
    1, & y\in \Lambda'^\vee/\Lambda, y\not\in\Lambda'/\Lambda, \\
    0, & y\not\in \Lambda'^\vee/\Lambda.
  \end{cases}
$$ Since $\Lambda'/\Lambda\subseteq V_\Lambda$ is totally isotropic, the number of nonzero isotropic vectors $y\in \Lambda'/\Lambda$ equals $\#(\Lambda'/\Lambda)-1=q^{2(d-1)}-1$. The number of isotropic vectors $y\in \Lambda'^\vee/\Lambda$, $y\not\in \Lambda'/\Lambda$ equals $\#(\Lambda'/\Lambda)$ times the number of nonzero isotropic vectors in $\Lambda'^\vee/\Lambda'$, which evaluates to $q^{2(d-1)}\cdot (1+q^3)(q^{2}-1)$
. It follows that $$\sum_{y\in V_\Lambda\setminus \{0\}\atop (y,y)=0}\Intch(y)\cdot [\mathcal{V}(\Lambda')]=(q^{2(d-1)}-1)\cdot (1-q^2)+q^{2(d-1)}(1+q^3)(q^2-1)=-(1-q^2)(1+q^{2d+1}).$$ Hence $$-\frac{1}{1+q^{2d+1}}\sum_{y\in V_\Lambda\setminus \{0\}\atop (y,y)=0}\Intch(y)\cdot [\mathcal{V}(\Lambda')]=(1-q^2)=\Intch(x)\cdot [\mathcal{V}(\Lambda')],$$ and the desired identity (\ref{eq:cupwithtypethree}) holds.
\end{proof}

\begin{lemma}\label{lem:Intcformula}
  Let $\Lambda\in \Ver^{2d+1}$. Then $$\Intc(x)=c(d)\cdot
  \begin{cases}
    (1-q^{2d}), & x\in \Lambda, \\
    1, & x\in \Lambda^\vee\setminus\Lambda, \val(x)\ge0,\\
    0, & \text{otherwise}.
  \end{cases}$$
  Here $c(1)=1$ and $c(d)=\prod_{i=1}^{d-1}(1-q^{2i})$.
\end{lemma}

\begin{proof}
  We induct on $d$. The base case $d=1$ follows from Lemma \ref{int Lam}. By the same proof as in Lemma \ref{int Lam}, we know that $\Intc(x)=0$ unless $x\in \Lambda^\vee$ and $\val(x)\ge0$. By the $\Lambda$-invariance of $\Intc$ in Lemma \ref{lem:Lambdainvariant}, to show the result it remains to show that
  \begin{equation}
    \label{eq:Intc}
    \Intc(0)=c(d)(1-q^{2d}),\quad \Intc(x)=c(d)
  \end{equation}
  for any $x\in V_\Lambda\setminus\{0\}$ with $(x,x)=0$.

  By Lemma \ref{lem:firstchern}, we have $$\Intc(0)=\Intch(0)^{d-1}\Intch(0)=\Intch(0)^{d-1}\left( -\frac{1}{1+q^{2d+1}}\sum_{y\in V_\Lambda\setminus \{0\}\atop (y,y)=0}\Intch(y) \right).$$ By the projection formula, we have $$\Intch(0)^{d-1}\Intch(y)=c_{1, \mathcal{V}(\Lambda+\langle y\rangle)}(0)^{d-1},$$ which by induction equals $$c_{\mathcal{V}(\Lambda+\langle y\rangle)}(0)=c(d-1)(1-q^{2(d-1)})$$ since $t(\Lambda+\langle y\rangle)=2d-1$. The number of nonzero isotropic vectors $y\in V_\Lambda$ equals $(1+q^{2d+1})(q^{2d}-1)$. 
  Hence $$\Intc(0)=(q^{2d+1}+1)(q^{2d}-1)\cdot \left( -\frac{1}{1+q^{2d+1}} c(d-1)(1-q^{2(d-1)})\right)=(1-q^{2d})(1-q^{2(d-1)})\cdot c(d-1).$$   On the other hand, for any $x\in V_\Lambda\setminus\{0\}$ with $(x,x)=0$. by the projection formula, we have $$\Intc(x)=\Intch(x)^{d-1}\Intch(x)=c_{1,\mathcal{V}(\Lambda+\langle x\rangle)}(x)^{d-1},$$ which by induction equals $$c_{\mathcal{V}(\Lambda+\langle x\rangle)}(x)=c(d-1)(1-q^{2(d-1)})$$ since $t(\Lambda+\langle x\rangle)=2d-1$. The desired identity (\ref{eq:Intc}) then follows as $c(d-1)(1-q^{2(d-1)})=c(d)$.
\end{proof}

\begin{lemma}\label{lem:Intc}
  Let $\Lambda\in \Ver^{2d+1}$. Then $$\Intc=\frac{c(d)}{c'(d)}\sum_{\Lambda' \in \Ver^3\atop \Lambda'\supseteq \Lambda}\Int_{\mathcal{V}(\Lambda')}.$$ Here $c'(1)=1$ and $c'(d)=\prod_{i=2}^{d}(1+q^{2i+1})$ when $d\geq 2$.
\end{lemma}

\begin{proof}
We distinguish three cases.
\begin{enumerate}
\item For $x\in \Lambda$, we have $\Int_{\mathcal{V}(\Lambda')}(x)=1-q^2$ for any $\Lambda'$ in the sum by Lemma \ref{lem: FT Int V}. The number of such $\Lambda'$ is the number of $(d-1)$-dimensional totally isotropic subspaces in $V_\Lambda$, which equals $S_{2d+1,d-1}$ (in the notation of Lemma \ref{lem:Smb}). 
  Hence 
  the right-hand-side evaluates to 
  $$\frac{c(d)}{c'(d)}S_{2d+1,d-1}(1-q^2)=\frac{c(d)}{\prod_{i=2}^d(1+q^{2i+1})}\frac{\prod_{i=4}^{2d+1}(1-(-q)^i)}{\prod_{i=1}^{d-1}(1-q^{2i})}(1-q^2)=c(d)(1-q^{2d}),$$ which equals $\Intc(x)$ by Lemma \ref{lem:Intcformula}.
\item For $x\in \Lambda^\vee\setminus \Lambda$ with $\val(x)\ge0$, we have $$\Int_{\mathcal{V}(\Lambda')}(x)=  \begin{cases}
    (1-q^2), & x\in \Lambda',\\
    1,  & x\in \Lambda'^\vee\setminus \Lambda'.
  \end{cases}
$$ The number of $\Lambda'$ such that $x\in \Lambda'$ is the number of $(d-2)$-dimensional totally isotropic subspaces in $V_{\Lambda+\langle x\rangle}$, which equals $S_{2d-1,d-2}$. The number of $\Lambda'$ such that $x\in \Lambda'^\vee\setminus \Lambda'$ is the number of $(d-1)$-dimensional totally isotropic subspaces  $W\subseteq V_{\Lambda}$ such that $x\not\in W$ but $x\in W^\perp$. In this case the map $W\mapsto W+\langle x\rangle/\langle x\rangle$ gives a surjection onto the set of $(d-1)$-dimensional totally isotropic subspaces in $\langle x\rangle^\perp/\langle x\rangle$, whose fiber has size equal to the number of $(d-1)$-dimensional subspaces of $W+\langle x\rangle$ not containing $\langle x\rangle$. Hence the number of such $W$ is equal to $S_{2d-1,d-1}\cdot q^{2d-2}$. So the right-hand-side evaluates to
\begin{align*}
  &\quad  \frac{c(d)}{c'(d)}(S_{2d-1,d-2}(1-q^2)+S_{2d-1,d-1}\cdot q^{2d-2})\\ &=\frac{c(d)}{\prod_{i=2}^d(1+q^{2i+1})}\left(\frac{\prod_{i=4}^{2d-1}(1-(-q)^i)}{\prod_{i=1}^{d-2}(1-q^{2i})}(1-q^2)+\frac{\prod_{i=2}^{2d-1}(1-(-q)^i)}{\prod_{i=1}^{d-1}(1-q^{2i})}q^{2d-2}\right)\\
  &=c(d),
\end{align*} which equals $\Intc(x)$ by Lemma \ref{lem:Intcformula}.
\item If $x\not\in \Lambda^\vee$ or $\val(x)<0$, then both sides are zero.\qedhere
\end{enumerate}
\end{proof}

\begin{corollary}\label{pro:Intc}
   Let $\Lambda\in \Ver^{2d+1}$. Then $\Intc\in C_c^\infty(\mathbb{V})$ satisfies $$\wh{\Intc}=\gamma_\mathbb{V}\,\Intc.$$
\end{corollary}

\begin{proof}
  It follows immediately from Lemmas \ref{pro:Intc} and \ref{lem: FT Int V}.
\end{proof}

\begin{theorem}[$K$-theoretic local modularity]\label{thm:highermodularity}
Let $\Lambda\in \Ver^{2d+1}$.  For any linear map $l: \barK\rightarrow \mathbb{Q}$, the function $l\circ \IntK$ extends to a (necessarily unique) function in $C_c^\infty(\mathbb{V})$ and satisfies $$\wh{l\circ\IntK}=\gamma_\mathbb{V}\,  l\circ\IntK.$$Here, we refer to Definition \ref{def:K CV} for $\IntK$.
\end{theorem}

\begin{proof}
  Since the Tate classes are spanned by the cycle classes of Deligne--Lusztig subvarieties $\mathcal{V}(\Lambda')\subseteq \mathcal{V}(\Lambda)$, it follows that from the perfect pairing (\ref{eq:tatepairing}) that under the identification (\ref{eq:Tateclasses}) any linear map $\barK\rightarrow \mathbb{Q}$ is 
  a linear combination of linear maps $$l: \barCh^i(\mathcal{V}(\Lambda))\rightarrow \mathbb{Q}$$ given by the intersection product with $[\mathcal{V}(\Lambda')]$ for varying 
  $\mathcal{V}(\Lambda')\subseteq \mathcal{V}(\Lambda)$.   So we may assume $l$ has the form $$l= \cdot[\mathcal{V}(\Lambda')] :\barCh^{d-d'}(\mathcal{V}(\Lambda))\rightarrow \mathbb{Q},$$ where $d'=\dim \mathcal{V}(\Lambda')\le d$. Then $$l\circ \IntK(x)=\ch_{d'}(\mathcal{V}(\Lambda)\jiao \mathcal{Z}(x))\cdot [\mathcal{V}(\Lambda')]$$ which by the projection formula equals to $$l\circ \IntK(x)=\ch_{d'}(\mathcal{V}(\Lambda')\jiao \mathcal{Z}(x))=\frac{c_{\mathcal{V}(\Lambda')}(x)}{d'!}.$$  By induction on $d$, we may assume $d=d'$ and we are reduced to show that $\Intc\in C_c^\infty(\mathbb{V})$ satisfies $$\wh{\Intc}=\gamma_\mathbb{V}\,\Intc.$$ This is exactly Corollary \ref{pro:Intc}. 
\end{proof}
Now we return to the function $\Int_{\CV(\Lambda)}$ defined by \eqref{eq:Int CV}.
\begin{corollary}[Higher local modularity]\label{cor:highermodularity}
  Let $\Lambda\in \Ver^{2d+1}$. Then $\Int_{\CV(\Lambda)}$ extends to a (necessarily unique) function in $C_c^\infty(\mathbb{V})$ and satisfies
$$
\wh{\Int_{\CV(\Lambda)}}=\gamma_\BV\Int_{\CV(\Lambda)}.
$$
\end{corollary}

\begin{proof}
  Consider the linear map given by the Euler--Poincar\'e characteristic $$\chi: K_0(\mathcal{V}(\Lambda))_\mathbb{Q}\rightarrow \mathbb{Q},\quad [\mathcal{F}]\mapsto \chi(\mathcal{V}(\Lambda), \mathcal{F}).$$ By the Grothendieck--Riemann--Roch theorem, we have $\chi(\mathcal{V}(\Lambda), \mathcal{F})=\deg(\ch(\mathcal{F})\cdot\mathrm{Td}(\mathcal{V}(\Lambda)))$, where $\mathrm{Td}(\mathcal{V}(\Lambda))$ is the Todd class of the tangent bundle of $\mathcal{V}(\Lambda)$. It follows from (\ref{eq:chcl}) that $\chi$ factors through $\barK$ and thus defines a linear map $\chi:  \barK\rightarrow \mathbb{Q}$. By definition we have $\chi\circ\IntK=\Int_{\mathcal{V}(\Lambda)}$. The result then follows from Theorem \ref{thm:highermodularity}.
\end{proof}

\begin{remark}\label{rem:avoid B3}
  Corollary \ref{cor:highermodularity} allows us to give an alternative proof of Corollary \ref{cor:FT int} without a priori knowing that only the $(n-1)$-th graded piece of the derived Kudla--Rapoport cycle contributes to $\Int_{L^\flat, \sV}(x)$ in the decomposition (\ref{eq:Intdecomp}), in particular, without using \cite[(B.3)]{Zhang2019} for a formal scheme.
\end{remark}

\section{Fourier transform: the analytic side}
\label{s:FT ana}

In this section we allow $F_0$ to be a non-archimedean local field of characteristic not equal to $2$ (but possibly with residual characteristic $2$), and $F$ an unramified quadratic extension.

\subsection{Lattice-theoretic notations}\label{sec:latt-theor-notat}
Recall that $\BV=\BV_n$ is the hermitian space defined in \S \ref{sec:herm-space-mathbbv} (in particular it is non-split).
We continue to let  $L^\flat\subset \BV=\BV_n$ be an integral $O_F$-lattice of rank $n-1$, such that $L_F^\flat$ is non-degenerate.
 Define
\begin{equation}
(L^\flat)^{\vee,\circ}=\{x\in (L^\flat)^\vee \mid (x,x)\in O_F\}.
\end{equation}
The fundamental invariants of $L^\flat$ are denoted by
$$(a_{1},\cdots,a_{n-1})\in\BZ^{n-1},
$$where $0\leq a_1\leq\cdots \leq a_{n-1}$. 
Denote the largest invariant by
\begin{align}e_{\max}(L^\flat)=a_{n-1}.
\end{align}

Let 
\begin{equation}\label{eqn M(L)}
M=M(L^\flat)=L^\flat  \obot \pair{ u}
\end{equation} be the lattice characterized by the following condition: $u\in \BV$ is a vector satisfying $u\perp L^\flat$ and with  valuation  $a_{n-1}$ or $a_{n-1}+1$ (only one of these two is possible due to the parity of $\val(\det(\BV))$).
In other words, the rank one lattice $\pair{ u}$ is the set of all $x\perp L^\flat$ with $\val(x)\geq a_{n-1}$. Then the fundamental invariants of $M(L^\flat)$ are
$$(a_{1},\cdots,a_{n-1},a_{n-1}), \quad\text{or} \quad(a_{1},\cdots,a_{n-1},a_{n-1}+1).
$$

Finally we note that, if $L^\flat\subset L'^\flat$ are two integral lattices of rank $n-1$, then 
\begin{equation}\label{ineq e max}
e_{\max}(L'^\flat)\leq e_{\max}(L^\flat)
\end{equation}  and $M(L^\flat)\subset M(L'^\flat)$. The above inequality follows from the characterization of $-e_{\max}(L^\flat)$ as the minimal valuation of vectors in the lattice $(L^\flat)^\vee$.

\subsection{Lemmas on lattices}
In this subsection, we do not require the lattice $L^\flat$ to be integral.
\begin{lemma}\label{lem:Lat}
Let $L'^\flat\subset L_F^\flat$ be an $O_F$-lattice (of rank $n-1$). Denote  $$
{\rm Lat}(L'^\flat):=\{O_F\text{-\rm lattices }L'\subset \BV\mid {\rm rank }\,L'=n,\quad L'^\flat=L'\cap L_F^\flat\}.
$$
Then there is a bijection
\begin{equation}
	\xymatrix@R=0ex{
[(\BV/L'^\flat) \setminus (L'^\flat_F/L'^\flat) ]/O_F^\times\ar[r]^-\sim  &  {\rm Lat}(L'^\flat) \\
		u \ar@{|->}[r]  & L'^\flat+\pair{ u}.
		}
\end{equation}

\end{lemma}
\begin{proof}
The indicated map is well-defined and clearly injective. To show the surjectivity, we note that $L'/L'^\flat$ is free for any $L'\in  {\rm Lat}(L'^\flat) $. Choose any element $u\in L'$ whose image in   $L'/L'^\flat$ is a generator. Then it is clear that $L'=  L'^\flat+\pair{ u}$. 
\end{proof}

Let  $\pair{x}_F=F\, x$ be the $F$-line generated by $x\in\BV\setminus L_F^\flat $.
Corresponding to the (not necessarily orthogonal) decomposition $\BV=L^\flat_F\oplus   \pair{x}_F$, there are  two projection maps
$$\pi_\flat\colon
\xymatrix{ \BV\ar[r]&L^\flat_F},\quad \pi_x\colon
\xymatrix{ \BV\ar[r]&  \pair{x}_F.}
$$

\begin{lemma}
\label{lem:lat cyc}
Let $L'\subset \BV$ be an $O_F$-lattice (of rank $n$).
Denote 
$$
L'^\flat=L'\cap L_F^\flat,\quad L'_x=L'\cap \pair{x}_F.
$$The natural projection maps induce isomorphisms of $O_F$-modules
$$
\xymatrix{ \pi_\flat(L')/L'^\flat &\ar[l]_-\sim L'/(L'^\flat\oplus L'_x)\ar[r]^-\sim& \pi_x(L')/L'_x .}
$$
In particular, all three abelian groups are $O_F$-cyclic modules.
\end{lemma}
\begin{proof}
Consider the map$$
\phi\colon\xymatrix{ L'\ar[r]& \pi_x(L')/L'_x .}
$$
We show that the kernel of $\phi$
 is $L'^\flat\oplus L'_x$; the other assertion  can be proved similarly.

Let $u\in L'$ and write $u=u^\flat+u^\nat$ uniquely for $u^\flat\in L^\flat_F, \, u^\nat\in \langle x\rangle_F$. Then $\phi(u)=u^\nat\mod L'_x$. If $u\in \ker(\phi)$, then $u^\nat\in L'$. It follows that $u^\flat=u-u^\nat\in L'$, and hence $u^\flat\in L'^\flat$. Therefore $u\in L'^\flat\oplus L'_x$ and $\ker(\phi)\subset L'^\flat\oplus L'_x$. Conversely, if $u\in L'^\flat\oplus L'_x$, then $u^\flat\in L'^\flat, u^\nat\in L'_x$, and clearly $\phi(u)=0$. This completes the proof.
\end{proof}

Now we assume that $x\perp L^\flat$. We rename the projection map to the line $\pair{x}_F=(L_F^{\flat})^{\perp}$ as $\pi_\perp$. Then we have a formula relating the volume of $L'$ to that of $L'^\flat=L'\cap L^\flat_F$ and of the image of the projection $\pi_\perp$  (the analog of ``base $\times$ height'' formula for the area of a parallelogram)
\begin{equation}\label{eqn:vol}
\vol(L')=\vol(L'^\flat)\vol(\pi_\perp(L')).
\end{equation}
In fact, this is clear if  $L'=L'^\flat\obot \pi_\perp(L')$ and in general we obtain the formula by Lemma \ref{lem:lat cyc}:
$$
\frac{\vol(L')}{\vol(L'^\flat\obot  L'_{x})}=\frac{\vol(\pi_\perp(L')) }{\vol(L'_{x})}.
$$


\subsection{Local constancy of $\pDen_{L^\flat,\sV}$}
We now resume the convention in \S\ref{sec:latt-theor-notat}. In particular, $L^\flat$ is now an integral lattice. When $\rank L=n$ with $\val(L)$ odd, recall that the derived local density is (Corollary \ref{cor: pDen})
$$
\pDen(L)=\sum_{L\subset L' \subset L'^\vee} \fkm(t(L')),\quad
$$
where
$$  \fkm(a)= \begin{cases}(1+q)(1-q^2)\cdots (1-(-q)^{a-1}),& a\geq 2,\\
1,& a\in\{0,1\}.
\end{cases}
$$

\begin{definition}
For $ x\in \BV\setminus L^\flat_F$, define
\begin{align*}
\pDen_{L^\flat}(x)\coloneqq\pDen(L^\flat+\pair{ x}).
\end{align*}
Then
\begin{align}\label{def pDen x}
\pDen_{L^\flat}(x)=
\sum_{L^\flat\subset L' \subset L'^\vee} \fkm(t(L')) {\bf 1}_{L'}(x) ,
\end{align}
where the sum is over all integral lattices $L'\subset\BV$ of rank $n$.
Note that this is a finite sum for a given $ x\in \BV\setminus L^\flat_F$. However,  when varying $x\in \BV\setminus L^\flat_F$, infinitely many $L'$ can appear.
\end{definition}

\begin{definition}\label{def:verpDen}
Recall that we have defined the horizontal part $\pDen_{L^\flat,\sH}$ in Definition \ref{def:horpDen}. Now define the \emph{vertical part of the derived local density}
\begin{align}\label{pDen V}
\pDen_{L^\flat,\sV}(x)\coloneqq \pDen_{L^\flat}(x)-\pDen_{L^\flat,\sH}(x),\quad x\in\BV\setminus L_F^\flat.
\end{align}
\end{definition}

Obviously the functions $\pDen_{L^\flat,\sH}$ and $\pDen_{L^\flat}$ are locally constant on  $\BV\setminus L_F^\flat$. Hence $\pDen_{L^\flat,\sV}$ is also locally constant on $\BV\setminus L_F^\flat$.

\begin{definition}\hfill
\begin{altenumerate}
\renewcommand{\theenumi}{\alph{enumi}}
\item 
Let $\rm L^1_c(\BV)$ be the space of integrable functions that are defined on a dense open subset of $\BV$ and vanish outside a compact subset of $\BV$. 
\item Let  $\BW$ be a non-degenerate co-dimension one subspace of $\BV$. We say that a smooth function $f$ on $\BV\setminus \BW$ has \emph{logarithmic singularity along $\BW$ near $w\in \BV$}, there is a neighborhood  $\CU_w$ of $w$ in $\BV$ such that $$
f(u)=C_0\log |(u^\perp,u^\perp)|+C_1 
$$
holds for all $u\in \CU_w\setminus \BW$, where $u^\perp\in \BW^\perp$ denotes the orthogonal projection of $u$ to $\BW^\perp$, and $C_0,C_1$ are constants (depending on $w$). We say that a smooth function $f$ on $\BV\setminus \BW$ has \emph{logarithmic singularity along $\BW$} if it does so near  every $w\in \BV$.
\end{altenumerate}
\end{definition}

\begin{proposition}\label{prop:LC Den}
\hfill
\begin{altenumerate}
\renewcommand{\theenumi}{\alph{enumi}}
\item
The functions $\pDen_{L^\flat,\sH}$ and $\pDen_{L^\flat}$ lie in $\rm L^1_c(\BV)$, and they have logarithmic singularity along $L_F^\flat$.
\item  The function $\pDen_{L^\flat,\sV}$ extends to a (necessarily unique) element in $C_c^\infty(\BV)$ (we will still denote this extension by $\pDen_{L^\flat,\sV}$), i.e., there exists an element in $C_c^\infty(\BV)$ whose restriction to the open dense subset $\BV\setminus L_F^\flat$ is equal to  $\pDen_{L^\flat,\sV}$.
\end{altenumerate}
\end{proposition}

\begin{proof}
Consider the set 
\begin{align} \label{eqn:N(L)}
N(L^\flat):=\{x\in\BV\mid \pair{ x}+L^\flat\text{ is  integral}\}.
\end{align}
We claim that $N(L^\flat)$ is a {\em compact open} subset of $\BV$.  To show the claim, we rewrite  the above set as $$
N(L^\flat)=\{ x\in\BV\mid(x, L^\flat)\subset O_F,\quad (x,x)\in O_F\}.
$$
Write $x=x^\flat+x^\perp$ according to the orthogonal direct sum $\BV=L_F^\flat\obot (L_F^{\flat})^{\perp}$. Then the condition $(x, L^\flat)\subset O_F$ is equivalent to $x^\flat\in (L^\flat)^\vee$. Since  $(L^\flat)^\vee$ is a compact subset of $L^\flat_F$, $(x^\flat,x^\flat)$ is bounded in $F$. Together with the integrality of the norm $(x,x)$, it follows that $(x^\perp,x^\perp)$ is also bounded. Therefore $x^\perp$ lies in a bounded subset $\CL$ of the $F$-line $(L_F^{\flat})^{\perp}$. It follows that $N(L^\flat)$ is contained in a bounded set $ (L^\flat)^\vee\obot \CL$. Since $N(L^\flat)$ is open and closed in $\BV$, it must be compact.


Note that  all three function $ \pDen_{L^\flat}$, $ \pDen_{L^\flat,\sH}$ and  $\pDen_{L^\flat,\sV}$ vanish outside $ N(L^\flat)$.  It follows that all three vanish outside a compact subset of $\BV$, and are smooth functions on $\BV\setminus L^\flat_F$. To show part (a) it suffices to show both functions  $ \pDen_{L^\flat}$, $ \pDen_{L^\flat,\sH}$ have logarithmic singularity along $L_F^\flat$ near each $e\in L^\flat_F$ (then the integrability follows by the consideration of the support). To show part (b), it suffices to show that $\pDen_{L^\flat,\sV}$ is a constant near each $e\in L^\flat_F$.

We now consider the behavior of the three functions $\pDen_{L^\flat}$, $\pDen_{L^\flat,\sH}$ and $\pDen_{L^\flat,\sV}$ near each $e\in L^\flat_F$. By the above discussion on the support, we may assume $e\in L^\flat_F\cap N(L^\flat)$. 


First we consider the case $e\in L^\flat$, and we consider its neighborhood $M(L^\flat)=L^\flat  \obot \pair{ u}$, the lattice defined by \eqref{eqn M(L)}. Obviously the three functions are all invariant under $L^\flat$-translation. By Lemma \ref{lem:diff},  both $\pDen_{L^\flat,\sH}$ and $\pDen_{L^\flat}$ have logarithmic singularity along $L_F^\flat$ near such $e$.

Again by Lemma \ref{lem:diff},  when $x\in \pair{ \varpi  u}$ is non-zero, we have
$$
\pDen_{L^\flat,\sV}(x)-\pDen_{L^\flat,\sV}(\varpi^{-1}x)= \Den(-q, L^\flat)-\frac{1}{ \vol( L^\flat)} \Den((-q)^{-1}, L^\flat),
$$
which vanishes by the functional equation for $\Den(X, L^\flat)$ evaluated at $X=-q$, cf. \eqref{eq: FE}.  It follows that $
\pDen_{L^\flat,\sV}(x)=\pDen_{L^\flat,\sV}(\varpi^{-1} x)
$ when $x\in \pair{ \varpi  u}$ is non-zero. Therefore $\pDen_{L^\flat,\sV}$ is a constant in $M(L^\flat)\setminus L_F^\flat$. 

Next we consider the case $e\in  L^\flat_F\cap N(L^\flat)$ but $e\notin L^\flat$. 
We denote  $\widetilde{ L}^\flat:=\pair{ e}+L^\flat$. Choose an orthogonal basis $e_1,\cdots,e_{n-1}$ of  the lattice $L^\flat$ and write $e=\lambda_1 e_1+\cdots +\lambda_{n-1}e_{n-1}$.  Up to re-ordering these basis vectors, we may assume that $\lambda_1$ attains the {\em minimal} valuation among all of the coefficients $\lambda_i, 1\leq i\leq n-1$. Since we are assuming $e\notin L^\flat$, we have $\lambda_1\notin O_F$.  Let us denote 
$$\lambda:=\lambda_1^{-1}\in O_F.$$
Then we have $$e_1=\lambda e-\lambda_1^{-1}\lambda_2 e_2-\cdots- \lambda_1^{-1}\lambda_{n-1}e_{n-1}\in \langle \lambda e,e_2,e_3,\ldots, e_{n-1}\rangle.$$
Fix a basis vector  $e_{n}$ of the line $(L^\flat_F)^\perp$. Since all three functions are invariant under $L^\flat$-translation,  it suffices to consider the behavior of the function:
$$
\xymatrix{t\in O_F \ar@{|->}[r]  & \pDen_{L^\flat}(e+te_{n})}
$$when $t$ is near $0\in O_F$,  and the similar functions for $\pDen_{L^\flat,\sH}, \pDen_{L^\flat,\sV}$ respectively.

Set $x_t:=e+te_{n}$ and $M_t:=\pair{e_2,e_3,\cdots, e_{n-1}, x_t}$. Then we have
\begin{align*}
L^\flat+\pair{x_t}&=\pair{e_1,e_2,e_3,\cdots, e_{n-1}, e+t e_{n}}\\
&=\pair{\lambda t e_{n},e_2,e_3,\cdots, e_{n-1}, e+t e_{n} }\\
&=M_t+\pair{\lambda t e_{n}}.
\end{align*}
Note that the vector space $\BV$ has an orthogonal basis $ \{e_2,e_3,\cdots, e_{n-1}, e_1+\lambda t e_n, e_n'\}$, where $e_n':=e_n-\sigma(t)\mu e_1$ and $\mu=\sigma(\lambda)\frac{(e_n,e_n)}{(e_1,e_1)}\in F$, where $\sigma(t)$ denotes the Galois conjugation of $t$. A straightforward computation shows that, when $|t|$ is sufficiently small, with respect to $\BV= \pair{e_n'}_F\obot\pair{e_n'}_F^\perp$,  the projection of $e_n$ to $\pair{e_n'}_F^\perp$ lies in $M_t$, and to the line $\pair{e_n'}_F$ is $u_t e'_n$
where $u_t\in O_F^\times $ is a unit (since $u_t\to 1$ as $t\to 0$).  It follows that, when $|t|$ is sufficiently small,
$$
L^\flat+\pair{x_t}=M_t+\pair{\lambda t e_n' }.
$$
Note that, when $|t|$ is sufficiently small,  the lattice $M_t$ has the same fundamental invariants as $\widetilde{L}^\flat=\pair{ e}+L^\flat=\pair{e_2,e_3,\cdots, e_{n-1}, e}$, and $(e_n',e_n' )$ differs from $(e_n,e_n)$ by a unit.  Since  $ \pDen_{L^\flat}(x_t)=\pDen( L^\flat+\pair{x_t})$ depends only on the fundamental invariants of the lattice $L^\flat+\pair{x_t}$, we obtain that, when $|t|$ is sufficiently small, 
\begin{align}\label{eqn diff}
\pDen_{L^\flat}(x_t)=\pDen_{\widetilde{L}^\flat}(\lambda t e_n  ).
\end{align}
Now by Lemma \ref{lem:diff}, the function $\pDen_{L^\flat}$ has logarithmic singularity near $e$.

Next we investigate the behavior of  $ \pDen_{L^\flat,\sH}(x_t)$ when $t\to 0$.  By \eqref{pDen H}, we have
\begin{align}\label{eqn diff H}
\pDen_{L^\flat,\sH}(x_t)- \pDen_{\widetilde{L}^\flat,\sH}(x_t)=\sum_{ L^\flat\subset  L'^\flat \subset( L'^\flat )^\vee\atop e\notin L'^\flat,\,  t(L'^\flat)\leq1 } 
\sum_{ L'^\flat+\pair{x_t}\subset  L'\subset
      L'^\vee \atop L'\cap L_F^\flat=L'^\flat }\fkm(t(L')).
\end{align}
We {\em claim} that, when $|t|$ is sufficiently small, the right hand side is a constant  dependent on $L^\flat$ and $e$ but not on $t$. The outer sum has only finitely many terms. Therefore, to show the claim, it suffices to show that each of the inner sums is a constant dependent on $L'^\flat$ and $e$ but not on $t$.
Now fix an integral lattice $ L'^\flat\supset L^\flat$ such that  $e\notin L'^\flat$ and $ t(L'^\flat)\leq1 $. We may further assume that $L'^\flat+\pair{e}$ is integral (otherwise the inner sum for such $L'^\flat$ is empty when  $|t|$ is sufficiently small). Then $L'^\flat$ must be an orthogonal sum $M\obot \pair{f}$ for some self-dual lattice $M$ of rank $n-2$, and a rank one lattice $\pair{f}$. Denote by $e^\ast$ the orthogonal projection of $e$ to the line $\pair{f}_F$. Since $L'^\flat+\pair{e}$ is integral  and  $e\notin L'^\flat$, we must have $e-e^\ast\in M$ and $f=\xi e^\ast$ for some $\xi\in O_F$ but $\xi\notin O_F^\times$. Let $\BW=\pair{f,e_n}_F$ be the orthogonal complement of $M_F$. Then the inner sum associated to $L'^\flat$ is equal to
$$
\sum_{ \pair{f,e^\ast + te_n}\subset  L'\subset
      L'^\vee \subset \BW \atop L'\cap L_F^\flat=\pair{f}}\fkm(t(L'))=\pDen( \pair{f, e^\ast + te_n})- \pDen( \pair{\varpi^{-1}f,e^\ast + te_n}).
$$Now it is easy to see that, when $|t|$ is sufficiently small, the lattice $ \pair{f,e^\ast + te_n}= \pair{\xi e^\ast,e^\ast + te_n}$ (resp., $ \pair{\varpi^{-1} f,e^\ast + te_n}= \pair{\varpi^{-1} \xi e^\ast, e^\ast + te_n}$) has  the same fundamental invariants  as $ \pair{ \xi t e_n, e^\ast} $ (resp., $ \pair{ \varpi^{-1}  \xi t e_n, e^\ast} $). By Lemma \ref{lem:diff}, when $|t|$ is sufficiently small, the difference  
\begin{align*}
&\pDen( \pair{f, e^\ast + te_n})- \pDen( \pair{\varpi^{-1}f,e^\ast + te_n})\\=&\pDen(\pair{ \xi t e_n, e^\ast} )- \pDen(  \pair{ \varpi^{-1}  \xi t e_n, e^\ast} )\\
=&\pDen_{ \pair{e^\ast}}(  \xi t e_n )- \pDen_{ \pair{e^\ast}}( \varpi^{-1} \xi t e_n )\\
=&\Den(-q, \pair{e^\ast})
\end{align*}
is a constant independent of $t$. This proves the claim. Note that $\pDen_{\widetilde{L}^\flat,\sH}(x_t)=\pDen_{\widetilde{L}^\flat,\sH}(te_n)$ and it has  logarithmic singularity along $\widetilde{L}^\flat_F=L^\flat_F$  by Lemma \ref{lem:diff}. It follows from the claim that the function $\pDen_{L^\flat,\sH}$ also has logarithmic singularity along $L_F^\flat$ near $e$. Now we have completed the proof of part (a).

By \eqref{eqn diff} and \eqref{eqn diff H},  let $t\to 0$ and denote by $C$ the constant equal to \eqref{eqn diff H}:
\begin{align*}
\pDen_{L^\flat,\sV}(x_t)&=\pDen_{L^\flat}(x_t)-\pDen_{L^\flat,\sH}(x_t)\\
&=\pDen_{\widetilde{L}^\flat}(\lambda t e_n  )-\pDen_{\widetilde{L}^\flat,\sH}(t e_n)-C\\
&=(\pDen_{\widetilde{L}^\flat}(\lambda t e_n  )-\pDen_{\widetilde{L}^\flat}(t e_n  ))+\pDen_{\widetilde{L}^\flat,\sV}(t e_n )-C.
\end{align*}
 By Lemma \ref{lem:diff}, the term $(\pDen_{\widetilde{L}^\flat}(\lambda t e_n  )-\pDen_{\widetilde{L}^\flat}(t e_n  ))$ is a constant dependent on  $L^\flat$ and $e$ but not on $t$. By the previous case (replacing $L^\flat$ by $\widetilde{L}^\flat$) that we have considered,  the term $\pDen_{\widetilde{L}^\flat,\sV}(t e_n )$ is a constant when $t\to 0$. This shows that  $\pDen_{L^\flat,\sV}$ is a constant near $e$, and we have completed the proof of part (b).
\end{proof}

\begin{lemma}
\label{lem:diff}
Assume that $x\perp L^\flat$ and $\val(x) \geq1+ e_{\max}(L^\flat)$. Then
$$
\pDen_{L^\flat}(x)-\pDen_{L^\flat}(\varpi^{-1} x)= \Den(-q, L^\flat),
$$
and
\begin{align*}
\pDen_{L^\flat,\sH}(x)-\pDen_{L^\flat,\sH}(\varpi^{-1}x)=\frac{1}{ \vol( L^\flat)} \Den((-q)^{-1}, L^\flat).
\end{align*}
\end{lemma}
\begin{proof}
The first part follows from the induction formula in Proposition \ref{prop: ind}
$$
\Den(X,L)=X^2\Den(X,L')+(1-X)\Den(-qX,L^\flat),
$$
where
$$
L'=L^\flat\obot\pair{  \varpi^{-1} x}, \quad L=L^\flat\obot\pair{  x}.
$$

Now we consider the second part. By the definition \eqref{pDen H} of the function $\pDen_{L^\flat,\sH}$,  we obtain
$$
\pDen_{L^\flat,\sH}(x)-\pDen_{L^\flat,\sH}(\varpi^{-1} x)=\sum_{L^\flat\subset L'\subset L'^\vee,\atop t( L'^\flat )\leq 1, \, L'\cap  \pair{x}_F=\pair{ x}} \fkm(t(L')),
$$ 
where we recall that $L'^\flat=L'\cap L^\flat_F$, cf. \eqref{eq:L' flat}. 
This sum can be rewritten as a double sum, first over all $L'$ with a given $L'\cap L^\flat_F=L'^\flat$ then over all $L'^\flat$ 
\begin{align}\label{sum pden H}
\sum_{L'^\flat\subset (L'^\flat)^\vee\atop L^\flat\subset L'^\flat, \, t(L'^\flat)\leq 1 }\sum_{L'\subset L'^\vee \atop L'\cap L^\flat_F= L'^\flat , L'\cap \pair{x}_F=\pair{ x}} \fkm(t(L')).
\end{align}

Fix $L'^\flat$ with $t(L'^\flat)\leq 1$ and we consider the inner sum.
Since $t(L'^\flat)\leq 1$, we may assume that $L'^\flat$ has an orthogonal basis $e_1',\cdots, e'_{n-1}$ such that $\val(e_1')=\val(e_2')=\cdots=\val(e_{n-2}')=0$, and $a_{n-1}'\coloneqq \val(e_{n-1}')$. 

By Lemmas \ref{lem:Lat} and \ref{lem:lat cyc}, each lattice $L'$ in the inner sum is  of the form $L'^\flat+ \pair{  u}$ where $u$ satisfies
$$
(u, L'^\flat)\subset O_F,\quad (u,u)\in O_F.
$$

Write $u=u^\flat+u^\perp$ according to the orthogonal direct sum $\BV=L_F^\flat\obot (L_F^{\flat})^{\perp}$. 
We {\em claim} that $\val(u^\flat)\geq 0 $ and $\val(u^\perp)\geq 1$.

To prove the claim, we first note that the condition $(u, L'^\flat)\subset O_F$ above is equivalent to $u^\flat\in (L'^\flat)^\vee$. Therefore we may write $u^\flat=\lambda_1e_1'+\cdots+\lambda_{n-1}e'_{n-1}$ where $\lambda_i\in O_F$ ($1\leq i\leq n-2$) and $\lambda_{n-1}\in \varpi^{-a_{n-1}'} O_F$.  By Lemma \ref{lem:lat cyc}, we have 
\begin{align}\label{eqn: cyclic}
\xymatrix{\frac{ \pi_\perp(L')}{L'\cap \pair{x}_F}=\frac{\pair{ u^\perp} }{\pair{ x}}\ar[r]^-\sim& \frac{\pi_\flat(L')}{L'^\flat} = \frac{L'^\flat+\pair{u^\flat}}{L'^\flat}\simeq 
\frac{O_F+\lambda_{n-1} \cdot O_F}{O_F} }.
\end{align}
This isomorphism implies that
\begin{align}\label{eqn: val}
\max\{0,-2\,\val(\lambda_{n-1})\}=-\val(u^\perp)+\val(x).
\end{align}
Now if $\val(u^\flat)< 0$, from $(u,u)\in O_F$ it follows that $\val(u^\flat)=\val(u^\perp)<0$ and $\val(u^\flat)=2\,\val(\lambda_{n-1})+a_{n-1}'$. Hence $2\, \val(\lambda_{n-1})<- a'_{n-1}$ (in particular,  $\val(\lambda_{n-1})<0$). It follows from \eqref{eqn: val} that $\val(x)=a_{n-1}'$, which contradicts $\val(x)>a_{n-1}\geq a_{n-1}'$. Therefore we must have $\val(u^\flat)\geq 0 $ and $\val(u^\perp)\geq 0$.  It then follows that $\val(\lambda_{n-1}e'_{n-1})\geq0$, or equivalently $2\, \val(\lambda_{n-1})+a'_{n-1}\geq 0$. By \eqref{eqn: val}, we have either $\val(u^\perp)=\val(x)\geq 1$ or  
$$\val(u^\perp)= \val(x)+2\,\val(\lambda_{n-1})\geq (1+a_{n-1})- a'_{n-1}\geq 1.$$ Here the last inequality follows from $e_{\max}(L'^\flat)=a_{n-1}' \leq e_{\max}(L^\flat)= a_{n-1}
$ by  \eqref{ineq e max} applied to $L^\flat\subset L'^\flat$. We have thus completed the proof of the claim.

Now we define $\wit L'^\flat\coloneqq \pi_\flat(L') = L'^\flat+\pair{  u^\flat}$. Then $\wit L'^\flat$ is an integral lattice. By  $\val(u^\perp)\geq 1$, we obtain
$$
t(L')=t(\wit L'^\flat)+1.
$$
Moreover, for a given integral lattice $\wit L'^\flat\supset L'^\flat$,  the set of desired integral lattices $L'$ is bijective to the set of generators of the cyclic $O_F$-module $\wit L'^\flat/L'^\flat$. Therefore the inner sum in \eqref{sum pden H} is equal to 
\begin{align}\label{eqn: wt L'}
&\sum_{ L'^\flat\subset \wit L'^\flat } \fkm(t(\wit L'^\flat)+1)[\wit L'^\flat : L'^\flat]\cdot\begin{cases} 1,&\text{if }\wit L'^\flat =L'^\flat, 
\\ (1-q^{-2}),&\text{if } \wit L'^\flat \neq L'^\flat,  \end{cases}
\end{align}
where the index $[\wit L'^\flat : L'^\flat]=\frac{\vol (\wit L'^\flat)}{ \vol( L'^\flat)}$.
For the sum \eqref{eqn: wt L'}, we distinguish three cases.
\begin{enumerate}
\item
If $t(L'^\flat)=0$, i.e., $a_{n-1}'=0$, then the sum is equal to $1$.
\item If $a_{n-1}'>0$ is odd, then   the sum is equal to
$$
(1+q) (1+(q^2-1)+\cdots +(q^{a'_{n-1}-1}-q^{a'_{n-1}-3}))=q^{a'_{n-1}-1}(1+q).
$$
\item If $a_{n-1}'>0$ is even, then   the sum is equal to
$$
(1+q) (1+(q^2-1)+\cdots +(q^{a'_{n-1}-2}-q^{a'_{n-1}-4}))+(q^{a'_{n-1}}-q^{a'_{n-1}-2})=q^{a'_{n-1}-1}(1+q).
$$
\end{enumerate}
Therefore the inner sum  in \eqref{sum pden H}  is equal to
\begin{align}\label{eqn: wt L' 1}
\begin{cases}1,&  t(L'^\flat)=0,\\
(1+q^{-1})\frac{1}{ \vol( L'^\flat)}, & t(L'^\flat)=1.
\end{cases}
\end{align}
We obtain that \eqref{sum pden H} is equal to
\begin{align*}
\sum_{L^\flat\subset L'^\flat,\, t(L'^\flat)= 0}  1 +
\sum_{L^\flat\subset L'^\flat,\, t(L'^\flat)= 1}  (1+q^{-1}) \frac{1}{ \vol( L'^\flat)}
=\frac{1}{ \vol( L^\flat)} \Den((-q)^{-1}, L^\flat),
\end{align*}
by \eqref{eq:valueat1}, and hence
\begin{align*}
\pDen_{L^\flat,\sH}(x)-\pDen_{L^\flat,\sH}(\varpi^{-1} x)= \frac{1}{ \vol( L^\flat)} \Den((-q)^{-1}, L^\flat).
\end{align*}
This completes the proof.
\end{proof}

We introduce an auxiliary function on $\BV\setminus L^\flat_F$,
$$
\wit\pDen_{L^\flat}(x)=\sum_{L^\flat\subset L'\subset L'^\vee }1_{L'}(x).
$$

Similar to Proposition \ref{prop:LC Den}, we have:

\begin{lemma}
\label{lem: aux den}
The function $\wit\pDen_{L^\flat}$ lies in $\rm L^1_c(\BV)$, having logarithmic singularity along $L_F^\flat$.
\end{lemma}
\begin{proof}
It suffices to 
show the logarithmic singularity near $0\in \BV$. The behavior of $\wit\pDen_{L^\flat}$ near an arbitrary $e\in \BV$ is then reduced to this case by the same argument as the proof of part (a) in  Proposition \ref{prop:LC Den} for $\pDen_{L^\flat}$. More precisely, the equality \eqref{eqn diff} also holds for the function $\wit\pDen_{L^\flat}$, since $\wit\pDen_{L^\flat}(x)$ depends only on the fundamental invariants of the lattice $L^\flat+\pair{x}$.

Note that the function $\wit\pDen_{L^\flat}$ is invariant under $L^\flat$-translation. It suffices to show that, when $x\perp L^\flat_F$ and $\val(x)>2\,e_{\max}(L^\flat)$, 
$$
\wit\pDen_{L^\flat}(x)=C_0 \,\val(x)+C_1
$$
for some constants $C_0,C_1$. Write the function as a double sum:
$$
  \wit\pDen_{L^\flat}(x)=\sum_{L^\flat\subset L'^\flat\subset (L'^\flat)^\vee } \sum_{ L'\subset L'^\vee\atop  L'\cap L_F^\flat=L'^\flat}{\bf 1}_{L'}(x).
$$
Since the outer sum has only finitely many terms, it suffices to prove the desired  logarithmic singularity  for the inner sum associated to each integral lattice $L'^\flat\supset L^\flat$. Fix such an $L'^\flat$. It suffices to show that, when $x\perp L^\flat_F$ and $\val(x)>2e_{\max}(L'^\flat)$ (we remind the reader that $e_{\max}(L'^\flat)\leq e_{\max}(L^\flat)$ by \eqref{ineq e max}), the cardinality
\begin{align}\label{eqn: card}
\#\{ L'\mid L'\subset L'^\vee,\, L'\cap L_F^\flat=L'^\flat,   L'\cap  (L_F^{\flat})^{\perp} =\pair{x}\}
\end{align}
is independent of $x$.

Following the proof of Lemma
\ref{lem:diff}, each lattice $L'$ in the above set is  of the form $L^\flat+ \pair{  u}$ where
$$
(u, L'^\flat)\subset O_F,\quad (u,u)\in O_F.
$$
Write $u=u^\flat+u^\perp$ according to the orthogonal direct sum $\BV=L_F^\flat\obot (L_F^{\flat})^{\perp}$. 
We {\em claim} that $\val(u^\perp)\geq 1$. In fact, by $(u, L'^\flat)\subset O_F$,  we obtain $u^\flat\in (L'^\flat)^\vee$, and hence $\length_{O_F} \frac{L'^\flat+\pair{ u^\flat}}{L'^\flat}\leq e_{\max}(L'^\flat)$.
Comparing the lengths of the $O_F$-modules in \eqref{eqn: cyclic},  we obtain
$$
-\val(u^\perp)+\val(x)=2\,\length_{O_F} \frac{L'^\flat+\pair{ u^\flat}}{L'^\flat}\leq 2\,e_{\max}(L'^\flat) .
$$
The claim follows.

It follows that the cardinality \eqref{eqn: card} is given by \eqref{eqn: wt L'} without the weight factor $\fkm(t(\wit L'^\flat)+1)$, hence  independent of $x$. This completes the proof.
\end{proof}

By Proposition \ref{prop:LC Den}, the functions $\pDen_{L^\flat}, \pDen_{L^\flat,\sH}$ and $\pDen_{L^\flat,\sV}$ are all in $\rm L^1(\BV)$, hence Fourier transforms exist for all of them. 

\begin{corollary}
\label{cor: FT}

The Fourier transforms of $\pDen_{L^\flat}$ and $\pDen_{L^\flat,\sH}$ are given by (point-wisely) absolutely  convergent sums
\begin{align}\label{eq:FT pDen H}
\wh\pDen_{L^\flat,\sH}(x)=\sum_{L^\flat\subset L'\subset L'^\vee,\,\, t(L'^\flat)\leq 1 }\vol(L')\fkm(t(L'))1_{L'^\vee}(x),
\end{align}
and
\begin{align}\label{eq:FT pDen}
\wh\pDen_{L^\flat}(x)=\sum_{L^\flat\subset L'\subset L'^\vee }\vol(L')\fkm(t(L'))1_{L'^\vee}(x),
\end{align}
where $x\in\BV$ in both equalities. 
\end{corollary}
\begin{proof}By Proposition \ref{prop:LC Den} and Lemma \ref{lem: aux den}, the two functions $\wit\pDen_{L^\flat}$ and $\pDen_{L^\flat,\sH}$ are $\rm L^1$ and given by a sum of point-wisely non-negative functions.
Since $|\fkm(t(L'))|$ is bounded in the sum defining  $\pDen_{L^\flat}$, the assertion follows from the dominated convergence theorem.
\end{proof}

\subsection{Fourier transform of $\pDen_{L^\flat,\sV}$}

\begin{theorem}\label{thm: pDen=0}
Assume that $x\perp L^\flat$ and $\val(x)<0$. Then
$$
\wh\pDen_{L^\flat,\sV}(x)=0.
$$
\end{theorem}
\begin{proof}This follows from Lemma \ref{lem: pDen FT} below, and the functional equation \eqref{eq: FE}
$$
\Den(-q,L^\flat+\pair{ u^\flat})=\frac{1}{\vol(L^\flat+\pair{ u^\flat})}\Den((-q)^{-1},L^\flat+\pair{ u^\flat}).\qedhere
$$
\end{proof}
\begin{lemma}\label{lem: pDen FT}
Assume that $x\perp L^\flat$ and $\val(x)<0$. Then
$$
\wh\pDen_{L^\flat}(x)=(1-q^{-2})^{-1} \vol (\pair{x}^\vee)\int_{L_F^\flat} \Den(-q,L^\flat+\pair{ u^\flat}) \,du^\flat,
$$
and
\begin{align*}
\wh\pDen_{L^\flat,\sH}(x)=(1-q^{-2})^{-1} \vol (\pair{x}^\vee) \int_{L_F^\flat}\frac{1}{\vol(L^\flat+\pair{ u^\flat})} \Den((-q)^{-1},L^\flat+\pair{ u^\flat}) \,du^\flat.
\end{align*}
Recall that $\pair{x}^\vee$ denotes the dual lattice of $\pair{x}$ in the line $\pair{x}_F$. Here we use the self-dual measures on $L_F^\flat$ and on $\pair{x}_F$ respectively, cf. \S\ref{ss:notation}.
\end{lemma}

\begin{proof}
First we consider the Fourier transform of $\pDen_{L^\flat}$. By \eqref{eq:FT pDen}, it is equal to the (point-wisely) absolutely  convergent sum
$$
\wh\pDen_{L^\flat}(x)=\sum_{L^\flat\subset L'\subset L'^\vee, \, x\in L'^\vee}\vol(L') \fkm(t(L'))  .
$$
For an integral lattice $L'^\flat\supset L^\flat$,  define
\begin{align}\label{def Sigma}
\Sigma(L'^\flat,x)=\{L'\subset \BV\mid x\in L'^\vee, L'\subset L'^\vee, L'^\flat=L'\cap L_F^\flat\}.
\end{align}Then
\begin{align}\label{FT pden}
\wh\pDen_{L^\flat}(x)=\sum_{L^\flat\subset L'^\flat\subset (L'^\flat)^\vee}\sum_{  L'\in \Sigma(L'^\flat,x)}\vol(L') \fkm(t(L')) .
\end{align}

By Lemmas \ref{lem:Lat} and Lemma \ref{lem:lat cyc}, 
we have a bijection  
\begin{equation}\label{eqn: Sigma L x}
	\xymatrix@R=0ex{
 [((\pair{x}+L'^\flat)^{\vee,\circ}/L'^\flat) \setminus (L^\flat_F/L'^\flat)  ] /O_F^\times  \ar[r]^-\sim  &  \Sigma(L'^{\flat},x)\\
		u \ar@{|->}[r]  & L'^\flat+\pair{ u}.
		}
\end{equation}
Here, though $(\pair{x}+L'^\flat)^{\vee,\circ}$ is not necessarily a lattice, it is invariant under $L'^\flat$-translation and $O_F^\times$-multiplication. Hence the quotient on the left hand side makes sense.

Now we follow the same argument as in the proof of Lemma \ref{lem:diff}. Write $u=u^\flat+u^\perp$ according to the orthogonal direct sum $\BV=L_F^\flat\obot (L_F^{\flat})^{\perp}$. 
Then the condition $x\in L'^\vee$ is equivalent to the projection $\pi_\perp(L')\subset \pair{x}^\vee$ (inside the line $(L_F^{\flat})^{\perp}=\pair{x}_F$), or equivalently, $(x,u^\perp)\in O_F$. Since $\val(x)<0$, we must have $\val(u^\perp)>0$ (due to $2\,\val((x,u^\perp))=\val(x)+\val(u^\perp)$).
It follows from the integrality of the norm $(u,u)$  and $(u^\perp, u^\perp)$ that $u^\flat$ also has integral norm and hence
$u^\flat\in (L'^\flat)^{\vee,\circ}$. Thus we have 
$$
(\pair{x}+L'^\flat)^{\vee,\circ}=(L'^\flat)^{\vee,\circ}\obot\pair{x}^\vee
$$ and a bijection with the left hand side of \eqref{eqn: Sigma L x} 
\begin{equation*}
	\xymatrix@R=0ex{(L'^\flat)^{\vee,\circ}/L'^\flat \times \frac{\pair{x}^\vee\setminus \{0\}}{O_F^\times}  \ar[r]^-\sim  &
 [(L'^\flat)^{\vee,\circ}/L'^\flat \times ( \pair{x}^\vee\setminus \{0\}) ] /O_F^\times }
\end{equation*}
sending  $( u^\flat, O_F^\times \cdot \varpi^{m} x)$ to the $O_F^\times$-orbit of $(u^\flat, \varpi^{m} x)$ (with the diagonal $O_F^\times$-action)\footnote{The bijection depends on the choice of a basis vector of $\pair{x}$, and here we have simply chosen $x$.}. We have the resulting  bijection
\begin{equation*}
	\xymatrix@R=0ex{
 (L'^\flat)^{\vee,\circ}/L'^\flat \times \frac{\pair{x}^\vee\setminus \{0\}}{O_F^\times}    \ar[r]^-\sim  &  \Sigma(L'^{\flat},x).
		}
\end{equation*}
The second factor 
$ \frac{\pair{x}^\vee\setminus \{0\}}{O_F^\times}$ can be further identified with the set of lattices 
contained in $\pair{x}^\vee$ (corresponding to $\pair{ u^\perp}=\pi_\perp(L')$).
We write $\wit L'^\flat\coloneqq  \pi_\flat(L')= L'^\flat+\pair{  u^\flat}$. Then $\wit L'^\flat$ is an integral lattice. By  $\val(u^\perp)\geq 1$, we obtain
$$
t(L')=t(\wit L'^\flat)+1,
$$
and by \eqref{eqn:vol},
$$
\vol(L')=\vol(L'^\flat)\vol (\pi_\perp(L')).
$$
Therefore the inner sum in \eqref{FT pden} is equal to 
\begin{align*}
&\vol(L'^\flat)\sum_{u^\flat\in\frac{ (L'^\flat)^{\vee,\circ}}{L'^\flat} } \fkm(t(\wit L'^\flat)+1) \sum_{N\subset \pair{x}^\vee} \vol(N)
\\
=&\vol(L'^\flat) \vol(\pair{x}^\vee)\bigl(\, \sum_{i\geq 0}q^{-2i}\,\bigr)   \sum_{u^\flat\in\frac{ (L'^\flat)^{\vee,\circ}}{L'^\flat} } \fkm(t(\wit L'^\flat)+1)
\\=&\vol(L'^\flat) \vol(\pair{x}^\vee) (1-q^{-2})^{-1} \sum_{u^\flat\in\frac{ (L'^\flat)^{\vee,\circ}}{L'^\flat} } \fkm(t(\wit L'^\flat)+1).
\end{align*}

We now return to the sum \eqref{FT pden}, which is now equal to 
\begin{align}\label{FT pden 1}
\wh\pDen_{L^\flat}(x)&= \vol (\pair{x}^\vee)\sum_{L^\flat\subset    L'^\flat \subset(L'^\flat)^\vee }\vol(L'^\flat)(1-q^{-2})^{-1} \sum_{u^\flat\in\frac{ (L'^\flat)^{\vee,\circ}}{L'^\flat} } \fkm(t(\wit L'^\flat)+1).
\end{align}
For a given integral lattice $\wit L'^\flat$ such that $\wit L'^\flat/L'^\flat$ is a cyclic $O_F$-module,  the number of $u^\flat\in\frac{ (L'^\flat)^{\vee,\circ}}{L'^\flat}$ such that $L'^\flat+\pair{  u^\flat}=\wit L'^\flat$ is 
$$\begin{cases} [\wit L'^\flat : L'^\flat](1-q^{-2})=\frac{\vol (\wit L'^\flat)}{ \vol( L'^\flat)}(1-q^{-2}), &\text{ if }\wit L'^\flat \neq L'^\flat,\\ 
 1,&\text{ if }\wit L'^\flat= L'^\flat.
 \end{cases}
 $$ We thus obtain
\begin{align*}
\wh\pDen_{L^\flat}(x)=&\vol (\pair{x}^\vee) \sum_{L^\flat\subset    L'^\flat  \subset(L'^\flat)^\vee}\vol(L'^\flat)
 \sum_{L'^\flat\subset  \wit L'^\flat,\, \wit L'^\flat/ L'^\flat \text{ cyclic}   } \frac{\vol (\wit L'^\flat)}{ \vol( L'^\flat)}\fkm(t(\wit L'^\flat)+1)\\
&+q^{-2}(1-q^{-2})^{-1}\vol (\pair{x}^\vee)\sum_{L^\flat\subset    L'^\flat \subset(L'^\flat)^\vee }\vol(L'^\flat)  \fkm(t( L'^\flat)+1).
\end{align*}
Here we split the contribution of the factor corresponding to $\wit L'^\flat= L'^\flat$ into two pieces $q^{-2}+(1-q^{-2})$. 
Interchanging the sum over $ L'^\flat$ and $\wit L'^\flat$, we obtain
\begin{align}\label{FT pDen 2}
\wh\pDen_{L^\flat}(x)=&\vol (\pair{x}^\vee)\sum_{L^\flat\subset  \wit L'^\flat \subset (\wit L'^\flat)^\vee}\vol(\wit L'^\flat)  \fkm(t(\wit L'^\flat)+1) \sum_{ L^\flat\subset L'^\flat\subset  \wit L'^\flat,\, \wit L'^\flat/ L'^\flat \text{ cyclic}}1 \\
&+q^{-2}(1-q^{-2})^{-1}\vol (\pair{x}^\vee)\sum_{L^\flat\subset    L'^\flat \subset (L'^\flat )^\vee}\vol(L'^\flat)  \fkm(t( L'^\flat)+1).\notag
\end{align}

Next we consider the integral
$$
\int_{L_F^\flat}\Den(-q,L^\flat+\pair{ u^\flat}) du^\flat.
$$
This can be written as a weighted sum over integral lattices $M\subset L^\flat_F$ such that $L^\flat\subset  M$ and $M/L^\flat$ is a cyclic $O_F$-module, with the weight factor $$
\begin{cases} \vol(M)(1-q^{-2}), &\text{if } M\neq L^\flat,\\
\vol(L^\flat) ,& \text{if } M= L^\flat.
 \end{cases}
 $$
Therefore we obtain
\begin{align}\label{FT pDen 3}
\int_{L_F^\flat}\Den(-q,L^\flat+\pair{ u^\flat}) &du^\flat=q^{-2} \vol(L^\flat)\Den(-q,L^\flat)
\\&+(1-q^{-2}) \sum_{L^\flat\subset  M\subset M^\vee,\, M/L^\flat \text{ cyclic} } \vol(M)\Den(-q,M).\notag
\end{align}
Again here we split the contribution of the factor corresponding to $M=L^\flat$ into two pieces $q^{-2}+(1-q^{-2})$.
By the formula \eqref{eq: Den -q}, the first term is equal to 
\begin{align}\label{FT pDen 3.2}
q^{-2} \vol(L^\flat)\Den(-q,L^\flat)=q^{-2} \sum_{L^\flat\subset    L'^\flat \subset  (L'^\flat)^\vee }\vol(L'^\flat)  \fkm(t( L'^\flat)+1).
\end{align}
Again by \eqref{eq: Den -q}, the second term in \eqref{FT pDen 3} is equal to
\begin{align*} &\sum_{L^\flat\subset  M\subset M^\vee,\, M/L^\flat \text{ cyclic} } \vol(M)\Den(-q,M)\\
=&\sum_{L^\flat\subset  M\subset L'^\flat \subset  (L'^\flat)^\vee,\, M/L^\flat \text{ cyclic} } \vol(M)\frac{\vol( L'^\flat)}{\vol(M)}\fkm(t(L'^\flat)+1)\notag\\
= &\sum_{L^\flat\subset  L'^\flat \subset  (L'^\flat)^\vee}\vol( L'^\flat)\fkm(t(L'^\flat)+1) \cdot \#\{M\mid L^\flat\subset  M\subset L'^\flat ,\, M/L^\flat \text{ cyclic}\}.
\end{align*}
Now note that we have an equality
$$
\#\{M\mid L^\flat\subset  M\subset L'^\flat ,\, M/L^\flat \text{ cyclic}\}=\#\{M\mid L^\flat\subset  M\subset L'^\flat ,\, L'^\flat/M \text{ cyclic}\}.
$$
In fact, the right hand side is the same as
$$\#\{M^\vee \mid L'^{\flat,\vee}\subset  M^\vee \subset L^{\flat,\vee} ,\, M^\vee/ L'^{\flat,\vee} \text{ cyclic}\}.
$$
and this is equal to the left hand side, due to the (non-canonical) isomorphism of finite $O_F$-modules 
$$
\xymatrix@R=0ex{L'^\flat/L^\flat \ar@{-}[r]^-\sim &(L^\flat)^\vee/(L'^\flat)^\vee.}
$$
It follows that
\begin{align}\label{FT pDen 3.1}
&\sum_{L^\flat\subset  M\subset L'^\flat  \subset  (L'^\flat)^\vee,\, M/L^\flat \text{ cyclic} } \vol(M)\Den(-q,M)
\\=&\sum_{L^\flat\subset  L'^\flat  \subset  (L'^\flat)^\vee}\vol( L'^\flat)\fkm(t(L'^\flat)+1)\cdot \#\{M\mid L^\flat\subset  M\subset L'^\flat ,\, L'^\flat/M \text{ cyclic}\}.\notag
\end{align}
By \eqref{FT pDen 3}, \eqref{FT pDen 3.2} and \eqref{FT pDen 3.1}, we obtain
\begin{align}\label{eq:pDen 4}
\int_{L_F^\flat} \Den(-q,L^\flat+\pair{ u^\flat}) du^\flat=&(1-q^{-2})\sum_{L^\flat\subset  L'^\flat  \subset  (L'^\flat)^\vee}\vol( L'^\flat)\fkm(t(L'^\flat)+1)\cdot \sum_{L^\flat\subset  M\subset L'^\flat,\, L'^\flat/M \text{ cyclic}}1\\&+ q^{-2} \sum_{L^\flat\subset    L'^\flat   \subset  (L'^\flat)^\vee}\vol(L'^\flat)  \fkm(t( L'^\flat)+1).\notag
\end{align}

Comparing \eqref{eq:pDen 4} with \eqref{FT pDen 2} we obtain 
$$
\wh\pDen_{L^\flat}(x)=(1-q^{-2})^{-1} \vol (\pair{x}^\vee)\int_{L_F^\flat} \Den(-q,L^\flat+\pair{ u^\flat}) du^\flat,
$$
and this completes the proof of the first part concerning $\wh\pDen_{L^\flat}$.

Similarly, let us consider the horizontal part $\pDen_{L^\flat,\sH}$. By \eqref{eq:FT pDen H},  we have a (point-wisely) absolutely  convergent sum
\begin{align}\label{FT pden H}
\wh\pDen_{L^\flat,\sH}(x)=\sum_{L^\flat\subset L'^\flat   \subset  (L'^\flat)^\vee \atop t(L'^\flat)\leq 1 }\sum_{  L'\in \Sigma(L'^\flat,x)} \fkm(t(L')) \vol(L').
\end{align}Here $\Sigma(L'^\flat,x)$ is the set defined by \eqref{def Sigma}. Similar to the equation \eqref{FT pden 1} for $\wh\pDen_{L^\flat}$, we obtain  
\begin{align*}
\wh\pDen_{L^\flat,\sH}(x)&=\vol (\pair{x}^\vee)\sum_{L^\flat\subset    L'^\flat  \subset  (L'^\flat)^\vee \atop t(L'^\flat)\leq 1 }\vol(L'^\flat)(1-q^{-2})^{-1} \sum_{u^\flat\in\frac{ (L'^\flat)^{\vee,\circ}}{L'^\flat} } \fkm(t(\wit L'^\flat)+1).
\end{align*}
The inner sum is equal to \eqref{eqn: wt L'}, hence equal to \eqref{eqn: wt L' 1}. We obtain
\begin{align*}
\wh\pDen_{L^\flat,\sH}(x)&=(1-q^{-2})^{-1} \vol (\pair{x}^\vee) \sum_{L^\flat\subset    L'^\flat  \subset  (L'^\flat)^\vee\atop t(L'^\flat)\leq 1 }\vol(L'^\flat)\begin{cases}1,&  t(L'^\flat)=0,\\
q^{-1} \fkm(t(L'^\flat)+1)\frac{1}{ \vol( L'^\flat)}, & t(L'^\flat)=1
\end{cases}
\\&=(1-q^{-2})^{-1} \vol (\pair{x}^\vee) \sum_{L^\flat\subset    L'^\flat   \subset  (L'^\flat)^\vee,~ t(L'^\flat)\leq 1 } \begin{cases}1,&  t(L'^\flat)=0,\\
1+q^{-1}, & t(L'^\flat)=1.
\end{cases}
\end{align*}
From the formula  \eqref{eq:valueat1}, it follows that
\begin{align*}
&\int_{L_F^\flat}\frac{1}{ \vol( L^\flat+\pair{ u^\flat})}  \Den((-q)^{-1},L^\flat+\pair{ u^\flat}) \,du^\flat\\
=&\sum_{L^\flat\subset  L'^\flat   \subset  (L'^\flat)^\vee,~ t(L'^\flat)=0} \int_{L_F^\flat}{\bf 1}_{L'^\flat}(u^\flat)\,du^\flat \\&+\sum_{L^\flat\subset L'^\flat   \subset  (L'^\flat)^\vee,~ t(L'^\flat)= 1}  q^{-1}\fkm(t(L'^\flat)+1) \frac{1}{ \vol( L'^\flat)}     \int_{L_F^\flat}{\bf 1}_{L'^\flat}(u^\flat)\,du^\flat \\
=&\sum_{L^\flat\subset  L'^\flat   \subset  (L'^\flat)^\vee,~ t(L'^\flat)=0  } 1+\sum_{L^\flat\subset L'^\flat  \subset  (L'^\flat)^\vee,~ t(L'^\flat)= 1}(1+q^{-1}).
\end{align*}
This completes the proof of the second part concerning the horizontal part.
\end{proof}

\section{Uncertainty principle and the proof of the main theorem}
\subsection{Uncertainty principle}

\subsubsection{Quadratic case}
In this subsection we first let $\BV$ be a (non-degenerate) quadratic space of dimension $d$ over a non-archimedean local field $F$ with characteristic not equal to $2$. Here we allow the residue characteristic to be $p=2$. We denote  by  $(\ ,\ )$ the symmetric bi-linear form on $\mathbb{V}$. Let $ \BV^{\circ} $ (resp. $ \BV^{\circ\circ} $) denote the ``positive cone'' (resp. ``strictly positive cone''), defined by 
\begin{align}
\label{eqn: V circ}
\BV^{\circ}=\{x\in \BV\mid \val((x,x)/2)\geq 0\},\quad \BV^{\circ\circ}=\{x\in \BV\mid \val((x,x)/2)> 0\}.
\end{align}
Fix an unramified additive character $\psi: F\to \BC^\times$ and, similar to the hermitian case (cf. \S\ref{ss:notation}), we define the Fourier transform on 
$C_c^\infty(\BV)$ by
\begin{align}\label{eq:fourierO}
\wh f(x):=\int_\mathbb{V} f(y)\psi ((x,y))\rd y,\quad x\in \mathbb{V}.  
\end{align}

\begin{proposition}\label{uncert o}
Let $\phi\in C_c^\infty(\BV)$ satisfy
\begin{itemize}
\item $\supp(\phi)\subset  \BV^{\circ\circ}$, and 
\item $\supp(\wh\phi)\subset  \BV^{\circ}$.
\end{itemize}
Then $\phi=0$.

\end{proposition}

\begin{proof}
If $\dim\BV$ is odd, we consider the ``doubling'' quadratic space $\BV\obot \BV$ and the function $\phi\otimes \phi\in C_c^\infty(\BV\obot \BV)$. It is easy to see that $\supp(\phi\otimes\phi)=\supp(\phi)\times \supp(\phi)$, $\wh{\phi\otimes \phi}=\wh{\phi}\otimes\wh{ \phi}$, $\BV^{\circ}\times \BV^{\circ}\subset (\BV\obot \BV)^{\circ}$, and $\BV^{\circ\circ}\times \BV^{\circ\circ}\subset (\BV\obot \BV)^{\circ\circ}$. Therefore it suffices to consider the case when $\dim\BV$ is even, which we assume from now on.  We use the the Weil representation $\omega$ of $\SL_2(F)$. The group $\SL_2(F)$ acts on $C_c^\infty(\BV)$ by the following formula
\begin{align}\label{eqn weil}
\omega\left(\begin{matrix} a& \\
& a^{-1}
\end{matrix}\right)\phi(x)&=\chi_{\BV}(a)|a|^{d/2}\phi(ax),\notag\\
\omega\left(\begin{matrix} 1&b \\
&1
\end{matrix}\right)\phi(x)&=\psi\left(\frac{1}{2}\,b\, (x,x)\right)\phi(x),\\
\omega\left(\begin{matrix} &1\\
-1&
\end{matrix}\right)\phi(x)&=\gamma_{\BV}\,\wh\phi(x) ,\notag
\end{align}
where $\chi_{\BV}$ is a quadratic character of $F^\times$ associated to the quadratic space $\BV$, and  $\gamma_{\BV}$ is the Weil constant.

 By the assumption on the support, the functions $\phi$ and $\wh\phi$ are fixed by $N(\varpi^{-1}O_{F})$ and  $N(O_{F})$ respectively, where $N$ denotes the unipotent subgroup of the standard Borel of $\SL_2$ of upper triangular matrices. Therefore $\phi$ is fixed by $N(\varpi^{-1}O_{F})$ and $N_-(O_{F})$ (the transpose of $N(O_{F})$). However, $N(\varpi^{-1}O_{F})$ and $N_-(O_{F})$ generate $\SL_2(F)$. We sketch a proof of this well-known fact. Using the following identity in $\SL_2(F)$
 $$
 \left(\begin{matrix} a&b \\
c&d
\end{matrix}\right)= \left(\begin{matrix} 1&a/c \\
&1
\end{matrix}\right)  \left(\begin{matrix} &-1/c \\
c&
\end{matrix}\right)\left(\begin{matrix} 1&d/c \\
&1
\end{matrix}\right),\quad c\neq 0,
 $$
it is easy to show that the group $\SL_2(F)$ is generated by $N(F)$ and any single element in $\SL_2(F)\setminus B(F)$. Now we first apply the above equality to $\left(\begin{matrix} 1& \\
1&1
\end{matrix}\right)$ (resp. $\left(\begin{matrix} 1& \\
\varpi&1
\end{matrix}\right)$) to generate $\left(\begin{matrix} &-1 \\
1&\end{matrix}\right)$ (resp.  $\left(\begin{matrix} &-1/\varpi \\
\varpi&
\end{matrix}\right)$). Then we note that $\left(\begin{matrix} &-1 \\
1&
\end{matrix}\right)\left(\begin{matrix} &-1/\varpi \\
\varpi&
\end{matrix}\right)=\left(\begin{matrix} -\varpi \\
&-1/\varpi
\end{matrix}\right) $ and this element together with $N(\varpi^{-1}O_{F})$ generate $N(F)$. 

It follows that  $\phi$ is fixed by  $\SL_2(F)$ and therefore $\supp(\phi)$ is contained in the null cone $\{x\in\BV: (x,x)=0\}$ (e.g., by using the invariance under the diagonal torus, or $N(F)$). Since $\phi$ is locally constant, it must vanish identically. 
\end{proof}

\begin{remark}
The uncertainty principle is also used in the new proof by   Beuzart-Plessis \cite{BP} of the Jacquet--Rallis fundamental lemma.
\end{remark}

\begin{corollary}
Let $\phi\in C_c^\infty(\BV)$ satisfy
\begin{itemize}
\item $\supp(\phi)\subset \BV^{\circ\circ}$, and 
\item $\wh\phi$ is a multiple of $ \phi$.
\end{itemize}
Then $\phi=0$.
\end{corollary}

\subsubsection{Hermitian case}\label{sec:uncert-princ}
Now we return to the case of hermitian space with respect to a (possibly ramified) quadratic extension $F/F_0$ where $F_0$ is non-archimedean local field  with characteristic not equal to $2$. Define $\BV^\circ$ and $\BV^{\circ\circ}$ 
\begin{align}\label{eqn: V circ U}
\BV^{\circ}=\{x\in \BV\mid \val(x)\geq 0\},\quad \BV^{\circ\circ}=\{x\in \BV\mid \val(x)> 0\},
\end{align}
where we recall that $\val(x)=\val((x,x))$ for  the hermitian form $(\,,\,)$  on $\BV$, cf. \S\ref{ss:notation}. 

\begin{proposition}\label{uncert}
Let $\phi\in C_c^\infty(\BV)$ satisfy
\begin{itemize}
\item $\supp(\phi)\subset  \BV^{\circ\circ}$, and 
\item $\supp(\wh\phi)\subset  \BV^{\circ}$.
\end{itemize}
Then $\phi=0$.

\end{proposition}

\begin{proof}
Consider the symmetric bilinear form on the underlying $F_0$-vector space of $\mathbb{V}$, $$(x,y)_{F_0}:=\tr_{F/F_0}(x,y)\in F_0,\quad x,y\in\BV.$$
Then the Fourier transform on $C_c^\infty(\BV)$ defined in  \S\ref{ss:notation} using the hermitian form $(\,,\,)$ on $\mathbb{V}$ is the same as the one in (\ref{eq:fourierO}) using $(\,,\,)_{F_0}$ on the underlying $F_0$-vector space of $\mathbb{V}$. Since $(x,x)_{F_0}/2=(x,x)$ for any $x\in \mathbb{V}$, the cones defined by \eqref{eqn: V circ} and \eqref{eqn: V circ U} coincide. Therefore the desired assertion follows from Proposition \ref{uncert o}.
\end{proof}

The uncertainty principle implies that, by Lemma \ref{lem: FT Int V}, the function $\Int_{L^\flat,\sV}$ is determined by its restriction to $$ \BV^{\circ} \setminus  \BV^{\circ\circ} =\{x\in\BV\mid\val(x)=0\}.
$$ 
Ideally one would like to prove the same conclusion as Lemma \ref{lem: FT Int V} holds for the function $\pDen_{L^\flat,\sV}$.  Then, by induction on $\dim\BV$, we can prove the main Theorem \ref{thm: main}.  However, we have not succeeded finding a direct proof the analog of Lemma \ref{lem: FT Int V} for $\pDen_{L^\flat,\sV}$. Nevertheless, a weaker version of  the uncertainty principle suffices to prove the identity $\Int_{L^\flat,\sV}=\pDen_{L^\flat,\sV}$ and this is what we will actually do in the next subsection.  A posteriori we can deduce that the function $\pDen_{L^\flat,\sV}$ also satisfies the same identity as $\Int_{L^\flat,\sV}$ does in Lemma \ref{lem: FT Int V}.

\subsection{The proof of Theorem \ref{thm: main}}\label{ss:proof}
We now prove the main Theorem \ref{thm: main}.
Fix  a rank $n-1$ lattice $L^\flat\subset\BV$ such that $L_F^\flat$ is non-degenerate. We want to prove an identity of functions on $\BV\setminus L^\flat_F$
$$\Int_{L^\flat}=\pDen_{L^\flat}.$$

By Theorem \ref{thm H}, equivalently we need to show

\begin{theorem}Let $L^\flat\subset\BV$ be a rank $n-1$ lattice such that $L_F^\flat$ is non-degenerate. Then
\begin{equation}\label{eqn:int=den}
\Int_{L^\flat,\sV}=\pDen_{L^\flat,\sV}
\end{equation}
as elements in $C_c^\infty(\BV).$
\end{theorem}

\begin{proof}
We prove the assertion by induction on $\val(L^\flat)$.  Let $(a_{1},a_2,\cdots,a_{n-1})$ be the fundamental invariants of the lattice $L^\flat$, cf. \S \ref{sec:latt-theor-notat}. Let $M=M(L^\flat)=L^\flat  \obot\pair{  u}$ be the lattice defined by \eqref{eqn M(L)}.
\begin{lemma} \label{lem:ind}
Let $x\in\BV\setminus L^\flat_F$ and let $(a'_{1},a'_2,\cdots,a'_{n})$ be the fundamental invariants of the lattice $L^\flat+\pair{x}$. Then the inequality 
\begin{equation}\label{ineq}
a_1'+\cdots+a_{n-1}'\geq a_1+\cdots+a_{n-1}
\end{equation}
holds if and only if $x\in M$.
\end{lemma}

\begin{proof} For $x\in M$,  we write $x=x^\flat+x^\perp$ where $x^\flat\in L^\flat$ and $x^\perp\perp L^\flat$. Then $L^\flat+\pair{x}=L^\flat+\pair{x^\perp}$. Therefore we may assume that $x\perp L^\flat$. It follows that $\val(x)\geq a_{n-1}$ by the definition of the lattice $M$,  and $a_i'=a_i$ for $1\leq i\leq n-1$. Hence  $a_1'+\cdots+a_{n-1}'=a_1+\cdots+a_{n-1}$, and the equation  \eqref{ineq} holds.

We now assume that  the inequality  \eqref{ineq} holds. We start with a special case. If $x\perp L^\flat$, the fundamental invariants of the lattice $L^\flat+\pair{x}$ is a re-ordering of $(a_{1},a_2,\cdots,a_{n-1},\val(x))$. From the inequality  \eqref{ineq},   it follows that $\val(x)\geq a_{n-1}$, and therefore  $x\in M$.

Now we consider the general case.  Let $\{e_1,\cdots, e_{n-1}\}$ be an orthogonal basis of $L^\flat$ such that $(e_i,e_i)=\varpi^{a_i}$. Write $$
x=\lambda_1e_1+\cdots +\lambda_{n-1}e_{n-1}+x^\perp,
$$
where $\lambda_i\in F, 1\leq i\leq n-1$ and $x^\perp\perp L^\flat$. The fundamental matrix of the basis $\{e_1,\cdots, e_{n-1},x\}$ of $L^\flat+\pair{x}$ is of the form
$$
T=\begin{pmatrix}\varpi^ {a_1}&&& (e_1,x)\\ 
&\ddots&&\vdots\\
&&\varpi^ {a_{n-1}}& (e_{n-1},x)\\
    (x,e_1)&\cdots&(x,e_{n-1})& (x,x)
\end{pmatrix}.
$$
We now use the characterization of the sum $a_1'+\cdots+a_{n-1}'$ as the minimum among the valuations of the determinants of all $(n-1)\times(n-1)$-minors of $T$. The set of such minors is bijective to the set of $(i,j)$-th entry: removing $i$-th row and $j$-th column to get such a minor. The valuation of the determinant of the $(n,i)$-th minor is $$\val((e_i,x)) -a_i+(a_1+\cdots+a_{n-1}).$$ 
From the inequality  \eqref{ineq}, it follows that
$$
\val((e_i,x)) \geq a_i,
$$
or equivalently $\lambda_i\in O_F$, for all $1\leq i\leq n-1$. Therefore $x-x^\perp\in L^\flat$, and $L^\flat+\pair{x}=L^\flat+\pair{x^\perp}$. Now we can assume that $x\perp L^\flat$ and by the special case above we complete the proof.
\end{proof}

Now we assume that the equation \eqref{eqn:int=den} 
$$
\Int_{L'^\flat,\sV}=\pDen_{L'^\flat,\sV}
$$holds for $L'^\flat$ such that $\val(L'^\flat)<\val(L^\flat)$. 
We may further assume that $L^\flat+\pair{x}$ is integral and has a basis $(e'_1,e'_2,\cdots, e'_n)$ such that $\val(e_i')=a_i'$. Let $L'^\flat=\pair{e_1',\cdots,e_{n-1}'}$. Then  we have
$$
\Int_{L^\flat,\sV}(x)=\Int_{L'^\flat,\sV}(x'),\text{ and}\quad \pDen_{L^\flat,\sV}(x)=\pDen_{L'^\flat,\sV}(x'),
$$
where $x'=e'_n$. 
By Lemma \ref{lem:ind}, if $x\notin M$, then we have a strict inequality 
$$
a_1'+\cdots+a_{n-1}'< a_1+\cdots+a_{n-1}.
$$ And so $\val(L'^\flat)<\val(L^\flat)$. By induction hypothesis, we have  
$$
\Int_{L'^\flat,\sV}(x')=\pDen_{L'^\flat,\sV}(x').
$$
It follows that
the support of the difference 
$$
\phi=\Int_{L^\flat,\sV}-\pDen_{L^\flat,\sV}\in C_c^\infty(\BV)
$$
is contained in the lattice $M$.

By Corollary \ref{cor:FT int}, we know
$$
\wh\Int_{L^\flat,\sV}(x)=-\Int_{L^\flat,\sV}(x).
$$
In particular, if $\val(x)<0$, then
$$
\wh\Int_{L^\flat,\sV}(x)=0.
$$
We know a little less about $\pDen_{L^\flat,\sV}$: by Theorem \ref{thm: pDen=0}, the vanishing $\wh\pDen_{L^\flat,\sV}(x)=0$ holds when  $\val(x)<0$ {\em and}  $x\perp L^\flat$. It follows that,  when  $\val(x)<0$ and $x\perp L^\flat$,
$$
\wh \phi(x)=0.
$$

Obviously the function $\phi$ is invariant under $L^\flat$. By the constraints imposed by the support of $\phi$ (being contained in $M$), it is of the form $$
\phi={\bf 1}_{L^\flat}\otimes \phi_\perp,
$$
where $\phi_\perp\in C_c^\infty((L_F^{\flat})^{\perp})$ is supported on the (rank one) lattice $M_\perp=\pair{   u}$. Then
$$
\wh \phi=\vol(L^\flat){\bf 1}_{L^{\flat,\vee}}\otimes \wh\phi_\perp.
$$
Here $\wh\phi_\perp$ is invariant under the translation by the dual lattice $M_\perp^\vee= \pair{   u^\vee}$, where $u^\vee=\varpi^{-a_{n}} u$. Note that $\val(u^\vee)=-a_{n}<0$. Now the Fourier transform $\wh \phi_\perp$ vanishes at every $x\perp L^\flat$ such that $\val(x)<0$.  It follows that $\wh\phi_\perp$ vanishes identically. Therefore $\phi=0$. This completes the proof. 
\end{proof}

\part{Local Kudla--Rapoport conjecture: the almost self-dual case}\label{part:local-kudla-rapoport-1}

\section{Local density for  an  almost self-dual lattice}\label{sec:local-density-an}
\subsection{Local density for  an  almost self-dual lattice} In this section we allow $F_0$ to be a non-archimedean local field of characteristic not equal to $2$ (but possibly with residue characteristic $2$), and $F$ an unramified quadratic extension.

Recall that we have defined the local density for 
two hermitian $O_F$-lattices $L$ and $M$
$$
\Den(M,L)=\lim_{N\rightarrow +\infty}\frac{\#\Rep_{M,L}(O_{F_0}/\varpi^N)}{q^{N\cdot\dim (\Rep_{M,L})_{F_0}}}
$$
in terms of the scheme $\Rep_{M,L}$, cf. \eqref{def: Rep} in Section \ref{ss:loc den}. 

Let $L$ be a hermitian $O_F$-lattice of rank $n$. For $k\ge0$, set
$$
M=\iden^{n-1+k}\obot \pair{\varpi},\quad \wit M=\iden^{n+1+k},
$$and
\begin{align}\label{eq:L 2 wit L}
L^\sharp=L\obot \ell,\quad\ell= \langle u_0\rangle,\quad (u_0,u_0)=\varpi.
\end{align}
We then have the following ``cancellation law''.
\begin{lemma}\label{lem Den alm dual}
Let $k\geq 0$. Then 
$$\Den(M,L)=\frac{\Den(\wit M,L^\sharp)}{\Den(\wit M,\ell)}.
$$
\end{lemma}
\begin{proof}

For any hermitian $O_F$-lattice $L$, we denote $$L_i =L\otimes_{O_F}O_F/\varpi^i,$$
endowed with the reduction of the hermitian form.

Then the restriction to $\ell_i$ defines a map
$$\rm{Res}\colon
	\xymatrix@R=0ex{
\Herm(L^\sharp_i, \wit M_i)\ar[r]  & \Herm(\ell_i, \wit M_i) \\
		\varphi \ar@{|->}[r]  & \varphi|_{\ell_i}.}
$$
Let  $\varphi\in \Herm(L^\sharp_i, \wit M_i)$. Denote by $\varphi(\ell_i)^\perp$ the orthogonal complement in $\wit M_i$ of the image $\varphi(\ell_i)$, i.e.,
$$
\varphi(\ell_i)^\perp=\{x\in\wit M_i\mid (x, \varphi(\ell_i))=0 \}.
$$

Now let $i\geq 2$. We {\em claim} that there is an isomorphism of hermitian modules over $O_{F}/\varpi^i$:
$$
\xymatrix{ \varphi(\ell_i)^\perp \ar@{->}[r]^-{\sim} & M_i}.
$$
Since the norm of $u_0$ has valuation one, so does its image $w_0\coloneqq \varphi(u_0)\in\wit M_i$ (this makes sense when $i\ge2$). Hence $w_0\notin \varpi \wit M_i$, i.e., $w_0\bmod\varpi\neq 0\in \wit M_1=\wit M_i\otimes_{O_F/\varpi^i} O_F/\varpi$. By the non-degeneracy of the hermitian form on the reduction $\wit M_i$, the map
$$
	\xymatrix@R=0ex{
 \wit M_i \ar[r]  & O_F/\varpi^i   \\
		x\ar@{|->}[r]  &(x,w_0)}
$$
is  surjective, and its kernel is $\varphi(\ell_i)^\perp$ by definition. The kernel is a free module over $O_F/\varpi^i $ (since it must be flat, being the kernel of a surjective morphism between finite free modules; alternatively, look at the reduction $\bmod\, \varpi$ and apply Nakayama's lemma). 

Now there exists $w_0'\in \wit M_i$ such that $(w'_0,w_0)=1$. Then  $\{w_0,w_0'\}$ span a self-dual submodule of rank two, which must  be an orthogonal direct summand of $\wit M_i$, again by non-degeneracy of the hermitian form on $\wit M_i$.
This reduces the  assertion $\varphi(\ell_i)^\perp \simeq  M_i$ to the case 
$\rank \wit M_i=2$. In the rank two case, it is easy to verify the desired isomorphism, e.g., using the basis $\{w_0,w_0'\}$. This proves the claim.

Note that the fiber of the map $\rm{Res}$ above $\varphi|_{\ell_i}$ is the set $\Herm(L_i, \varphi(\ell_i)^\perp) $ (and $\varphi(\ell_i)^\perp$ depends only on the restriction $\varphi|_{\ell_i}$).
It follows from the claim that the fiber has a constant cardinality (in particular, the map $\rm{Res}$ is surjective), namely that of 
$\Herm( L_i, M_i) $. Hence, $$\# \Herm(L^\sharp_i, \wit M_i)=\# \Herm( L_i,  M_i)\cdot \# \Herm( \ell_i, \wit M_i).$$ The result then follows from $$r( L^\sharp)(2 r( \wit M)-r(L^\sharp))=r( L)(2 r(  M)-r(L))+ r( \ell )(2 r(  \wit M)-r(\ell )),$$
where $r$ denotes the rank, cf. (\ref{eq:dimRep}).
\end{proof}

Recall that by  \eqref{eq: iden}
$$
\Den(\iden^{n-1+k},\iden^{n-1})=\prod_{i=1}^{n-1}(1-(-q)^{-i}X)\bigg|_{X= (-q)^{-k}}.
$$
\begin{theorem}\label{thm Den alm dual}
Let $\Lambda= \iden^{n-1}\obot \pair{\varpi}$. Let $k\geq 0$ and $L$ be a hermitian $O_F$-lattice of rank $n$. Then
$$\frac{\Den(\Lambda\obot \iden^{k} ,L)}{\Den( \iden^{n-1+k},\iden^{n-1})}=\Den(X, L^\sharp) \bigg|_{X= (-q)^{-k}}.
$$
\end{theorem}

\begin{proof}
By   \eqref{eq: iden}, we have
\begin{equation*}
\Den(\iden^{n+1+k}, \iden^{1})=(1-(-q)^{-1}X)\bigg|_{X= (-q)^{-n-k}}.
\end{equation*}
and
\begin{equation*}
\Den(\iden^{n+1+k}, \iden^{n+1})=\prod_{i=1}^{n+1}(1-(-q)^{-i}X)\bigg|_{X= (-q)^{-k}}.
\end{equation*}
It follows that
\begin{align*}
\frac{ \Den(\iden^{n+1+k}, \iden^{n+1}) }{\Den(\iden^{n+1+k}, \iden^{1})}&=\prod_{i=1}^{n}(1-(-q)^{-i}X)\bigg|_{X= (-q)^{-k}}
\\&=\Den(\iden^{n+k}, \iden^n).
\end{align*}
(Alternatively, repeat the proof of Lemma \ref{lem Den alm dual} in the case $\ell$ a self-dual lattice of rank one.)

By Example \ref{ex rank=1}, we have $\Den(X,\ell)=1-X$, and hence
$$
\frac{\Den(\iden^{n+1+k},\ell)}{ \Den(\iden^{n+1+k}, \iden^1) }=\Den((-q)^{-n-k},\ell)=(1-  (-q)^{-n}X)\bigg|_{X= (-q)^{-k}}.
$$
It follows that 
\begin{align*}
\frac{ \Den(\iden^{n+1+k}, \iden^{n+1}) }{\Den(\iden^{n+1+k},\ell)}&=\prod_{i=1}^{n-1}(1-(-q)^{-i}X)\bigg|_{X= (-q)^{-k}}\\
&=\Den(\iden^{n-1+k},\iden^{n-1}).
\end{align*}

Finally, by Lemma \ref{lem Den alm dual}, we obtain
\begin{align*}
\frac{\Den(\iden^{n-1+k}\obot\langle \varpi\rangle,L)}{\Den(\iden^{n-1+k},\iden^{n-1})}&=\frac{\Den(\iden^{n+1+k}, L^\sharp)/\Den(\iden^{n+1+k},\ell) }{  \Den(\iden^{n+1+k}, \iden^{n+1})/ \Den(\iden^{n+1+k},\ell )}
\\&=\Den(X, L^\sharp) \bigg|_{X= (-q)^{-k}}.
\end{align*}
This completes the proof.
\end{proof}

\begin{example}[The case $\rank L=2$]\label{ex asd rank=2} If $\rank L=2$,  Theorem \ref{thm Den alm dual} above specializes to Sankaran's formula \cite[Proposition 3.1]{Sankaran2017} which we recall now. Let $L=\pair{\varpi^a}\obot \pair{\varpi^b}, a\leq b$, $a+b$ even. Define 
$$
\epsilon=\begin{cases} 0,& \text{if $b$ is even}\\1,& \text{if $b$ is odd}.
\end{cases}
$$
Then the formula  {\em loc. cit.} asserts that the LHS of  Theorem \ref{thm Den alm dual} is equal to 
\begin{align}\label{San n=2}
&(1-X)(X^2-(q^2-q)X+1)^\epsilon+\frac{1-X}{1-q^{-1}X}\biggl\{ qX(1-q)\frac{(qX)^b-(qX)^\epsilon}{qX-1}\notag
 \\ &+X^2(q-q^{-1}X)\frac{X^{2b}-X^{2\epsilon}}{X^2-1}+ \left[-q^{b+1}(X-1)+qX^{b+1}-q^{-1}X^{b+2}\right]\frac{X^{a+1}-X^{b+1}}{X^2-1} \biggr\}.
\end{align}
On the other hand, this is consistent with the explicit formula for $\Den(X,L^\sharp)$ given by \cite[proof of Theorem 5.2]{Terstiege2013}.
\begin{align}\label{Ter n=3}
 \Den(X,L^\sharp)=\frac{1}{1+X}\left\{\sum_{l=0}^{b+1} X^l(q^l-q^{1+b-l}X^{a+1})-\sum_{l=0}^{b-1} X^{1+l}(q^{2+l}-q^{1+b-l}X^{a+1})\right\}.
\end{align}
In fact, two functions on $(a,b)\in (\BZ_{\geq0})^2$ (not only for $(a+b)$ such that $2\mid a+b$) are characterized by the following properties: 
\begin{itemize}
\item The value at $(0,0)$ (resp., $(1,1)$)  is $1-X$ (resp., $(1-X)(X^2-(q^2-q)X+1)$).
\item The term involving $a$ is  
\begin{align*}&\frac{1-X}{1-q^{-1}X}\cdot \left[-q^{b+1}(X-1)+qX^{b+1}-q^{-1}X^{b+2}\right]\cdot\frac{X^{a+1}}{X^2-1}\\
=&
\frac{1}{1+X} X^{a+1}\left\{- \sum_{l=0}^{b+1} X^l q^{1+b-l}+\sum_{l=0}^{b-1} X^{1+l}q^{1+b-l}\right\}.
\end{align*}The two expressions come from \eqref{San n=2}  and  \eqref{Ter n=3} respectively.
\item The term not involving $a$ is a function $\phi$ in one variable $b\in\BZ_{\geq 0}$, which satisfies a difference equation
 $$
\phi(b)-\phi(b-1)=\frac{1}{1+X}q^{b+1} X^{b}(X- 1).
$$
The difference equation is easy to see from \eqref{Ter n=3}, and from \eqref{San n=2} by a straightforward calculation. 
\end{itemize}

\end{example}

\begin{definition}
After Theorem \ref{thm Den alm dual}, define the (normalized) \emph{local Siegel series relative to $\Lambda= \iden^{n-1}\obot \pair{\varpi}$} as the polynomial $\Den_\Lambda(X,L)\in \mathbb{Z}[X]$ 
such that 
\begin{align}
\Den_\Lambda((-q)^{-k},L)=\frac{\Den(\Lambda\obot \iden^{k} ,L)}{\Den(\iden^{n-1}\obot \iden^{k},\iden^{n-1})}.
\end{align}
\end{definition}
Then by Theorem \ref{thm Den alm dual},
\begin{align}\label{eq:Den sharp}
\Den_\Lambda(X,L)=\Den(X,L^\sharp)\in \mathbb{Z}[X].
\end{align}
In particular, if $\val(L)$ is even, then $\Den_\Lambda(1, L)=0$. In this case, we denote the central derivative of local density by $$\pDen_\Lambda(L)\coloneqq -\frac{\rd}{\rd X}\bigg|_{X=1}\Den_\Lambda(X,L).$$

\subsection{Relation with local Whittaker functions}\label{sec:relation-with-local-1} Let $\Lambda=\langle 1\rangle^{n-1}\obot \langle\varpi\rangle$ be an almost self-dual hermitian $O_F$-lattice. Let $L$ be a hermitian $O_F$-lattice of rank $n$. Let $T=((x_i, x_j))_{1\le i,j\le n}$ be the fundamental matrix of an $O_F$-basis $\{x_1,\ldots,x_n\}$ of $L$, an $n\times n$ hermitian matrix over $F$. Associated to the standard Siegel--Weil section of the characteristic function $\varphi_{1}=\mathbf{1}_{\Lambda^n}$ and the unramified additive character $\psi: F_0\rightarrow \mathbb{C}^\times$, there is a local (generalized) Whittaker function $W_T(g, s, \varphi_1)$ (see \S\ref{sec:four-coeff-deriv}, \S\ref{sec:incoh-eisenst-seri} for the precise definition). By \cite[Proposition 10.1]{Kudla2014}, when $g=1$, it satisfies the interpolation formula for integers $s=k\ge0$ (notice $\gamma_p(V)=-1$ in the notation there), $$W_T(1,k, \varphi_1)=(-q)^n\cdot\Den(\Lambda \obot \langle 1\rangle^{2k}, L).$$ So its value at $s=0$ is $$W_T(1, 0, \varphi_1)=(-q)^{-n}\cdot\Den(\Lambda, L)=(-q)^{-n}\cdot\Den_\Lambda(L)\cdot \Den(\langle 1\rangle^{n-1}, \langle 1\rangle^{n-1}),$$ and its derivative at $s=0$ is $$W_T'(1, 0, \varphi_1)=(-q)^{-n}\cdot\pDen_\Lambda(L)\cdot \Den(\langle 1\rangle^{n-1}, \langle 1\rangle^{n-1})\cdot\log q^2.$$ Plugging in \eqref{eq: iden}, we obtain
\begin{align}
W_T(1, 0, \varphi_1)&=\Den_\Lambda(L)\cdot (-q)^{-n}\prod_{i=1}^{n-1}(1-(-q)^{-i}),\\ W_T'(1, 0, \varphi_1)&=\pDen_\Lambda(L)\cdot (-q)^{-n}\prod_{i=1}^{n-1}(1-(-q)^{-i})\cdot \log q^2.    \label{eq:localWhittaker2}
\end{align}

\section{Kudla--Rapoport cycles in the almost principally polarized case}

\subsection{Rapoport--Zink spaces $\mathcal{N}^1_n$ with almost self-dual level}\label{sec:rapoport-zink-with}We recall the construction from \cite[\S5]{Rapoport2018}. For a $\Spf \OFb$-scheme $S$, we consider triples  $(Y,\iota, \lambda)$ over $S$ as in \S\ref{sec:rapoport-zink-spaces}, except that $\lambda$ is no longer principal; instead, it is required that $\ker\lambda\subseteq Y[\iota(\varpi)]$ has order $q^2$. Up to $O_F$-linear quasi-isogeny compatible with polarizations, there is a unique such triple $(\mathbb{Y}, \iota_{\mathbb{Y}}, \lambda_{\mathbb{Y}})$ over $\Spec \bar k$. Let $\mathcal{N}^1=\mathcal{N}_n^1=\mathcal{N}_{F/F_0, n}^1$ be the formal scheme over $\Spf \OFb$ which represents the functor sending each $S$ to the set of isomorphism classes of tuples $(Y, \iota, \lambda, \rho)$, where the \emph{framing} $\rho: Y\times_S \bar S\rightarrow \mathbb{Y}\times_{\Spec \bar k}\bar S$ is an $O_F$-linear quasi-isogeny of height 0 such that $\rho^*((\lambda_\mathbb{Y})_{\bar S})=\lambda_{\bar S}$. 

The Rapoport--Zink space $\mathcal{N}^1=\mathcal{N}_n^1$ is a formal scheme, locally formally of finite type, regular, of relative dimension $n-1$ and  has semi-stable reduction over $\Spf \OFb$ (\cite[Theorem 5.1]{Rapoport2018}, \cite[Theorem 1.2]{Cho2018}). Denote $$
\BW_n=\Hom^\circ_{O_F}(\ov\BE,\BY),
$$
and endow it with the hermitian form by the formula similar to $\BV_n$ (cf. \S\ref{sec:herm-space-mathbbv}). Then $\BW_n$ is a {\em split} hermitian space of dimension $n$. Similar to \S\ref{sec:kudla-rapop-cycl}, for every non-zero $x\in \BW_n$ we can define the special divisor on $\CN^1_n$, denoted by $\CY(x)$  (resp. $\CY'(x)$), over which the special homomorphism $x$ (resp. $\lambda_{\BY}\circ x$) extends to a homomorphism $\bar{\mathcal{E}}_S\rightarrow Y$ (resp. $\bar{\mathcal{E}}_S\rightarrow Y^\vee$).  Then, by a similar argument to the self-dual case in \cite[Proposition~3.5]{Kudla2011}, the special divisors $\CY(x)$ and  $\CY'(x)$ are Cartier (cf. \cite[Proposition~5.9]{Cho2018}, denoted  by $\CZ(x)$ and $\CY(x)$ in {\it loc. cit.}).

For the later use, we recall from \cite[Proposition~5.10]{Cho2018} 
\begin{equation}\label{eq:isom2}
 \begin{cases} \CY(x)\simeq \CN_{n-1}^1, &\text{ when } \val(x)=0,\\
   \CY'(x)\simeq\CN_{n-1},&\text{ when } \val(x)=-1.
   \end{cases}
\end{equation}
 We only indicate the construction of the second isomorphism $\CY'(x)\simeq\CN_{n-1}$, since the first one is rather close to the self-dual case (cf. \eqref{eq:inc M}). We may assume that $(x,x)=\varpi^{-1}$. Fix an $O_F$-linear isomorphism 
\begin{equation}
\beta\colon \BX_{n-1}\times\ov\BE\to \BY,
\end{equation}
such that $\beta^*(\lambda_{\BY})=\lambda_{\BX_{n-1}}\times \varpi\lambda_{\ov \BE}$ and the restriction of $\beta$ to the second factor is $\varpi x \in \BW_{n}$.  Then we have an orthogonal decomposition $\BW_n=\BV_{n-1}\obot\pair{x}_F$. Then we define a map $\delta: \CN_{n-1}\to \CN_{n}^1$ sending $(X^\flat, \iota_{X^\flat}, \lambda_{X^\flat}, \rho_{X^\flat})\in \CN_{n-1}(S)$ to $ (X^\flat\times\ov\CE, \iota_{X^\flat}\times\iota_{\ov\CE}, \lambda_{X^\flat}\times \varpi \lambda_{\ov\CE}, \rho_{X^\flat}\times\rho_{\ov\CE})\in  \CN_{n}^{1}(S)$. Then the homomorphism $(0,\lambda_{\ov\CE}):\ov\CE\to (X^\flat)^\vee\times\ov\CE^\vee$ extends $\lambda_{\BY}\circ x\in \BW_n$ and the map $\delta$ defines an isomorphism $\delta: \CN_{n-1}\simeq \CY'(x)$.

\begin{definition}
Let $L\subset \BW_n$ be an $O_F$-lattice of rank $n$ and let  $x_1,\cdots,x_n$ be an $O_F$-basis of $L$. Then we define 
 \begin{align}\label{eq:def Int'}
\Int'(L)\coloneqq \chi\bigl(\mathcal{N}^1_n, \mathcal{O}_{\mathcal{Y}(x_1)} \otimes^\mathbb{L}\cdots \otimes^\mathbb{L}\mathcal{O}_{\mathcal{Y}(x_n)}\bigr ).
\end{align} 
\end{definition}
 We have not justified the independence of the choice of the basis, which will be proved under a conjectural relation between $\CN_n^1$ and some auxiliary Rapoport--Zink spaces. It turns out that $\Int'(L)$ is not equal to the derived local density $\pDen_\Lambda(L)$ (cf. Theorem \ref{thm: main2'} below). This is a typical phenomenon in the presence of bad reductions, cf. \cite{Kudla2000,Sankaran2017,RSZ1, Rapoport2018}. Therefore, we will instead define a variant $\Int(L)$ of  $\Int'(L)$, which will give an  
exact identity $\Int(L)=\pDen_\Lambda(L)$ (Theorem \ref{thm: main2}).

\subsection{Auxiliary Rapoport--Zink spaces}\label{sec:auxil-rapop-zink}
 Before we present our variant, we need an auxiliary moduli space (cf. \cite{KRSZ}). Let $(\mathbb{X}_{n+1}, \iota_{\mathbb{X}_{n+1}}, \lambda_{\mathbb{X}_{n+1}})$ be as in \S\ref{sec:rapoport-zink-spaces}. Fix an $O_F$-linear isogeny
\begin{equation}
\alpha\colon \BY\times\ov\BE\to \BX_{n+1} ,
\end{equation}
such that $\ker\alpha\subset (\BY\times\ov \BE)[\varpi]$ and $\alpha^*(\lambda_{\BX_{n+1}})=\lambda_\BY\times \varpi\lambda_{\ov \BE}$. 
Let $x_0\in \BV_{n+1}=\Hom^\circ_{O_F}(\ov\BE,\BX_{n+1})$ be the restriction of $\alpha$ to the second factor. Then the assumption implies that the norm of $x_0$ is
$(x_0,x_0)=\varpi,$ and we have an orthogonal decomposition
$$
\BV_{n+1}=\BW_n\obot \pair{ x_0}_F.
$$

We denote by 
\begin{equation}\label{eq:KRSZ0}
\wit\CN_{n}^{1}\incl \CN_{n}^{1}\times_{\Spf \OFb}\CN_{n+1}
\end{equation}
the closed formal subscheme consisting of tuples $(Y, \iota_Y, \lambda_Y,  \rho_Y,  X, \iota_X, \lambda_X, \rho_X)$ such that $\alpha$ lifts to an isogeny $\wit\alpha\colon Y\times\ov\CE\to X$. If $\alpha$ lifts, then $\wit\alpha$ is unique and satisfies $\ker\wit\alpha\subset (Y\times\ov\CE)[\varpi]$ and $\wit\alpha^*(\lambda_X)=\lambda_Y\times \varpi \lambda_{\ov\CE}$. 

We therefore obtain a diagram 
\begin{align}\label{eq:KRSZ}
   \xymatrix{
	       &\wit\CN_{n}^{1}\ar[dl]_-{\pi_1} \ar[dr]^-{\pi_2}\\
	 \CN_{n}^{1}  &  &\CZ(x_0)\ar[r]&\CN_{n+1}  , 
	}
\end{align}
where $\pi_1$ and $\pi_2$ are induced by the two projection maps in \eqref{eq:KRSZ0}. Recall from \cite{Terstiege2013b} that the formal scheme $\CZ(x_0)$ is regular. 

\begin{remark}\label{rem:corr}
Let  $\Lambda= \iden^{n-1}\obot \pair{\varpi}$ be as before. Let $\Lambda^\sharp$ be a self-dual lattice of rank $n+1$ containing $ \Lambda \oplus \pair{\varpi}$; there are $q+1$ such lattices  in the vector space $\Lambda_F \oplus \pair{\varpi}_F$. Then we have a natural embedding of hermitian spaces
$$
W_n\coloneqq\Lambda\otimes_{O_F} F \incl V_{n+1}\coloneqq\Lambda^\sharp\otimes_{O_F} F
$$
and their isometry groups $\U(W_n)\incl \U(V_{n+1})$. Let $K=\Aut(\Lambda)$ be the stabilizer of $\Lambda$, and similarly let $K^\sharp=\Aut(\Lambda^\sharp)$. Define $\wit K\coloneqq K\cap K^\sharp$ where the intersection is taken inside the unitary group $\U(V_{n+1})$:
\[
   \xymatrix{
	       &\wit K=K\cap K^\sharp \ar[dl] \ar[dr]\\
K=\Aut(\Lambda)  &  &K^\sharp=\Aut(\Lambda^\sharp) .
	}
\]
Intuitively, the Rapoport--Zink spaces $\CN^1_n,\wit\CN^1_n$, and $\CN_{n+1}$ correspond to the level structure $K,\wit K$, and $K^\sharp$ respectively. 
It is easy to see that the generic fiber of the map $\pi_1: \wit\CN_{n}^{1}\to \CN_{n}^{1}$ is finite \'etale of degree $[K:\wit K]=q+1$, and the generic fiber of the map  $\pi_2: \wit\CN_{n}^{1}\to \CZ(x_0)$ is an isomorphism. 
Therefore, $\CZ(x_0)$ is a regular integral model of a finite \'etale covering of the generic fiber of $\CN_{n}^1$. 

\end{remark}

Let $x\in\BW_n\subset \BV_{n+1}$. Denote by $\CZ^\flat(x)$ the restriction of the special divisor $\CZ(x)$ (on $\CN_{n+1}$) to $\CZ(x_0)$, i.e.,
$$
\CZ^\flat(x)\coloneqq\CZ(x_0)\cap \CZ(x)
$$
viewed as a formal subscheme on $\CZ(x_0)$. 

\begin{remark}\label{rem:altintegralmodel}It will be clear (cf. Theorem \ref{thm:D x}) that the generic fiber of $\CZ^\flat(x)$ (viewed as a divisor on the generic fiber of $\wit\CN_n^1$ since $\pi_2$ is an isomorphism on the generic fibers) is equal to the pull back along $\pi_1$ of the generic fiber of $\CY(x)$ on $\CN_n^1$. Therefore, we may use $\CZ^\flat(x)$ as an integral model of the pull-back of the generic fiber of $\CY(x)$.  
\end{remark}
Motivated by Remark \ref{rem:corr},  we now define a variant of $\Int'(L)$. 
\begin{definition}
Let $L\subset \BW_n$ be an $O_F$-lattice of rank $n$ and let  $x_1,\cdots, x_n$ be a basis of $L$. Then we define
\begin{align}\label{eq:Int L basis}
\Int(L; x_1,\cdots,x_n)=\frac{1}{q+1} \chi\bigl( \CZ(x_0),\CZ^\flat(x_1)\jiao\cdots \jiao  \CZ^\flat(x_n)\bigr ),
\end{align}
where the derived tensor product is taken as $\CO_{\CZ(x_0)}$-sheaves.   \end{definition}

\subsection{The $\Int=\pDen$ theorem}
The following theorem justifies our definition of the variant of intersection numbers. 
\begin{theorem}\label{thm: main2}
Let $L\subseteq \mathbb{V}$ be an $O_F$-lattice of full rank $n$. Then, for any basis $x_1,\cdots, x_n$ of $L$, we have $$\Int(L; x_1,\cdots,x_n)=\frac{1}{q+1}\pDen_\Lambda(L).$$
In particular, $\Int(L; x_1,\cdots,x_n)$ is independent of the choice of the basis and we therefore denote it by $\Int(L)$.  
\end{theorem}

\begin{proof}Let $x\in\BW_n$ be non-zero. Then $x\perp x_0$. 
By Lemma \ref{lem:two div}, we have
$$
\CO_{\CZ^\flat(x)}=\CO_{\CZ(x)}\otimes^\BL \CO_{\CZ(x_0)}
$$
as elements in $K_0'(\CZ(x_0))$.  Therefore, 
$$
\chi\bigl(\CZ(x_0), \CZ^\flat(x_1)\jiao\cdots \jiao  \CZ^\flat(x_n)\bigr )=\chi\bigl(\CN_{n+1},\CZ(x_0)\jiao \CZ(x_1)\jiao\cdots \jiao  \CZ(x_n)\bigr),
$$
which is $\Int(L^\sharp)$. By Theorem \ref{thm: main}, this is equal to $\pDen( L^\sharp)$. By \eqref{eq:Den sharp}  we obtain $\pDen( L^\sharp)=\pDen_\Lambda(L)$ and the proof is complete.
\end{proof}

\begin{remark}\label{rem:localKR2}
In the notation of \S\ref{sec:relation-with-local-1},  it follows immediately from Theorem \ref{thm: main2} and \eqref{eq:localWhittaker2} that $$\Int(L)=\frac{W_T'(1, 0, \varphi_1)}{\log q^2}\cdot  \frac{(-q)^n-1}{q+1}\cdot \prod_{i=1}^{n}(1-(-q)^{-i})^{-1}.$$ 
\end{remark}

\subsection{The comparison of two divisors} 
We compare the two divisors $ \CY(x)$ and $\CZ^\flat(x)$ after pulling-back to $\wit\CN_n^1$ along the diagram \eqref{eq:KRSZ}. The result is conditional on the conjectural  relation between $\CN^1_n$, 
$\wit\CN^1_n$ and $\CZ(x_0)$. 

Recall from \eqref{eq:KRSZ} that there are two projections $\pi_1$ and $\pi_2$. 
Let $\Ver^0(\BW_n)$ be the set of self-dual lattices $\Lambda$  in $\BW_n$. For each type $1$ lattice in $\BV_{n+1}$ of the form $
\Lambda\oplus \pair{x_0}$ where $\Lambda\in \Ver^0(\BW_n)$, there is a closed stratum  $\CV(\Lambda\oplus \pair{x_0})\subset\CN_{n+1}^{\red}$ which consists of a superspecial point (cf. \S\ref{sec:bruh-tits-strat})  contained in  $ \CZ(x_0)$. Let $\CZ(x_0)^{\rm ss}\subset \CZ(x_0)$ be the union of all of such superspecial points. Note that $\CZ(x_0)^{\rm ss}$ does not contain {\em all}  superspecial points on $\CZ(x_0)$. 

By the Bruhat--Tits stratification of the reduced locus of $\CN_n^1$ in \cite{Cho2018}, there exist a family of (disjoint) closed formal subschemes $\BP_{\Lambda}=\BP^{n-1}\incl \CN_n^1$ indexed by  $\Lambda\in \Ver^0(\BW_n)$ (cf. Remark 2.15 of {\it loc. cit.}). Denote by $\CN_n^{1,{\rm ss}}$ the  (disjoint) union of them.  

The following conjecture was observed by Kudla and Rapoport in an unpublished manuscript.
\begin{conjecture}\label{conj:KRSZ}
    \begin{altenumerate}
  \item The morphism $\pi_1$ is finite flat of degree $q+1$, \'etale away from $\CN_n^{1,{\rm ss}}$, and totally ramified along $\CN_n^{1,{\rm ss}}$.
   \item  The morphism $\pi_2$ is the blow-up of $\CZ(x_0)$ along the zero-dimensional subscheme $\CZ(x_0)^{\rm ss}$. \item The preimage of $\CN_n^{1,{\rm ss}}$ under $\pi_1$ is exactly the exceptional divisor on $\wit\CN^1_n$.
    \end{altenumerate}
\end{conjecture}
In \cite{KRSZ} the authors will prove this conjecture, which from now on we assume to hold. It follows from the conjecture that $\wit\CN_n^1$ is regular and the exceptional divisor above $\CV(\Lambda\oplus \pair{x_0})$ for  $\Lambda\in \Ver^0(\BW_n)$ is a projective space isomorphic to $\BP^{n-1}$ which we will again denote it by $\BP_{\Lambda}$. For $\Lambda^\sharp\in\Ver(\BV_{n+1})$ containing $x_0$, we let $\wit\CV(\Lambda^\sharp)\subset\wit\CN_n^1$ be the strict transform of $\CV(\Lambda^\sharp)\subset \CZ(x_0)$ under the blow-up morphism $\pi_2$.

We remind the reader of our convention in  \S\ref{sec:notat-form-schem}.
\begin{lemma}\label{lem:D x0}
Let $n\geq 1$. Let $x \in \BW_n$ be non-zero vector. Then both $\pi_1^{-1} (\CY(x))$ and $\pi_2^{-1}( \CZ^\flat(x))$ are
Cartier divisors on $\wit\CN^{1}_n$. Moreover,  we have $\pi_1^\ast (\CY(x))=\pi_1^{-1} (\CY(x))$ (in $K_0^{\pi_1^{-1} (\CY(x))}(\wit\CN^{1}_n)\simeq K_0'(\pi_1^{-1} (\CY(x)))$) and $\pi_2^\ast( \CZ^\flat(x))=\pi_2^{-1}( \CZ^\flat(x))$ (in $K_0^{\pi_2^{-1} (\CZ^\flat(x))}(\wit\CN^{1}_n)\simeq K_0'(\pi_2^{-1} (\CZ^\flat(x)))$).
\end{lemma}
\begin{proof} We have the following observation which applies to both $\pi_1$ and  $\pi_2$. Let $\pi:\CX\to\CY$ be a morphism between regular formal schemes such that for every $z\in \CX^{\red}$ the induced map on local rings $\pi_z^\sharp:\CO_{\CY,\pi(z)}\to \CO_{\CX,z} $ is injective.  Let $D$ be a Cartier divisor on $\CY$. Then $\pi^{-1}(D)$ is a Cartier divisor. 

Now that both $\pi_1^{-1}(\CY(x))$ and $\pi_2^{-1}(\CZ^\flat(x))$ are Cartier divisors on $\wit\CN_n^1$, it follows  that $\pi_1^\ast( \CO_{\CY(x)})= \CO_{\pi_1^{-1}(\CY(x))}$ and  $\pi_2^\ast( \CO_{\CZ^\flat(x)})= \CO_{\pi_2^{-1}(\CZ^\flat(x))}$ (e.g., by the argument of \cite[Lemma~B.2~$(i)$]{Zhang2019}, which applies even if the vertical map in {\it loc. cit.} is not a closed immersion). 
\end{proof}

\begin{theorem}
\label{thm:D x}
Let $n\geq 1$. Let $x \in \BW_n$ be non-zero vector. Define a locally finite (Cartier) divisor on $\wit\CN^1_n$
$$
{\rm Exp}(x)\coloneqq\sum_{\Lambda\in \Ver(x)} \BP_{\Lambda}, 
$$where 
$$ \Ver(x)\coloneqq\{\Lambda\subset \BW_n\mid \Lambda^\vee=\Lambda, x\in \Lambda\}.
$$
Then  there is an equality of (Cartier) divisors on $\wit\CN^{1}_n$
$$
\pi_1^{-1}( \CY(x))=\pi_2^{-1} (\CZ^\flat(x))-{\rm Exp}(x).
$$
\end{theorem}

%
%
%

\begin{proof}


We first note that a point $\CV(\Lambda\obot\pair{x_0})$ in $ \CZ(x_0)^{\rm ss}$ corresponding to $\Lambda\in\Ver^0(\BW_n)$ lies on $\CZ^\flat(x)$ if and only if $x\in \Lambda$. 

From the moduli interpretations, it is clear that $\pi_1^{-1} (\CY(x))\subset \pi_2^{-1} (\CZ^\flat(x))$. Denote $$\mathfrak{D}:=\pi_2^{-1} (\CZ^\flat(x))-\pi_1^{-1} (\CY(x)).$$ Let $\mult_\Lambda(x)\in\BZ_{\geq 0}$ be the multiplicity of $\BP_\Lambda$ (as a Cartier divisor on $\wit\CN_n^1$) in $\mathfrak{D}$. Set 
$$
\mathfrak{D}':=\mathfrak{D}- \sum_{\Lambda\in \Ver(x)} \mult_\Lambda(x)\,\BP_{\Lambda},
$$
Then $\mathfrak{D}'$ is an effective Cartier divisor and it does not contain any of the $\BP_\Lambda$. It suffices to show that $\mathfrak{D}'=0$ and $ \mult_\Lambda(x)=1$ for every  $\Lambda\in \Ver(x)$.

We proceed by induction on $n\geq 1$. When $n=1$, we know that $\CN^1_{1}\simeq \Spf \OFb$,  $\pi_2$ is an isomorphism $\wit\CN_1^1\simeq\CZ(x_0)$ where both $\wit\CN_1^1$ and $\CZ(x_0)$ are isomorphic to a quasi-canonical lifting, the degree $q+1$ ramified cover $\Spf \OFb_{,1}$ of $ \Spf \OFb$.   Then $\CY(x)$ is non-empty unless $\val(x)\geq 2$ (note that $\val(x)$ is even), in which case it has $\OFb$-length $\frac{\val(x)}{2}$ by the theory of canonical lifting. By \cite{Kudla2011}, we also know that the divisor $\CZ^\flat(x)=\CZ(x)\cap\CZ(x_0)$ has $\OFb_{,1}$-length $1+(q+1)\frac{\val(x)}{2}$.
Therefore the desired equality of Cartier divisors on $\wit\CN_1^1$ follows by comparing the $\OFb_{,1}$-lengths.


Now let $n\geq 2$. We note that the case $n=2$ is slightly different (and in fact easier). We {\em claim} that, when $n=2$, the divisor $\mathfrak{D}'$ is vertical (i.e., its structure sheaf is annihilated by a power of $\varpi$). To show the claim, we first consider the case: $(x,x)\in O_F$ and $(x,x)\neq 0$. By \cite[Theorem~2.9]{Sankaran2017} and the finite flatness of $\pi_1$ (Conjecture  \ref{conj:KRSZ} $(i)$), the horizontal component of $\pi_1^{-1} (\CY(x))$ has degree $q+1$ (over $\OFb$); similarly by Theorem \ref{thm:horizontal} and Conjecture  \ref{conj:KRSZ} $(iii)$, the horizontal component of $\pi_2^{-1} (\CZ^\flat(x))$ also has degree $q+1$. By $\pi_1^{-1} (\CY(x))\subset \pi_2^{-1} (\CZ^\flat(x))$, their horizontal components must cancel out in $\mathfrak{D}'$. We then consider the remaining case: $(x,x)=0$ and $x\neq 0$. We apply Lemma \ref{lem:tate} to deduce that $\CZ^\flat(x)=\CZ(x_0)\cap\CZ(x)\subset\CN_3$ has no horizontal component (otherwise the rank two lattice $\pair{x,x_0}$ is embedded into the self-dual lattice $L$ in Lemma \ref{lem:tate}; however, the orthogonal complement of $\pair{x_0}_F$ in $L_F$ is a two dimensional non-split hermitian space which has no non-zero isotropic vector). This proves the claim. Now  by \cite[Theorem~0.2]{Terstiege2013}, the special fiber of $\CZ(x_0)$ (as a Cartier divisor on $\CZ(x_0)$) is the sum of $\wit\CV(\Lambda^\sharp)$ for all $\Lambda^\sharp\in\Ver^3(\BV_3)$ containing $x_0$. It follows that we may express $\mathfrak{D}'$ as a (locally finite) sum of Cartier divisors
\begin{align}\label{eq:D' n=2}
\mathfrak{D}'=\sum_{\Lambda^\sharp\in\Ver^3(\BV_3),\,x_0\in \Lambda^\sharp} \mult_{\Lambda^\sharp} (x)\,\wit\CV(\Lambda^\sharp).
\end{align}

We now return to $n\geq 2$. The basic idea is to intersect the given divisors with (many) well-positioned special divisors that are isomorphic to $\CN_{n-1}^1$ or  $\CN_{n-1}$ (cf. \eqref{eq:isom2}). We first determine $\mult_\Lambda(x)$ and show that, when $n\geq 3$, the divisor $\mathfrak{D}'$ does not intersect any of the $\BP_\Lambda$ for $\Lambda\in\Ver(x)$. We fix a $\Lambda_0\in \Ver(x)$. Since $\Lambda_0$ is self-dual of rank  $n\geq 2$, there exists a vector $e\in \Lambda_0$ such that $\val(e)=0$ and that $e$ is linearly independent of $x$. We have an orthogonal decomposition $\BW_{n}=\BW_{n-1}\obot\pair{e}$, and let $x^\flat\in \BW_{n-1}$ be the projection of $x$ to $\BW_{n-1}$. Then $x^\flat\neq 0$. By \eqref{eq:isom2}, the special divisor $\CY(e)$ on $\CN_n^1$ is isomorphic to $\CN^1_{n-1}$. We consider the commutative diagrams  with the obvious maps 
\begin{align}\label{eq:diag case1}
   \xymatrix{
	       &\wit\CN_{n-1}^{1}\ar[d]_{\wit\delta} \ar[dl]_-{\pi_1^\flat} \ar[dr]^-{\pi_2^\flat}\\
	 \CN_{n-1}^{1}\simeq \CY(e) \ar[d]  &   \wit\CN_{n}^{1} \ar[dl]_-{\pi_1}\ar[dr]^-{\pi_2}&\CZ(\pair{e,x_0})\ar[d]\ar[r]& \CZ(e)\simeq \CN_{n}\ar[d]  \\
	 \CN_{n}^{1} &&\CZ(x_0)\ar[r]&\CN_{n+1}.
	}
\end{align}
 The leftmost and the rightmost squares are cartesian (but the middle square is not). All the vertical maps are closed immersions. The pull-back of the divisor $\mathfrak{D}$  along the map $\wit\delta$ is the analogous Cartier divisor $\mathfrak{D}^\flat=(\pi_2^{\flat})^{-1} (\CZ^\flat(x^\flat))-(\pi_1^\flat)^{-1} (\CY(x^\flat))$ on $\wit\CN_{n-1}^1$ (cf. \cite[Proposition~5.11~(1, 2)]{Cho2018}). The pull-back of $\mathfrak{D}'$ along $\wit\delta$ is then a Cartier divisor
\begin{align}\label{eq:delta D'}
\wit\delta^{-1} (\mathfrak{D}') =\mathfrak{D}^\flat-\sum_{\Lambda\in \Ver(x)\atop e\in \Lambda} \mult_\Lambda(x)\,\BP_{\Lambda^\flat},
\end{align}
where $\Lambda^\flat\in\Ver^0(\BW_{n-1})$ denotes the orthogonal complement of $e$ in $\Lambda$. 
By induction hypothesis, we have $ \mathfrak{D}^\flat={\rm Exp}(x^\flat)$, which is a sum over $\Lambda\in\Ver(x^\flat)$ (and $\Ver(x^\flat)$ is bijective to the set of $\Lambda=\Lambda^\flat\obot\pair{e}$ in the sum \eqref{eq:delta D'}), but with known multiplicity $ \mult_{\Lambda^\flat}(x^\flat)=1$. When $n=2$, by \eqref{eq:D' n=2} and the fact that the strict transform $\wit\CV(\Lambda^\sharp)$ does not intersect the image of $\wit\delta$, we have $\wit\delta^{-1} (\mathfrak{D}') =0$, and hence we can already deduce that $\mult_{\Lambda_0}(x)=1$ for every $\Lambda_0\in \Ver(x)$. 

When $n\geq 3$, we can only deduce that $\mult_{\Lambda_0}(x)\leq 1$.  To see that $\mult_{\Lambda_0}(x)\neq 0$, we look at the intersection number between $ \mathfrak{D}'$ and a certain line $\BP^1$ in $\BP_{\Lambda_0}$
$$
\chi(\wit\CN_n^1, \BP^1\jiao  \mathfrak{D}')=\chi(\wit\CN_n^1, \BP^1\jiao  \mathfrak{D})-\mult_{\Lambda_0}(x)\,\chi(\wit\CN_n^1, \BP^1\jiao \BP_{\Lambda_0} ).
$$
 Here we choose the line $\BP^1$ to be $\BP_{M}$ for some rank two self-dual sublattice $M\subset \Lambda_0$ such that $x\notin M_F^\perp$. 
Since $\mathfrak{D}'$ does not contain $\BP_{\Lambda_0}$, its restriction to  $\BP_{\Lambda_0}$ is an effective Cartier divisor and therefore the left hand side is non-negative. On the other hand, $\chi(\wit\CN_n^1, \BP^1\jiao  \mathfrak{D}) =-1$ is strictly negative (e.g., by repeating \eqref{eq:diag case1} $n-2$ times to reduce to the case $n=2$). Therefore $\mult_{\Lambda_0}(x)\neq 0$ and we deduce that $\mult_{\Lambda_0}(x)=1$.  This is true for every $\Lambda_0\in \Ver(x)$. It follows that $\wit\delta^{-1} (\mathfrak{D}')=0$, i.e., $\mathfrak{D}'$ does not intersect $\wit\CN^1_{n-1}$ in \eqref{eq:diag case1}. This then implies that $\mathfrak{D}'$ does not intersect $\BP_{\Lambda_0}$  for every $\Lambda_0\in \Ver(x)$ (otherwise the intersection $\mathfrak{D}'\cap \BP_{\Lambda_0}$ would be a non-zero Cartier divisor on $\BP_{\Lambda_0}$, which necessarily intersects with the hyperplane $\BP_{\Lambda_0}\cap \wit\CN^1_{n-1}$ in $\BP_{\Lambda_0}\simeq \BP^{n-1}$ when $n\geq 3$; hence $\mathfrak{D}'$ must intersect $\wit\CN^1_{n-1}$, a contradiction!).

Now let us identify the open formal subscheme $\wit\CN_n^1\setminus \bigsqcup_{\Lambda\in\Ver^0(\BW_n)} \BP_{\Lambda}$ with $\CZ(x_0)\setminus \CZ(x_0)^{\rm ss}$ by Conjecture \ref{conj:KRSZ} $(iii)$. It remains to show that $\mathfrak{D}'$ is locally trivial at every point  $z\in \CZ(x_0)(\ov k)\setminus \CZ(x_0)^{\rm ss}$ (i.e., the local equations are all units).   Suppose that there exists a point $z\in \CZ(x_0)(\ov k)\setminus \CZ(x_0)^{\rm ss}$ where the local equation defining $\mathfrak{D}$ is not a unit. By \eqref{eq:KRstrat}, there exists a unique {\em maximal} vertex lattice $\Lambda^\sharp\subset\BV_{n+1}$ such that $z\in \CV(\Lambda^\sharp)\subset \CZ(\pair{x_0,x})^\red$.  

We first assume that $z$ is a non-super-general point (\S\ref{sec:bruh-tits-strat}). Then the type  $t(\Lambda^\sharp)\leq n$. Therefore there exists $e\in\Lambda^\sharp$ such that $\val(e)=0$ and $e$ is linearly independent of $x$.  Then $\pair{e,x_0}$ is a vertex lattice of rank two. There are exactly two cases:
   \begin{altenumerate}
    \item $\pair{e,x_0}$ has type $1$.
  \item $\pair{e,x_0}$ has type $0$.
     \end{altenumerate}
In the first case, we may assume that $e\perp x_0$.  Then we again consider the commutative diagrams \eqref{eq:diag case1} and retain the notation there. Now the point $z$ belongs to $\wit\CN_{n-1}^{1}(\ov k)$ (and not on the exceptional divisor). We have assumed that locally at $z$ the divisor $\mathfrak{D}$  is not defined by a unit. It follows that $\mathfrak{D}^\flat$ is not defined by a unit locally at $z$, contradicting the induction hypothesis. In the second case, let $e^\flat\in \BW_n$ be the orthogonal projection of $e$ to $\BW_n$. Then $\val(e^\flat)=-1$, and we may assume that $(e^\flat,e^\flat)=\varpi^{-1}$. Then by \eqref{eq:isom2}, the special divisor $\CY'(e^\flat)$ on $\CN_n^1$ is isomorphic to $\CN_{n-1}$. We consider the commutative diagrams 
\begin{align}
   \xymatrix{
	       &\CN_{n-1}\ar[d]_{\wit\delta} \ar[dl] \ar[dr]\\
	 \CN_{n-1}\simeq \CY'(e^\flat) \ar[d]  &   \wit\CN_{n}^{1} \ar[dl]_-{\pi_1}\ar[dr]^-{\pi_2}&\CZ(\pair{e,x_0})\simeq\CN_{n-1}\ar[d]\ar[r]&\CZ(e)\simeq\CN_{n}\ar[d]  \\
	 \CN_{n}^{1} &&\CZ(x_0)\ar[r]&\CN_{n+1}
	}
\end{align}
where the only non-obvious map $\wit\delta: \CN_{n-1}\to \wit\CN_n^1$ is defined as follows. The natural morphisms $\CN_{n-1}\simeq\CY'(e^\flat)\incl \CN_{n}^1$ and  $\CN_{n-1}\simeq\CZ(\pair{e,x_0})\incl \CZ(x_0)\incl\CN_{n+1}$  induced a morphism $\CN_{n-1}\to  \CN_{n}^1\times \CN_{n+1}$. Then it is straightforward to verify that this morphism factors through \eqref{eq:KRSZ0}, and therefore defines a morphism $\wit\delta$. 
The rest of the proof is similar to the first case (using \cite[Proposition~5.11~(3, 4)]{Cho2018}) instead), and we omit the detail. 
When $n=2$, this already implies $\mathfrak{D}'=0$ because every curve $\wit\CV(\Lambda^\sharp)$ in \eqref{eq:D' n=2} must pass through some non-super-general point $z\in \CZ(x_0)(\ov k)\setminus \CZ(x_0)^{\rm ss}$ (there are $q^3+1$ of them on $\wit\CV(\Lambda^\sharp)$  and only $q+1$ lie on the exceptional divisor). 


Finally, we assume that $n\geq 3$  and $z\in\mathfrak{D}'(\ov k)$ is a super-general point in $\CN_{n+1}(\ov k)$.  Then $n+1$ is odd and  $t(\Lambda^\sharp)=n+1$. The intersection $\mathfrak{D}'\cap\wit\CV(\Lambda^\sharp)$ is a closed subscheme of $\wit\CV(\Lambda^\sharp)$ with codimension one. It is also contained in the open subscheme $\CV(\Lambda^\sharp)^\circ$. This is impossible due to the affineness of $\CV(\Lambda^\sharp)^\circ$ \cite[Corollary 2.8]{Lusztig1976/77} and its dimension $\frac{n}{2}\geq 2$. This completes the induction.
\end{proof}

\subsection{The intersection number $\Int'(L)$ } 
We are now ready to complete the computation of the intersection number $\Int'(L)$ defined by \eqref{eq:def Int'}. Note that the result is conditional on Conjecture \ref{conj:KRSZ}.
\begin{theorem}\label{thm: main2'}
Let $L\subseteq \mathbb{V}$ be an $O_F$-lattice. Then 
$$\Int'(L)=\frac{1}{q+1}\left(\pDen_\Lambda(L)-\Den(L)\right).$$
In particular, the definition \eqref{eq:def Int'} is independent of the choice of the basis.
\end{theorem}
\begin{remark}
The case $n=2$ is due to \cite{Sankaran2017}.
\end{remark}
\begin{example}[The case $n=1$] When $n=1$, let $L=\pair{x}\subset \BW_1$.
It is easy to see that
 $$\Int'(L)=
 \begin{cases} \frac{\val(x)}{2}, & \val(x)\geq 0,\\
 0,& \text{otherwise}.
 \end{cases}
 $$
  On the other hand, the local density formula shows that $$\pDen_\Lambda(L)= \begin{cases}
  1+(q+1)\frac{\val(x)}{2},& \val(x)\geq 0,\\
 0,& \text{otherwise}.
 \end{cases}$$ 
This verifies the theorem in the case $n=1$.
\end{example}

\begin{proof}
By the projection formula for the finite flat map $\pi_1$, we obtain an equality in $K'_0(\CY(x_1)\cap\cdots\cap\CY(x_n))$
$$
\pi_{1\ast}(\pi_1^\ast( \CO_{\CY(x_1)})\otimes^\BL \cdots \otimes^\BL\pi_1^\ast (\CO_{\CY(x_n)}))=\deg(\pi_1) \, \CO_{\CY(x_1)}\otimes^\BL \cdots \otimes^\BL \CO_{\CY(x_n)},
$$
and hence 
\begin{align*}
\Int'(L)=&
\chi \left( \CN_n^1, \CO_{\CY(x_1)}\otimes^\BL \cdots \otimes^\BL \CO_{\CY(x_n)}\right)
\\=&\frac{1}{\deg(\pi_1)} \chi \left(\wit\CN_{n}^1,\pi_1^\ast (\CO_{\CY(x_1)})\otimes^\BL \cdots \otimes^\BL\pi_1^\ast (\CO_{\CY(x_n)})\right).
\end{align*}

For two Cartier divisors $D_1$ and $D_2$ on a regular formal scheme $\CX$, we have $\CO_{D_1+D_2}=\CO_{D_1}+\CO_{D_2}\in K_0^{D_1\cup D_2}(\CX)/{\rm F}^2K_0^{D_1\cap D_2}(\CX)$.  This allows us to apply Lemma \ref{lem:D x0} and the equality of Cartier divisors in Theorem \ref{thm:D x} and to obtain an equality in  $ K_0'( \pi_2^{-1}(\CZ^\flat(x)))/{\rm F}^1K_0'({\rm Exp}(x))$
$$
\pi_1^\ast  (\CO_{\CY(x)})=\pi_2^{\ast} (\CO_{\CZ^\flat(x)})- \CO_{{\rm Exp}(x)}.
$$
Fix a connected component $\BP_\Lambda$ of ${\rm Exp}(x)$ and let $i_\Lambda:\BP_\Lambda\to \wit\CN_n^1$ be the closed immersion. 
For any $\CF\in {\rm F}^1 K_0(\BP_\Lambda)$ and $n-1$ Cartier divisors $D_1,\cdots,D_{n-1}$ on $\wit\CN_n^1$, we have an equality in $K_0(\BP_\Lambda)$
$$
\CF\otimes_{\CO_{\wit\CN_n^1}}^\BL\CO_{D_1}\otimes_{\CO_{\wit\CN_n^1}}^\BL\cdots\otimes_{\CO_{\wit\CN_n^1}}^\BL\CO_{D_{n-1}}= \CF\otimes_{\CO_{\BP_\Lambda}}^\BL i_\Lambda^\ast (\CO_{D_1})\otimes_{\CO_{\BP_\Lambda}}^\BL\cdots\otimes_{\CO_{\BP_\Lambda}}^\BL i_\Lambda^\ast (\CO_{D_{n-1}}).
$$
Since $i_\Lambda^\ast (\CO_{D_i})\in  {\rm F}^1 K_0(\BP_\Lambda)$,
the above product belongs to $\mathrm{F}^{n} K_0(\BP_\Lambda)$ by \cite[B.3]{Zhang2019} (applied to the scheme $\BP_\Lambda$). Since $\dim\BP_\Lambda=n-1$, we conclude that $\CF\otimes_{\CO_{\wit\CN_n^1}}^\BL\CO_{D_1}\otimes_{\CO_{\wit\CN_n^1}}^\BL\cdots\otimes_{\CO_{\wit\CN_n^1}}^\BL\CO_{D_{n-1}}=0$. It follows that
\begin{align}\label{eq:Int'2}
\Int'(L)=\frac{1}{\deg(\pi_1)} \chi \left(\wit\CN_{n}^1,\bigl( \pi_2^{\ast} (\CO_{\CZ^\flat(x_1)})- \CO_{{\rm Exp}(x_1)} \bigr)\otimes^\BL \cdots \otimes^\BL \bigl( \pi_2^{\ast} (\CO_{\CZ^\flat(x_n)})- \CO_{{\rm Exp}(x_n)} \bigr)\right).
\end{align}

We apply  the projection formula to $\pi_2:\wit\CN_n^1\to\CZ(x_0)$
$$
\pi_{2\ast} \bigl(\CO_{\BP_\Lambda}\otimes_{\CO_{\wit\CN_n^1}}^\BL \pi_2^\ast(\CF)\bigr)=\pi_{2\ast}(\CO_{\BP_\Lambda})\otimes_{\CO_{\CZ(x_0)}}^\BL  \CF,
$$
where $\Lambda\in \Ver^0(\BW_n)$ and $\CF\in K_0(\CZ(x_0))$. Since the first factor $\pi_{2\ast}(\CO_{\BP_\Lambda})$ is supported on a zero-dimensional  subscheme of $\CZ(x_0)$, 
we have 
\begin{align}\label{eq:van chi}
\chi\bigl(\CO_{\BP_\Lambda}\otimes_{\CO_{\wit\CN_n^1}}^\BL \pi_2^\ast(\CF)\bigr)=0,
\end{align}
for any $\CF\in \mathrm{F}^1K_0(\CZ(x_0))$. 

Next, for $\Lambda_1,\cdots,\Lambda_n\in\Ver^0(\BW_n)$, the intersection numbers between exceptional divisors are equal to
$$
\chi\bigl(\wit\CN_n^1,\BP_{\Lambda_1}\jiao\cdots\jiao\BP_{\Lambda_n}\bigr)=\begin{cases}
(-1)^{n-1},&  \Lambda_1=\cdots=\Lambda_n,\\
0,&\text{otherwise}.\end{cases}
$$
Together with \eqref{eq:Int'2}  and \eqref{eq:van chi} we obtain
\begin{align}\label{eq:Int'3} 
\deg(\pi_1)\Int'(L)=\chi \left( \wit\CN_n^1, \pi_2^{\ast} (\CO_{\CZ^\flat(x_1)}) \otimes^\BL \cdots  \otimes^\BL\pi_2^{\ast} (\CO_{\CZ^\flat(x_n)})\right) +(-1)^n \sum_{ \Lambda\in\Ver^0(\BW_n) \atop L\subset\Lambda}(-1)^{n-1}.
\end{align}

By the projection formula for $\pi_2$, we have
\begin{align*}
\pi_{2\ast}(\pi_2^\ast (\CO_{\CZ^\flat(x_1)}))&=\pi_{2\ast}(\CO_{\wit\CN_n^1}\otimes_{\CO_{\wit\CN_n^1}} \pi_2^\ast (\CO_{\CZ^\flat(x_1)}))
\\&=\pi_{2\ast}(\CO_{\wit\CN_n^1})\otimes_{\CO_{\CZ(x_0)}}^\BL \CO_{\CZ^\flat(x_1)}\quad\in K'_0(\CZ(x_0)).
\end{align*}
Note that $\pi_{2\ast}(\CO_{\wit\CN_n^1})-\CO_{\CZ(x_0)}$ is supported on $\CZ(x_0)^{\rm ss}$ which is zero-dimensional. We obtain
\begin{align*}
&
\chi \left( \wit\CN_n^1, \pi_2^{\ast} (\CO_{\CZ^\flat(x_1)}) \otimes^\BL \cdots \otimes^\BL \pi_2^{\ast} (\CO_{\CZ^\flat(x_n)})\right)
\\=&\chi \left(\CZ(x_0),  \CZ^\flat(x_1) \jiao\cdots\jiao   \CZ^\flat(x_n)\right)
\\=&\pDen_\Lambda(L),
\end{align*}
where the last equality follows from Theorem \ref{thm: main2}.

Finally,  by \eqref{eq: Den L} we have
$$
\#\left\{\Lambda\in\Ver^0(\BW_n) \mid L\subset\Lambda\right\}=\Den(L).
$$
By \eqref{eq:Int'3} and $\deg(\pi_1)=q+1$,  the proof is complete.
\end{proof}

\begin{remark}
In the notation of \S\ref{sec:relation-with-local} and \S\ref{sec:relation-with-local-1},  it follows immediately from Theorem \ref{thm: main2'},  \eqref{eq:localWhittaker0}, \eqref{eq:localWhittaker2} that $$\Int'(L)=\left(\frac{W_T'(1, 0, \varphi_1)}{\log q^2}\cdot  \frac{(-q)^n-1}{q+1}-W_T(1,0,\varphi_0)\cdot\frac{1}{q+1}\right)\cdot\prod_{i=1}^{n}(1-(-q)^{-i})^{-1}.$$ 
\end{remark}

\part{Semi-global and global applications: arithmetic Siegel--Weil formula}\label{part:semi-global-global}

In this part we apply our main Theorem \ref{thm: main} to prove an identity between the local intersection number of Kudla--Rapoport cycles on (integral models of) unitary Shimura varieties  at an inert prime with hyperspecial level and the derivative of a Fourier coefficient of Siegel--Eisenstein series on unitary groups (also known as the \emph{local arithmetic Siegel--Weil formula}). This is achieved by relating the Kudla--Rapoport cycles on unitary Shimura varieties to those on unitary Rapoport--Zink spaces via the $p$-adic uniformization, and by relating the Fourier coefficients to local representation densities. This deduction is more or less standard (see \cite{Kudla2014} and \cite{Terstiege2013}), and we will state the results for more general totally real base fields and level structures, making use of the recent advance on integral models of unitary Shimura varieties (\cite{Rapoport2017}). We will also apply the main Theorem \ref{thm: main2} in the almost self-dual case to deduce a similar identity at an inert prime with almost self-dual level. Finally, combining these semi-global identities with archimedean identities of Liu \cite{Liu2011} and Garcia--Sankaran \cite{Garcia2019} will allow us to deduce the \emph{arithmetic Siegel--Weil formula} for Shimura varieties with minimal levels at inert primes, at least when the quadratic extension is unramified at all finite places.

\section{Shimura varieties and semi-global integral models}

\subsection{Shimura varieties}\label{sec:shimura-varieties} We will closely follow \cite{Rapoport2017} and \cite{RSZsurvey}. In this part we switch to global notations. Let $F$ be a CM number field, with $F_0$ its totally real subfield of index 2. We fix  a CM type $\Phi\subseteq \Hom(F, \overline{\mathbb{Q}})$ of $F$ and a distinguished element $\phi_0\in \Phi$. We fix an embedding $\overline{\mathbb{Q}}\hookrightarrow \mathbb{C}$ and identify the CM type $\Phi$ with the set of archimedean places of $F$, and also with the set of archimedean places of $F_0$. Let $V$ be an $F/F_0$-hermitian space of dimension $n\ge2$. Let $V_\phi=V \otimes_{F, \phi}\mathbb{C}$ be the associated $\mathbb{C}/\mathbb{R}$-hermitian space for $\phi\in\Phi$. Assume the signature of $V_\phi$ is given by
$$  (r_\phi,r_{\bar\phi})=\begin{cases}
  (n-1,1), & \phi=\phi_0,\\
  (n, 0), & \phi\in \Phi\setminus\{\phi_0\}.
\end{cases}$$
Define a variant $G^\mathbb{Q}$ of the unitary simulate group $\GU(V)$ by $$G^\mathbb{Q}\coloneqq \{g\in \Res_{F_0/\mathbb{Q}}\GU(V): c(g)\in \mathbb{G}_m\},$$  where $c$ denotes the similitude character. Define a cocharacter $$h_{G^\mathbb{Q}}: \mathbb{C}^\times\rightarrow G^\mathbb{Q}(\mathbb{R})\subseteq \prod_{\phi\in\Phi}\GU(V_\phi)(\mathbb{R})\simeq\prod_{\phi\in\Phi} \GU(r_\phi,r_{\bar \phi})(\mathbb{R}),$$ where its $\phi$-component is given by $$h_{G^\mathbb{Q},\phi}(z)=\diag\{ z\cdot 1_{r_\phi},\bar z \cdot 1_{r_{\bar\phi}}\}.$$ 
Then its $G^\mathbb{Q}(\mathbb{R})$-conjugacy class defines a Shimura datum $(G^\mathbb{Q},\{h_{G^\mathbb{Q}}\})$. Let $E_{r}=E(G^\mathbb{Q}, \{h_{G^\mathbb{Q}}\})$ be the reflex field, i.e., the subfield of $\overline{\mathbb{Q}}$ fixed by $\{\sigma\in\Aut(\overline{\mathbb{Q}}/\mathbb{Q}): \sigma^*(r)=r\}$, where $r: \Hom(F, \overline{\mathbb{Q}})\rightarrow \mathbb{Z}$ is the function defined by $r(\phi)=r_\phi$.

We similarly define the group $Z^\mathbb{Q}$ (a torus) associated to a totally positive definite $F/F_0$-hermitian space of dimension 1 (i.e., of signature $\{(1,0)_{\phi\in\Phi}\}$) and a cocharacter $h_{Z^\mathbb{Q}}$ of $Z^\mathbb{Q}$. The reflex field $E_\Phi=E(Z^\mathbb{Q}, \{h_{Z^\mathbb{Q}}\})$ is equal to the reflex field of the CM type $\Phi$, i.e., the subfield of $\overline{\mathbb{Q}}$ fixed by $\{\sigma\in\Gal(\overline{\mathbb{Q}}/\mathbb{Q}): \sigma\circ\Phi=\Phi\}$. 

Now define a Shimura datum $(\wit G, \{h_{\wit G}\})$ by  $$\wit G\coloneqq Z^\mathbb{Q}\times_{\mathbb{G}_m} G^\mathbb{Q}=\{(z,g)\in Z^\mathbb{Q}\times G^\mathbb{Q}: \mathrm{Nm}_{F/F_0}(z)=c(g)\},\quad h_{\wit G}=(h_{Z^\mathbb{Q}}, h_{G^\mathbb{Q}}).$$ Its reflex field $E$ is equal to the composite $E_rE_\Phi$, and the CM field $F$ becomes a subfield of $E$ via the embedding $\phi_0$. Let $\KG\subseteq \wit G(\mathbb{A}_f)$ be a compact open subgroup. Then the associated Shimura variety $\Sh_{\KG}=\Sh_{\KG}(\wit G,\{h_{\wit G}\})$ is of dimension $n-1$ and has a canonical model over $\Spec E$.  We remark that $E=F$ when $F/\mathbb{Q}$ is Galois, or when $F=F_0K$ for some imaginary quadratic $K/\mathbb{Q}$ and the CM type $\Phi$ is induced from a CM type of $K/\mathbb{Q}$ (e.g., when $F_0=\mathbb{Q}$).

\subsection{Semi-global integral models at hyperspecial levels}\label{sec:semi-global-integral}

Let $p$ be a prime number. If $p=2$, then we assume all places $v$ of $F_0$ above $p$ are unramified in $F$. Let $\nu$ be a place of $E$ above $p$. It determines places $v_0$ of $F_0$ and $w_0$ of $F$ via the embedding $\phi_0$. To specify the level $\KG$, notice that for $G\coloneqq \Res_{F_0/\mathbb{Q}}\U(V)$ we have an isomorphism
\begin{equation}
  \label{eq:tildeG}
  \wit G\simeq Z^\mathbb{Q}\times G,\quad (z,g)\mapsto (z,z^{-1}g).
\end{equation}
  We consider the open compact subgroup of the form $$\KG\simeq K_{Z^\mathbb{Q}}\times K_G$$ under the decomposition (\ref{eq:tildeG}). We assume that $K_{Z^\mathbb{Q}}$ is the unique maximal open compact subgroup of $Z^\mathbb{Q}(\mathbb{A}_f)$ and  $$K_G=\prod_{v|p} K_{G,v}\times K_{G}^p.$$ In this subsection, we assume

\begin{enumerate}[label=(H\arabic*)]
\item\label{item:H1}  $v_0$ is  inert in $F$ (possibly ramified over $p$).
\item\label{item:H2} $V_{v_0}$ is split and we take $K_{G,v_0}$ to be the stabilizer of a self-dual lattice $\Lambda_{v_0}\subseteq V_{v_0}$, a hyperspecial subgroup of $\U(V)(F_{0,v_0})$.
\item For each place $v\ne v_0$ of $F_0$ above $p$, let $K_{G,v}^\circ$ be the maximal compact subgroup of $\U(V)(F_{0,v})$ given by  the stabilizer of a vertex lattice $\Lambda_v\subseteq V_v$. We take $K_{G,v}=K_{G,v}^\circ$ if $v$ is nonsplit in $F$. We take $K_{G,v}\subseteq K_{G,v}^\circ$ to be any open compact subgroup if $v$ is split in $F$.
\item $K_G^p\subseteq G(\mathbb{A}_f^p)$ is any open compact subgroup.
\end{enumerate}

Under these conditions, Rapoport--Smithling--Zhang \cite[\S4.1]{Rapoport2017} and \cite[\S4--5]{RSZsurvey} (see also \cite[Proposition C.20]{Liu2018}) construct a smooth integral model $\mathcal{M}_{\KG}\tildeG$ of $\Sh_{\KG}\tildeGh$ over $O_{E,(\nu)}$. More precisely, for a locally noetherian $O_{E,(\nu)}$-scheme $S$, we consider $\mathcal{M}_{\KG}\tildeG(S)$ to be the groupoid of tuples $(A_0, \iota_0, \lambda_0, A,\iota,\lambda, \bar\eta^p, \bar\eta_p^{v_0})$, where
\begin{enumerate}[label=(M\arabic*)]
\item\label{item:M1} $A_0$ (resp. $A$) is an abelian scheme over $S$.
\item $\iota_0$ (resp. $\iota$) is an action of $O_F \otimes \mathbb{Z}_{(p)}$ on $A_0$ (resp. $A$) satisfying the Kottwitz condition of signature $\{(1,0)_{\phi\in\Phi}\}$ (resp. signature $\{(r_\phi,r_{\bar \phi})_{\phi\in\Phi}\}$).
\item $\lambda_0$ (resp. $\lambda$) is a polarization of $A_0$ (resp. $A$) whose Rosati involution induces the automorphism given by the nontrivial Galois automorphism of $F/F_0$ via $\iota_0$ (resp. $\iota$).
\item $\bar\eta^p$ is a $K_G^p$-orbit of $\mathbb{A}_{F,f}^p$-linear isometries between lisse $\mathbb{A}_{F,f}^p$-sheaves $$\eta^p: \Hom_F(\hat V^p(A_0), \hat V^p(A))\simeq V \otimes_F \mathbb{A}_{F,f}^p.$$ Here $\hat V^p(\cdot)$ denotes the $\mathbb{A}_{F,f}^p$-Tate module.
\item\label{item:splitlevel} $\bar \eta_p^{v_0}$ is a collection $\{\bar \eta_v\}$, where $v\ne v_0$ runs over places of $F_0$ above $p$ such that $v$ is split in $F$, and each $\bar\eta_v$ is a $K_{G,v}$-orbit of $F_w$-linear isomorphisms between lisse $F_w$-sheaves $$\eta_v: \Hom_{O_{F_w}}(A_0[w^\infty], A[w^\infty])\otimes_{O_{F_w}}F_w\simeq V \otimes_F F_w.$$ Here $w$ is the unique place over $v$ determined by the CM type $\Phi$, and we view $K_{G,v}$ as an open subgroup of  $\GL(V_w)\cong \U(V_v)$ under the decomposition $V_v\cong V_w \oplus V_{\overline{w}}$. Notice that by the Kottwitz signature condition, both $A_0[w^\infty]$ and $A[w^\infty]$ are  \'etale $O_{F_w}$-modules (cf. \cite[Definition C.19]{Liu2018}). 
\end{enumerate}
Such a tuple is required to satisfy the following extra conditions:
\begin{enumerate}[resume,label=(M\arabic*)]
\item\label{item:M5} $(A_0,\iota_0,\lambda_0)\in \mathcal{M}_0^{\mathfrak{a},\xi}(S)$. Here $\mathcal{M}_0^{\mathfrak{a},\xi}$ is an integral model of $\Sh_{K_{\mathbb{Z}^\mathbb{Q}}}(Z^\mathbb{Q},\{h_{Z^\mathbb{Q}}\})$ coming from an axillary moduli problem depending on a choice of an nonzero  coprime-to-$p$ ideal $\mathfrak{a}$ of $O_{F_0}$ and $\xi$ a certain similarity class of 1-dimensional hermitian $F/F_0$-hermitian spaces (\cite[\S3.2]{Rapoport2017}). These axillary choices are made to ensure that the unitary group in 1-variable with $\mathfrak{a}$-level structure exists and so $\mathcal{M}_0^{\mathfrak{a},\xi}$ is non-empty. In particular, the polarization $\lambda_0$ is coprime-to-$p$. We remark that when $F/F_0$ is ramified at some finite place, one may choose $\mathfrak{a}$ to be the trivial ideal. Moreover, when $F_0=\mathbb{Q}$, there is only one choice of $\xi$, and the condition $(A_0,\iota_0,\lambda_0)\in \mathcal{M}_0^{\mathfrak{a},\xi}(S)$ is nothing but requiring $\lambda_0$ to be principal.
\item\label{item:M6} For each place $v$ of $F_0$ above $p$, $\lambda$ induces a polarization $\lambda_v$ on the $p$-divisible group $A[v^\infty]$. We require $\ker \lambda_v\subseteq A[\iota(\varpi_v)]$ of rank equal to the size of $\Lambda_v^\vee/\Lambda_v$, where $\varpi_v$ is a uniformizer of $F_{0,v}$. In particular, we require $\lambda_{v_0}$ to be principal.
\item\label{item:Eisenstein} For the place $v_0$, we further require the \emph{Eisenstein condition} in \cite[\S5.2,\ case (2)]{RSZsurvey}. We remark the Eisenstein condition is automatic when $v_0$ is unramified over $p$, 
\item\label{item:M7} For each place $v\ne v_0$ of $F_0$ above $p$, we further require the \emph{sign condition} and \emph{Eisenstein condition} as explained in \cite[\S4.1]{Rapoport2017}. We remark that the sign condition is automatic when $v$ is split in $F$, and the Eisenstein condition is automatic when the places of $F$ above $v$ are unramified over $p$.
\end{enumerate}

A morphism $(A_0, \iota_0, \lambda_0, A,\iota,\lambda, \bar\eta^p, \bar\eta_p^{v_0})\rightarrow (A_0', \iota_0', \lambda_0', A',\iota',\lambda', \bar\eta^p{}', \bar\eta_p^{v_0}{}')$ in this groupoid is an isomorphism $(A_0, \iota_0, \lambda_0)\xrightarrow{\sim}(A_0', \iota_0', \lambda_0')$ in $\mathcal{M}_0^{\mathfrak{a},\xi}(S)$ and an $O_{F,(p)}$-linear quasi-isogeny $A\rightarrow A'$ inducing an isomorphism $A[p^\infty]\xrightarrow{\sim}A'[p^\infty]$, pulling $\lambda'$ back to $\lambda$, pulling $\bar\eta^p{}'$ back to $\bar\eta^p$ and pulling $ \bar\eta_p^{v_0}{}'$ back to $\bar\eta_p^{v_0}$.

By  \cite[Theorem 4.1]{Rapoport2017}, \cite[Theorem 5.6 (c)]{RSZsurvey}, the functor $S\mapsto \mathcal{M}_{\KG}\tildeG(S)$ is represented by a Deligne--Mumford stack $\mathcal{M}_{\KG}\tildeG$ smooth  over $\Spec O_{E,(\nu)}$. For $K_G^p$ small enough, $\mathcal{M}_{\KG}\tildeG$ is relatively representable over $\mathcal{M}_0^{\mathfrak{a},\xi}$, with generic fiber naturally isomorphic to the canonical model of $\Sh_{\KG}\tildeGh$ over~$\Spec E$. 

\subsection{Semi-global integral models at almost self-dual parahoric levels}\label{sec:almost-self-dual}
With the same set-up as \S\ref{sec:semi-global-integral}, but replace the assumptions \ref{item:H1} and \ref{item:H2} by
\begin{enumerate}[label=(A\arabic*)]
\item $v_0$ is inert in $F$ and unramified over $p$. 
\item $V_{v_0}$ is nonsplit and we take $K_{G,v_0}$ to be the stabilizer of an almost self-dual lattice $\Lambda_{v_0}\subseteq V_{v_0}$, a maximal parahoric subgroup of $\U(V)(F_{0,v_0})$.
\end{enumerate}
For a locally noetherian $O_{E,(\nu)}$-scheme $S$, we consider $\mathcal{M}_{\KG}\tildeG(S)$ to be the groupoid of tuples $\allowbreak(A_0, \iota_0, \lambda_0, A,\iota,\lambda, \bar\eta^p, \bar\eta_p^{v_0})$ satisfying \ref{item:M1}---~\ref{item:M7}. In particular, $\lambda_{v_0}$ is almost principal instead of principal in \ref{item:M6}.

By \cite[Theorem 4.7]{Rapoport2017}, the functor $S\mapsto \mathcal{M}_{\KG}\tildeG(S)$ is represented by a Deligne--Mumford stack $\mathcal{M}_{\KG}\tildeG$ flat over $\Spec O_{E,(\nu)}$. For $K_G^p$ small enough, $\mathcal{M}_{\KG}\tildeG$ is relatively representable over $\mathcal{M}_0^{\mathfrak{a},\xi}$, with generic fiber naturally isomorphic to the canonical model of $\Sh_{\KG}\tildeGh$ over~$\Spec E$. Moreover, when $\nu$ is unramified over $p$ (e.g., all $p$-adic places of $F$ are unramified over $p$), $\mathcal{M}_\KG\tildeG$ has semi-stable reduction over $\Spec O_{E,(\nu)}$ by  \cite[Theorem 4.7]{Rapoport2017} and \cite[Proposition 1.4]{Cho2018}.

\subsection{Semi-global integral models at split primes}\label{sec:split} With the same set-up as \S\ref{sec:semi-global-integral}, but replace the assumption \ref{item:H1} by
\begin{enumerate}[label=(S)]
\item\label{item:S} $v_0$ is split in $F$ (possibly ramified over $p$).
\end{enumerate}
For a locally noetherian $O_{E,(\nu)}$-scheme $S$, we consider $\mathcal{M}_{\KG}\tildeG(S)$ to be the groupoid of tuples $\allowbreak(A_0, \iota_0, \lambda_0, A,\iota,\lambda, \bar\eta^p, \bar\eta_p^{v_0})$ satisfying \ref{item:M1}---~\ref{item:M7}. We further require 
\begin{enumerate}[label=(MS)]
\item\label{item:MS} when $p$ is locally nilpotent on $S$, the $p$-divisible group $A[w_0^\infty]$ is a Lubin--Tate group of type $r|_{w_0}$ (\cite[\S8]{Rapoport2017a}). We remark that this condition is automatic when $v_0$ is unramified over $p$.
\end{enumerate}

By \cite[Theorem 4.2]{Rapoport2017}, as in the hyperspecial case, the functor $S\mapsto \mathcal{M}_{\KG}\tildeG(S)$ is represented by a Deligne--Mumford stack $\mathcal{M}_{\KG}\tildeG$ smooth  over $\Spec O_{E,(\nu)}$. For $K_G^p$ small enough, $\mathcal{M}_{\KG}\tildeG$ is relatively representable over $\mathcal{M}_0^{\mathfrak{a},\xi}$, with generic fiber naturally isomorphic to the canonical model of $\Sh_{\KG}\tildeGh$ over~$\Spec E$.

\subsection{Semi-global integral models with Drinfeld levels at split primes}\label{sec:drinfeld} With the same set-up as \S\ref{sec:split}, we may consider semi-global integral models with Drinfeld levels by further requiring
\begin{enumerate}[label=(D)]
\item\label{item:D}
  \begin{enumerate}
  \item the place $\nu$ of $E$ matches the CM type $\Phi$ (in the sense of \cite[\S 4.3]{Rapoport2017}): if $\phi\in\Hom(F, \overline{\mathbb{Q}})$ induces the $p$-adic place $w_0$ of $F$ (via $\nu: E\hookrightarrow \overline{\mathbb{Q}_p}$), then $\phi\in\Phi$. We remark that the matching condition is automatic when $F=F_0K$  for some imaginary quadratic $K/\mathbb{Q}$ and the CM type $\Phi$ is induced from a CM type of $K/\mathbb{Q}$ (e.g., when $F_0=\mathbb{Q}$), or when $v_0$ is of degree one over $p$.
  \item the extension $E_{\nu}/E_{r|_{v_0}}$ is unramified, where $E_{r|_{v_0}}$ is the local reflex field as defined in \cite[\S4.1]{Rapoport2017}.  We remark that this condition is automatic if all $p$-adic places of $F$ are unramified over $p$.
  \end{enumerate}
\end{enumerate}
For $m\ge0$, we consider the open compact subgroup $K_{G}^m\subseteq K_{G}$ such that $K_{G,v_0}^m\subseteq K_{G,v_0}$ is the principal congruence subgroup modulo $\varpi_{v_0}^m$, and $K_{G,v}^m=K_{G,v}$ for $v\ne v_0$. Write $\KG^m=K_{Z^\mathbb{Q} }\times K_G^m$. Notice that $\KG^0=\KG$. We define a semi-global integral model $\mathcal{M}_{\KG^m}\tildeG$ of $\Sh_{\KG^m}\tildeGh$ over $O_{E,(\nu)}$ as follows. For a locally noetherian $O_{E,(\nu)}$-scheme $S$, we consider $\mathcal{M}_{\KG^m}\tildeG(S)$ to be the groupoid of tuples $(A_0, \iota_0, \lambda_0, A,\iota,\lambda, \bar\eta^p, \bar\eta_p^{v_0}, \eta_{w_0})$, where $(A_0, \iota_0, \lambda_0, A,\iota,\lambda, \bar\eta^p, \bar\eta_p^{v_0})\in \mathcal{M}_{\KG}\tildeG(S)$ and the additional datum  $\eta_{w_0}$ is a Drinfeld level structure:
\begin{enumerate}[label=(MD)]
\item\label{item:MD} when $p$ is locally nilpotent on $S$, $\eta_{w_0}$ is an $O_{F,w_0}$-linear homomorphism of finite flat group schemes $$\eta_{w_0}: \varpi_{w_0}^{-m}\Lambda_{w_0}/\Lambda_{w_0}\rightarrow \underline{\Hom}_{O_{F,w_0}}( A_0[w_0^m], A[w_0^m]).$$
\end{enumerate}

By \cite[Theorem 4.5]{Rapoport2017} (where the second condition in \ref{item:D} should be added to ensure regularity), the functor $S\mapsto \mathcal{M}_{\KG^m}\tildeG(S)$ is represented by a regular Deligne--Mumford stack $\mathcal{M}_{\KG}\tildeG$, flat over $\Spec O_{E,(\nu)}$ and finite flat over $\mathcal{M}_{\KG}\tildeG$,  with generic fiber naturally isomorphic to the canonical model of $\Sh_{\KG^m}\tildeGh$ over~$\Spec E$. 

\subsection{Semi-global integral models at ramified primes}\label{sec:ramified} With the same set-up as \S\ref{sec:semi-global-integral}, but replace the assumption \ref{item:H1} by
\begin{enumerate}[label=(R)]
\item\label{item:R} $v_0$ is ramified in $F$ (so $p\ne2$) and unramified over $p$. 
\end{enumerate}
For a locally noetherian $O_{E,(\nu)}$-scheme $S$, we consider $\mathcal{M}_{\KG}\tildeG(S)$ to be the groupoid of tuples $\allowbreak(A_0, \iota_0, \lambda_0, A,\iota,\lambda, \bar\eta^p, \bar\eta_p^{v_0})$ satisfying \ref{item:M1}---~\ref{item:M7}. We further require
\begin{enumerate}[label=(MR)]
\item\label{item:MR} when $p$ is locally nilpotent on $S$, the $p$-divisible group $A[w_0^\infty]$ satisfies the Pappas wedge condition (\cite[Definition 2.4]{Kudla2014}, \cite[\S5.2]{RSZsurvey}).
\end{enumerate}
By \cite[Theorem 5.4]{RSZsurvey}, the functor $S\mapsto \mathcal{M}_{\KG}\tildeG(S)$ is represented by a Deligne--Mumford stack $\mathcal{M}_{\KG}\tildeG$ flat over $\Spec O_{E,(\nu)}$. For $K_G^p$ small enough, $\mathcal{M}_{\KG}\tildeG$ is relatively representable over $\mathcal{M}_0^{\mathfrak{a},\xi}$, with generic fiber naturally isomorphic to the canonical model of $\Sh_{\KG}\tildeGh$ over~$\Spec E$. By \cite[Theorem 6.7]{RSZsurvey}, it has isolated singularities and we may further obtain a regular model by blowing up (the \emph{Kr\"amer model}, see \cite[Definition 6.10]{RSZsurvey}) which we still denote by $\mathcal{M}_K$.

\section{Incoherent Eisenstein series}

\subsection{Siegel Eisenstein series}\label{sec:sieg-eisenst-seri} Let $W$ be the standard split $F/F_0$-skew-hermitian space of dimension $2n$. Let $G_n=\U(W)$. Write $G_n(\mathbb{A})=G_n(\mathbb{A}_{F_0})$ for short. Let $P_n(\mathbb{A})=M_n(\mathbb{A})N_n(\mathbb{A})$ be the standard Siegel parabolic subgroup of $G_n(\mathbb{A})$, where
\begin{align*}
  M_n(\mathbb{A})&=\left\{m(a)=\begin{pmatrix}a & 0\\0 &{}^t\bar a^{-1}\end{pmatrix}: a\in \GL_n(\mathbb{A}_F)\right\},\\
  N_n(\mathbb{A}) &= \left\{n(b)=\begin{pmatrix} 1_n & b \\0 & 1_n\end{pmatrix}: b\in \Herm_n(\mathbb{A}_F)\right\}.
\end{align*}

Let $\eta: \mathbb{A}_{F_0}^\times/F_0^\times\rightarrow \mathbb{C}^\times$ be the quadratic character associated to $F/F_0$. Fix $\chi: \mathbb{A}_{F}^\times\rightarrow \mathbb{C}^\times$ a character such that $\chi|_{\mathbb{A}_{F_0}^\times}=\eta^n$. We may view $\chi$ as a character on $M_n(\mathbb{A})$ by $\chi(m(a))=\chi(\det(a))$ and extend it to $P_n(\mathbb{A})$ trivially on $N_n(\mathbb{A})$. Define the \emph{degenerate principal series} to be the unnormalized smooth induction $$I_n(s,\chi)\coloneqq \Ind_{P_n(\mathbb{A})}^{G_n(\mathbb{A})}(\chi\cdot |\cdot|_{F}^{s+n/2}),\quad s\in \mathbb{C}.$$ For a standard section $\Phi(-, s)\in I_n(s,\chi)$ (i.e., its restriction to the standard maximal compact subgroup of $G_n(\mathbb{A})$ is independent of $s$), define the associated \emph{Siegel Eisenstein series} $$E(g,s, \Phi)\coloneqq \sum_{\gamma\in P_n(F_0)\backslash G_n(F_0)}\Phi(\gamma g, s),\quad g\in G_n(\mathbb{A}),$$ which converges for $\Re(s)\gg 0$ and admits meromorphic continuation to $s\in \mathbb{C}$. Notice that $E(g,s,\Phi)$ depends on the choice of $\chi$.

\subsection{Fourier coefficients and derivatives}\label{sec:four-coeff-deriv}
By class field theory\footnote{This should be well-known, but we include the argument for the convenience of the readers. Let $\psi_0=\psi_\mathbb{Q}\circ\tr_{F_0/\mathbb{Q}}: \mathbb{A}_{F_0}/F_0\rightarrow \mathbb{C}^\times$, where $\psi_\mathbb{Q}: \mathbb{A}_\mathbb{Q}/\mathbb{Q}\rightarrow \mathbb{C}^\times$ is the standard additive character (so $\psi_{\mathbb{Q},\infty}(x)=e^{2\pi ix}$). Then the conductor of $\psi_0$ is the different ideal $\delta_{F_0/\mathbb{Q}}$ of $F_0/\mathbb{Q}$. Let $H$ be the Hilbert class field of $F_0$. Since $F/F_0$ is ramified at infinite places, we know that $H$ and $F$ are linearly disjoint over $F_0$. It follows that $\Gal(H/F_0)$ is generated by the Frobenii associated to $\mathrm{Spl}(F/F_0)$. Hence by class field theory, the ideal class group of $F_0$ is generated by $\mathrm{Spl}(F/F_0)$. In particular, we may find $a\in F_0^\times$ such that the ideal $a\delta_{F_0}$ is supported on $\mathrm{Spl}(F/F_0)$. Then the character $\psi(x)=\psi_0(ax)$ works. Moreover, by a theorem of Hecke (\cite[Theorem 176]{Hecke1954}, see also \cite{Armitage1967}), the ideal class of $\delta_{F_0/\mathbb{Q}}$ is a square. Hence we may further choose $a\in F_0^\times$ to be a square.}, we may and do choose an additive character $\psi: \mathbb{A}_{F_0}/F_0\rightarrow \mathbb{C}^\times$ such that $\psi$ is unramified outside $\mathrm{Spl}(F/F_0)$ (the set of finite places of $F_0$ split in $F$). 
We have a Fourier expansion $$E(g,s,\Phi)=\sum_{T\in\Herm_n(F)}E_T(g,s,\Phi),$$ where $$E_T(g,s,\Phi)=\int_{N_n(F_0)\backslash N_n(\mathbb{A})} E(n(b)g,s,\Phi)\psi(-\tr(Tb))\,\rd n(b),$$ and the Haar measure $\rd n(b)$ is normalized to be self-dual with respect to $\psi$. When $T$ is nonsingular, for factorizable $\Phi=\otimes_v\Phi_v$ we have a factorization of the Fourier coefficient into a product $$E_T(g,s,\Phi)=\prod_v W_{T,v}(g_v, s, \Phi_v),$$ where the \emph{local (generalized) Whittaker function} is defined by $$W_{T,v}(g_v, s, \Phi_v)=\int_{N_n(F_{0,v})}\Phi_v(w_n^{-1}n(b)g,s)\psi(-\tr(Tb))\, \rd n(b),\quad w_n=
\begin{pmatrix}
0  & 1_n\\
  -1_n & 0\\
\end{pmatrix}.$$ and has analytic continuation to $s\in \mathbb{C}$. Thus we have a decomposition of the derivative of a nonsingular Fourier coefficient at $s=s_0$,
\begin{equation}
  \label{eq:eissum}
E_T'(g, s_0, \Phi)=\sum_v E'_{T,v}(g, s_0,\Phi),
\end{equation}
 where  
\begin{equation}
  \label{eq:eisenfactor}
  E'_{T,v}(g, s, \Phi)=W_{T,v}'(g_v, s,\Phi_v)\cdot \prod_{v'\ne v}W_{T,v'}(g_{v'},s,\Phi_{v'}).
\end{equation}

\subsection{Incoherent Eisenstein series} \label{sec:incoh-eisenst-seri} Let $\mathbb{V}$ be an $\mathbb{A}_F/\mathbb{A}_{F_0}$-hermitian space of rank $n$.  Let $\sS(\mathbb{V}^n)$ be the space of Schwartz functions on $\mathbb{V}^n$. The fixed choice of $\chi$ and $\psi$ gives a \emph{Weil representation} $\omega=\omega_{\chi,\psi}$ of $G_n(\mathbb{A})\times \U(\mathbb{V})$ on $\sS(\mathbb{V}^n)$. Explicitly, for $\varphi\in \sS(\mathbb{V}^n)$ and $\mathbf{x}\in \mathbb{V}^n$,
\begin{align*}
\omega(m(a))\varphi(\mathbf{x})&=\chi(m(a))|\det a|_F^{n/2}\varphi(\mathbf{x}\cdot a),&m(a)\in M_n(\mathbb{A}),\\
\omega(n(b))\varphi(\mathbf{x})&=\psi(\tr b\,T(\mathbf{x}))\varphi(\mathbf{x}),&n(b)\in N_n(\mathbb{A}),\\
\omega_\chi(w_n)\varphi(\mathbf{x})&=\gamma_{\mathbb{V}}^n\cdot\widehat \varphi(\mathbf{x}),&w_n=\left(\begin{smallmatrix}
0  & 1_n\\
  -1_n & 0\\
\end{smallmatrix}\right),\\
\omega(h)\varphi(\mathbf{x})&=\varphi(h^{-1}\cdot\mathbf{x}),& h\in \U(\mathbb{V}).
\end{align*}
Here $T(\mathbf{x})=((x_i,x_j))_{1\le i,j\le n}$ is the \emph{fundamental matrix} of $\mathbf{x}$, $\gamma_{\mathbb{V}}$ is the Weil constant (see \cite[(10.3)]{Kudla2014}), and $\widehat\varphi$ is the Fourier transform of $\varphi$ using the self-dual Haar measure on $\mathbb{V}^n$ with respect to $\psi\circ\tr_{F/F_0}$.

For $\varphi\in \sS(\mathbb{V}^n)$, define a function $$\Phi_\varphi(g)\coloneqq \omega(g)\varphi(0),\quad g\in G_n(\mathbb{A}).$$ Then $\Phi_\varphi\in I_n(0,\chi)$. Let $\Phi_\varphi(-,s)\in I_n(s,\chi)$ be the associated standard section, known as \emph{the standard Siegel--Weil section} associated to $\varphi$. For $\varphi\in\sS(\mathbb{V}^n)$, we write $$E(g,s,\varphi)\coloneqq E(g, s, \Phi_\varphi), \quad E_T(g,s,\varphi)\coloneqq E_T(g,s,\Phi_\varphi),\quad E'_{T,v}(g,s,\varphi)\coloneqq E'_{T,v}(g,s,\Phi_\varphi),$$ and similarly for $W_{T,v}(g_v,s,\varphi_v)$. We say $\mathbb{V}$ (resp. $\Phi_\varphi$, $E(g,s,\varphi)$) is \emph{coherent} if $\mathbb{V}=V \otimes_{F_0}\mathbb{A}_{F_0}$ for some $F/F_0$-hermitian space $V$, and \emph{incoherent} otherwise. When $E(g,s,\varphi)$ is incoherent,  its central value $E(g,0,\varphi)$ automatically vanishes (cf. \cite[\S9]{Kudla2014}). In this case, we write the central derivatives as $$\pEis(g,\varphi)\coloneqq E'(g,0,\varphi),\quad \pEis_{T}(g,\varphi)\coloneqq E'_{T}(g,0,\varphi), \quad \pEis_{T,v}(g,\varphi)\coloneqq E'_{T,v}(g,0,\varphi).$$ Let $T\in \Herm_n(F)$ be nonsingular. Then $W_{T,v}(g_v,0,\varphi_v)\ne0$ only if $\mathbb{V}_v$ represents $T$, hence $\pEis_{T,v}(g,\varphi)\ne0$ only if $\mathbb{V}_{v'}$ represents $T$ for all $v'\ne v$. Let $\Diff(T, \mathbb{V})$ be the set of finite places $v$ such that $\mathbb{V}_{v}$ does not represent $T$. Since $\mathbb{V}$ is incoherent,  by (\ref{eq:eissum}) we know that $\pEis_T(g,\varphi)\ne0$ only if $\Diff(T, \mathbb{V})=\{v\}$ is a singleton, and in this case $v$ is necessarily nonsplit in $F$ (cf. \cite[Lemma 9.1]{Kudla2014}). Thus
 \begin{equation}\label{eq:Diffset}
    \pEis_T(g,\varphi)\ne0\Rightarrow \Diff(T,\mathbb{V})=\{v\},\quad \pEis_T(g,\varphi)=\pEis_{T,v}(g,\varphi).
 \end{equation}

We say $\varphi_{v}\in \sS(\mathbb{V}_v^n)$ is \emph{nonsingular} if its support lies in $\{\mathbf{x}\in \mathbb{V}_v^n: \det T(\mathbf{x})\ne0\}$.  By \cite[Proposition 2.1]{Liu2011a}, we have
\begin{equation}
  \label{eq:twononsingular}
  \varphi\text{ is nonsingular at two finite places} \Longrightarrow \pEis_T(g,\varphi)=0\text{ for any singular } T.
\end{equation} 

\subsection{Classical incoherent Eisenstein series associated to the Shimura datum}\label{sec:incoh-eisenst-seri-1}

Assume that we are in the situation in \S \ref{sec:shimura-varieties}. Let $\mathbb{V}$ be the incoherent hermitian space obtained from $V$ so that $\mathbb{V}$ has signature $(n,0)_{\phi\in\Phi}$ and $\mathbb{V}_v\cong V_v$ for all finite places $V_v$. 

The hermitian symmetric domain for $G_n=\U(W)$ is the \emph{hermitian upper half space}
\begin{align*}
  \mathbb{H}_n&=\{\sz\in \mathrm{Mat}_{n}(F_\infty): \frac{1}{2i}\left(\sz-{}^t\bar\sz\right)>0\}\\
  &=\{\sz=\sx+i\sy:\ \sx\in\Herm_n(F_\infty),\ \sy\in\Herm_n(F_\infty)_{>0}\},  
\end{align*}
where $F_\infty=F \otimes_{F_0} \mathbb{R}^\Phi\cong \mathbb{C}^\Phi$. Define the \emph{classical incoherent Eisenstein series} to be $$E(\sz,s,\varphi)\coloneqq \chi_\infty(\det(a))^{-1}\det(\sy)^{-n/2}\cdot E(g_\sz,s, \varphi), \quad g_\sz\coloneqq n(\sx)m(a)\in G_n(\mathbb{A}),$$ where $a\in\GL_n(F_\infty)$ such that $\sy=a{}^{t}\bar a$. Notice that $E(\sz,s,\Phi)$ does not depend on the choice of $\chi$. We write the central derivatives as $$\pEis(\sz,\varphi)\coloneqq E'(\sz,0,\varphi),\quad \pEis_T(\sz, \varphi)\coloneqq E'_T(\sz,0,\varphi),\quad \pEis_{T,v}(\sz, \varphi)\coloneqq E'_{T,v}(\sz,0,\varphi).$$ Then we have a Fourier expansion
\begin{equation}
  \label{eq:eisz}
  \pEis(\sz,\varphi)=\sum_{T\in\Herm_n(F)}\pEis_T(\sz,\varphi)
\end{equation}
By (\ref{eq:Diffset}) we know that
\begin{equation}
  \label{eq:diffana}
  \pEis_T(\sz, \varphi)\ne0\Rightarrow \Diff(T,\mathbb{V})=\{v\},\quad \pEis_T(\sz,\varphi)=\pEis_{T,v}(\sz, \varphi).
\end{equation}
For the fixed open compact subgroup $\KG\subseteq \wit G(\mathbb{A}_f)$, we will choose $$\varphi= \varphi_{\KG} \otimes \varphi_{\infty}\in \sS(\mathbb{V}^n)$$ such that $\varphi_K\in\sS(\mathbb{V}^n_f)$ is $\KG$-invariant (where $\KG$ acts on $\mathbb{V}$ via the second factor $K_G$) and $\varphi_{\infty}$ is the Gaussian function $$\varphi_{\infty}(\mathbf{x})=\psi_\infty(i\tr T(\mathbf{x})).$$ 
For our fixed choice of Gaussian $\varphi_\infty$, we write $$E(\sz,s,\varphi_K)=E(\sz,s,\varphi_K \otimes \varphi_\infty),\quad \pEis(\sz, \varphi_K)=\pEis(\sz, \varphi_K \otimes \varphi_\infty)$$ and so on for short. When $T>0$ is totally positive definite, by \cite[Proposition 4.5 (2)]{Liu2011} the archimedean Whittaker function is $$W_{T,\infty}(\sz, 0, \varphi_\infty)=c_\infty\cdot q^{T}, \quad q^T\coloneqq \psi_\infty(\tr (T\sz))$$ for some constant $c_\infty$ independent of $T$. It follows from \eqref{eq:eisenfactor} that we have a factorization
\begin{equation}
  \label{eq:pEisfactorzation}
  \pEis_{T,v}(\sz, \varphi_K)=c_\infty\cdot W_{T,v}'(1,0,\varphi_{\KG, v})\cdot\prod_{v'\ne v, v'\nmid\infty}W_{T,v'}(1,0,\varphi_{\KG, v'})\cdot q^T.
\end{equation}

\section{The semi-global identity at inert primes}

In this section we assume that we are in the situation of \S\ref{sec:semi-global-integral} (hyperspecial level) or \S\ref{sec:almost-self-dual} (almost self-dual level). We fix the level $\KG$ as above and write $\mathcal{M}=\mathcal{M}_K$ for short. 

\subsection{$p$-adic uniformization of the supersingular locus of $\mathcal{M}$} \label{sec:p-adic-unif}  Let $\widehat{\mathcal{M}^\mathrm{ss}}$ be the completion of the base change $\mathcal{M}_{O_{\breve E_\nu}}$ along the supersingular locus $\mathcal{M}_{k_\nu}^\mathrm{ss}$ of its special fiber $\mathcal{M}_{k_\nu}$. Here $E_\nu$ is the completion of $E$ at $\nu$ and $k_\nu$ is its residue field. Assume $p>2$. Then we have a $p$-adic uniformization theorem (\cite{RZ96}, \cite[Theorem 4.3]{Cho2018}, see also the proof of \cite[Theorem 8.15]{Rapoport2017}),
\begin{equation}
  \label{eq:padicunif}
  \widehat{\mathcal{M}^\mathrm{ss}}\simeq \wit G'(\mathbb{Q})\backslash [\mathcal{N}'\times \wit G(\mathbb{A}_f^p)/\KG^p].
\end{equation}
Here $\wit G'=Z^\mathbb{Q}\times_{\mathbb{G}_m}G'^{\mathbb{Q}}$ is the group associated to a $F/F_0$-hermitian space $V'$ obtained from $V$ by changing the signature at $\phi_0$ from $(n-1,1)$ to $(n, 0)$ and the invariant at $v_0$ from $+1$ (resp. $-1$) to $-1$ (resp. $+1$) (i.e., $V'_{v_0}$ is a non-split (resp. split) $F_{w_0}/F_{0,v_0}$-hermitian space) in the hyperspecial case (resp. the almost self-dual case). The relevant Rapoport--Zink space $\mathcal{N}'$ associated to $\wit G'$ is given by $$\mathcal{N}'\simeq (Z^\mathbb{Q}(\mathbb{Q}_p)/K_{Z^\mathbb{Q},p})\times \mathcal{N}_{O_{\breve E_\nu}}\times \prod_{v\ne v_0}\U(V)(F_{0,v})/K_{G,v},$$ where the product is over places  $v\ne v_0$ of $F_0$ over $p$, and $\mathcal{N}$ is isomorphic to $\mathcal{N}_{F_{w_0}/F_{0,v_0},n}$, the Rapoport--Zink space defined in \S\ref{sec:rapoport-zink-spaces}\footnote{We use the convention $(1,n-1)$ for the signature of Rapoport--Zink spaces while the convention $(n-1,1)$ for Shimura varieties; each of these two conventions is more preferable in its respective setting. Strictly speaking, \cite[Theorem 4.3]{Cho2018} and \cite[Theorem 8.15]{Rapoport2017} assume that $v_0$ is unramified over $p$. This assumption can be dropped due to the Eisenstein condition in \ref{item:Eisenstein}. In fact, \cite[Definition 2.8 (ii)]{Mihatsch2016} specializes to the Eisenstein condition in \ref{item:Eisenstein} for signature $(1, n-1)$, so Mihatsch's theorem \cite[Theorem 4.1]{Mihatsch2016} is still applicable even when $v_0$ is ramified over $p$.} in the hyperspecial case, or isomorphic to $\mathcal{N}_{F_{w_0}/F_{0,v_0},n}^1$, the Rapoport--Zink space defined in \S\ref{sec:rapoport-zink-with} in the almost self-dual case.

\subsection{The hermitian lattice $\mathbb{V}(A_0,A)$}\label{sec:herm-latt-mathbbv} For a locally noetherian $O_{E,(\nu)}$-scheme $S$ and a point $(A_0,\iota_0, \lambda_0, A,\iota, \lambda,\bar\eta^p)\in \mathcal{M}(S)$, define the \emph{space of special homomorphisms} to be $$\mathbb{V}(A_0, A)\coloneqq \Hom_{O_F}(A_0, A) \otimes \mathbb{Z}_{(p)},$$ a free $O_{F,(p)}\coloneqq O_F \otimes \mathbb{Z}_{(p)}$-module of finite rank. Then $\mathbb{V}(A_0,A)$ carries a $O_{F,(p)}$-valued hermitian form: for $x,y\in \mathbb{V}(A_0,A)$, the pairing $(x,y)\in O_{F,(p)}$ is given by $$(A_0\xrightarrow{x}A\xrightarrow{\lambda}A^\vee \xrightarrow{y^{\vee}}A_0^\vee\xrightarrow{\lambda_0^{-1}}A_0)\in\End_{O_F}(A_0) \otimes \mathbb{Z}_{(p)}=\iota_0(O_{F,(p)})\simeq O_{F,(p)}.$$ Notice that  $\lambda_0^{-1}$ makes sense as the polarization $\lambda_0$  is coprime-to-$p$ by \ref{item:M5}.

Let $m\ge1$. Given an $m$-tuple $\mathbf{x}=[x_1,\ldots,x_m]\in \mathbb{V}(A_0,A)^m$, define its \emph{fundamental matrix} to be $$T(\mathbf{x})\coloneqq ((x_i,x_j))_{1\le i,j\le m}\in \Herm_m(O_{F,(p)}),$$ an $m\times m$ hermitian matrix over $O_{F,(p)}$.

\subsection{Semi-global Kudla--Rapoport cycles $\mathcal{Z}(T,\varphi_K)$}\label{sec:semi-global-kudla}

We say a Schwartz function $\varphi_K\in \sS(\mathbb{V}^m_f)$ is \emph{$v_0$-admissible} if it is $K$-invariant and $\varphi_{K,v}=\mathbf{1}_{(\Lambda_{v})^m}$ for all $v$ above $p$ such that $v$ is nonsplit in $F$. First we consider a special $v_0$-admissible Schwartz function of the form
\begin{equation}
  \label{eq:testfunction}
  \varphi_K=(\varphi_i)\in \sS(\mathbb{V}^m_f),\quad \varphi_{i}=\mathbf{1}_{\Omega_i},\quad i=1,\ldots,m,
\end{equation}
where $\Omega_i\subseteq \mathbb{V}_f$ is a $K$-invariant open compact subset such that $\Omega_{i,v}=\Lambda_v$ for all $v$ above $p$ such that $v$ is nonsplit in $F$.  Given such a special Schwartz function $\varphi_K$ and $T\in \Herm_m(O_{F,(p)})$, define a semi-global  \emph{Kudla--Rapoport cycle} $\mathcal{Z}(T,\varphi_K)$ over $\mathcal{M}$ as follows. For a locally noetherian $O_{E,(\nu)}$-scheme $S$, define $\mathcal{Z}(T,\varphi_K)(S)$ to be the groupoid of tuples $(A_0,\iota_0, \lambda_0, A,\iota, \lambda,\bar\eta^p, \bar\eta_p^{v_0}, \mathbf{x})$ where
\begin{enumerate}
\item $(A_0,\iota_0, \lambda_0, A,\iota, \lambda,\bar\eta^p, \bar\eta_p^{v_0})\in \mathcal{M}(S)$,
\item $\mathbf{x}=[x_1,\ldots, x_m]\in \mathbb{V}(A_0,A)^m$ with fundamental matrix $T(\mathbf{x})=T$.
\item $\eta^p(\mathbf{x})\in(\Omega^{(p)}_i)\subseteq (\mathbb{V}^{(p)}_f)^m$.
\item $\eta_v(\mathbf{x})\in (\Omega_{i, w})\subseteq \mathbb{V}_w^m$ for all $v\ne v_0$ split in $F$, where $\bar\eta_p^{v_0}=\{\bar\eta_v\}$ (cf. \ref{item:splitlevel}). 
\end{enumerate}
The functor $S\mapsto \mathcal{Z}(T,\varphi_K)(S)$ is represented by a (possibly empty) Deligne--Mumford stack which is finite and unramified over $\mathcal{M}$ (\cite[Proposition 2.9]{Kudla2014}), and thus defines a cycle $\mathcal{Z}(T,\varphi_K)\in \mathrm{Z}^*(\mathcal{M})$. For a general $v_0$-admissible Schwartz function $\varphi_K\in \sS(\mathbb{V}^m_f)$, by  extending $\mathbb{C}$-linearly we obtain a cycle $\mathcal{Z}(T,\varphi_K)\in\mathrm{Z}^*(\mathcal{M})_\mathbb{C}$.

\subsection{Variants of semi-global Kudla--Rapoport cycles $\mathcal{Z}^\flat(T, \varphi_K)$ at almost self-dual level}\label{sec:variants-semi-global}

Assume that we are in the situation of \ref{sec:semi-global-kudla} and $\Lambda_{v_0}$ is almost self-dual. We will define a variant $\mathcal{Z}^\flat(T, \varphi_K)$ of the semi-global Kudla--Rapoport cycle. To do so, consider a diagram of Shimura varieties
\begin{equation}
  \label{eq:diagSh}
  \xymatrix{
  &\Sh_{K\cap K^\sharp} \ar[dl]_{\pi_1} \ar[dr]^{\pi_2}\\
  \Sh_K  &  &\Sh_{K^\sharp},}
\end{equation}
where the level at $v_0$ is modified as in Remark \ref{rem:corr}.

More precisely, consider a hermitian space of dimension $n+1$, $$V^\sharp:=V \obot \langle x_0\rangle,$$ where $u_0=(x_0,x_0)$ has valuation 1 at $v_0$ and valuation 0 for all places $v\ne v_0$ of $F_0$ above $p$. We take the level $K^\sharp\subseteq \wit G^\sharp(\mathbb{A}_f)$ such that
\begin{enumerate}
\item $K^\sharp_{G,v_0}$ is the stabilizer of a self-dual lattice $\Lambda^\sharp_{v_0}\subseteq V^\sharp_{v_0}$,
\item for $v\ne v_0$ a place of $F_0$ above $p$, $K^\sharp_{G,v}$ is the stabilizer of the lattice $\Lambda_{v} \obot \langle x_0\rangle\subseteq V^\sharp_v$, 
\item $K^p\subseteq K^{\sharp,p}\cap \wit G(\mathbb{A}_f^p)$.
\end{enumerate} Denote by $$\Sh_{K^\sharp}=\Sh_{K^\sharp}(\wit G^\sharp, \{h_{\wit G^\sharp}\}),\quad \Sh_{K\cap K^\sharp}=\Sh_{K\cap K^\sharp}(\wit G, \{h_{\wit G}\})$$ the Shimura varieties defined in \S\ref{sec:shimura-varieties}.

Let $\mathcal{M}^\sharp$ be the semi-global model of $\Sh_{K^\sharp}$ over $O_{E,(\nu)}$ as defined in \S\ref{sec:semi-global-integral}. Define the semi-global integral model $\mathcal{M}_{K\cap K^\sharp}$ of $\Sh_{K\cap K^\sharp}$  over $O_{E,(\nu)}$ as follows. For a locally noetherian $O_{E,(\nu)}$-scheme $S$, define $\mathcal{M}_{K\cap K^\sharp}(S)$ to be the groupoid of tuples $(A_0,\iota_0, \lambda_0, A,\iota, \lambda,\bar\eta^p, \bar\eta_p^{v_0}, A^\sharp,\iota^\sharp, \lambda^\sharp, \bar\eta^{p,\sharp}, \bar\eta_p^{v_0,\sharp}, \alpha)$, where
\begin{enumerate}
\item $(A_0,\iota_0, \lambda_0, A,\iota, \lambda,\bar\eta^p, \bar\eta_p^{v_0})\in \mathcal{M}(S)$,
\item $(A_0,\iota_0, \lambda_0, A^\sharp,\iota^\sharp, \lambda^\sharp, \bar\eta^{p,\sharp}, \bar\eta_p^{v_0,\sharp})\in \mathcal{M}^\sharp(S)$,
\item $\alpha: A\times A_0\rightarrow A^\sharp$ is an isogeny of degree $q_{v_0}$ such that $\ker\alpha\subseteq (A\times A_0)[u_0]$ and $\alpha^*(\lambda^\sharp)=\lambda\times u_0\lambda_0$. 
\end{enumerate} Then the diagram (\ref{eq:diagSh}) extends to semi-global integral integral models $$   \xymatrix{
  &\mathcal{M}_{K\cap K^\sharp} \ar[dl]_{\pi_1} \ar[dr]^{\pi_2}\\
  \mathcal{M}  &  &\mathcal{M}^\sharp.}$$ The $p$-adic uniformization theorem (\ref{eq:padicunif}) of \cite{RZ96} then holds for $\mathcal{M}_{K\cap K^\sharp}$ with $\mathcal{N}=\wit{\mathcal{N}}^1_{F_{w_0}/F_{0,v_0},n}$, the auxiliary Rapoport--Zink space defined in \S\ref{sec:auxil-rapop-zink}.

Analogous to Remark \ref{rem:altintegralmodel}, 
we obtain a cycle $\mathcal{Z}^\flat(T,\varphi_K)$ on $\mathcal{M}_{K\cap K^\sharp}$, which can serve as an integral model of the pullback of the generic fiber of $\mathcal{Z}(T,\varphi_K)$ along $\pi_1$. More precisely, first assume that $\varphi_K$ is a special $v_0$-admissible Schwartz function as in (\ref{eq:testfunction}). For a locally noetherian $O_{E,(\nu)}$-scheme $S$, define $\mathcal{Z}^\flat(T,\varphi_K)(S)$ to be the groupoid of tuples $(A_0,\iota_0, \lambda_0, A,\iota, \lambda,\bar\eta^p, \bar\eta_p^{v_0}, A^\sharp,\iota^\sharp, \lambda^\sharp, \allowbreak \bar\eta^{p,\sharp}, \bar\eta_p^{v_0,\sharp}, \alpha, \mathbf{x})$, where
\begin{enumerate}
\item $(A_0,\iota_0, \lambda_0, A,\iota, \lambda,\bar\eta^p, \bar\eta_p^{v_0}, A^\sharp,\iota^\sharp, \lambda^\sharp, \bar\eta^{p,\sharp}, \bar\eta_p^{v_0,\sharp}, \alpha)\in \mathcal{M}_{K\cap K^\sharp}(S)$,
\item $(A_0,\iota_0, \lambda_0, A^\sharp,\iota^\sharp, \lambda^\sharp, \bar\eta^{p,\sharp}, \bar\eta_p^{v_0,\sharp}, \mathbf{x})\in \mathcal{Z}^\sharp(T,\varphi_K)$ (the semi-global Kudla--Rapoport cycle on $\mathcal{M}^\sharp$ defined in \S\ref{sec:semi-global-kudla}).
\end{enumerate} The functor $S\mapsto \mathcal{Z}^\flat(T,\varphi_K)(S)$ is represented by a (possibly empty) Deligne--Mumford stack which is finite and unramified over $\mathcal{M}_{K\cap K^\sharp}$ and thus defines a cycle $\mathcal{Z}^\flat(T,\varphi_K)\in \mathrm{Z}^*(\mathcal{M}_{K\cap K^\sharp})$. For a general $v_0$-admissible Schwartz function $\varphi_K\in \sS(\mathbb{V}^m_f)$, by  extending $\mathbb{C}$-linearly we obtain a cycle $\mathcal{Z}^\flat(T,\varphi_K)\in\mathrm{Z}^*(\mathcal{M}_{K\cap K^\sharp})_\mathbb{C}$.

\subsection{The local arithmetic intersection number $\Int_{T,v_0}(\varphi_K)$} \label{sec:local-arithm-inters} Assume $T\in \Herm_n(O_{F,(p)})_{>0}$ is totally positive definite. Let $t_1,\ldots,t_n$ be the diagonal entries of $T$. 
Let $\varphi_K\in \sS(\mathbb{V}^n_f)$ be a special Schwartz function as in (\ref{eq:testfunction}).

When $\Lambda_{v_0}$ is self-dual, define
\begin{equation}\label{eq:localint}
\Int_{T,\nu}(\varphi_K)\coloneqq \chi(\mathcal{Z}(T,\varphi_K), \mathcal{O}_{\mathcal{Z}(t_1,\varphi_1)} \otimes^\mathbb{L}\cdots \otimes^\mathbb{L}\mathcal{O}_{\mathcal{Z}(t_n,\varphi_n)})\cdot\log q_\nu,  
\end{equation}
where $q_\nu$ denotes the size of the residue field $k_\nu$ of $E_{\nu}$, $\mathcal{O}_{\mathcal{Z}(t_i,\varphi_i)}$ denotes the structure sheaf of the semi-global Kudla--Rapoport divisor $\mathcal{Z}(t_i,\varphi_i)$, $\otimes^\mathbb{L}$ denotes the derived tensor product of coherent sheaves on $\mathcal{M}$, and $\chi$ denotes the Euler--Poincar\'e characteristic (an alternating sum of lengths of $\mathcal{O}_{E,(\nu)}$-modules). 

When $\Lambda_{v_0}$ is almost self-dual, define
\begin{equation}\label{eq:localint2}
\Int_{T,\nu}(\varphi_K)\coloneqq \frac{1}{\deg \pi_1}\chi(\mathcal{Z}^\flat(T,\varphi_K), \mathcal{O}_{\mathcal{Z}^\flat(t_1,\varphi_1)} \otimes^\mathbb{L}\cdots \otimes^\mathbb{L}\mathcal{O}_{\mathcal{Z}^\flat(t_n,\varphi_n)})\cdot\log q_\nu,  
\end{equation} where $\deg \pi_1$ is the generic degree of the generically finite morphism $\pi_1$ (\S\ref{sec:variants-semi-global}).

Finally, when $\Lambda_{v_0}$ is self-dual or almost self-dual, define $$\Int_{T,v_0}(\varphi_K)\coloneqq \frac{1}{[E:F_0]}\cdot \sum_{\nu|v_0}\Int_{T,\nu}(\varphi_K).$$ We extend the definition of $\Int_{T,v_0}(\varphi_K)$ to a general $v_0$-admissible $\varphi_K\in \sS(\mathbb{V}^n_f)$ by  extending $\mathbb{C}$-linearly.

\subsection{The semi-global identity} Recall that we are in the situation of \S\ref{sec:semi-global-integral} (hyperspecial level) or \S\ref{sec:almost-self-dual} (almost self-dual level).

\begin{theorem}\label{thm:semi-global-identity}
 Assume $p>2$. Assume $\varphi_\KG\in \sS(\mathbb{V}^n_f)$ is $v_0$-admissible (\S\ref{sec:semi-global-kudla}). Then for any $T\in \Herm_n(O_{F,(p)})_{>0}$, $$\Int_{T,v_0}(\varphi_{\KG})q^T=c_K\cdot \pEis_{T,v_0}(\sz,\varphi_{\KG}),$$ where $c_K=\frac{(-1)^n}{\vol(K)}$ is a nonzero constant independent of $T$ and $\varphi_K$, and $\vol(K)$ is the volume of $K$ under a suitable Haar measure on $\wit G(\mathbb{A}_f)$.
\end{theorem}

\begin{proof}
  As explained in \cite[Remark 7.4]{Terstiege2013}, this follows routinely from our main Theorem \ref{thm: main} in the hyperspecial case. We briefly sketch the argument. The support of $\mathcal{Z}(T)$ lies in the supersingular locus $\mathcal{M}^\mathrm{ss}_{k_\nu}$ by the same proof of \cite[Lemma 2.21]{Kudla2014}.  We may then compute the left-hand-side via $p$-adic uniformization \S\ref{sec:p-adic-unif} as the product of the arithmetic intersection numbers on the Rapoport--Zink space $\mathcal{N}$ and a point-count. The arithmetic intersection number is equal to $W_{T,v_0}'(1, 0, \varphi_{\KG,v_0})$ up to a nonzero constant independent of $T$ and $\varphi_K$ by our main Theorem \ref{thm: main} and Remark \ref{rem:localKR1} (as $p>2$). The point-count gives a theta integral of $\varphi_{\KG,f}^{v_0}$ which can be evaluated using the Siegel--Weil formula (due to Ichino \cite[\S6]{Ichino2004} in our case) and becomes $\prod_{v\ne v_0,v\nmid \infty}W_{T,v}(1, 0, \varphi_{\KG,v})$ up to a constant independent of $T$ and $\varphi_K$. The result then follows from the factorization (\ref{eq:pEisfactorzation}) of the right-hand-side $\pEis_{T, v_0}$.

  The identity follows in a similar way from our main Theorem \ref{thm: main2} and Remark \ref{rem:localKR2} in the almost self-dual case. In fact, by the same proof of \cite[Theorem 4.13]{Sankaran2017}, it remains to check that for $\Lambda=\langle1\rangle^{n-1} \obot\langle \varpi\rangle$ an almost self-dual lattice and $L\subseteq \mathbb{V}$ any $O_F$-lattice of full rank $n$, we have the following identity
  \begin{equation}
    \label{eq:sankaran}
    \frac{\Den(\Lambda,\Lambda)}{\Den(\langle1\rangle^{n-1}, \langle1\rangle^{n-1})}=\frac{\pDen_\Lambda(L)}{\Int(L)}.
  \end{equation}
 By Theorem \ref{thm Den alm dual}, the left-hand-side of (\ref{eq:sankaran}) is equal to $\Den(\Lambda^\sharp)$. By (\ref{eq: Den L}), $\Den(\Lambda^\sharp)$ is equal to the number of self-dual lattices containing $\Lambda^\sharp$. Since $\Lambda^\sharp$ is a vertex lattice of type 2, the latter is equal to the number of isotropic lines in a 2-dimensional nondegenerate $k_F$-hermitian space, which is $q+1$ (cf. Remark \ref{rem:corr}). By Theorem \ref{thm: main2}, the right-hand-side of (\ref{eq:sankaran}) is also equal to $q+1$, and thus the desired identity (\ref{eq:sankaran}) is proved.
\end{proof}

\section{Global integral models and the global identity}

\subsection{Global integral models at minimal levels} \label{sec:glob-integr-models}In this subsection we will define a global integral model over $O_E$ of the Shimura variety $\Sh_{\KG}\tildeG$ introduced in \S \ref{sec:shimura-varieties}. We will be slightly more general than \cite[\S 5]{Rapoport2017}, allowing $F/F_0$ to be unramified at all finite places.

We consider an $O_F$-lattice $\Lambda\subseteq V$  and let  $$K_G^\circ=\{g\in G(\mathbb{A}_{f}): g(\Lambda\otimes_{O_F}\widehat{O}_F)=\Lambda\otimes_{O_F}\widehat{O}_F\}.$$ Assume that for any finite place $v$ of $F_0$ (write $p$ its residue characteristic),
\begin{enumerate}[label=(G\arabic*)]\setcounter{enumi}{-1}
\item\label{item:G0} if $p=2$, then $v$ is unramified in $F$.    
\item\label{item:G1} if $v$ is inert in $F$ and $V_v$ is split, then $\Lambda_v\subseteq V_v$ is self-dual.
\item\label{item:G2} if $v$ is inert in $F$ and $V_v$ is nonsplit, then $v$ is unramified over $p$ and $\Lambda_v\subseteq V_v$ is almost self-dual.
\item\label{item:G3} if $v$ is split in $F$, then $\Lambda_v\subseteq V_v$ is self-dual.
\item\label{item:G4} if $v$ is ramified in $F$, then $v$ is unramified over $p$ and $\Lambda_v\subseteq V_v$ is self-dual.
\end{enumerate}
We take $\KG^\circ=K_{Z^\mathbb{Q}}\times K_{G}^\circ$, where $K_{Z^\mathbb{Q}}$ is the unique maximal open compact subgroup of $Z^\mathbb{Q}(\mathbb{A}_f)$ as in \S\ref{sec:semi-global-integral}. 

Notice the assumptions \ref{item:G0}---\ref{item:G4} ensure that each finite place $v_0$ and the level $K_{G,v_0}$ belongs one of the four cases considered in \S\ref{sec:semi-global-integral}, \S\ref{sec:almost-self-dual}, \S\ref{sec:split}, \S\ref{sec:ramified}.  Define an integral $\mathcal{M}_{\KG^\circ}\tildeG$ of $\Sh_{\KG^\circ}\tildeGh$ over $O_E$ as follows. For a locally noetherian $O_E$-scheme $S$, we consider $\mathcal{M}_{\KG^\circ}\tildeG(S)$ to be the groupoid of tuples $(A_0,\iota_0,\lambda_0,A, \iota,\lambda)$, where
\begin{enumerate}
\item $A_0$ (resp. $A$) is an abelian scheme over $S$.
\item $\iota_0$ (resp. $\iota$) is an action of $O_F$ on $A_0$ (resp. $A$) satisfying the Kottwitz condition of signature $\{(1,0)_{\phi\in\Phi}\}$ (resp. signature $\{(r_\phi,r_{\bar \phi})_{\phi\in\Phi}\}$).
\item $\lambda_0$ (resp. $\lambda$) is a polarization of $A_0$ (resp. $A$) whose Rosati involution induces the automorphism given by the nontrivial Galois automorphism of $F/F_0$ via $\iota_0$ (resp. $\iota$).
\end{enumerate}
We require that the triple $(A_0, \iota_0, \lambda_0)$ satisfies \ref{item:M5}, and for any finite place $\nu$ of $E$ (write $p$ its residue characteristic), the triple $(A,\iota,\lambda)$ over $S_{O_{E,(\nu)}}$ satisfies the conditions \ref{item:M6}, \ref{item:M7}, and moreover \ref{item:MS} when $v_0$ is split in $F$ and \ref{item:MR} when $v_0$ is ramified in $F$. We may and do choose the axillary ideal $\mathfrak{a}\subseteq O_{F_0}$ in \ref{item:M5} to be divisible only by primes split in $F$.

Then the functor $S\mapsto \mathcal{M}_{\KG^\circ}\tildeG(S)$ is represented by a Deligne--Mumford stack $\mathcal{M}_{\KG^\circ}=\mathcal{M}_{\KG^\circ}\tildeG$ flat over $\Spec O_{E}$. It has isolated singularities only in ramified characteristics, and we may further obtain a regular model by blowing up (the \emph{Kr\"amer model}) which we still denote by $\mathcal{M}_K$. For each finite place $\nu$ of $E$, the base change $\mathcal{M}_{\KG^\circ,O_{E,(\nu)}}$ is canonically isomorphic to the semi-global integral models defined in \S\ref{sec:semi-global-integral}, \S\ref{sec:almost-self-dual}, \S\ref{sec:split}, \S\ref{sec:ramified}.

\subsection{Global integral models at Drinfeld levels}\label{sec:globaldrinfeld} With the same set-up as \S\ref{sec:glob-integr-models}, but now we allow Drinfeld levels at split primes. Let $\mathbf{m}=(m_v)$ be a collection of integers $m_v\ge0$ indexed by finite places $v$ of $F_0$. Further assume
\begin{enumerate}[label=(G\arabic*)]\setcounter{enumi}{4}
\item\label{item:G6} if $m_v>0$, then $v$ satisfies \ref{item:S}, and each place $\nu$ of $E$ above $v$ satisfies \ref{item:D}.
\end{enumerate}
We take $K_G^\mathbf{m}\subseteq K_G^\circ$ such that $(K_{G}^\mathbf{m})_v=(K_{G}^\circ)_v$ if $m_v=0$ and $(K_{G}^\mathbf{m})_v=(K_{G}^\circ)_v^{m_v}$ to be the principal congruence subgroup mod $\varpi_{v}^{m_v}$ if $m_v>0$. Write $\KG^\mathbf{m}=K_{Z^\mathbb{Q}}\times K_G^\mathbf{m}$. Define $\mathcal{M}_{\KG^\mathbf{m}}$ to be the normalization of $\mathcal{M}_{\KG^\circ}$ in $\Sh_{\KG^\mathbf{m}}(\wit G, h_{\wit G})$.

Then $\mathcal{M}_{\KG^\mathbf{m}}$ is a Deligne--Mumford stack finite flat over $\mathcal{M}_{\KG^\circ}$. Moreover for each finite place $\nu$ of $E$, the base change $\mathcal{M}_{\KG^\mathbf{m},O_{E,(\nu)}}$ is canonically isomorphic to the semi-global integral models defined in \S\ref{sec:semi-global-integral}, \S\ref{sec:almost-self-dual}, \S\ref{sec:split}, \S\ref{sec:drinfeld}, \S\ref{sec:ramified}. Thus $\mathcal{M}_{\KG^\mathbf{m}}$ is smooth at places over $v_0$ in \ref{item:G1}, \ref{item:G3}, semi-stable at places over $v_0$ in \ref{item:G2} when $\nu$ is unramified over $p$, and regular at places over $v_0$ in \ref{item:G4}, \ref{item:G6}. In particular, assume all places $\nu$ over $v_0$ in \ref{item:G2} are unramified over $p$, then $\mathcal{M}_{\KG^\mathbf{m}}$ is regular. When $\mathbf{m}$ is sufficiently large, $\mathcal{M}_{\KG^\mathbf{m}}$ is relatively representable over $\mathcal{M}_0^{\mathfrak{a},\xi}$.

\subsection{Global Kudla--Rapoport cycles $\mathcal{Z}(T,\varphi_K)$}\label{sec:glob-kudla-rapop-1} We continue with the same set-up as \S\ref{sec:globaldrinfeld}. From now on write $K=\KG^\mathbf{m}$ and $\mathcal{M}=\mathcal{M}_{\KG^\mathbf{m}}$ for short. Let $\varphi_K=(\varphi_i)\in \sS(\mathbb{V}^m_f)$ be $K$-invariant. Let $t_1,\ldots,t_m\in F$. Let $Z(t_i,\varphi_i)$ be the (possibly empty) Kudla--Rapoport cycle on the generic fiber of $\mathcal{M}$ (defined similarly as in \S\ref{sec:semi-global-kudla}) and let $\mathcal{Z}(t_i,\varphi_i)$ be its Zariski closure in the global integral model $\mathcal{M}$. Then we have a decomposition into the global \emph{Kudla--Rapoport cycles} $\mathcal{Z}(T,\varphi_K)$ over $\mathcal{M}$ (cf. \cite[(11.2)]{Kudla2014}), $$\mathcal{Z}(t_1,\varphi_1)\cap\cdots\cap\mathcal{Z}(t_m,\varphi_m)=\bigsqcup_{T\in \Herm_{m}(F)} \mathcal{Z}(T, \varphi_K),$$ here $\cap$ denotes taking fiber product over $\mathcal{M}$, and the indexes $T$ have diagonal entries $t_1,\ldots,t_m$.

\subsection{The arithmetic intersection number $\Int_{T}(\varphi_K)$}\label{sec:arithm-inters-numb-1}  For nonsingular $T\in \Herm_n(F)$, define  $$\Int_T(\varphi_K)\coloneqq \sum_{v} \Int_{T,v}(\varphi_K)$$ to be the sum over all finite places $v$ of $F$ of local arithmetic intersection numbers defined as in \S\ref{sec:local-arithm-inters}. By the same proof of \cite[Lemma 2.21]{Kudla2014}, this sum is nonzero only if $\Diff(T,\mathbb{V})=\{v\}$ is a singleton, and in this case $v$ is necessarily nonsplit in $F$. Hence
\begin{equation}
  \label{eq:diffgeom}
  \Int_T(\varphi_K)\ne0\Longrightarrow \Diff(T,\mathbb{V})=\{v\} \text{ and } \Int_T(\varphi_K)=\Int_{T,v}(\varphi_K).
\end{equation}

\subsection{The global Kudla--Rapoport conjecture for nonsingular Fourier coefficients}\label{sec:glob-kudla-rapop} Assume that we are in the situation of \S\ref{sec:globaldrinfeld}. We say $\varphi_K\in\sS(\mathbb{V}^m_f)$ is \emph{inert-admissible} if it is $v$-admissible at all $v$ inert in $F$ (\S\ref{sec:semi-global-kudla}).  When $\varphi_K$ is inert-admissible, the base change of the global Kudla--Rapoport cycle $\mathcal{Z}(T,\varphi_K)$ to $\Spec O_{E,(\nu)}$ above an inert prime agrees with the semi-global Kudla--Rapoport cycle defined in \S\ref{sec:semi-global-kudla}. We say a nonsingular $T\in \Herm_n(F)$ is \emph{inert} if $\Diff(T,\mathbb{V})=\{v\}$ where $v$ is inert in $F$ and not above 2. 
\begin{theorem}\label{thm:totallypositive}
  Assume $\varphi_K\in\sS(\mathbb{V}^n_f)$ is inert-admissible. Let $T\in\Herm_n(F)$ be  inert. Then $$\Int_T(\varphi_K)q^T=c_K\cdot \pEis_T(\sz,\varphi_K),$$ where $c_K=\frac{(-1)^n}{\vol(K)}$ as in Theorem \ref{thm:semi-global-identity}.
\end{theorem}

\begin{proof}
 Since $T$ is inert, we know that $T>0$, and moreover by (\ref{eq:diffgeom}) and (\ref{eq:diffana}) both sides are contributed non-trivially only by the term at $\Diff(T,\mathbb{V})=\{v\}$. Since $\varphi_K$ is inert-admissible, both sides are zero unless $T\in\Herm_{n}(O_{F,(p)})$ ($p$ the residue characteristic of $v$).  So we can apply Theorem \ref{thm:semi-global-identity} to obtain $\Int_{T,v}(\varphi_K)q^T=c_K\cdot \pEis_{T,v}(\sz, \varphi_K)$.
\end{proof}

\begin{corollary}
Kudla--Rapoport's global conjecture \cite[Conjecture 11.10]{Kudla2014} holds.
\end{corollary}

\begin{proof}
We take $F_0=\mathbb{Q}$ and $K=\KG^\circ$. We also take the axillary ideal $\mathfrak{a}$ to be trivial (see \ref{item:M5}). Then the global integral model $\mathcal{M}_{\KG^\circ}$ agrees with the moduli stack $\mathcal{M}^V$ in \cite[Proposition 2.12]{Kudla2014}. The test function $\varphi$ in \cite{Kudla2014} satisfies $\varphi_{K}=\mathbf{1}_{(\hat \Lambda)^n}$ and $\varphi_{\infty}$ is the Gaussian function, so $\varphi_K$ is inert-admissible. The assumption  $\mathrm{Diff}_0(T)=\{p\}$ with $p>2$ in \cite[Conjecture 11.10]{Kudla2014} ensures that $T$ is inert. The result then follows from Theorem \ref{thm:totallypositive}.
\end{proof}

\section{The arithmetic Siegel--Weil formula}

\subsection{Complex uniformization} Assume we are in the situation of \S\ref{sec:shimura-varieties}.  Under the decomposition (\ref{eq:tildeG}), we may identify the the $\wit G(\mathbb{R})$-conjugacy class $\{h_{\wit G}\}$ as the product $\{h_{Z^\mathbb{Q}}\}\times \prod_{\phi\in\Phi}\{h_{G,\phi}\}$. Notice $\{h_{Z^\mathbb{Q}}\}$ is a singleton as $Z^\mathbb{Q}$ is a torus, and $\{h_{G,\phi}\}$ is also a singleton for $\phi\ne\phi_0$ as $h_{G,\phi}$ is the trivial cocharacter. For $\phi=\phi_0$ the cocharacter is given by $h_{G,\phi_0}(z)=\diag\{1_{n-1}, \bar z/z\}$, and $\{h_{G,\phi_0}\}$ is the hermitian symmetric domain $$\mathcal{D}_{n-1}\cong\U(n-1,1)/(\U(n-1)\times \U(1)).$$ We may identify  $\mathcal{D}_{n-1}\subseteq \mathbb{P}(V_{\phi_0})(\mathbb{C})$ as the open subset of negative $\mathbb{C}$-lines in $V_{\phi_0}$, and $\wit G(\mathbb{R})$ acts on $\mathcal{D}_{n-1}$ via its quotient $\mathrm{PU}(V_{\phi_0})(\mathbb{R})$.  We may also identity it with the open $(n-1)$-ball $$\mathcal{D}_{n-1}\xrightarrow{\sim} \{z\in \mathbb{C}^{n-1}: |z|<1\},\quad [z_1,\ldots,z_n]\mapsto (z_1/z_n,\ldots, z_{n-1}/z_n),$$ under the standard basis of $V_{\phi_0}$. In this way we obtain a complex uniformization (via $\phi_0$),
\begin{equation}
  \label{eq:complexuniform}
  \Sh_{\KG}\tildeGh(\mathbb{C})=\wit G(\mathbb{Q})\backslash[\mathcal{D}_{n-1}\times \wit G(\mathbb{A}_f)/K].
\end{equation}

\subsection{Green currents}
Write $\mathcal{D}=\mathcal{D}_{n-1}$ for short. Let $x\in V_{\phi_0}$ be a nonzero vector. For any $z\in \mathcal{D}$, we let $x=x_z+x_{z^\perp}$ be the orthogonal decomposition with respect to $z$ (i.e., $x_z\in z$ and $x_{z^\perp}\perp z$). Let $R(x,z)=-(x_z,x_z)$. Define $$\mathcal{D}(x)=\{z\in \mathcal{D}: z\perp x\}=\{z\in \mathcal{D}: R(x,z)=0\}.$$  Then $\mathcal{D}(x)$ is nonempty if and only if $(x,x)>0$, in which case $\mathcal{D}(x)$ is an analytic divisor on $\mathcal{D}$. Define \emph{Kudla's Green function} to be $$g(x,z)=-\Ei(-2\pi R(x,z)),$$ where $\Ei(u)=-\int_{1}^\infty \frac{e^{ut}}{t}dt$  is the exponential integral. Then $g(x,-)$ is a smooth function on $\mathcal{D}\setminus\mathcal{D}(x)$ with a logarithmic singularity along $\mathcal{D}(x)$. By \cite[Proposition 4.9]{Liu2011}, it satisfies the $(1,1)$-current equation for $\mathcal{D}(x)$, $$\rd\rd^c[g(x)]+\delta_{\mathcal{D}(x)}=[\omega(x)],$$ where $\omega(x,-)=e^{2\pi (x,x)}\varphi_\mathrm{KM}(x,-)$, and $\varphi_\mathrm{KM}(-,-)\in (\sS(V_{\phi_0}) \otimes A^{1,1}(\mathcal{D}))^{\U(V_{\phi_0})(\mathbb{R})}$ is the \emph{Kudla--Millson Schwartz form} (\cite{Kudla1986}). Here we recall $\rd=\partial+\bar\partial$, $\rd^c=\frac{1}{4\pi i}(\partial-\bar\partial)$ and $\rd\rd^c=-\frac{1}{2\pi i}\partial\bar\partial$.

More generally, let $\mathbf{x}=(x_1,\ldots,x_m)\in V_{\phi_0}^m$ such that its fundamental matrix $T(\mathbf{x})=((x_i,x_j))_{1\le i,j\le m}$ is nonsingular. Define $$\mathcal{D}(\mathbf{x})=\mathcal{D}(x_1)\cap\cdots \cap\mathcal{D}(x_m),$$ which is nonempty if and only if $T(\mathbf{x})>0$. Define Kudla's Green current by taking star product $$g(\mathbf{x})\coloneqq [g(x_1)]*\cdots*[g(x_m)].$$ It satisfies the $(m,m)$-current equation for $\mathcal{D}(\mathbf{x})$, $$\rd\rd^c(g(\mathbf{x}))+\delta_{\mathcal{D}(\mathbf{x})}=[\omega(x_1)\wedge \cdots\wedge \omega(x_m)].$$ Here we recall that $$[g(x)]*[g(y)]\coloneqq [g(x)]\wedge \delta_{D(y)}+[\omega(x)]\wedge [g(y)].$$ 

\subsection{The local arithmetic Siegel--Weil formula at archimedean places}\label{sec:local-arithm-sieg} Let $T\in \Herm_m(F)$ be nonsingular. Let $\varphi_{K}\in\sS(\mathbb{V}^m_f)$ be $K$-invariant. Let $Z(T,\varphi_K)$ be the (possibly empty) Kudla--Rapoport cycle on the generic fiber $\Sh_{\KG}\tildeGh$. Then $$Z(T,\varphi_K)(\mathbb{C})=\sum_{\mathbf{x}\in \wit G(\mathbb{Q})\backslash V^m(F)\atop T(\mathbf{x})=T}\sum_{\tilde g\in \wit G_\mathbf{x}(\mathbb{A}_f)\backslash\wit G(\mathbb{A}_f)/K }\varphi_{K}(\tilde g^{-1}\mathbf{x})\cdot Z(\mathbf{x},\tilde g)_K,$$ where we define the cycle on $\Sh_K(\mathbb{C})$ via the complex uniformization (\ref{eq:complexuniform}), $$Z(\mathbf{x},\tilde g)_K=\wit G_\mathbf{x}(\mathbb{Q})\backslash [\mathcal{D}(\mathbf{x})\times \wit G_\mathbf{x}(\mathbb{A}_f)\tilde g K/K],$$ and $\wit G_\mathbf{x}\subseteq \wit G$ is the stabilizer of $\mathbf{x}$. Define a Green current for $Z(T, \varphi_K)(\mathbb{C})$ by $$g(\sy_{\phi_0}, T,\varphi_K,z,\tilde g)\coloneqq \sum_{\mathbf{x}\in V^m(F)\atop T(\mathbf{x})=T}\varphi_{K}(\tilde g^{-1}\mathbf{x})\cdot g(\mathbf{x}\cdot a,z),\quad (z,\tilde g)\in \mathcal{D}\times \wit G(\mathbb{A}_f),$$ where $a\in \GL(V_{\phi_0})\cong\GL_n(\mathbb{C})$ and $\sy_{\phi_0}=a{}^t\bar a$. Define the \emph{archimedean arithmetic intersection number} (depending on the parameter $y_{\phi_0}$) to be $$\Int_{T,\phi_0}(\sy_{\phi_0},\varphi_K)\coloneqq \frac{1}{2}\int_{\Sh_{\KG}\tildeGh(\mathbb{C})}g(\sy_{\phi_0}, T,\varphi_K).$$ Replacing the choice of $\phi_0$ by another $\phi\in\Phi$ (\S\ref{sec:shimura-varieties}) gives rise to a Shimura variety $\Sh_\KG^\phi$ conjugate to $\Sh_\KG$, associated to a hermitian space $V^\phi$ whose signature at $\phi_0,\phi$ are swapped compared to $V$. Thus we can define in the same way the archimedean intersection number for any $\phi\in\Phi$,
\begin{equation}
  \label{eq:archiheight}
  \Int_{T,\phi}(\sy_{\phi}, \varphi_K)\coloneqq \frac{1}{2}\int_{\Sh_{K}^\phi\tildeGh(\mathbb{C})}g(\sy_{\phi},T,\varphi_K).
\end{equation}

\begin{theorem}\label{thm:archimedean}
 Assume $\varphi_K\in\sS(\mathbb{V}^n_f)$ is $K$-invariant.  Let $T\in\Herm_n(F)$ be nonsingular and $\phi\in\Phi$. Then $$\Int_{T,\phi}(\sy,\varphi_K)q^T=c_K\cdot\pEis_{T,\phi}(\sz,\varphi_K),$$  where $c_K=\frac{(-1)^n}{\vol(K)}$ as in Theorem \ref{thm:semi-global-identity}.
\end{theorem}
\begin{proof}
 By the main archimedean result of \cite[Proposition 4.5, Theorem 4.17]{Liu2011} (the archimedean analogue of our main Theorem \ref{thm: main})  and the standard unfolding argument, we can express the integral \eqref{eq:archiheight} as a product involving the derivative $W_{T,\phi}'(g_\sz,0,\varphi_K)q^T$ and the product of values $\prod_{v\ne\phi}W_{T,v}(g_\sz,0,\varphi_K)$ from the Siegel--Weil formula, up to a nonzero constant independent of $T$. The result then follows from the factorization \eqref{eq:eisenfactor} of Fourier coefficients and comparing the constant with that of Theorem \ref{thm:semi-global-identity}. See the proof of \cite[Theorem 4.20]{Liu2011} and the proof in the orthogonal case \cite[Theorem 7.1]{Bruinier2018} for details. When $V$ is anisotropic (e.g., when $F_0\ne \mathbb{Q}$), the result also follows from \cite[(1.19)]{Garcia2019} for $r=p+1=n$ in the notation there.
\end{proof}

\subsection{Arithmetic degrees of Kudla--Rapoport cycles}\label{sec:arithm-degr-kudla} Let us come back to the situation of \S\ref{sec:globaldrinfeld}. Let $T\in \Herm_n(F)$ be nonsingular. Let $\varphi_{K}=(\varphi_i)\in \sS(\mathbb{V}^n_f)$ be $K$-invariant. Define the \emph{arithmetic degree} (depending on the parameter $\sy=(\sy_\phi)_{\phi\in\Phi}$)
\begin{equation}
  \label{eq:arihtmeticdegree}
  \wdeg_{T}(\sy,\varphi_K)\coloneqq \Int_T(\varphi_K)+\sum_{\phi\in\Phi}\Int_{T,\phi}(\sy_\phi, \varphi_K)
\end{equation}
 to be the sum of all nonarchimedean and archimedean intersection numbers.  Define the \emph{generating series of arithmetic degrees} of Kudla--Rapoport cycles to be $$\wdeg(\sz, \varphi_K)\coloneqq \sum_{T\in\Herm_n(F)\atop \det T\ne0}\wdeg_T(\sy,\varphi_K) q^T.$$

It is related to the usual arithmetic degree on arithmetic Chow groups as we now explain. For nonzero $t_1,\ldots,t_n\in F$, we have classes in the Gillet--Soul\'e arithmetic Chow group (with $\mathbb{C} $-coefficients) of the regular Deligne--Mumford stack $\mathcal{M}_K$ (\cite{Gillet1990,Gillet2009}), $$\widehat{\mathcal{Z}}(\sy,t_i,\varphi_i)\coloneqq (\mathcal{Z}(t_i,\varphi_i), g(\sy, t_i,\varphi_i))\in \widehat{\Ch}^1_\mathbb{C}(\mathcal{M}_K).$$ We have an arithmetic intersection product on $n$ copies of $\wCh^1_\mathbb{C}(\mathcal{M}_K)$, $$\langle\ ,\cdots,\ \rangle_\mathrm{GS}: \wCh^1_\mathbb{C}(\mathcal{M}_K)\times \cdots \times\wCh^1_\mathbb{C}(\mathcal{M}_K)\rightarrow \wCh^n_\mathbb{C}(\mathcal{M}_K),$$ and when $\mathcal{M}_K$ is proper over $O_E$, a degree map on the arithmetic Chow group of 0-cycles, $$\wdeg:\wCh^n_\mathbb{C}(\mathcal{M}_K)\rightarrow \mathbb{C}.$$ We may compose these two maps and obtain a decomposition $$\wdeg\langle \widehat{\mathcal{Z}}(\sy,t_1,\varphi_1),\cdots, \widehat{\mathcal{Z}}(\sy,t_n,\varphi_n)\rangle_\mathrm{GS}=\sum_{T}\wdeg_T(\sy,\varphi_K),$$ where the matrices $T$  have diagonal entries $t_1,\ldots,t_n$. The terms corresponding to nonsingular $T$ agree with (\ref{eq:arihtmeticdegree}), at least in the hyperspecial case at inert primes.

\subsection{The arithmetic Siegel--Weil formula when $F/F_0$ is unramified} Assume that we are in the situation of \S\ref{sec:globaldrinfeld}.

\begin{theorem}[Arithmetic Siegel--Weil formula]\label{sec:arithm-sieg-weil}
  Assume that $F/F_0$ is unramified at all finite places and split at all places above 2. Assume that $\varphi_K\in \sS(\mathbb{V}^n_f)$ is inert-admissible (\S\ref{sec:glob-kudla-rapop}) and nonsingular (\S\ref{sec:incoh-eisenst-seri}) at two places split in $F$. Then $$\wdeg(\sz, \varphi_K)=c_K\cdot \pEis(\sz,\varphi_K),$$ where $c_K=\frac{(-1)^n}{\vol(K)}$ as in Theorem \ref{thm:semi-global-identity}.
\end{theorem}

\begin{remark}
  The assumption that $F/F_0$ is unramified at all finite places implies that $F_0\ne \mathbb{Q}$ and hence the Shimura variety $\Sh_{K}\tildeGh$ is projective and the global integral model $\mathcal{M}_K$ is proper over $O_E$. Moreover, this assumption forces that the hermitian space $V$ to be nonsplit at some inert place, and thus it is necessary to allow almost self-dual level at some inert place (as we did in \ref{item:G2}).
\end{remark}

\begin{remark}
  The Schwartz function $\varphi_K$ satisfying the assumptions in Theorem \ref{sec:arithm-sieg-weil} exists for a suitable choice of $K$ since we allow arbitrary Drinfeld levels at split places.
\end{remark}

\begin{proof}
  Since $\varphi_K$ is nonsingular at two places, by \eqref{eq:twononsingular} we know that only nonsingular $T$ contributes non-trivially to the sum \eqref{eq:eisz}. For a nonsingular $T$, by \eqref{eq:diffana} we know that $\Diff(T,\mathbb{V})=\{v\}$ for $v$ nonsplit in $F$. By the assumption on $F/F_0$, we know that either $T$ is inert or $v$ is archimedean. The result then follows from Theorem \ref{thm:totallypositive} and Theorem \ref{thm:archimedean} depending on $T$ is inert or $v$ is archimedean.
\end{proof}

\bibliographystyle{alpha}
\bibliography{KR}

\end{document}